\documentclass[a4paper]{amsart}
\usepackage{amssymb}
\usepackage[polutonikogreek, english]{babel}
\usepackage{graphicx}
\usepackage{pstricks}
\usepackage{enumerate}
\theoremstyle{plain}
\newtheorem{theorem}{Theorem}[section]
\newtheorem{corollary}[theorem]{Corollary}
\newtheorem{lemma}[theorem]{Lemma}
\newtheorem{axiom}{Axiom}
\begin{document}
\title{ The Fan Theorem, its strong negation, and the determinacy of games}
\author{Wim Veldman}
\address{Institute for Mathematics, Astrophysics and Particle Physics, Faculty of Science, Radboud University Nijmegen,
Postbus 9010, 6500 GL Nijmegen, the Netherlands}
\email{Wim.Veldman@ru.nl}
\date{}

\begin{abstract}
In the context of   a weak formal theory called {Basic Intuitionistic Mathematics $\mathsf{BIM}$}, we study Brouwer's  \textit{Fan Theorem} and  a strong negation of the Fan Theorem,  \textit{Kleene's Alternative (to the Fan Theorem)}.
We prove that  the Fan Theorem is  equivalent to   \textit{contrapositions} of a number of  intuitionistically accepted axioms of countable choice and that Kleene's Alternative is equivalent to  \textit{strong negations} of these statements. We  discuss finite and infinite games and introduce a constructively useful notion of \textit{determinacy}. We prove that the Fan Theorem is equivalent to the   \textit{Intuitionistic Determinacy Theorem}. This theorem  says that every subset of Cantor space is, in our constructively meaningful sense, determinate.   Kleene's Alternative is equivalent to a strong negation of  a special case of this theorem. We also consider  a \textit{uniform intermediate value theorem} and a \textit{compactness theorem for classical propositional logic}. The Fan Theorem is equivalent to each of these theorems and  Kleene's Alternative is equivalent to strong negations of them.   We end with a note on  \textit{`stronger'} Fan Theorems. The paper is a sequel to \cite{veldman2011b}. 
\end{abstract}

\maketitle

\section{Introduction}

 \subsection{Intuitionistic reverse mathematics}\label{SS:intuitionisticreversemathematics} L.E.J.~Brouwer did not present  his intuitionistic mathematics as a formal axiomatic theory. He did not like formalism and formalization and anxiously maintained the distinction between a mathematical proof and the linguistic expression that should help us to recover the proof but may fail to do so. 
The challenge to develop formal theories coming close to Brouwer's intentions  was taken up  by A.~Heyting, G.~Gentzen, S.C.~Kleene, G.~Kreisel,  \\J.~Myhill, A.S.~Troelstra, and others.  

Given  a preferably weak formal basic theory $\Gamma$   and  a formal proof in $\Gamma$  of  a statement $T$ from some extra assumption $A$, one may ask: is there also a formal proof in $\Gamma$ of this  extra assumption  $A$ from the statement $T$?  The study of such questions, as far as they belong  to the field of classical analysis or second-order arithmetic, is called \textit{reverse mathematics}, see \cite{Simpson}. The weak basic theory  
there is $\mathsf{RCA}_0$.

\subsection{The basic theory $\mathsf{BIM}$}Our subject is \textit{intuitionistic reverse mathematics}.

 The weak basic theory we use is the two-sorted first-order intuitionistic theory $\mathsf{BIM}$ (\textit{Basic Intuitionistic Mathematics}), introduced in \cite{veldman2011b}.  The domain of discourse of $\mathsf{BIM}$ consists of two kinds of   objects: natural numbers and infinite sequences of natural numbers. The axioms express some  basic assumptions like the (full) principle of induction on the set $\omega$ of the natural numbers,  and the fact that  the  set $\omega^\omega$ of the infinite sequences of natural numbers  is closed under the  recursion-theoretic operations.
 
 The reason that we use a basic theory different in spirit from the basic theory used in classical reverse mathematics is that, in intuitionistic analysis, one prefers the notion of an infinite sequence of natural numbers as a primitive notion above the notion of a subset of the set of the natural numbers, see \cite[Section 5]{veldman2011b}.

 For the  intuitionistic mathematician, the set $\omega^\omega$ of all infinite sequences of natural numbers is {\it not}, as one sometimes says when explaining the notion of `set' that lies at the basis of   classical set theory, the result of taking together the earlier constructed and completed items that are to be the \textit{`elements'} of the set. The set $\omega^\omega$ is a realm of possibilities: it is a framework for constructing, in the future, in all kinds of possibly as yet unforeseen ways, the objects that will be called the elements of the set. There are several kinds of infinite sequences  $\alpha = \bigl(\alpha(0), \alpha(1), \ldots \bigr)$ of natural numbers.  Sometimes, $\alpha$ is the result of executing a \textit{program},  a finitely formulated {\it algorithm}. It is also possible that   $\alpha$ is the result of a more or less {\it free step-by-step construction} that is not governed by a rule formulated at the start.  

\subsection{Two interpretations} The axioms of  $\mathsf{BIM}$  hold for their \textit{intended interpretation}, the interpretation given to them by the intuitionistic mathematician.  The axioms of $\mathsf{BIM}$ also become true for her if she assumes that the elements of $\omega^\omega$ are just the    Turing-computable functions from $\omega$ to $\omega$.  Turing-computable functions may be represented by the natural number coding their program, and 
$\mathsf{BIM}$ may be seen to be  a \textit{conservative extension of first-order intuitionistic arithmetic $\mathsf{HA}$}, \textit{Heyting arithmetic}. 

The model given by the computable functions thus is the \textit{second  interpretation} of $\mathsf{BIM}$. Our study of this model is a contribution to \textit{intuitionistic recursive analysis}.

In the weak context given by $\mathsf{BIM}$ 
one may study the further assumptions of the intuitionistic mathematician. They fall  into three groups:
\begin{enumerate}
\item Axioms of Countable Choice,
\item Brouwer's Continuity Principle and the Axioms of Continuous Choice, and 
\item Brouwer's Thesis on bars in  Baire space $\omega^\omega$ and the Fan Theorem.
\end{enumerate}
The intuitionistic mathematican is prepared to argue the plausibility of  these assumptions for her  intended interpretation.

 She defends the Axioms of Countable Choice, for instance,  by explaining that  the functions promised by the axioms may be constructed step by step.  
 
She has no argument for the truth of the further assumptions under the second interpretation, where   every function is assumed to be given by an algorithm. It is not clear to her if the Axioms of Countable Choice then are true.  
 
  Brouwer's Continuity Principle and its extensions, the Axioms of Continuous Choice,  certainly fail in the second  interpretation.
  
 The Thesis on bars in $\omega^\omega$ was introduced by Brouwer for the sake of the  Fan Theorem. The Fan Theorem itself, dating from 1924, see \cite{brouwer27}, might be called \textit{the Thesis on Bars in Cantor space $\mathcal{C}$},  see \cite{veldman2006b}, \cite{veldman2008}, and \cite[Subsection 2.3]{veldman2011b}.
 
 In 1950, see \cite{kleene52}, Kleene  saw that, in our second  interpretation, also the Fan Theorem, and, a fortiori, the Thesis on bars in $\omega^\omega$,  do not hold. Actually,  a positively formulated {\it strong negation} of the Fan Theorem  becomes true.  In \cite{veldman2011b}, we called this statement 
  \textit{Kleene's Alternative (to the Fan Theorem)}.

 \subsection{Strong negations}\label{SS:strongnegations}The (strict) \textit{Fan Theorem}, $\mathbf{FT}$, see  \ref{SSS:fan}, is the statement: $$\forall \alpha [Bar_\mathcal{C}(D_\alpha) \rightarrow \exists n[ Bar_\mathcal{C}(D_{\overline\alpha n})]],$$ and \textit{Kleene's Alternative (to the Fan Theorem)}, $\mathbf{KA}$, see \ref{SSS:ka}, is the statement $$\exists \alpha [Bar_\mathcal{C}(D_\alpha) \;\wedge\;  \forall n[ \neg Bar_\mathcal{C}(D_{\overline\alpha n})]].$$ We  want to call $\mathbf{KA}$ the \textit{strong negation} $\mathbf{\neg ! FT}$ of $\mathbf{FT}$.    In general,  if we decide to call a statement $B$ the \textit{strong negation} of a statement $A$, $B$ will be a statement  more positive than the negation of $A$  that constructively implies the negation of $A$. We do not require that the statement $B$ is completely positive in the sense that the corresponding formula does not contain $\neg$ and $\rightarrow$\footnote{See, for instance, the sentences \ref{SS:omegacantor} and \ref{SS:omega01} and  Subsection \ref{SS:infinitelymanymoves}.}. We do not introduce \textit{strong negation} as a syntactical operation on formulas.  It is important to realize that, once we have understood that statement $A$ is equivalent to statement $B$, it may be the case that statements $C,D$, which have been chosen to be called the strong negations of $A,B$, respectively,  fail to be  equivalent. 
 
 If we have decided to call the formula (denoted by) $B$ the strong negation of the formula (denoted by) $A$, we will write $B=\neg!A$, but note that this notation belongs to the meta-language of $\mathsf{BIM}$. $\neg!$ is \textit{neither} a connective belonging to the language of $\mathsf{BIM}$ \textit{nor} a syntactical operation on formulas.
 
 We shall prove a number of results of the form: \begin{quote} In $\mathsf{BIM}$, $A$ is equivalent to $B$ and $\neg! B$ is equivalent to $\neg ! A$. \end{quote}
 
 When we  do so, we try to explain that the conclusions $A\rightarrow B$ and $\neg ! B\rightarrow \neg ! A$ have a common ground and that also the conclusions $B\rightarrow A$ and $\neg ! A\rightarrow \neg ! B$ have a common ground. 
 \subsection{Contrapositions or reversals} We may compare \textit{Weak K\"onig's Lemma}, $\mathbf{WKL}$, see \ref{SSS:wkl}:  $$\forall \alpha [\forall n[\neg Bar_\mathcal{C}(D_{\overline\alpha n})]\rightarrow \exists \gamma\in\mathcal{C} \forall n [ \alpha(\overline{\gamma}
 n) = 0 ]].$$ with the (strict) \textit{Fan Theorem}, $\mathbf{FT}$:$$\forall \alpha [Bar_\mathcal{C}(D_\alpha) \rightarrow \exists n[ Bar_\mathcal{C}(D_{\overline\alpha n})]].$$ We like to say that $\mathbf{WKL}$ is a \textit{reversal} or \textit{contraposition} of $\mathbf{FT}$ and also that $\mathbf{FT}$ is a \textit{reversal} or \textit{contraposition} of $\mathbf{WKL}$.  
 \\We  like to write: $\mathbf{WKL}=\overleftarrow{\mathbf{FT}}$ and $\mathbf{FT}=\overleftarrow{\mathbf{WKL}}$.
 
 \smallskip In general, if we call a  statement $B$ the \textit{contraposition} or \textit{reversal} $\overleftarrow A$ of a  statement $A$, both $A$ and $B$ will be largely positively formulated statements and the classical mathematician would think $ A$ and $B$ are equivalent, but, constructively, $ A$ and $B$ will have quite different meanings. 
 
 This is clear from the above example as $\mathbf{FT}$ is intuitionistically true (under our first interpretation) and $\mathbf{WKL}$ is false (in both interpretations).  
 
 It may happen also that both $A$ and $B$ are intuitionistically true (at least under the first interpretation), although they make different sense. An important example of this phenomenon is given by $\mathbf{\Pi}^0_1$-$\mathbf{AC}_{\omega,2}$, see \ref{SS:pi01axomega2}, and $\mathbf{\Sigma}^0_1$-$\overleftarrow{\mathbf{AC}_{\omega,2}} $, see \ref{SS:sigma01contraxomega2} and Lemma \ref{L:3resolutions}. 
 
 Note that \textit{Kleene's Alternative (to the Fan Theorem)}, $\mathbf{KA}$,
 might be called the strong negation $\neg!\mathbf{WKL}$ of  $\mathbf{WKL}$ as well as the strong negation $\neg !\mathbf{FT}$ of $\mathbf{FT}$. 
 
 We do not claim that, given a statement $A$, there always is a unique candidate for being called the \textit{contraposition} of $A$. We do not introduce \textit{contraposition} as a syntactical operation on formulas and use the term only in certain specific cases.
 It is important to realize that, once we have understood that statement $A$ is equivalent to statement $B$, it may be the case that statements $C,D$, which one would like to call contrapositions of $A,B$, respectively,  fail to be  equivalent. 
\subsection{Non-intuitionistic assumptions} The reader may wonder why we pay attention  to statements that fail in both our models, like Weak K\"onig's Lemma, $\mathbf{WKL}$, and Bishop's Omniscience Principles, $\mathbf{LPO}$, see \ref{SSS:lpo},  and $\mathbf{LLPO}$, see \ref{SSS:llpo}. Doing so, however, we come to understand that certain other statements, being equivalent, in $\mathsf{BIM}$,  to one of them, also do not make sense in either one of our two interpretations.

\subsection{Our aim} As in \cite{veldman2009b} and \cite{veldman2011b}, it is our aim, in this paper, to find  statements that are, in $\mathsf{BIM}$, equivalent to either $\mathbf{FT}$ or $\neg!\mathbf{FT}=\mathbf{KA}$.  
\subsection{The contents of the paper}Apart from this introduction, the paper contains Sections 2-13. 

Section 2 is divided into two Subsections. In Subsection 2.1 we introduce the formal system $\mathsf{BIM}$.   Subsection 2.2 lists a number of assumptions  one might study in the context of $\mathsf{BIM}$. 

In Section 3 we prove that, in $\mathsf{BIM}$,    the $\mathbf{\Sigma}^0_1$-\textit{Separation Principle} $\mathbf{\Sigma}^0_1$-$\mathbf{Sep}$ is  equivalent to  $\mathbf{WKL}$. 

In Section 4 we formulate some special cases of the  \textit{First Axiom of Choice} $\mathbf{AC}_{\omega,\omega}$, among them $\mathbf{\Pi}^0_1$-$\mathbf{AC}_{\omega,2}$, the $\mathbf{\Pi}^0_1$-\textit{Axiom of Countable Binary Choice}.

In Section 5 we introduce $\mathbf{\Sigma}^0_1$-$\overleftarrow{\mathbf{AC}_{\omega,2}}$, a \textit{contraposition} of $\mathbf{\Pi}^0_1$-$\mathbf{AC}_{\omega,2}$, and we prove: in $\mathsf{BIM}$, $\mathbf{\Sigma}^0_1$-$\overleftarrow{\mathbf{AC}_{\omega,2}}$  is   equivalent to  $\mathbf{FT}$, and  a strong negation of $\mathbf{\Sigma}^0_1$-$\overleftarrow{\mathbf{AC}_{\omega,2}}$ is  equivalent to   $\neg!\mathbf{FT}=\mathbf{KA}$.

 Section 5 thus shows that a  contraposition of $\mathbf{\Pi}^0_1$-$\mathbf{AC}_{\omega,2}$ fails in our second interpretation. This gives us no conclusion about the validity of $\mathbf{\Pi}^0_1$-$\mathbf{AC}_{\omega,2}$ itself in our second interpretation.

In Section 6 we formulate some special cases of the Second Axiom Scheme of Countable Choice $\mathbf{AC}_{\omega,\omega^\omega}$, among them $\mathbf{\Pi}^0_1$-$\mathbf{AC}_{\omega,\mathcal{C}}$, the \textit{$\mathbf{\Pi}^0_1$-Axiom  of Countable Compact Choice}.

In Section 7 we introduce $\mathbf{\Sigma}^0_1$-$\overleftarrow{\mathbf{AC}_{\omega,\mathcal{C}}}$,  a contraposition of $\mathbf{\Pi}^0_1$-$\mathbf{AC}_{\omega,\mathcal{C}}$, and we prove: in $\mathsf{BIM}$, $\mathbf{\Sigma}^0_1$-$\overleftarrow{\mathbf{AC}_{\omega,\mathcal{C}}}$  is  equivalent to $\mathbf{FT}$, and  a strong negation of $\mathbf{\Sigma}^0_1$-$\overleftarrow{\mathbf{AC}_{\omega,\mathcal{C}}}$ is  equivalent to $\neg!\mathbf{FT}$.

   In Section 8  we introduce $\mathbf{\Sigma}^0_1$-$\overleftarrow{\mathbf{AC}_{2,\mathcal{C}}}$, a contraposition of a statement provable in $\mathsf{BIM}$, to wit, the $\mathbf{\Pi}^0_1$-\textit{``axiom''} of \textit{Twofold Compact Choice}. 
   
   We prove:   in $\mathsf{BIM}+\mathbf{\Pi}^0_1$-$\mathbf{AC}_{\omega,2}$, $\mathbf{\Sigma}^0_1$-$\overleftarrow{\mathbf{AC}_{2,\mathcal{C}}}$ is  equivalent to $\mathbf{FT}$.  There is no  companion result for $\neg!\mathbf{FT}$.

   In Section 9 we consider finite and infinite games. We explain in what sense we want to call such games \textit{$I$-determinate} or \textit{$II$-determinate}. We see that $\mathbf{\Sigma}^0_1$-$\overleftarrow{\mathbf{AC}_{\omega,2}}$ can be read as the statement that certain 2-move games are  $I$-determinate. We  prove: in $\mathsf{BIM}$, $\mathbf{FT}$ is equivalent to the statement that every subset of Cantor space $\mathcal{C}$ is (weakly) $I$-determinate, and  $\neg!\mathbf{FT}$ is equivalent to the statement that there exists an open subset of $\mathcal{C}$ that positively fails to be $I$-determinate.  
   
   In Section 10 we consider   a \textit{Uniform Contrapositive Intermediate Value Theorem} $\overleftarrow{\mathbf{UIVT}}$ and we prove: in $\mathsf{BIM}$, $\mathbf{FT}$ is equivalent to  $\overleftarrow{\mathbf{UIVT}}$ and $\neg !\mathbf{FT}$ is equivalent to a strong negation of $\overleftarrow{\mathbf{UIVT}}$.
   
    In Section 11 we see that, if one formulates the  compactness theorem for classical propositional logic carefully and contrapositively, one obtains a statement that, in $\mathsf{BIM}$, is equivalent to $\mathbf{FT}$. $\neg!\mathbf{FT}$ is equivalent to a strong negation of this statement.
    
    In Section 12 we ask the reader's attention for the \textit{Approximate-Fan Theorem} $\mathbf{AppFT}$, a statement  stronger than $\mathbf{FT}$. We did so already in \cite[Subsection 10.2]{veldman2011b}, see also \cite{veldman2011c}. $\mathbf{AppFT}$ is studied further in \cite{veldman2011c}.
   
   Section 13 contains a list of defined notions. This Section  may be used as a reference by the reader of the preceding Sections.
   
   \smallskip 
    I  thank the referees of earlier versions of this paper. Their thorough comments and  criticisms were very useful and led to numerous improvements. The last referee deserves special mention. He  wrote three very detailed reports and found many points where I expressed myself wrongly or confusingly.
    
     I also thank U.~Kohlenbach for providing some references and for making me repair a mistaken observation in Section \ref{S:ivt}.

\section{The formal system $\mathsf{BIM}$}\label{S:BIM}

 \subsection{The basic axioms}
 $\mathsf{BIM}$, introduced in \cite[Section 6]{veldman2011b},  has two kinds of variables, \textit{numerical} variables $m,n,p,\ldots$,
 whose intended range is the set  $\omega$ of the natural numbers, and \textit{function}
 variables $\alpha, \beta, \gamma, \ldots$, whose intended range is the set
 $\omega^\omega$ of
 all infinite sequences of natural numbers. There is a numerical constant $0$. There are five unary function
 constants: $Id$, a name for the identity function, $\underline{0}$, a name for the zero function,  $S$, a name for
 the successor function, and $K$, $L$, names for the projection
 functions. There is one binary function symbol $J$, a name for the pairing
 function on $\omega$. From these symbols \textit{numerical terms} are formed in the usual 
 way. The basic terms are the numerical variables and the numerical constant and other terms are obtained from earlier constructed terms by the use of
 a function symbol. The function constants $\mathit{Id}$, $\underline 0$, $S$, $K$ and $L$ and the function variables are  the only 
 \textit{function 
 terms}. 
 
 $\mathsf{BIM}$ has two equality symbols, $=_0$ and $=_1$. The first symbol may be placed
 between numerical terms and the second one between function terms.
 When confusion seems improbable we simply write $=$ and not $=_0$ or $=_1$. The usual axioms for equality are part of $\mathsf{BIM}$.  A
 \textit{basic formula} is an equality between numerical terms or an equality
 between function terms. A \textit{basic formula in the strict sense} is an
 equality between numerical terms. We obtain the formulas of the theory from the
 basic formulas by using the connectives, the numerical quantifiers and the
 function quantifiers. 
 
 The logic of the theory is  intuitionistic
 logic.
 
 Our first axiom is
 \begin{axiom}[Axiom of Extensionality]\label{ax:ext}
 
 \[\forall \alpha \forall \beta [ \alpha =_1 \beta  \leftrightarrow \forall n [
 \alpha(n) =_0 \beta(n) ] ] \]
 \end{axiom}
 The Axiom of Extensionality guarantees that every formula will be provably
 equivalent to a formula built up by means of connectives and quantifiers from
 basic formulas in the strict sense.
 
 The second axiom is \textit{the axiom on  the unary function constants} $Id$, $\underline 0$, $S$, $K,L$,   \textit{and the binary function constant} $J$.
 
\begin{axiom}\label{ax:const}\[\forall n[Id(n) = n]\;\wedge\]
 \[ \forall n [ \neg(S(n) = 0) ] \;\wedge\; \forall m \forall n [ S(m) = S(n)
 \rightarrow m = n ]\;\wedge \]\[ 
\forall n [ \underline{0}(n) = 0]\;\wedge\]\[ \forall m \forall n [ K\bigl(J(m,n)\bigr) = m \;\wedge\; L\bigl(J(m,n)\bigr) = n \;\wedge\; J\bigl(K(n),L(n)\bigr)=n]\]
 \end{axiom}

 Thanks to the presence of the pairing function we may treat binary, ternary and
 other non-unary operations on $\omega$ as unary functions. ``$\alpha(m,n,p)$'',
 for instance, will be an abbreviation of ``$\alpha\bigl(J(J(m,n),p)\bigr)$''.
  
  We also introduce the following notation: for each $n$, $n':= K(n)$ and $n'':= L(n)$, and, for all $m,n,$
  $(m,n) := J(m,n)$.
  The last part of  Axiom \ref{ax:const} now reads as follows: $\forall m\forall n[(m,n)'=m \;\wedge\; (m,n)''=n \;\wedge\; (n',n'')=n]$.
  
 \medskip
 The next axiom\footnote{A referee made us see that this Axiom 3, as formulated in \cite{veldman2011b}, is a little bit too weak.} asks for the \textit{closure of the set $\omega^\omega$ under composition, pairing, primitive recursion and unbounded search}.
 \begin{axiom}\label{ax:recop}
 \[\forall \alpha \forall \beta \exists \gamma \forall n [ \gamma(n) =
 \alpha \bigl(\beta(n)\bigr) ]\;\wedge\]
 \[\forall \alpha \forall \beta  \exists \gamma\forall n[\gamma(n) = \bigl(\alpha(n), \beta(n)\bigr)]\;\wedge\]
   \[\forall \alpha \forall \beta \exists \gamma \forall m \forall n [ \gamma(m,0)
 = \alpha(m) \wedge \gamma \bigl(m,S(n)\bigr) = \beta\bigl(m,n,\gamma(m,n)\bigr) ]\;\wedge\]
 \[\forall \alpha  [ \forall n \exists m [ \alpha(n,m) = 0 ] \rightarrow
 \exists \gamma \forall n [ \alpha\bigl(n, \gamma(n)\bigr) = 0 ] ]\]
 
 \end{axiom}

 We also need the \textit{Axiom Scheme of Induction}: 
 \begin{axiom}\label{ax:ind}

 \[(P(0) \wedge \forall n [ P(n) \rightarrow P\bigl(S(n)\bigr)]) \rightarrow
 \forall n [P(n)]\]
 \end{axiom}
  
  Instances of this axiom scheme are obtained by 
  substituting a formula \\$\phi=\phi(m_0, m_1, \ldots, m_{k-1}, \alpha_0,\alpha_1, \ldots, \alpha_{l-1}, n)$ for $P$ and taking the universal closure of the resulting formula. The reader should understand the further axiom schemes we are to mention in this paper in the same way.
    
    \smallskip
  The axioms 1-4, together with the usual axioms for equality,  define the system $\mathsf{BIM}$.
  
  Note  $\mathsf{BIM}$ has the full induction scheme whereas $\mathsf{RCA}_0$ has only $\Sigma^0_1$-induction, see \cite[Definition II.1.5]{Simpson}, a fact that is indicated by the suffix $_0$.  We did not study the possibility of restricting induction likewise and we do not answer the question if our results might have been obtained in a system that probably should be called $\mathsf{BIM}_0$.

We form a conservative extension of $\mathsf{BIM}$ by adding  constants for all primitive recursive functions and relations 
 and making  their defining
 equations into axioms.  Primitive recursive relations are present via their characteristic functions. `$x<y$', for instance, will be short for: `$\chi_<(x,y)\neq 0$'. 
 Somewhat loosely, we also denote this conservative extension of $\mathsf{BIM}$ by the acronym $\mathsf{BIM}$ although one might decide to use the acronym $\mathsf{BIM}^+$, see \cite{vafeiadou2018}. 
 
 \smallskip
 $\mathsf{BIM}$ 
  may be compared  to the system \textbf{H} introduced in
  \cite{Howard-Kreisel} and to the system $\mathbf{EL}$
 occurring in \cite{Troelstra-van Dalen} and to the system $\mathbf{IRA}$, proposed by J.R.~Moschovakis and G.~Vafeiadou, see \cite{moschovakisvafeiadou12}. A precise proof of the fact that $\mathsf{BIM}$ and these systems are essentially equivalent may be found in \cite{vafeiadou2018}.

\subsection{Possible further assumptions}\label{SS:fa}

\subsubsection{First Axiom Scheme of Countable Choice, $\mathbf{AC}_{\omega,\omega} (=\mathbf{AC}_{0,0})$:}
$$\forall n\exists m[R(n,m)] \rightarrow \exists \gamma\forall n[R\bigl(n,\gamma(n)\bigr)].$$

The intuitionistic mathematician accepts $\mathbf{AC_{\omega,\omega}}$. If $\forall n\exists m[R(n,m)]$, she builds the promised $\gamma$ step by step, first choosing $\gamma(0)$ such that $R\bigl(0,\gamma(0)\bigr)$, then choosing $\gamma(1)$ such that $R\bigl(1,\gamma(1)\bigr)$, and so on. In her view, there is no need to give a finite description or algorithm that determines the infinitely many values of $\gamma$ at one stroke.
\subsubsection{Second Axiom Scheme of Countable Choice, $\mathbf{AC}_{\omega,\omega^\omega}=\mathbf{AC}_{0,1}$:}
$$\forall n \exists \gamma[R(n,\gamma)]\rightarrow  \exists \gamma\forall n[R(n, \gamma^{\upharpoonright n} )].\footnote{For the notation $\gamma^{\upharpoonright n}$ see Subsection \ref{SS:subs}.}$$

\smallskip
  The intuitionistic mathematician  accepts $\mathbf{AC}_{\omega, \omega^\omega}$. If $\forall n \exists \gamma[R(n, \gamma)]$, she first starts a project for
  building $\gamma^{\upharpoonright 0}$ with the property $R(0,\gamma^{\upharpoonright 0})$ and determines $\gamma^{\upharpoonright 0}(0)$, she
  then starts a project for building $\gamma^{\upharpoonright 1}$ with the property $R(1,\gamma^{\upharpoonright 1})$
  and determines $\gamma^{\upharpoonright 1}(0)$, and also, continuing the project started earlier, $\gamma^{\upharpoonright 0}(1)$, she then starts a project for building
  $\gamma^{\upharpoonright 2}$ with the property $R(2,\gamma^{\upharpoonright 2})$ and determines $\gamma^{\upharpoonright 2}(0)$ and also, continuing the projects started earlier,
  $\gamma^{\upharpoonright 1}(1)$ and $\gamma^{\upharpoonright 0}(2)$, $\ldots$.

\subsubsection{The Fan Theorem as an Axiom Scheme, $\mathbf{FAN}$:} $$\forall\beta[\bigl(Fan(\beta)\;\wedge\;Bar_{\mathcal{F}_\beta}(B)\bigr)\rightarrow \exists a[D_a\subseteq B\;\wedge\;Bar_{\mathcal{F}_\beta}(D_a)]].$$

\subsubsection{The (strict) Fan Theorem, $\mathbf{FT}$:}\label{SSS:fan}
\[\forall \alpha[\mathit{Bar}_\mathcal{C}(D_\alpha)
  \rightarrow \exists m[\mathit{Bar}_\mathcal{C}(D_{\overline \alpha m})]], \;or, \;equivalently,\]
\[\forall \beta[Explfan(\beta)\rightarrow\forall \alpha[\mathit{Bar}_\mathcal{\mathcal{F}_\beta}(D_\alpha)
  \rightarrow \exists m[\mathit{Bar}_\mathcal{\mathcal{F}_\beta}(D_{\overline \alpha m})]]]\;\mathit{or,}\;\mathit{equivalently},\]
\[\forall \beta[Explfan(\beta)\rightarrow\forall \alpha[\mathit{Bar}_\mathcal{\mathcal{F}_\beta}(D_\alpha)
  \rightarrow \exists m\forall \gamma\in\mathcal{F}_\beta\exists n\le m[\overline \gamma n \in D_\alpha]]].\]
  
  \begin{theorem}\label{T:{thinfan}}$\mathsf{BIM}\vdash \mathbf{FT}\leftrightarrow \\\forall \alpha[\bigl(Thinbar_\mathcal{C}(D_\alpha)\;\wedge\; D_\alpha\subseteq Bin\bigr)\rightarrow \exists n\forall m>n[m\notin D_\alpha]]$. \end{theorem}
  \begin{proof} The proof is left to the reader. \end{proof}
  \subsubsection{The strengthened (strict) Fan Theorem, $\mathbf{FT}^+$:}\label{SSS:fan+}

\[\forall \beta[Fan(\beta)\rightarrow\forall \alpha[\mathit{Bar}_\mathcal{\mathcal{F}_\beta}(D_\alpha)
  \rightarrow \exists m[\mathit{Bar}_\mathcal{\mathcal{F}_\beta}(D_{\overline \alpha m})]]]\;\mathit{or,}\;\mathit{equivalently},\]
\[\forall \beta[Fan(\beta)\rightarrow\forall \alpha[\mathit{Bar}_\mathcal{\mathcal{F}_\beta}(D_\alpha)
  \rightarrow \exists m\forall \gamma\in\mathcal{F}_\beta\exists n\le m[\overline \gamma n \in D_\alpha]]].\]
\begin{theorem}\label{T:{thinfan+}}$\mathsf{BIM}\vdash \mathbf{FT}^+\leftrightarrow \\\forall \beta[Fan(\beta)\rightarrow\forall\alpha[\bigl(Thinbar_{\mathcal{F}_\beta}(D_\alpha)\;\wedge\; \forall s \in D_\alpha[\beta(s)=0]\bigr)\rightarrow \exists n\forall m>n[m\notin D_\alpha]]]$. \end{theorem}
\begin{proof} The proof is left to the reader. \end{proof}
 \smallskip Note: $\mathsf{BIM}+\mathbf{FAN}\vdash \mathbf{FT}^+$. 
 
 In this paper,  we restrict attention to $\mathbf{FT}$.  For the equivalence of the three formulations of $\mathbf{FT}$, see \cite[Section 2.3, Theorem 9.6(ii) and Theorem 7.7(v)]{veldman2011b}, and  also \cite{troelstra} and \cite[vol. I, Chapter 4, Section 7.5]{Troelstra-van Dalen}\footnote{At the latter two places, no attention is given to the r\^ole of  \textit{Weak}-$\mathbf{\Pi}^0_1$-$\mathbf{AC}_{\omega, \omega}$, see \ref{SS:weakac}.}. 
 
 In \ref{SS:weakac}, we introduce   \textit{Weak}-$\mathbf{\Pi}^0_1$-$\mathbf{AC}_{\omega, \omega}$, a special case of $\mathbf{AC}_{\omega, \omega}$. 
 
  $\mathsf{BIM}+$\textit{Weak}-$\mathbf{\Pi}^0_1$-$\mathbf{AC}_{\omega, \omega}\vdash \mathbf{FT}\rightarrow \mathbf{FT}^+$, see Theorem \ref{T:fanfan+}. 

 \subsubsection{Brouwer's Thesis: Bar Induction as an Axiom Scheme, $\mathbf{BARIND}$:}\label{SS:barind}
 \[\bigl(Bar_{\omega^\omega}(B) \;\wedge\; \forall s[s\in B\rightarrow s\in C]\;\wedge\;\forall s[\forall n[s\ast\langle n \rangle \in C]\leftrightarrow s\in C]\bigr)\rightarrow \langle \;\rangle \in C.\]
 Brouwer derived $\mathbf{FAN}$ from  $\mathbf{\textbf{BARIND}}$, 
see \cite[Subsections 2.2 and 2.3]{veldman2011b}. We repeat the proof, in order to prepare the reader for Theorem \ref{T:almftscheme}.

\begin{theorem}\label{T:ftscheme} $\mathsf{BIM}+\mathbf{BARIND}\vdash \mathbf{FAN}$. \end{theorem}

\begin{proof} Let $\beta$ be given such that $Fan(\beta)$ and $\beta(\langle\;\rangle)=0$.\footnote{If $\beta(\langle\;\rangle)\neq 0$, then $\mathcal{F}_\beta=\emptyset$ and there is nothing to prove.}  Assume $Bar_{\mathcal{F}_\beta}(B)$.

Define $B':=B\cup\{s\mid \beta(s)\neq 0\}$.

We claim: $Bar_{\omega^\omega}(B')$. We prove this claim as follows.
Let $\gamma$ be given. Define $\gamma^\ast$ such that, for each $n$, if $\beta\bigl(\overline\gamma(n+1)\bigr)=0$, then $\gamma^\ast(n)=\gamma(n)$, and, if not, then $\gamma^\ast(n)=\mu p[\beta(\overline \gamma n\ast \langle p\rangle)=0]$. Note: $\gamma^\ast \in \mathcal{F}_\beta$ and find $n$ such that $\overline{\gamma^\ast}n\in B$. \textit{Either}  $\overline \gamma n=\overline{\gamma^\ast}n$ and  $\overline \gamma n \in B$ \textit{or} $\overline \gamma n\neq\overline{\gamma^\ast}n$ and $\beta(\overline\gamma n)\neq 0$. In both cases, $\overline \gamma n \in B'$.

We thus see: $\forall \gamma \exists n[\overline \gamma n \in B']$, i.e. $Bar_{\omega^\omega}(B')$. 

\smallskip Let $C$ be the set of all $s$ such that\\ either: $\beta(s)\neq 0$ or: $\beta(s) = 0$ and $\exists a[ D_a \subseteq B \;\wedge\; Bar_{\mathcal{F}_\beta\cap s}(D_a)]$. 

\smallskip
 For every $s$, if $\beta(s) =0$ and $s\in B$, define $a:=\overline{\underline 0}s\ast\langle 1 \rangle$ and note: $\{s\}=D_a\subseteq B$    and $Bar_{\mathcal{F}_\beta\cap s}(D_a)$. Conclude: $B\subseteq C$. 

\smallskip Now let $s$ be given such that $\forall m[s\ast\langle m \rangle \in C]$. \\Find $n$ such that $\forall m\ge n[\beta(s\ast\langle m\rangle) \neq 0]$. Find $b$ in $\omega^n$ such that \\$\forall m<n[\beta(s\ast\langle m \rangle)=0 \rightarrow \bigl(D_{b(m)}\subseteq B\;\wedge \;Bar_{\mathcal{F}_\beta\cap s\ast\langle m \rangle}(D_{b(m)})\bigr)]$ and \\$\forall m<n[\beta(s\ast\langle m \rangle)\neq 0\rightarrow b(m)=\langle\;\rangle]$. \\Find $a$ such that $p:=length(a)=\max_{m<n}length\bigl(b(m)\bigr)$ and \\$\forall t<p[a(t)\neq 0\leftrightarrow \exists m<n[\bigl(b(m)\bigr)(t)\neq 0]]$. \\Note: $D_a\subseteq B$ and $Bar_{\mathcal{F}_\beta\cap s}(D_a)$, and  conclude: $s\in C$.

We thus see: $\forall s[\forall m[s\ast\langle m \rangle\in C]\rightarrow s \in C]$. 

Obviously: $\forall s\forall m[s\in C\rightarrow s\ast\langle m \rangle \in C]$. 

Using $\mathbf{BARIND}$, we conclude: $\langle\;\rangle\in C$, i.e. $\exists a[ D_a \subseteq B \;\wedge\; Bar_{\mathcal{F}_\beta}(D_a)]$. 
\end{proof}

 \subsubsection{Church's Thesis, $\mathbf{CT}$:}
\[\exists \tau \exists \psi \forall \alpha \exists e\forall n \exists z[z=\mu i[\tau(e, n, i )\neq 0] \;\wedge\;  \psi(z) = \alpha(n)].\]

Kleene has shown  that $\mathbf{CT}$ is true in the model of $\mathsf{BIM}$ given by the computable functions.  He provided \textit{K\'almar-elementary} functions $\tau,\psi$ satisfying the above conditions. 

 Note that our formulation of $\mathbf{CT}$ is cautious and somewhat weaker than the usual one. We do not require that the set $\{(e,n,z)\mid\tau(e,n,z) \neq 0\}$ coincides with Kleene's set $T$, but only ask that the set behaves appropriately. A similar `abstract' approach to Church's Thesis has been advocated by F.~Richman, see \cite{richman83} and \cite[Chapter 3, Section 1]{bridgesrichman}.

\subsubsection{Kleene's Alternative (to the Fan Theorem), $\neg!\mathbf{FT}$:}\label{SSS:ka} 
 \[ \exists \alpha[ \mathit{Bar}_\mathcal{C}(D_\alpha)\;\wedge\; \forall m [ \neg \mathit{Bar}_\mathcal{C}(D_{\overline \alpha m})]].\]

The following result is due to Kleene. 

\begin{theorem}\label{T:ctka} $\mathsf{BIM} + \mathbf{CT} \vdash \neg!\mathbf{FT}$. \end{theorem} \begin{proof}

 Let $\tau, \psi$ be as in $\mathbf{CT}$. Define $\alpha$  such that, for all $m$, for all $c$ in $Bin_m$, $$\alpha(c) = 0\leftrightarrow \forall e<m\forall z < m[  z=\mu i[ \tau( e, e, i)\neq 0]\rightarrow c(e) = 1-\psi(z)].$$ 

 Let $\gamma$ in $\mathcal{C}$ be given.   Find $e$ such that, for all $n$, $\gamma(n) = \psi(\mu i[\tau(  e,n,i)\neq 0])$.  Define  $z=\mu i[\tau( e, e, i )\neq 0]$.  Note: $ \gamma(e)=\psi(z)$. Define $m := \max(e, z) +1$ and note:  $\alpha(\overline \gamma m) \neq 0$.
 Conclude: $\mathit{Bar}_\mathcal{C}(D_\alpha)$. 
 
 Let $m$ be given. Find $c$ in $\mathit{Bin}_m$ such that \\ $\forall e < m\forall z <m[z=\mu i[\tau( e, e, i)\neq 0]\rightarrow c(e) = 1-\psi(z)]$. Note:  $\forall n \le m[\alpha(\overline c n) = 0]$. Define $\gamma:=c\ast\underline 0$ and  note:  $\forall n [c\ast\underline{\overline 0} n > m]$ and conclude: $\neg \exists k[\overline\gamma k<m \;\wedge\;\alpha(\overline \gamma k)\neq 0]$ and: $\neg\mathit{Bar}_\mathcal{C}(D_{\overline \alpha m})$. \end{proof}
 
 The above proof may be found in \cite[vol. I, Chapter 4, Subsection 7.6]{Troelstra-van Dalen}.  In \cite[Section 3]{veldman2011b} one finds two proofs.
 
 We do not know the answer to the question if $\mathsf{BIM}+\neg!\mathbf{FT}\vdash \mathbf{CT}$. 

\subsubsection{Brouwer's (unrestricted) Continuity Principle as an Axiom Scheme, $\mathbf{BCP}$:}
 $$\forall \alpha \exists n[\alpha R n]\rightarrow\forall \alpha \exists m \exists n \forall \beta[ \overline \alpha m  \sqsubset \beta \rightarrow \beta R n].$$

 \smallskip
Brouwer used this principle for the first time in 1918, see \cite[Section 1, page 13]{brouwer18}. The intuitionistic mathematician believes  the axiom to be plausible. 

She argues as follows. 

If $\forall \alpha\exists n[\alpha Rn]$, I must be able, given any $\alpha$, to find effectively an $n$ as promised, \textit{also if the values of $\alpha$ are disclosed one by one and nothing is told about a law governing the development of $\alpha$ as a whole}. I will decide upon $n$ at some point of time and, at that point of time, only finitely many values of $\alpha$, say \\$\alpha(0), \alpha(1), \ldots,\alpha(m-1)$, will be known to me. The number $n$ will satisfy any infinite sequence that is a continuation of $\alpha(0), \alpha(1), \ldots,\alpha(m-1)$.

The Continuity Principle is revolutionary and changes one's mathematical perspective, see \cite{veldman2001}.

The classical mathematician may ask for a consistency proof for  $\mathbf{BCP}$.  Kleene proved, using realizability methods,  that his formal system $\mathbf{FIM}$ for intuitionistic analysis,  actually an extension of $\mathsf{BIM}+\mathbf{FT}+\mathbf{BCP}$,  is (simply) consistent, see \cite[Chapter II, Subsection 9.2]{Kleene-Vesley}. Kleene's consistency proof should convince  both the classical mathematician and the constructive mathematician, should the  latter accept $\mathbf{BARIND}$ but be plagued by doubts concerning $\mathbf{BCP}$.

 It is not difficult to see that  $\mathsf{BIM}+ \mathbf{CT} + \mathbf{BCP}$ is inconsistent, as it implies $\forall \alpha \exists m\forall \beta[\overline \beta m = \overline \alpha m \rightarrow \beta = \alpha]$, see \cite[vol. I, Chapter 4, Theorem 6.7]{Troelstra-van Dalen}. 
 
 \smallskip
 The next two axioms strengthen $\mathbf{BCP}$. Kleene's consistency proof extends to these stronger forms of the Continuity Principle.
 
 \subsubsection{The First Axiom Scheme of Continuous Choice, $\mathbf{AC}_{\omega^\omega,\omega} (=\mathbf{AC}_{1,0}$):}
 $$\forall \alpha \exists n[\alpha R n]\rightarrow\exists \varphi:\omega^\omega\rightarrow\omega\forall \alpha  [\alpha R \varphi(\alpha)].$$
 
  \subsubsection{The Second Axiom Scheme of Continuous Choice, $\mathbf{AC}_{\omega^\omega,\omega^\omega} (=\mathbf{AC}_{1,1}$):}
 $$\forall \alpha \exists \beta[\alpha R \beta]\rightarrow\exists \varphi:\omega^\omega\rightarrow\omega^\omega\forall \alpha  [\alpha R \varphi|\alpha].$$
 \subsubsection{Weak K\"onig's Lemma, $\mathbf{WKL}$:}\label{SSS:wkl} \[\forall \alpha [\forall n\exists a \in  Bin_n\forall m\le n[\alpha(\overline a n)=0]\rightarrow \exists \gamma\in\mathcal{C} \forall n [ \alpha(\overline{\gamma}
 n) = 0 ]],\] \[or,\; equivalently, \;\forall \alpha [\forall n[\neg Bar_\mathcal{C}(D_{\overline\alpha n})]\rightarrow \exists \gamma\in\mathcal{C} \forall n [ \alpha(\overline{\gamma}
 n) = 0 ]].\] 
 
$\mathbf{WKL}$, as a classical theorem, dates from 1927, see \cite{koenig}. It is a contraposition of $\mathbf{FT}$ and,  from a classical point of view, the two are equivalent.
 The following result is due to U.~Kohlenbach, see \cite{kohlenbach3} and also \cite{ishihara2}.
 \begin{theorem}\label{T:wklft} $\mathsf{BIM}\vdash \mathbf{WKL} \rightarrow \mathbf{FT}$. \end{theorem} \begin{proof} Assume $\mathbf{WKL}$. 
 
 Let $\alpha$ be given such that  $\mathit{Bar}_\mathcal{C}(D_\alpha)$. We intend to prove: $\exists n[Bar_\mathcal{C}(D_{\overline \alpha n})]$. 
 
 Define $\alpha^\ast$ such that $\forall s \in Bin[\alpha^\ast(s)= 0 \leftrightarrow \forall t\sqsubseteq s[\alpha(t)= 0]]$.
 
 Define $\alpha^{\ast\ast}$ such that, for all $n$, for all $s$ in $Bin_n$, $\alpha^{\ast\ast}(s)= 0$ if and only if \\\textit{either} $\alpha^\ast(s)= 0$ \textit{or} $\neg\exists t \in Bin_n[\alpha^\ast(t)=0]$. 

Clearly,  for each $n$, there exists $s$ in $Bin_n$ such that $\alpha^{\ast\ast}(s)=0$.

  Applying $\mathbf{WKL}$, find $\gamma$ in $\mathcal{C}$ such that $\forall n[\alpha^{\ast\ast}(\overline \gamma n)=0]$.
    
    Find $n$ such that $\alpha(\overline \gamma n)\neq 0$.  
    
    Conclude: $\neg \exists t\in Bin_n[\alpha^\ast(t)=0]$ and: $\forall \delta\in\mathcal{C}\exists j\le n[\alpha(\overline \delta j)\neq 0]$ and:\\  $\exists m[Bar_\mathcal{C}(D_{\overline\alpha m})]$.

    We thus see: $\forall \alpha[Bar_\mathcal{C}(D_\alpha)\rightarrow \exists m[Bar_\mathcal{C}(D_{\overline \alpha m})]]$,  i.e. $\mathbf{FT}$. 
  \end{proof}
\subsubsection{Weak K\"onig's Lemma with a uniqueness condition, $\mathbf{WKL!}$:}

\[\forall \alpha[\bigl(\forall m[\neg Bar_\mathcal{C}(D_{\overline \alpha m})]\;\wedge\;\forall \gamma\in \mathcal{C}\forall\delta\in\mathcal{C}[\gamma\perp\delta\rightarrow \exists m[\alpha(\overline\gamma m)\neq 0\;\vee\;\alpha(\overline\delta m)\neq 0]]\bigr)\rightarrow\]\[\exists \gamma\forall n[\alpha(\overline \gamma n)=0]].\] 
 The next two theorems, apart from being of interest in themselves, are useful for the discussion in Subsubsection \ref{SSS:bickford}. 
  \begin{theorem}\label{T:wkl!ft} $\mathsf{BIM}\vdash \mathbf{WKL!}\rightarrow \mathbf{FT}$. \end{theorem}
 
\begin{proof} \footnote{The equivalence of $\mathbf{FT}$ and $\mathbf{WKL!}$ is a result due to J. Berger and H. Ishihara, see \cite{bergerishihara}. Another proof has been given by   H. Schwichtenberg, see \cite{schwichtenberg}. }
 Assume $\mathbf{WKL!}$. 
 
 Let $\alpha$ be given such that $Bar_\mathcal{C}(D_\alpha)$. We intend to prove: $\exists m[Bar_\mathcal{C}(D_{\overline\alpha m})]$. 
 
 If $\alpha(0)=\alpha(\langle \;\rangle)\neq 0$, then $D_{\overline\alpha 1}=\{\langle \;\rangle\}$ and $Bar_\mathcal{C}(D_{\overline \alpha 1})$, and we are done. 
  
   We now assume $\alpha(\langle\;\rangle)=0$, that is: $\langle\;\rangle\notin D_\alpha$. 
 
 Define $\alpha^\ast$ such that $\forall s \in Bin[\alpha^\ast(s)= 0 \leftrightarrow \forall t\sqsubseteq s[\alpha(t)= 0]]$.

 Define $\alpha^{\ast\ast}$ such that $\alpha^{\ast\ast}(\langle\;\rangle)=0$ and, for all $n>0$, for all $s$ in $Bin_n$,\\$\alpha^{\ast\ast}(s)= 0$ if and only if \textit{either} $\alpha^\ast(s)= 0$ and $\forall t\in Bin_n[\alpha^\ast(t)= 0\rightarrow s\le_{lex}t]$ \textit{or} $\neg\exists t \in Bin_n[\alpha^\ast(t)=0]$ and $\exists t \in Bin_{n-1}[\alpha^{\ast\ast}(t)=0 \;\wedge\; s = t\ast\langle 0\rangle]$. 
 
 Using induction, one proves that,  for each $n$, there is exactly one $s$ in $Bin_n$ such that $\alpha^{\ast\ast}(s)=0$. 
   
   Let $\gamma, \delta$ in $\mathcal{C}$ be given such that $\gamma \perp \delta$.\\ Find $n$ such that $\overline \gamma n\perp\overline \delta n$ and note: $\alpha^{\ast\ast}(\overline \gamma n)\neq 0\;\vee\; \alpha^{\ast\ast}(\overline\delta n)\neq 0$. 
   
    We thus see: $\forall \gamma \in \mathcal{C}\forall \delta\in \mathcal{C}[\gamma \perp \delta \rightarrow \exists n[\alpha^{\ast\ast}(\overline \gamma n)\neq 0\;\vee\; \alpha^{\ast\ast}(\overline\delta n)\neq 0]]$.\\Applying $\mathbf{WKL!}$, find $\gamma$ in $\mathcal{C}$ such that $\forall n[\alpha^{\ast\ast}(\overline \gamma n)=0]$.
    
    Find $n$ such that $\alpha(\overline \gamma n)\neq 0$.  
    
    Conclude: $\neg \exists t\in Bin_n[\alpha^\ast(t)=0]$ and: $\forall \delta\in\mathcal{C}\exists j\le n[\alpha(\overline \delta j)\neq 0]$ and  \\$\exists m[Bar_\mathcal{C}(D_{\overline\alpha m})]$.

    We thus see: $\forall \alpha[Bar_\mathcal{C}(D_\alpha)\rightarrow \exists m[Bar_\mathcal{C}(D_{\overline \alpha m})]]$,  i.e. $\mathbf{FT}$. 
  \end{proof}
  \begin{theorem}\label{T:ftwkl!} $\mathsf{BIM}\vdash \mathbf{FT}\rightarrow \mathbf{WKL!}$. \end{theorem}

   \begin{proof} Assume $\mathbf{FT}$. Let $\alpha$ be given such that $\forall n [\neg Bar_\mathcal{C}(D_{\overline\alpha n})]]$ and \\$\forall \gamma \in \mathcal{C}\forall \delta \in \mathcal{C}[\gamma \perp \delta\rightarrow \exists n[\alpha(\overline \gamma n)
   \neq 0\;\vee\;\alpha(\overline\delta n)\neq 0]]$. \\We intend to prove: $\exists \gamma \in \mathcal{C}\forall n[\alpha(\overline \gamma n)=0]$.\\Define $\alpha^\ast$ such that $\forall s \in Bin[\alpha^\ast(s)=0\leftrightarrow \forall t\sqsubseteq s[\alpha(t)=0]]$. 
  
  \smallskip 
   Let $s$ in $Bin$ be given. \\Note: $\forall \gamma\in \mathcal{C}\forall \delta\in \mathcal{C}\exists k[\alpha(s\ast\langle 0\rangle\ast\overline \gamma k)\neq 0\;\vee\;\alpha(s\ast\langle 1\rangle\ast\overline \delta k)\neq 0]$. 
   
   Applying $\mathbf{FT}$, find $m$ such that  \\$\forall \gamma\in \mathcal{C}\forall \delta\in \mathcal{C}\exists k\le m[\alpha(s\ast\langle 0\rangle\ast\overline \gamma k)\neq 0\;\vee\;\alpha(s\ast\langle 1\rangle\ast\overline \delta k)\neq 0]$.
   
   Conclude: $\forall c\in Bin_m\forall d\in Bin_m[\alpha^\ast(s\ast\langle 0\rangle\ast c)\neq 0\;\vee\;\alpha^\ast(s\ast\langle 1\rangle\ast d)\neq 0]$.
   
   Conclude: $\forall c\in Bin_m[\alpha^\ast(s\ast\langle 0\rangle\ast c)\neq 0]\;\vee\;\forall d\in Bin_m[\alpha^\ast(s\ast\langle 1\rangle\ast d)\neq 0]$.
   
   \medskip This last step is justified by the fact that $\mathsf{BIM}$ proves the following scheme\footnote{This follows from our Lemma \ref{L:comb0}.}:
    \\ $\forall n[\forall k<n\forall l<n[A(k)\;\vee\;B(l)]\rightarrow(\forall k<n[A(k)]\vee\forall l<n[B(l)])]$.

   \medskip
   Now define $\zeta$ such that, for every $s$ in $Bin$, \\$\zeta(s)=\mu k[\forall c\in Bin_k[\alpha^\ast(s\ast\langle 0\rangle\ast c)\neq 0]\;\vee\;\forall d\in Bin_k[\alpha^\ast(s\ast\langle 1\rangle\ast d)\neq 0]$. 
   
   Then define $\delta$ in $\mathcal{C}$ such that, for every $s$ in $Bin$, \\$\forall c\in Bin_{\zeta(s)}[\alpha^\ast (s\ast\langle 1-\delta(s)\rangle \ast c)\neq 0]$.
   
   Finally, define $\gamma$ such that $\forall n[\gamma(n)=\delta(\overline \gamma n)]$. 
   
   One may prove by induction, using  $\forall n[ \neg Bar_\mathcal{C}(D_{\overline \alpha n})]$: \\for each $n$, $\forall m\ge n \exists s\in Bin_m[\overline\gamma n\sqsubseteq s \;\wedge\;\alpha^\ast(s) = 0]$. \
   
   Conclude: $\forall n[\alpha^\ast(\overline \gamma n)=0]$ and $\forall n[\alpha(\overline \gamma n)=0]$.
   
  \smallskip We thus see: for each $\alpha$, if $\forall n[\neg Bar_\mathcal{C}(D_{\overline \alpha n})]$ and \\$\forall \gamma \in \mathcal{C}\forall \delta \in \mathcal{C}[\gamma\perp\delta\rightarrow \exists n[\alpha(\overline\gamma n)\neq 0\;\vee\;\alpha(\overline \delta n)\neq 0]]$, then $\exists \gamma\in \mathcal{C}\forall n[\alpha(\overline \gamma n)=0]$,\\ i.e. $\mathbf{WKL!}$. \end{proof}

\subsubsection{The Lesser Limited Principle of Omniscience, $\mathbf{LLPO}$:}\label{SSS:llpo}\footnote{\smallskip $\mu n[P(n)]=k\;\;\leftrightarrow \;\;\bigl(P(k)\;\wedge\;\forall n<k[\neg P(n)]\bigr)$. Therefore: \\$\mu n[\alpha(n)\neq 0] \neq k \;\leftrightarrow \;\bigl(\alpha(k) = 0 \;\vee \;\exists i<k[\alpha(i)\neq 0]\bigr)$. } \[\forall \alpha\exists i<2\forall p[2p+i\neq \mu n[\alpha(n)\neq 0]].\]
\begin{theorem}\label{T:nonllpo} $\mathsf{BIM}\vdash \mathbf{BCP} \rightarrow \neg \mathbf{LLPO}$. \end{theorem}
\begin{proof} Assume $\mathbf{LLPO}$. 
\\Using $\mathbf{BCP}$, find $m$,$i$ such that $i<2$ and   $\forall \alpha[\overline{\underline 0}m\sqsubset \alpha\rightarrow \forall p[2p+i\neq\mu n[\alpha(n)\neq 0]]]$.
 \\Define $\alpha:= \overline{\underline 0}(2m+i)\ast \langle 1 \rangle \ast \underline 0$ and note: $\overline{\underline 0} m\sqsubset \alpha$ and $2m+i=\mu n[\alpha(n)\neq 0]$.  Contradiction. \end{proof}

\begin{theorem}\label{T:wklllpo} $\mathsf{BIM} \vdash \mathbf{WKL} \rightarrow \mathbf{LLPO}$. \end{theorem}

\begin{proof} Assume $\mathbf{WKL}$. Let $\alpha$ be given. Define $\beta$ such that, for every $s$, $\beta(s) = 0$ if and only if $\exists q\exists i <2[s=\overline{\underline i}q \;\wedge\;\forall p[2p+i<n \rightarrow 2p+i\neq\mu n[\alpha(n)\neq 0]]]$.   Note: $\forall m\exists i<2\forall q\le m[\beta(\overline{\underline i} q)=0]$. Using $\mathbf{WKL}$, find $\gamma$ such that $\forall n[\beta(\overline \gamma n)=0]$. Define $i:=\gamma(0)$ and conclude: $\gamma = \underline i$ and   $\forall p[ 2p+i\neq\mu n[\alpha(n)\neq 0]]$. \end{proof}

Using a weak axiom of countable choice, one may also prove $\mathbf{LLPO} \rightarrow \mathbf{WKL}$, see Theorem \ref{T:acllpowkl}.

\subsubsection{The Limited Principle of Omniscience, $\mathbf{LPO}$:}\label{SSS:lpo} $$\forall \alpha[\exists n[\alpha(n) \neq 0] \;\vee\; \forall n[\alpha(n) = 0]].$$
 
 $\mathbf{LPO}$ and $\mathbf{LLPO}$ were introduced by E. Bishop, see \cite[Chapter 1, Section 1]{bridgesrichman}. The following result is not difficult and well-known, see  \cite[Section 2.6]{mandelkern} and  \cite[Theorem 3.1]{akama}.
\begin{theorem}\label{T:lpollpo} $\mathsf{BIM} \vdash \mathbf{LPO} \rightarrow \mathbf{LLPO}$. \end{theorem}

\begin{proof} Let $\alpha$ be given.   Apply $\mathbf{LPO}$ and distinguish two cases.

\textit{Case (1)}. $\exists n[\alpha(n) \neq 0]$.  Find $q,k$ such that $k<2$  and $ 2q+k=\mu n[\alpha(n)\neq 0]$. Conclude:   $\forall p[2p+1-k\neq\mu n[\alpha(n)\neq 0]]$.

\textit{Case (2)}. $\forall n[\alpha(n) = 0]$. Then $\forall k<2\forall p[2p+k\neq\mu n[\alpha(n)\neq 0]]$.  \end{proof}

The following Lemma shows that, in $\mathsf{BIM}+\mathbf{BCP}$, not every closed subset of $\omega^\omega$ is a spread, see also 
 \cite[Theorem 2.10 (vi)]{veldman2019}. 
 
 \begin{lemma}\label{L:closedspread} $\mathsf{BIM}\vdash \forall \beta \exists \gamma[Spr(\gamma)\;\wedge\; \mathcal{F}_\beta=\mathcal{F}_\gamma]\rightarrow \mathbf{LPO}$. \end{lemma}
 
 \begin{proof} Let $\alpha$ be given. \\Define $\beta$ such that $\beta(0)=0$ and $\forall n \forall s[\beta(\langle n \rangle \ast s)=0 \leftrightarrow \alpha(n)\neq 0]$. \\Assume we find $\gamma$ such that $Spr(\gamma)$ and $\mathcal{F}_\beta=\mathcal{F}_\gamma$.\\ If $\gamma(0)=0$, then $\exists n[\alpha(n)\neq 0]$ and, if $\gamma(0)\neq 0$, then $\forall n[\alpha(n)=0]$. \end{proof}
 
 \subsubsection{\it Markov's principle, $\mathbf{MP}$:}\label{SSS:markov} $$\forall \alpha[\neg\neg\exists n[\alpha(n)\neq 0]\rightarrow \exists n[\alpha(n)\neq 0]].$$
For some discussion of this principle, see \cite[Volume I, Chapter 4, Section 5]{Troelstra-van Dalen}. In this paper, the principle figures only in Subsection \ref{SS:finiteness}. 
 
 \medskip
We would like to make a philosophical observation. For a constructive mathematician, the assumptions $\mathbf{LPO}$, $\mathbf{WKL}$, $\mathbf{LLPO}$ and $\mathbf{MP}$ make no sense, as he does not know a situation in which these assumptions are true. Theorem \ref{T:wklllpo}: $\mathbf{WKL} \rightarrow \mathbf{LLPO}$ concludes something which is never true from something which is never true. Nevertheless, the proof of Theorem \ref{T:wklllpo} makes sense. It shows us how to find, given any $\alpha$, a suitable $\beta$ such that if $\beta$  has the $\mathbf{WKL}$-property, then $\alpha$ has the $\mathbf{LLPO}$-property. This part of the argument does not use the assumption that every $\beta$ has the $\mathbf{WKL}$-property. Theorems \ref{T:wklft} and \ref{T:lpollpo} deserve a similar comment.

The reader may find more information on the axioms of intuitionistic analysis in
 \cite{brouwer27}, \cite{brouwer54}, \cite{heyting}, \cite{Kleene-Vesley}, \cite{Howard-Kreisel}, \cite{Troelstra-van Dalen} and \cite{veldman2021}.
 
\section{The $\mathbf{\Sigma}^0_1$-separation principle}\label{S:sepprinc}
\subsection{}In classical reverse mathematics, weak K\"onig's Lemma is equivalent to a principle called 
 \textit{ $\Sigma^0_1$-separation}, see \cite[Lemma IV.4.4]{Simpson}.
 
 We call $X\subseteq\omega$ \textit{enumerable} or $\mathbf{\Sigma}^0_1$ if and only if $\exists \alpha[X= E_\alpha]$.
 
 The following statement, formulated  in the intuitionistic language of $\mathsf{BIM}$, comes close to the just-mentioned classical  principle.

\smallskip
$\mathbf{\Sigma}^0_1$-\textit{separation principle}, $\mathbf{\Sigma}_1^0$-$\mathbf{Sep}$:
$$\forall \alpha[\neg \exists n\forall i<2[(n,i)\in E_\alpha ]\rightarrow\exists\gamma\forall n[  \bigl(n,\gamma(n)\bigr)\notin E_\alpha]].$$

The next Theorem seems  to confirm the just-mentioned classical result.
\begin{theorem}\label{T:sep=wkl}  $\mathsf{BIM}\vdash \mathbf{WKL} \leftrightarrow \mathbf{\Sigma}^0_1$-$\mathbf{Sep}$.

\end{theorem}

\begin{proof} 
(i) Assume $\mathbf{WKL}$. We prove: $\mathbf{\Sigma}^0_1$-$\mathbf{Sep}$.

Let $\alpha$ be given such that $\neg\exists n\forall i<2[(n,i)\in E_\alpha]$. 

Define $\beta$  such that
$\forall n \forall a\in Bin_n[\beta(a) = 0\leftrightarrow\forall m <n[\bigl(m,a(m)\bigr)\notin E_{\overline \alpha n} ]]
$.

\smallskip

 Let $n$ be given.  Note: $\forall m<n\exists i<2[(m,i)\notin E_{\overline \alpha n}]$ and find $a$ in $\mathit{Bin}_n$ such that $\forall m<n[\bigl(m,a(m)\bigr)\notin E_{\overline \alpha n}]$.  

Conclude: $\forall j \le m [\beta(\overline a j) = 0]$.

We thus see: $\forall n\exists a\in Bin_n\forall m\le n[\beta(\overline a n)=0]$.  

Using $\mathbf{WKL}$, find $\gamma$ in $\mathcal{C}$ such that  $\forall n[\beta(\overline \gamma n) = 0]$. 

Conclude: $\forall n \forall m<n[\bigl(m, \gamma(m)\bigr)\notin E_{\overline \alpha n}]$ and:   $\forall n[\bigl(n,\gamma(n)\bigr) \notin  E_{\alpha}]$. 

We thus see: $\forall \alpha[\neg \exists n\forall i<2[(n,i)\in E_\alpha]\rightarrow \exists \gamma\forall n[\bigl(n, \gamma(n)\bigr)\notin E_\alpha]$, i.e. $\mathbf{\Sigma}^0_1$-$\mathbf{Sep}$.

\medskip
 (ii) Assume $\mathbf{\Sigma}^0_1$-$\mathbf{Sep}$. We prove: $\mathbf{WKL}$.

Let $\alpha$ be given such that $\forall m [\neg Bar_\mathcal{C}(D_{\overline \alpha m})]$.

 Define $\beta$ such that, for both $i<2$, for all $n$, for all  $s$, \\\textit{if} $s\in \mathit{Bin}$ and    $Bar_{\mathcal{C}\cap s\ast\langle i \rangle}(D_{\overline \alpha n})$ and $ \neg Bar_{\mathcal{C}\cap s \ast\langle 1-i \rangle} (D_{\overline \alpha n})$, \textit{then} $\beta(n,s) = (s,i)+1$, and, \textit{if not}, \textit{then} $\beta(n,s) =0$.
 
 Note: $E_\beta$ is the set of all $s\ast\langle  i\rangle$ such that $  s \in Bin$ and  $ i<2$ and \\ $\exists n[Bar_{\mathcal{C}\cap s\ast\langle i \rangle}(D_{\overline \alpha n})\;\wedge\;\neg Bar_{\mathcal{C}\cap s\ast\langle 1-i \rangle}(D_{\overline \alpha n})]$.
 
 Note: for all $s$ in $Bin$, for all $i<2$,\\ if $s\ast \langle i \rangle\notin E_\beta$, then $\forall n[Bar_{\mathcal{C}\cap s\ast\langle i \rangle}(D_{\overline \alpha n})\rightarrow Bar_{\mathcal{C}\cap s\ast\langle 1-i \rangle}(D_{\overline \alpha n})]$.

Note: $\neg \exists s\forall i <2[(s,i)\in E_\beta]$.

\smallskip
Using $\mathbf{\Sigma}^0_1$-$\mathbf{Sep}$, find  $\gamma$ in $\mathcal{C}$ such that  $\forall s[\bigl(s,\gamma(s)\bigr) \notin E_{\beta}]$. 

Define $\delta$ in $\mathcal{C}$ by: $\forall n[\delta(n)= \gamma(\overline{\delta}n)]$.

\smallskip 
Suppose we find $n$ such that $\alpha(\overline {\delta}n) \neq 0$.  \\Define $q:=\overline \delta n +1$ and note: $\overline \delta n \in D_{\overline \alpha q}$, that is: $Bar_{\mathcal{C}\cap\overline \delta n}(D_{\overline \alpha q})$. 
\\We  prove, using backwards induction:  $\forall j\le n[Bar_{\mathcal{C}\cap\overline \delta j}(D_{\overline \alpha q})]$.
\\Our starting point is: $Bar_{\mathcal{C}\cap\overline \delta n}(D_{\overline \alpha q})$.
\\Now suppose $j+1 \le n$ and $Bar_{\mathcal{C}\cap\overline \delta (j+1)}(D_{\overline \alpha q})$. Note: $\overline\delta (j+1)=\overline\delta j\ast\langle \gamma(\overline \delta j)\rangle$. \\Also: $\bigl(\overline \delta j, \gamma(\overline \delta j)\bigr)\notin E_\beta$ and, therefore: $Bar_{\mathcal{C}\cap \overline \delta j\ast\langle 1-\gamma(\overline\delta(j)\rangle} (D_{\overline \alpha q})$ and:  $Bar_{\mathcal{C}\cap \overline \delta j} (D_{\overline \alpha q})$.
 \\This completes the proof of the induction step.
 \\After $n$ steps we reach the conclusion: $Bar_\mathcal{C}(D_{\overline \alpha q})$. \\This contradicts the assumption: $\forall m[\neg Bar_\mathcal{C}(D_{\overline \alpha m})]$.
 \\Conclude: $\forall n[\alpha(\overline{\delta}n) = 0]$.
 
 \smallskip
 We thus see: $\forall \alpha[\forall n[\neg Bar_\mathcal{C}(D_{\overline\alpha n})]\rightarrow \exists \delta\in \mathcal{C}\forall n[\alpha(\overline \delta n)=0]]$, i.e. $\mathbf{WKL}$.  
\end{proof}

Theorem \ref{T:sep=wkl} shows that $\mathbf{\Sigma}^0_1$-$\mathbf{Sep}$, like $\mathbf{WKL}$,  is not constructive, see Theorem \ref{T:wklllpo}. It is not true in intuitionistic analysis and it also fails
 in the model of $\mathsf{BIM}$ given by the recursive functions.

\section{$\mathbf{AC}_{\omega,\omega}$, some special cases}

 The following restricted version of $\mathbf{AC_{\omega,\omega}}$ is  provable in  $\mathsf{BIM}$  as it is a   consequence of Axiom \ref{ax:recop}, see Section \ref{S:BIM}. 
 
\subsection{}\textit{Minimal Axiom of Countable Choice}, $\mathbf{\Delta}^0_1$-$\mathbf{AC_{\omega,\omega}}$:
 \[\forall \alpha [ \forall n \exists m [ \alpha(n,m) = 0 ] \rightarrow
 \exists \gamma \forall n [ \alpha\bigl(n, \gamma(n)\bigr) = 0  ] ].\]

\subsection{}\textit{Axiom Scheme of Countable Unique Choice}, $\mathbf{AC}_{\omega,\omega}!=\mathbf{AC}_{0,0}!$: \[\forall n\exists!m[R(n,m)]\rightarrow\exists\gamma\forall n[R\bigl(n,\gamma(n)\bigr)].\]
where `$\forall n\exists!m[R(n,m)]$' abbreviates `$\forall n\exists m[R(n,m)\;\wedge\;\forall p[R(n,p)\rightarrow m=p]]$'. 

$\mathbf{AC}_{\omega,\omega}!$ is not a theorem of $\mathsf{BIM}$, see Subsection \ref{SSS:countbincunpr} and \cite{vafeiadou2018}.
 \subsection{}$\mathbf{\Sigma}^0_1$-\textit{First Axiom of Countable Choice}, $\mathbf{\Sigma}^0_1$-$\mathbf{AC}_{\omega,\omega}$:                                                                                                                                                                
\[\forall \alpha[\forall n \exists m[(n,m) \in E_{\alpha}] \rightarrow \exists \gamma \forall n[\bigl(n,\gamma(n)\bigr) \in E_{\alpha}]].\]
 \begin{theorem} $\mathsf{BIM}\vdash \mathbf{\Sigma}^0_1$-$\mathbf{AC}_{\omega,\omega}$. \end{theorem}\label{T:acoosigma01} \begin{proof}
Assume $\forall n \exists m[(n,m) \in E_{\alpha}]$. Then $\forall n \exists m \exists p[\alpha(p) = (n,m)+1]$.
Find $\delta$ such that $\forall n[\delta(n)=\mu q[\alpha(q')=(n, q'')+1]]$.
Define $\gamma$ such that $\forall n[\gamma(n) = \delta''(n)]$ and note:
$\forall n[\bigl(n,\gamma(n)\bigr) \in E_{\alpha}]$.\end{proof}

\smallskip
We call $X\subseteq\omega$  \textit{co-enumerable} or $\mathbf{\Pi}^0_1$ if and only if there exists $\alpha$ such that
$X= \omega\setminus E_\alpha$.

 \subsection{}$\mathbf{\Pi}_1^0$-\textit{First Axiom of Countable Choice}, $\mathbf{\Pi}^0_1$-$\mathbf{AC_{\omega,\omega}}$:
 \[\forall \alpha [ \forall n \exists m[(n, m) \notin E_{\alpha}] \rightarrow \exists \gamma \forall n [\bigl(n,\gamma(n)\bigr) \notin E_{\alpha}]].\] 

 $\mathbf{\Pi}^0_1$-$\mathbf{AC}_{\omega,\omega}$ is unprovable in $\mathsf{BIM}$, see Subsection \ref{SSS:countbincunpr}.
 In \cite[Section 6]{veldman2011b}, we introduced the following special case of this axiom. 
 \subsection{}\textit{Weak} $\mathbf{\Pi}^0_1$-\textit{First Axiom of Countable Choice}\label{SS:weakac}, \textit{Weak-}$\mathbf{\Pi}^0_1$-$\mathbf{AC}_{\omega,\omega}$:\[\forall \alpha[\forall m \exists n \forall p \ge n[\alpha(m,p) \neq 0] \rightarrow \exists \gamma \forall m \forall p \ge \gamma(m)[\alpha(m,p) \neq 0]].\]
\textit{Weak-}$\mathbf{\Pi}^0_1$-$\mathbf{AC}_{\omega,\omega}$ follows from  $\mathbf{AC}_{\omega,\omega}!=\mathbf{AC}_{0,0}!$. 

 \textit{Weak-}$\mathbf{\Pi}^0_1$-$\mathbf{AC}_{\omega,\omega}$ is a special case of the axiom scheme $AC_{0,0}^m$ introduced in \cite[Subsection 3.1]{moschovakisvafeiadou12}. We suspect that, in $\mathsf{BIM}$, \textit{Weak-}$\mathbf{\Pi}^0_1$-$\mathbf{AC}_{\omega,\omega}$ does not imply $\mathbf{\Pi}^0_1$-$\mathbf{AC}_{\omega,\omega}$, but we have no proof.
\begin{theorem}\label{T:fanfan+}\begin{enumerate}[\upshape (i)]\item   $\mathsf{BIM}+$\textit{Weak-}$\mathbf{\Pi}^0_1$-$\mathbf{AC}_{\omega,\omega}\vdash \forall \beta[Fan(\beta)\rightarrow Explfan(\beta)]$. \item $\mathsf{BIM}+$\textit{Weak-}$\mathbf{\Pi}^0_1$-$\mathbf{AC}_{\omega,\omega}\vdash \mathbf{FT}\rightarrow\mathbf{FT}^+$.\end{enumerate}\end{theorem} \begin{proof} The proof is left to the reader. \end{proof}
\smallskip One may also study  statements one obtains from $\mathbf{AC}_{\omega,\omega}$ by limiting the number of alternatives one has at each choice. 
 
\subsection{} \textit{  Axiom Scheme of Countable Binary Choice}, $\mathbf{AC}_{\omega,2}$:
  $$\forall n\exists m<2[R(n,m)] \rightarrow \exists \gamma \in \mathcal{C} \forall n[R\bigl(n,\gamma(n)\bigr)].$$

\smallskip
Here is a  restricted version of $\mathbf{AC}_{\omega,2}$:

\subsection{}\label{SS:pi01axomega2} $\mathbf{\Pi}^0_1$-\textit{Axiom of Countable Binary 
Choice}, $\mathbf{\Pi}^0_1$-$\mathbf{AC}_{\omega,2}$:
\[\forall \alpha  [\forall n\exists m <2[(n,m) \notin E_{\alpha} ] \rightarrow 
   \exists \gamma \in \mathcal{C}\forall n[\bigl(n,\gamma(n)\bigr) \notin E_{\alpha}]].\]

A result related to the following Theorem has been proven by U.~Kohlenbach, see \cite[Theorem 3]{kohlenbach}. A similar result is  mentioned  in \cite[Subsection 2.2]{akama}.

\begin{theorem}\label{T:acllpowkl} $\mathsf{BIM}\vdash(\mathbf{\Pi}^0_1$-$\mathbf{AC}_{\omega,2}\;\wedge\; \mathbf{LLPO}) \leftrightarrow \mathbf{WKL}$.\end{theorem}
 \begin{proof} (i)  Assume, in $\mathsf{BIM}$, $\mathbf{\Pi}^0_1$-$\mathbf{AC}_{\omega,2}$ and $\mathbf{LLPO}$. It suffices to prove: $\mathbf{\Sigma}^0_1$-$\mathbf{Sep}$, as, according to Theorem \ref{T:sep=wkl},  $\mathsf{BIM}\vdash \mathbf{\Sigma}^0_1$-$\mathbf{Sep}\leftrightarrow \mathbf{WKL}$.
 
 Let $\alpha$ be given such that  $\forall n\neg\forall m<2[(n,m)\in E_\alpha]$. Let $n$ be given.  \\Define $\beta$ such that $\forall q\forall i<2[\beta(2q+i) \neq 0\leftrightarrow q=\mu p[\alpha(p)=(n,i)+1]]$.   \\Apply $\mathbf{LLPO}$ and find $i<2$ such that  $\forall q[2q+i\neq\mu m[\beta(m)\neq 0]]$. \\Assume $(n,i)\in E_\alpha$. Then $(n, 1-i)\notin E_\alpha$ and $\neg \exists p[\alpha(p)=(n, 1-i)+1]$ and $\forall q[\beta(2q+1-i)=0]$. Find $q:=\mu p[\alpha(p)= (n,i)+1]$. Note $\beta(2q+i)\neq 0$ and $\forall m<2q+i[\beta(m)=0]$, so $2q+i=\mu m[\beta(m)\neq 0]$. Contradiction.  \\ Therefore, $(n,i)\notin E_\alpha$. 
 
 Conclude: $\forall n\exists i<2[(n,i)\notin E_\alpha]$. 
 
  Apply $\mathbf{\Pi}^0_1$-$\mathbf{AC}_{\omega,2}$ and find $\gamma$ in $\mathcal{C}$ such that $\forall n[\bigl(n,\gamma(n)\bigr)\notin E_\alpha]$.
  
   Conclude:   $\forall\alpha[\forall n[\neg \forall m<2[(n,m)\in E_\alpha]\rightarrow \exists\gamma\in\mathcal{C}\forall n[\bigl(n,\gamma(n)\bigr))\notin E_\alpha]]$, i.e.: $\mathbf{\Sigma}^0_1$-$\mathbf{Sep}$. 
   
   \smallskip (ii)
   Note that $\neg(P\;\wedge\;Q)\leftrightarrow \neg\neg(\neg P\;\vee\;\neg Q)$ is a valid scheme of intuitionistic logic, and, therefore:  \\$\mathsf{BIM}\vdash \mathbf{\Sigma}^0_1$-$\mathbf{Sep}\leftrightarrow\bigl(\forall \alpha[\forall n\neg\neg \exists m < 2[(n,m) \notin E_{\alpha} ] \rightarrow \exists \gamma \in \mathcal{C}\forall n[ \bigl(n,\gamma(n)\bigr) \notin E_{\alpha}]]\bigr).$
Conclude:  $\mathsf{BIM}\vdash\mathbf{\Sigma}^0_1$-$\mathbf{Sep}\rightarrow \mathbf{\Pi}^0_1$-$\mathbf{AC}_{\omega,2}$. \\ The conclusion $\mathsf{BIM}\vdash \mathbf{WKL}\rightarrow \mathbf{LLPO}$ has been drawn in Theorem \ref{T:wklllpo}. \end{proof} From a constructive point of view, $\mathbf{\Sigma}^0_1$-$\mathbf{Sep}$, or equivalently, $\mathbf{WKL}$, is an axiom of countable choice that is formulated too strongly.

   \subsection{$\mathbf{AC}_{\omega,2}!$ and $\mathbf{\Pi}^0_1$-$\mathbf{AC}_{\omega,2}$ are unprovable in $\mathsf{BIM}$}\label{SSS:countbincunpr}$\;$\\
 Note that the theory $\mathsf{BIM} +\mathbf{CT}$ may be translated  into intuitionistic arithmetic $\mathsf{HA}$, by interpreting functions from $\omega$ to $\omega$ as indices of total computable functions.
 The {\it negative translation} due to G\"odel and Gentzen, see \cite[vol. I, Ch. 3, Subsection 3.4]{Troelstra-van Dalen},  shows that first order classical (Peano) arithmetic $\mathsf{PA}$, the theory that results from $\mathsf{HA}$  by adding the axiom scheme $X\;\vee\neg X$, is consistent. It follows that also the theory $\mathsf{BIM}+\mathbf{CT}$  remains consistent upon adding the axiom scheme $X \;\vee\;\neg X$. 
 
 \subsubsection{} Note that the theory $\mathsf{BIM}+\mathbf{CT}+\; X\vee\neg X\;+\mathbf{AC}_{\omega,2}!$ is inconsistent. The argument is as follows. \\ Let $\tau, \psi$ be as in $\mathbf{CT}$.
 Define $H:=\{n\mid \exists z[\tau(n,n,z)\neq 0]\}$. \\Using classical logic, conclude: $\forall n\exists! i<2[i=0\leftrightarrow n\in H]$. \\Using $\mathbf{AC}_{\omega, 2}!$, find $\alpha$ such that $\forall n[\alpha(n)\neq 0\leftrightarrow n \in H]$. \\Define $\beta$ such that, for each $n$, if $\alpha(n)\neq 0$, then $\beta(n) = \psi(\mu z[\tau(n,n,z)\neq 0])+1$. \\Using $\mathbf{CT}$, find $n_0$ such that $\forall n[\beta (n)=\psi(\mu z[\tau(n_0, n,z)\neq 0])]$.\\ Note: $\alpha(n_0)\neq 0$ and: $\beta(n_0)=\beta(n_0)+1$. Contradiction. 
 
 \smallskip Conclude: If $\mathsf{HA}$ is consistent, then $\mathbf{AC}_{\omega, 2}!$ is not derivable in $\mathsf{BIM}$.
 
 \subsubsection{}  Theorem \ref{T:acllpowkl} implies: $\mathsf{BIM} + \neg!\mathbf{FT} + \mathbf{LLPO} + \mathbf{\Pi}^0_1$-$\mathbf{AC}_{\omega,2}$ is not consistent, as $\neg!\mathbf{FT}$ contradicts $\mathbf{WKL}$. \\Conclude: $\mathsf{BIM} + \neg!\mathbf{FT} + \mathbf{LLPO}\vdash \neg \mathbf{\Pi}^0_1$-$\mathbf{AC}_{\omega,2}$.
   
 On the other hand, $\mathsf{BIM} + \neg!\mathbf{FT} + \mathbf{LLPO}$ is consistent, as it is a subsystem of $\mathsf{BIM} + \mathbf{CT}+ \;X\vee\neg X$. It follows that $\mathbf{\Pi}^0_1$-$\mathbf{AC}_{\omega,2}$ is not derivable in $\mathsf{BIM}+\neg!\mathbf{FT}+\mathbf{LLPO}$ and, a fortiori,  not derivable in $\mathsf{BIM}$. The stronger axiom  $\mathbf{\Pi}^0_1$-$\mathbf{AC}_{\omega,\omega}$  and  the even stronger axiom
 $\mathbf{\Pi}^0_1$-$\mathbf{AC}_{\omega,\omega^\omega}$, to be introduced in   Section 6,   also are not derivable in $\mathsf{BIM}$. 
 
 From Theorem \ref{T:acllpowkl}, we also conclude $\mathsf{BIM} + \neg!\mathbf{FT}+ \mathbf{\Pi}^0_1$-$\mathbf{AC}_{\omega,2}\vdash \neg \mathbf{LLPO}$.
 
  \smallskip In the context of $\mathsf{HA}$, 
\textit{Church's Thesis}   is  sometimes introduced as an  axiom scheme, $\mathsf{CT}_0$, see \cite[vol. I, Ch. 4, Sect. 3]{Troelstra-van Dalen}:
\[\forall n \exists m[nAm] \rightarrow \exists e \forall n \exists z[T(\langle e,n,z\rangle) \;\wedge\; \forall i<z[\neg T(\langle e, n, i \rangle)] \;\wedge \; n A\bigl(U(z)\bigr)]. \footnote{$T$ is Kleene's $T$-predicate and $U$ his result-extracting function.}\]  It is not difficult to see that   $\mathsf{BIM}+ \mathbf{CT}+\mathbf{AC}_{\omega,\omega}$ and   also $\mathsf{BIM}+ \mathbf{CT}+\mathbf{AC}_{\omega,\omega^\omega}$ translate into $\mathsf{HA}+\mathsf{CT}_0$. 
There is no straightforward model for $\mathsf{HA} +\mathsf{CT_0}$ but, 
using 
\textit{realizability}, one may show that, if $\mathsf{HA}$ is consistent, then so is $\mathsf{HA}+\mathsf{CT_0}$, see \cite[vol. I, Ch. 4, Sect. 4]{Troelstra-van Dalen}. 
 
 It follows that, if $\mathsf{HA}$ is consistent, then $\mathsf{BIM}+\neg!\mathbf{FT}+\mathbf{\Pi}^0_1$-$\mathbf{AC}_{\omega,2}$ is consistent.
  \subsection{} \textit{  Axiom Scheme of Countable Finite Choice}, $\mathbf{AC_{\omega,<\omega}}$:
$$\forall \beta[\forall m \exists n \le \beta(m)[ R(n,m)] \rightarrow \exists \gamma \forall n[\gamma(n) \le \beta(n) \;\wedge\; nR\gamma(n)]].$$
 
 \subsection{} $\mathbf{\Pi}^0_1$-\textit{Axiom of Countable Finite 
Choice}, $\mathbf{\Pi}^0_1$-$\mathbf{AC}_{\omega,<\omega}$:
\[ \forall \alpha  [\forall n\exists m \le \alpha^{\upharpoonright 0}(n)[(n,m) \notin E_{\alpha^{\upharpoonright 1}}] \rightarrow 
   \exists \gamma \forall n[\gamma(n)\le \alpha^{\upharpoonright 0}(n)\;\wedge\;\bigl(n,\gamma(n)\bigr) \notin E_{\alpha^{\upharpoonright 1}}]].\]
   
We conjecture that  $\mathbf{\Pi}^0_1$-$\mathbf{AC}_{\omega,<\omega}$ is not provable in $\mathsf{BIM}+ \mathbf{\Pi}^0_1$-$\mathbf{AC}_{\omega,2}$, but we have no proof of this conjecture. There might be many  statements intermediate in strength like $ \mathbf{\Pi}^0_1$-$\mathbf{AC}_{\omega,3}$, the  $\mathbf{\Pi}^0_1$-Axiom of Countable Ternary Choice.

\section{Contrapositions of some special cases of $\mathbf{AC}_{\omega,\omega}$}\label{S:ccc}

The following statement is a contraposition of $\mathbf{AC}_{\omega,\omega}$:

\subsection{}\textit{First Axiom Scheme of Reverse  Countable Choice}, 
$\overleftarrow{\mathbf{AC}_{\omega,\omega}}$: 
 $$\forall\gamma\exists n[R\bigl(n,\gamma(n)\bigr)] \rightarrow \exists n \forall m[R(n,m)].$$

\smallskip
A special case is: \subsection{} \textit{ Minimal Axiom of  Reverse Countable Choice},
 $\mathbf{\Delta}^0_1$-$\overleftarrow{\mathbf{AC}_{\omega,\omega}}:$
$$\forall \alpha[\forall \gamma \exists n[\bigl(n, \gamma(n)\bigr)\in D_\alpha] \rightarrow \exists n \forall m[(n,m)\in D_\alpha]].$$
The following result may be found  in \cite[Section 2]{veldman82}.\begin{theorem} $\mathsf{BIM}+ \mathbf{\Delta}^0_1$-$\overleftarrow{\mathbf{AC}_{\omega,\omega}}\vdash \mathbf{LPO}$. \end{theorem}

\begin{proof}
Let $\beta$ be given. 

Assuming $\mathbf{\Delta}^0_1$-$\overleftarrow{\mathbf{AC}_{\omega,\omega}}$, we  prove: $\exists n[\beta(n)\neq 0]\;\vee\;\forall n[\beta(n)=0]$.  

To this end, define $\alpha$ such that $\forall n\forall m[\alpha(n, m) = 0\leftrightarrow (\overline \beta n= \underline{\overline 0}n\;\wedge\; \overline \beta m\neq \underline{\overline 0}m)]$. 

For every $\gamma$, \textit{either}: $\overline \beta \bigl(\gamma(0)\bigr)= \overline{\underline 0}\bigl(\gamma(0)\bigr)$  and, therefore, $\alpha\bigl(0, \gamma(0)\bigr) \neq 0$, that is $\bigl(0, \gamma(0)\bigr)\in D_\alpha$, \textit{or}: $\overline \beta \bigl(\gamma(0)\bigr)\neq\overline{\underline 0}\bigl(\gamma(0)\bigr)$, and, therefore, $\alpha\bigl(\gamma(0), \gamma(\gamma(0))\bigr)\neq 0$, that is: $\bigl(\gamma(0), \gamma(\gamma(0))\bigr)\in D_\alpha$. 

 Conclude: $\forall \gamma \exists n[\bigl(n, \gamma(n)\bigr) \in D_\alpha]$.
  
  Applying $\mathbf{\Delta}^0_1$-$\overleftarrow{\mathbf{AC}_{\omega,\omega}}$,  find $n$ such that $\forall m[(n,m) \in D_\alpha]$. 
  
  \textit{Either}: $\overline \beta n \neq \overline{\underline 0}n$ and $\exists j [\beta(j) \neq 0]$, \textit{or}:
 $\overline \beta n = \overline{\underline 0}n$. In the latter case,  for each $m$, $\overline \beta m = \overline{\underline 0}m$ and
 $\forall j[\beta(j) = 0]$.\end{proof}

 \subsection{}\textit{Axiom Scheme of   Reverse Countable Binary Choice,}
 $\overleftarrow{\mathbf{AC}_{\omega,2}}: $
  $$\forall\gamma \in \mathcal{C}\exists n[R\bigl(n,\gamma(n)\bigr)] \rightarrow \exists n\forall i<2[R(i,m)].$$

In \cite[Section 4]{veldman82}, $\overleftarrow{\mathbf{AC}_{\omega,2}}$ has been  shown to be  a consequence
 of $\mathbf{FT}+\mathbf{AC}_{\omega, \omega^\omega}$.

 We introduce a restricted version: 

$\mathbf{\Delta}^0_1$-\textit{Axiom of Reverse Countable Binary Choice}, 
$\mathbf{\Delta}^0_1$-$\overleftarrow{\mathbf{AC}_{\omega,2}}$:  \[\forall \alpha[ \forall \gamma \in \mathcal{C} \exists n[\bigl(n,\gamma(n)\bigr) \in D_{\alpha}]
 \rightarrow \exists n\forall i<2[ (n,i) \in D_{\alpha}]].\]
 
 \begin{theorem} $\mathsf{BIM}\vdash \mathbf{\Delta}^0_1$-$\overleftarrow{\mathbf{AC}_{\omega,2}}$. \end{theorem}
\begin{proof} Let $\alpha$ be given such that $\forall \gamma \in \mathcal{C} \exists n[\bigl(n,\gamma(n)\bigr) \in D_{\alpha}]$.  
 Define $\gamma$ in
$\mathcal{C}$ such that $\forall n[(n,0)\in D_\alpha\leftrightarrow\gamma(n) = 1]$.   Find $n$ such
that $\bigl(n,\gamma(n)\bigr)\in D_\alpha$. Note: $\gamma(n) = 1$ and $\forall i<2[(n,i) \in D_\alpha]$. \end{proof}

 We now introduce a less restricted version:

\subsection{}\label{SS:sigma01contraxomega2}\textit{$\mathbf{\Sigma}_1^0$-Axiom of Reverse Countable Binary Choice,
 $\mathbf{\Sigma}^0_1$-$\overleftarrow{\mathbf{AC}_{\omega,2}}$}:  \[\forall \alpha[ \forall \gamma \in \mathcal{C} \exists n[\bigl(n,\gamma(n)\bigr) \in E_{\alpha}]
 \rightarrow \exists n\forall i<2[ (n,i) \in E_{\alpha}]].\]

We define a formula that we want to call its \textit{strong negation}:
\subsection{}$\neg !(\mathbf{\Sigma}^0_1$-$\overleftarrow{\mathbf{AC}_{\omega,2}})$:
$$\exists \alpha[ \forall \gamma \in \mathcal{C} \exists n[\bigl(n,\gamma(n)\bigr) \in E_{\alpha}]
 \;\wedge\; \neg \exists n\forall i<2[(n, i) \in E_\alpha] ].$$
 
 Note: $\mathsf{BIM}$ proves: \\$\forall\alpha[\neg \exists n\forall i<2[(n, i) \in E_\alpha]\leftrightarrow \forall n\forall p\forall q[\alpha(p) = (n,0)+1\rightarrow  \alpha(q)\neq (n,1)+1]]$. 
 
\begin{lemma}\label{L:3resolutions}
 $\mathsf{BIM}$ proves the following:
\begin{enumerate}[\upshape(i)]
 \item  $\mathbf{FT} \rightarrow \mathbf{\Sigma}^0_1$-$\overleftarrow{\mathbf{AC}_{\omega,2}} $ and 
 $\neg !(\mathbf{\Sigma}^0_1$-$\overleftarrow{\mathbf{AC}_{\omega,2}}) \rightarrow \neg!\mathbf{FT}$.
\item  $\mathbf{\Sigma}^0_1$-$\overleftarrow{\mathbf{AC}_{\omega,2}}\rightarrow \mathbf{FT}$ and
$ \neg!\mathbf{FT} \rightarrow\neg !(\mathbf{\Sigma}^0_1$-$\overleftarrow{\mathbf{AC}_{\omega,2}})$.
\end{enumerate}

\end{lemma}
\begin{proof} (i)  
 We prove, in $\mathsf{BIM}$: for each $\alpha$, there exists $\beta$ such that
$$(1)\;\forall \gamma \in \mathcal{C} \exists n[\bigl(n,\gamma(n)\bigr) \in E_{\alpha}]  \rightarrow Bar_\mathcal{C}(D_\beta)]\;\mathrm{and}$$ $$(2)\;\exists m[Bar_\mathcal{C}(D_{\overline\beta m})] \rightarrow\exists n\forall i<2[(n,i)\in E_\alpha ].$$ The two promised conclusions then follow easily. 
 
 \smallskip
 Let  $\alpha$   be given. Define $\beta$ such that $$\forall m\forall a \in Bin_m[\beta(a) \neq 0\leftrightarrow \exists n<m[\bigl(n, a(n)\bigr)\in E_{\overline \alpha m}]].$$
 
 \smallskip 
(1) Assume: $ \forall \gamma \in \mathcal{C} \exists n [ \bigl(n,\gamma(n)\bigr) \in E_{\alpha}]$. Let $\gamma$ in $\mathcal{C}$ be given. \\Find $n,p$ such that $ \bigl(n,\gamma(n)\bigr)\in E_{\overline \alpha p}$ and $n<p$. Note:  $\beta(\overline \gamma p) \neq 0$.
\\We thus see: $\forall \gamma\in\mathcal{C}\exists p[\beta(\overline\gamma p)\neq 0]$, i.e. $Bar_\mathcal{C}(D_\beta)$. 

\smallskip
 
(2)  Let 
 $m$ be given such that $Bar_\mathcal{C}(D_{\overline \beta m})$. \\ Note: $\forall a \in Bin_m[m<a]$ and: $\forall a \in Bin_m \exists n\le  m[\beta(\overline a n)\neq 0 ]$. 
 
  Assume $\forall n<m\exists i<2[ 
 (n,i)\notin E_{\overline\alpha m}]$.
\\Define $a$ in  $\mathit{Bin}_m$ such that $ \forall n<m[(n,0)\notin E_{\overline\alpha m}\leftrightarrow a(n)=0]$.  \\Then $\forall n<m[\bigl(n, a(n)\bigr)\notin E_{\overline \alpha m}]$ and  $\forall n \le m[\beta(\overline a n)= 0]$. 
Contradiction. 

Conclude: $\exists n<m\forall i<2[(n,i)\in E_{\overline \alpha m}\subseteq E_\alpha]$.

\medskip (ii)\footnote{The reader should compare this proof to the proof of Theorem \ref{T:sep=wkl}(ii).}  We prove, in $\mathsf{BIM}$: for each $\alpha$, there exists $\beta$ such that  
$$(1)\;Bar_\mathcal{C}(D_\alpha)\rightarrow \forall \gamma \in \mathcal{C}\exists n[\bigl( n, \gamma(n) \bigr) \in E_{\beta}]\;\mathrm{and}$$ $$(2)\;\exists n\forall i<2[(n,i)\in E_\beta ]\rightarrow \exists m[Bar_\mathcal{C}(D_{\overline \alpha m})].$$ The two promised conclusions then follow easily.  

\smallskip 
Let $\alpha$ be given. Define $\beta$ such that, for all $n$,    for every $s$ in $Bin$, for all $i<2$, \\ \textit{if}  \textit{either:} (a) $Bar_\mathcal{C}(D_{\overline\alpha n})$, \textit{or:}  (b) $Bar_{\mathcal{C}\cap s\ast\langle i \rangle}(D_{\overline \alpha n})$  and \textit{not} $Bar_{\mathcal{C}\cap s\ast\langle 1- i \rangle}(D_{\overline \alpha n})$, 
  \textit{then} $\beta(n, s\ast\langle i \rangle ) = (s,i)+1$, and,  (c) \textit{if} both (a) and (b) fail, \textit{then} $\beta(n,s\ast\langle i \rangle )=0$. \\Furthermore,  for all $n$, for all $s$, if  $s\notin Bin\setminus\{\langle\;\rangle\}$, then $\beta(n,s)=0$. 
 
 Note: $E_\beta$ is the set of all pairs $(s, i)$ such that $s\in Bin$ and $i<2$ and either $\exists n[Bar_\mathcal{C}(D_{\overline \alpha n})]$ or $\exists n[Bar_{\mathcal{C}\cap s\ast\langle i\rangle}(D_{\overline\alpha n})\;\wedge \neg  Bar_{\mathcal{C}\cap s\ast\langle 1- i\rangle}(D_{\overline\alpha n})]$.

Note that, for all $s$ in $Bin$, if $\forall i<2[(s,i) \in E_\beta]$, then $\exists n[Bar_\mathcal{C}(D_{\overline\alpha n})]$. 
 
 \smallskip
(1) Assume $Bar_\mathcal{C}(D_\alpha)$. We shall prove: $\forall \gamma\in\mathcal{C} \exists n[\bigl(n,\gamma(n)\bigr)\in E_\beta]$.
 
 Let $\gamma$ in $\mathcal{C}$ be given. Define $\delta$ in  $\mathcal{C}$ such that $\forall n [\delta(n) = \gamma(\overline{\delta}n)]$. Find $n$ such that 
 $\alpha(\overline{\delta}n) \neq 0$. Define $q:=\overline \delta n +1$ and note: $\overline \delta n\in D_{\overline \alpha q}$. 
 
 We claim: 
 for all $j \le n$, either $\exists i \le n[\bigl(\overline \delta i, \gamma(\overline \delta i)\bigr) \in E_{\beta}]$, or $Bar_{\mathcal{C}\cap\overline\delta j}(D_{\overline \alpha q})$.
We prove this claim by backwards induction, starting from $j = n$. Note:  
 $Bar_{\mathcal{C}\cap \overline{\delta}n}(D_{\overline \alpha q})$. Now assume  $j < n$ and $Bar_{\mathcal{C}\cap \overline\delta(j+1)}(D_{\overline \alpha q})$, that is: $Bar_{\mathcal{C}\cap \overline{\delta}(j)\ast \langle \delta(j) \rangle}(D_{\overline \alpha p})$. \\Find out if also $Bar_{\mathcal{C}\cap \overline{\delta}(j)\ast \langle 1-\delta(j) \rangle}(D_{\overline \alpha q})$.  
 \\\textit{If so}, then $Bar_{\mathcal{C}\cap \overline{\delta}(j)}(D_{\overline \alpha q})$, and, \textit{if not}, 
  then $\beta\bigl(q,\overline{\delta}(j+1)\bigr)= \bigl(\overline \delta j, \delta(j)\bigr)  + 1$, and $\bigl(\overline \delta j, \delta(j)\bigr)=\bigl(\overline\delta j, \gamma(\overline \delta (j)\bigr)\in E_\beta$.

We may conclude: either $\exists j\le n[(\overline \delta j, \delta(j)\bigr)\in E_\beta]$, or $Bar_\mathcal{C}(D_{\overline \alpha q})$.
\\Note that, if  $Bar_\mathcal{C}(D_{\overline \alpha q})$, then $\forall s \in Bin\forall i<2[\beta(q,s\ast\langle i \rangle) = (s, i) +1]$ and: \\$\forall s\in Bin[\bigl(s, \gamma(s)\bigr) \in E_\beta]$.

 We thus see: $\forall\gamma \in \mathcal{C} \exists s[ \bigl(s,\gamma(s)\bigr) \in E_{\beta}]$.

\smallskip
(2)  Now assume:  $\exists n\forall i<2[(n,i)\in E_\beta]$.   Conclude, using the observation we made just after the definition of $\beta$:   $\exists n[Bar_\mathcal{C}(D_{\overline \alpha n})]$. 
\end{proof}
 \begin{theorem}\label{T:contrac<2}

 $\mathsf{BIM}$ proves: $\mathbf{FT} \leftrightarrow \mathbf{\Sigma}^0_1$-$\overleftarrow{\mathbf{AC}_{\omega,2}}$
and:  $ \neg!\mathbf{FT} \leftrightarrow \neg !(\mathbf{\Sigma}^0_1$-$\overleftarrow{\mathbf{AC}_{\omega,2}})$.

\end{theorem}

\begin{proof} Use Lemma \ref{L:3resolutions}.\end{proof}
 
\subsection{}\label{SS:cfc}\textit{Axiom Scheme of Reverse Countable Finite Choice}, $\overleftarrow{\mathbf{AC}_{\omega,<\omega}}$:
  $$\forall \beta[ \forall  \gamma \exists n[\gamma(n)\le \beta(n) \rightarrow  R\bigl(n,\gamma(n)\bigr)]\rightarrow \exists m \forall n \le \beta(m)[R(n,m) ]].$$

\smallskip
$\overleftarrow{\mathbf{AC}_{\omega,<\omega}}$ may be concluded from $\mathbf{FT}+\mathbf{AC}_{\omega^\omega,\omega}$,  by a slight extension of the argument given in \cite[Section 4]{veldman82}.

 \smallskip The following is a restricted version:

$\mathbf{\Sigma}^0_1$-$\overleftarrow{\mathbf{AC}_{\omega,<\omega}}$:
\[\forall \alpha [\forall \gamma \exists n[\gamma(n)\le\alpha^{\upharpoonright 0}(n)\rightarrow\bigl(n, \gamma(n)\bigr) \in E_{\alpha^{\upharpoonright 1}}]\rightarrow \exists n\forall i \le \alpha^{\upharpoonright 0}(n)[ (n,i )\in E_{\alpha^{\upharpoonright 1}}]].\]
We introduce a strong negation of this restricted version:

$\neg!(\mathbf{\Sigma}^0_1$-$\overleftarrow{\mathbf{AC_{\omega,<\omega}}})$:
\[\exists \alpha [\forall \gamma \exists n[\gamma(n)\le\alpha^{\upharpoonright 0}(n)\rightarrow\bigl(n,\gamma(n)\bigr)  \in E_{\alpha^{\upharpoonright 1}}]\;\wedge\; \neg\exists n\forall i \le \alpha^{\upharpoonright 0}(n)[(n, i) \in E_{\alpha^{\upharpoonright 1}}]].\]

Note: $\mathsf{BIM}$ proves: \\$\forall\alpha[\neg \exists n\forall i\le \alpha^{\upharpoonright 0}(n)[(n, i) \in E_\alpha]\leftrightarrow$

$ \forall n\forall t\in\omega^{\alpha^{\upharpoonright 0}(n)+1}\exists i\le \alpha^{\upharpoonright 0}(n)[\alpha\bigl(t(i)\bigr) \neq  (n,i)+1]]$.

 \begin{lemma}\label{L:b3}  $\mathsf{BIM}$ proves:
 \begin{enumerate}[\upshape (i)] \item $ \mathbf{\Sigma}^0_1$-$\overleftarrow{\mathbf{AC_{\omega,<\omega}}} \rightarrow \mathbf{\Sigma}^0_1$-$\overleftarrow{\mathbf{AC}_{\omega,2}}$ and $ \neg !(\mathbf{\Sigma^0_1}$-$\overleftarrow{\mathbf{AC}_{\omega,2}}) \rightarrow \neg !(\mathbf{\Sigma}^0_1$-$\overleftarrow{\mathbf{AC}_{\omega,<\omega}})$
\item $ \mathbf{FT}\rightarrow \mathbf{\Sigma}^0_1$-$\overleftarrow{\mathbf{AC_{\omega,<\omega}}}$ and 
 $  \neg !(\mathbf{\Sigma}^0_1$-$\overleftarrow{\mathbf{AC_{\omega,<\omega}}}) \rightarrow\neg!\mathbf{FT}$.
\end{enumerate}
\end{lemma}

\begin{proof}(i) 
We prove, in $\mathsf{BIM}$: for each $\alpha$, there exists $\beta$ such that $$(1)\; \forall \gamma \in \mathcal{C}\exists n[\bigl(n,\gamma(n)\bigr)\in E_{\alpha}]\rightarrow \forall \gamma \exists n[\gamma(n)\le\beta^{\upharpoonright 0}(n)\rightarrow\bigl(n, \gamma(n)\bigr) \in E_{\beta^{\upharpoonright1}}]\;\mathrm{and}$$
$$ (2)\; \exists n \forall i \le \beta^{\upharpoonright 0}(n)[(n,i)\in E_{\beta^{\upharpoonright 1}}]\rightarrow \exists n \forall i<2[(n,i)\in E_{\alpha}].$$ The two promised conclusions then follow easily.

\smallskip
Let $\alpha$ be given.  Define $\beta$ such that $\beta^
{\upharpoonright 0}=\underline 1$ and $\beta^{\upharpoonright 1} =\alpha$. 

\smallskip
(1) Assume $\forall \gamma \in \mathcal{C}\exists n[\bigl(n,\gamma(n)\bigr)\in E_{\alpha}]$. Let $\gamma$ be given. Define $\gamma^\ast$ such that $\forall n[\gamma^\ast(n) =\min\bigl(1, \gamma(n)\bigr)]$. Note: $\gamma^\ast \in \mathcal{C}$ and find $n$ such that $\bigl(n,\gamma^\ast(n) \bigr)\in  E_{\alpha}$. If $\gamma^\ast(n) \neq\gamma(n)$, then $\gamma(n)>1=\beta^{\upharpoonright 0}(n)$, and if $\gamma^\ast(n) =\gamma(n)$, then $\bigl(n,\gamma(n)\bigr)\in E_{\alpha}$. 

Conclude: $\forall \gamma \exists n[\gamma(n)\le\beta^{\upharpoonright 0}(n)\rightarrow\bigl(n, \gamma(n)\bigr) \in E_{\beta^{\upharpoonright 1}}]$.

\smallskip (2)  Let $n$ be given such that  $\forall i \le \beta^{\upharpoonright 0}(n)[(n,i)\in E_{\beta^{\upharpoonright 1}}]$. \\Clearly, then, $\forall i<2[(n,i)\in E_{\alpha}]$.

\medskip
(ii)\footnote{The argument extends the argument for Lemma \ref{L:3resolutions}(i).} 
 We prove, in $\mathsf{BIM}$: for each $\alpha$, there exists $\beta$ such that
  $$(1)\; \forall \gamma \exists n[\gamma(n)\le \alpha^{\upharpoonright 0}(n)\rightarrow \bigl(n,\gamma(n)\bigr) \in     E_{\alpha^{\upharpoonright 1}}] \rightarrow Bar_\mathcal{C}(D_\beta)\;\mathrm{and}$$ $$(2)\; \exists m[ Bar_\mathcal{C}(D_{\overline \beta m})]  \rightarrow \exists n\forall m \le \alpha^{\upharpoonright 0}(n)[(n, m) \in   E_{\alpha^{\upharpoonright 1}}].$$
 The two promised conclusions then follow easily.
 
\smallskip
 First, define $Cod_2:\omega\rightarrow\mathit{Bin}$ such that \\$Cod_2(\langle \;\rangle)=\langle \; \rangle$ and $\forall s\forall n[Cod_2(s\ast\langle n \rangle) = Cod_2(s)\ast \underline{\overline 0}n \ast\langle 1 \rangle]$.
  \\Note: $\forall t\in\mathit{Bin}\exists s\exists i[ t= Cod_2(s) \ast \underline{\overline 0}i]$.

\smallskip Now,  let $\alpha$ be given.  Define $\beta$  such that, for all $s, i$,
$\beta\bigl(Cod_2(s)\ast\underline{\overline 0}i\bigr) \neq 0$ if and only if  $\exists n < \mathit{length}(s)[s(n) >\alpha^{\upharpoonright 0}(n)\;\vee\; \bigl(n,s(n)\bigr)\in E_{\overline{\alpha^{\upharpoonright 1}}length(s)}\;\vee\;\alpha^{\upharpoonright 0}\bigl(\mathit{length}(s)\bigr)< i]$. 

\smallskip (1) Assume:  $\forall \gamma \exists n[\gamma(n) \le\alpha^{\upharpoonright 0}(n) \rightarrow  \bigl(n,\gamma(n)\bigr) \in E_{\alpha^{\upharpoonright 1}}]$.

Assume: $\delta\in\mathcal{C}$.  Define $\gamma$,   by induction, such that, for each $n$, \\\textit{if} $\exists i\le\alpha^{\upharpoonright 0}(n)[Cod_2(\overline \gamma n\ast\langle i\rangle)\sqsubset \delta]$, then $Cod_2\bigl(\overline\gamma(n+1)\bigr)\sqsubset \delta$, and,  \\\textit{if not}, then $\gamma(n)=0$.

Note: $\forall n[\gamma(n)\le\alpha^{\upharpoonright 0}(n)]$.

Find $n,p$ such that  $\bigl(n,\gamma(n)\bigr) \in E_{\overline{\alpha^{\upharpoonright 1}}p}$.  Define $q:=\max (n,p)$.

 Note: $\beta\bigl(Cod_2(\overline \gamma(q+1))\bigr)\neq 0$ and distinguish two cases. 
 
 \textit{Case (a)}.  $c:=Cod_2(\overline \gamma(q+1))\sqsubset\delta$ and $\beta(c)\neq 0$. 
 
 \textit{Case (b)}. $\exists m\forall i\le \alpha^{\upharpoonright 0}(m)[Cod_2(\overline\gamma m\ast\langle i \rangle)\perp\delta]$. 
 
 Find $m_0:=\mu m\forall i\le \alpha^{\upharpoonright 0}(m)[Cod_2(\overline\gamma m\ast\langle i \rangle)\perp\delta]$ and note:\\ $d:=Cod_2(\overline \gamma m_0)\ast \underline{\overline 0}\bigl(\alpha^{\upharpoonright 0}(m_0)+1\bigr)\sqsubset \delta$ and $\beta(d) \neq 0$.

 In both  cases $\exists p[\beta(\overline \delta p) \neq 0]$.
 
  We thus see: $\forall \delta \in \mathcal{C}\exists p[\beta(\overline \delta p)\neq 0]$, i.e. $Bar_\mathcal{C}(D_\beta)$.

\smallskip
(2) Let $m$ be given such that  $Bar_\mathcal{C}(D_{\overline \beta m})$. \\Suppose: $\forall i<m \exists j \le \alpha^{\upharpoonright 0}(i)[(i,j)\notin E_{\overline{ \alpha^1}m}]]$. \\Find $s$ such that
$\mathit{length}(s) = m$ and $\forall i<m[s(i)\le \alpha^{\upharpoonright 0}(i)\;\wedge\;\bigl(i,s(i)\bigr)\notin E_{\overline{ \alpha^{\upharpoonright 1}}m}]$. \\Note: $Cod_2(s)>m$ and $\forall t\sqsubseteq Cod_2(s)[\beta(t)=0 ]$, so $\neg \mathit{Bar}_\mathcal{C}(D_{\overline \beta m})$. 

 Contradiction.

Conclude:   $\exists i <m\forall j\le\alpha^{\upharpoonright 0}(i)[(i, j) \in E_{\overline{\alpha^{\upharpoonright 1}}m}\subseteq E_{\alpha^{\upharpoonright 1}}]$.
\end{proof}

\begin{theorem} $\mathsf{BIM}$ proves: 
$ \mathbf{FT} \leftrightarrow \mathbf{\Sigma}^0_1$-$\overleftarrow{\mathbf{AC}_{\omega,<\omega}}$  and 
$ \neg!\mathbf{FT} \leftrightarrow \neg !(\mathbf{\Sigma}^0_1$-$\overleftarrow{\mathbf{AC}_{\omega,<\omega}})$.

\end{theorem}
\begin{proof} These statements follow from Lemma \ref{L:b3} and Theorem \ref{T:contrac<2}. \end{proof} 

\subsection{No Double Negation Shift}\label{SSS:ndns} \hfill

Assume $\neg!\mathbf{FT}$. Using $\neg !(\mathbf{\Sigma}^0_1$-$\overleftarrow{\mathbf{AC}_{\omega, 2}})$, find $\alpha$  such that $\forall \gamma \in \mathcal{C} \exists n[ \bigl(n,\gamma(n)\bigr) \in E_{\alpha}]$ and: $\neg \exists n\forall m<2[(n,m)\in E_\alpha ]$.

 Then, for each $n$, $\neg \forall m <2[(n,m)\in E_\alpha]$ and: $\neg \forall m<2[\neg\neg\bigl((n,m)\in E_\alpha\bigr)]$ and: $\neg\neg\exists m<2[(n,m)\notin E_\alpha]$.
 
We thus see: $\forall n\neg \neg\exists m<2[(n,m) \notin E_{\alpha}]$.
 
  Also: $\neg \exists \gamma \in \mathcal{C}\forall n[ \bigl(n,\gamma(n)\bigr) \notin E_{\alpha}]$.
  
   Using  $\mathbf{\Pi}^0_1$-$\mathbf{AC}_{\omega,2}$, we conclude: $\neg \forall n\exists m<2[(n,m)\notin E_\alpha]$.

We thus see that if we assume both $\neg!\mathbf{FT}$ and  $\mathbf{\Pi}^0_1$-$\mathbf{AC}_{\omega,2}$  we can find $\mathbf{\Pi}^0_1$-subsets $P=\{n\mid (n,0)\notin E_\alpha\}$ and $Q=\{n\mid (n,1)\notin E_\alpha\}$ of $\omega$ such that 
\\$\forall n[\neg \neg\bigl(P(n) \vee Q(n)\bigr)]$ and $\neg \forall n[P(n) \vee Q(n)]$. 

S. Kuroda's scheme of Double Negation Shift $$\forall n[\neg\neg T(n)]\rightarrow \neg\neg\forall n[T(n)]$$ (see \cite[page 45]{kuroda} and \cite[page 105]{heyting}) thus is refuted.

In \cite[vol. I, Chapter 4, Proposition 3.4, Corollary 1]{Troelstra-van Dalen},  the same conclusion is obtained in $\mathsf{HA}$ from $\mathsf{CT}_0$.

\section{$\mathbf{AC}_{\omega,\omega^\omega}$, some special cases}\label{SS:choice2}

\subsection{} $\mathbf{\Sigma}^0_1$-\textit{Second Axiom of Countable Choice}, $\mathbf{\Sigma}^0_1$-$\mathbf{AC}_{\omega,\omega^\omega}$:
\[\forall \alpha[\forall n \exists \gamma[\gamma \in \mathcal{G}_{\alpha^{\upharpoonright n}}] \rightarrow \exists \gamma \forall n[\gamma^{\upharpoonright n} \in
\mathcal{G}_{\alpha^{\upharpoonright n}}]].\]

\begin{theorem} $\mathsf{BIM}\vdash \mathbf{\Sigma}^0_1$-$\mathbf{AC}_{\omega,\omega^\omega}$. \end{theorem} \begin{proof}  Assume $\forall n \exists \gamma[ \gamma \in \mathcal{G}_{\alpha^{\upharpoonright n}}]$. Then $\forall n \exists s[\alpha^{\upharpoonright n}(  s)\neq 0]$. \\Find $\delta$ such that $\forall n[\delta(n)=\mu s[\alpha^{\upharpoonright n}( s) \neq 0]$]. Find $\gamma$ such that $\forall n[\gamma^{\upharpoonright n} =  \delta(n) \ast \underline 0]$. Note:
$\forall n[ \gamma^{\upharpoonright n} \in \mathcal{G}_{\alpha^{\upharpoonright n}}]$. \end{proof}

\subsection{} $\mathbf{\Pi}^0_1$-\textit{Second Axiom of Countable Choice}, $\mathbf{\Pi}^0_1$-$\mathbf{AC}_{\omega,\omega^\omega}$:
\[\forall \alpha[\forall n \exists \gamma[ \gamma \notin \mathcal{G}_{\alpha^{\upharpoonright n}}] \rightarrow \exists \gamma \forall n [\gamma^{\upharpoonright n} \notin \mathcal{G}_{\alpha^{\upharpoonright n}}]].\]

 \begin{theorem}\label{T:pi01acNn} $\mathsf{BIM}\vdash \mathbf{\Pi}^0_1$-$\mathbf{AC}_{\omega,\omega^\omega} \rightarrow \mathbf{\Pi}^0_1$-$\mathbf{AC}_{\omega,\omega}$. \end{theorem}\begin{proof} Let $\alpha$ be given such that $ \forall n \exists m  [(n, m) \notin E_{\alpha}]$. 
Define $\beta$ such that 

$\forall n \forall a[\beta^{\upharpoonright n}(a) \neq 0 \leftrightarrow \exists m\exists b  \exists p\le a[ \alpha(p)=(n,m)+1\;\wedge\;a=\langle m\rangle\ast b]] $. 

Note: $\forall n \forall m[(n,m)\in E_\alpha\leftrightarrow \forall \gamma[\langle m\rangle\ast\gamma \in \mathcal{G}_{\beta^{\upharpoonright n}}]]$. 

Conclude: $\forall n \exists \gamma [\gamma \notin \mathcal{G}_{\beta^{\upharpoonright n}}]$. Using $\mathbf{\Pi}^0_1$-$\mathbf{AC}_{\omega,\omega^\omega}$, find $\gamma$ such that $\forall n[\gamma^{\upharpoonright n} \notin \mathcal{G}_{\beta^{\upharpoonright n}}]$. Define $\delta$  such that $\forall n[\delta(n)=\gamma^{\upharpoonright n}(0)]$ and note: $\forall n [\bigl( n, \delta(n)\bigr)\notin E_\alpha]$. \end{proof}One may conclude that $\mathbf{\Pi}^0_1$-$\mathbf{AC}_{\omega,\omega^\omega}$  is unprovable in $\mathsf{BIM}$, see Subsection \ref{SSS:countbincunpr}.

   Not every $\mathbf{\Pi}^0_1$ subset of $\omega^\omega$ is a spread, see Lemma \ref{L:closedspread}. For spreads, which are a special kind of $\mathbf{\Pi}^0_1$ sets, countable choice is easier:
 \begin{theorem}$\mathsf{BIM}\vdash\forall \alpha[\bigl(\forall n[Spr(\alpha^{\upharpoonright n})]\;\wedge\;\forall n \exists\gamma[ \gamma \notin \mathcal{G}_{\alpha^{\upharpoonright n}}]\bigr) \rightarrow \exists \gamma \forall n [\gamma^{\upharpoonright n} \notin \mathcal{G}_{\alpha^{\upharpoonright n}}]]$.\end{theorem}

 \begin{proof} Let $\alpha$ be given such that $\forall n[ Spr(\alpha^{\upharpoonright n})]$ and $\forall n\exists \gamma[ \gamma \notin \mathcal{G}_{\alpha^{\upharpoonright n}}] $. \\Define  $\gamma$ such that, for each $n$, 
   for each $m$, $\gamma^{\upharpoonright n}(m)=\mu k[\alpha^{\upharpoonright n} \bigl(  (\overline{\gamma^{\upharpoonright n}}m)\ast\langle k \rangle \bigr) = 0]$. \end{proof}

 \subsection{}\textit{ Axiom Scheme of Countable Compact Choice}, $\mathbf{AC}_{\omega,\mathcal{C}}$: $$\forall n \exists \gamma \in \mathcal{C}[R(n,\gamma)]\rightarrow \exists \gamma \in \mathcal{C}\forall n[R(n, \gamma^{\upharpoonright n})].$$

  Here   is a   restricted version of $\mathbf{AC}_{\omega,\mathcal{C}}$:
 \subsection{} $\mathbf{\Pi}_1^0$-\textit{Axiom of Countable Compact 
  Choice}, $\mathbf{\Pi}_1^0$-$\mathbf{AC}_{\omega,\mathcal{C}}$:
  \[\forall \alpha[ \forall n \exists \gamma \in \mathcal{C} [\gamma \notin \mathcal{G}_{\alpha^{\upharpoonright n}}] \rightarrow \exists \gamma \in \mathcal{C} \forall n[\gamma^{\upharpoonright n} \notin \mathcal{G}_{\alpha^{\upharpoonright n}}]]\]
\begin{theorem} $\mathsf{BIM}\vdash \mathbf{\Pi}^0_1$-$\mathbf{AC}_{\omega,\mathcal{C}}\rightarrow\mathbf{\Pi}^0_1$-$\mathbf{AC}_{\omega, 2}$. \end{theorem}\begin{proof} The proof is almost the same as the proof of Theorem \ref{T:pi01acNn} and is left to the reader.
\end{proof} 
  
  We may conclude: $\mathbf{\Pi}^0_1$-$\mathbf{AC}_{\omega,\mathcal{C}}$ is unprovable in $\mathsf{BIM}$, see Subsection \ref{SSS:countbincunpr}.  
 
 \smallskip
The treatment of real numbers in $\mathsf{BIM}$ is sketched in Subsection \ref{SSS:reals}. 
\subsection{}$\mathbf{\Pi}^0_1$-$\mathbf{AC}_{\omega,[0,1]}$:
  \[\forall \alpha[ \forall n \exists \delta \in [0,1] [\delta \notin \mathcal{H}_{\alpha^{\upharpoonright n}}] \rightarrow \exists \delta \in [0,1]^\omega \forall n[\delta^{\upharpoonright n} \notin \mathcal{H}_{\alpha^{\upharpoonright n}}]].\]
  
  \begin{theorem}\label{T:compactchoice} $\mathsf{BIM}\vdash \mathbf{\Pi}^0_1$-$\mathbf{AC}_{\omega,\mathcal{C}}\leftrightarrow \mathbf{\Pi}_1^0$-$\mathbf{AC}_{\omega,[0,1]}$.\end{theorem}
  
  \begin{proof}

First assume $\mathbf{\Pi}^0_1$-$\mathbf{AC}_{\omega,\mathcal{C}}$.

Using Lemma \ref{L:Conto[0,1]}, find $\sigma: \mathcal{C} \rightarrow [0,1]$ and $\psi:\omega^\omega\rightarrow\omega^\omega$ such that \begin{enumerate}[\upshape (1)] \item  $\forall \delta \in [0,1]  \exists \gamma\in\mathcal{C}[\delta =_\mathcal{R} \sigma|\gamma]$ and  \item   $\forall \alpha\forall \gamma \in \mathcal{C}[\gamma\in \mathcal{G}_{\psi|\alpha}\leftrightarrow \sigma|\gamma \in \mathcal{H}_\alpha]$. \end{enumerate}

Let $\alpha$ be given such that $\forall n \exists \delta \in [0,1] [\delta \notin \mathcal{H}_{\alpha^{\upharpoonright n}}]$. Then $\forall n \exists \gamma \in \mathcal{C}[\sigma|\gamma \notin \mathcal{H}_{\alpha^{\upharpoonright n}}]$ and $\forall n \exists \gamma \in \mathcal{C}[\gamma \notin \mathcal{G}_{\psi|(\alpha^{\upharpoonright n})}]$. Find $\gamma$ in $\mathcal{C}$ such that $\forall n[\gamma^{\upharpoonright n}\notin \mathcal{G}_{\psi|(\alpha^{\upharpoonright n})}]$. \\Conclude: $\forall n[\sigma|\gamma^{\upharpoonright n}  \notin \mathcal{H}_{\alpha^{\upharpoonright n}}]$ and $\exists \delta\forall n[\delta^{\upharpoonright n}\notin \mathcal{H}_{\alpha^{\upharpoonright n}}]$. 

\smallskip
Now assume $\mathbf{\Pi}^0_1$-$\mathbf{AC}_{\omega,[0,1]}$.

Using Lemma \ref{L:Cinto[0,1]}, find $\tau:\mathcal{C}\rightarrow [0,1]$ and $\chi:\omega^\omega\rightarrow\omega^\omega$ such that  \begin{enumerate}[\upshape (1)] \item $\forall \gamma \in \mathcal{C}\forall\delta\in\mathcal{C}[\gamma\;\#\;\delta\rightarrow \tau|\gamma\;\#_\mathcal{R}\;\tau|\delta]$, 
 and  \item $\forall \alpha \forall \gamma \in \mathcal{C}[\gamma \in \mathcal{G}_\alpha\leftrightarrow \tau|\gamma \in \mathcal{H}_{\chi|\alpha}]$. and \item $\forall\alpha\forall\delta\in [0,1]^\omega\exists \gamma \in \mathcal{C}\forall n[\delta^{\upharpoonright n}\;\#_\mathcal{R}\;\tau|(\gamma^{\upharpoonright n}) \rightarrow \delta^{\upharpoonright n} \in \mathcal{H}_{\chi|(\alpha^{\upharpoonright n})}]$. \end{enumerate}

Let $\alpha$ be given such that $\forall n \exists \gamma \in \mathcal{C}[\gamma \notin \mathcal{G}_{\alpha^{\upharpoonright n}}]$. 

Conclude: $\forall n\exists \gamma\in\mathcal{C}[\tau|\gamma\notin\mathcal{H}_{\chi|(\alpha^{\upharpoonright n})}]$ and: $\forall n\exists \delta\in[0,1][\delta\notin\mathcal{H}_{\chi|(\alpha^{\upharpoonright n})}]$.

 Using $\mathbf{\Pi}^0_1$-$\mathbf{AC}_{\omega,[0,1]}$, find $\delta$ in $[0,1]^\omega$ such that  $\forall n[\delta^{\upharpoonright n} \notin \mathcal{H}_{\chi|(\alpha^{\upharpoonright n})}]$.

Using (3),  find $\gamma$ in $\mathcal{C}$ such that $\forall n[\delta^{\upharpoonright n}\;\#_\mathcal{R}\;\tau|(\gamma^{\upharpoonright n}) \rightarrow \delta^{\upharpoonright n} \in \mathcal{H}_{\chi|(\alpha^{\upharpoonright n})}]$. 
 
 Conclude: $\forall n[\delta^{\upharpoonright n}=_\mathcal{R} \tau|(\gamma^{\upharpoonright n})]]$,
and, using (2):  $\forall n[\gamma^{\upharpoonright n} \notin \mathcal{G}_{\alpha^{\upharpoonright n}}]$.
 
 Clearly, $\exists \gamma \in \mathcal{C}\forall n[\gamma^{\upharpoonright n} \notin \mathcal{G}_{\alpha^{\upharpoonright n}}]$.
 
 \end{proof}
 \section{Contrapositions of some special cases of $\mathbf{AC}_{\omega,\omega^\omega} $}
 
 The  axiom $\mathbf{\Delta}^0_1$-$\overleftarrow{\mathbf{AC}_{\omega,\omega^\omega}}$   implies $\mathbf{\Delta}^0_1$-$\overleftarrow{\mathbf{AC}_{\omega,\omega}}$ and therefore $\mathbf{LPO}$. 
  Let us consider an \textit{Axiom Scheme of Reverse   Countable Compact Choice}:
 
 \subsection{}
 $\overleftarrow{\mathbf{AC}_{\omega,\mathcal{C}}}$:
 $\forall \gamma \in \mathcal{C}\exists n[R(n,\gamma^{\upharpoonright n})]\rightarrow \exists n \forall \gamma \in \mathcal{C}[R(n,\gamma)]$.

In  \cite{veldman82} it is shown that $\overleftarrow{\mathbf{AC}_{\omega,\mathcal{C}}}$ is a consequence of the First Axiom of Continuous Choice $\mathbf{AC}_{\omega^\omega,\omega}$ and  \textbf{FAN}.
  
  We now   require the relation $R$ to be $\mathbf{\Sigma}^0_1$ and obtain the  $\mathbf{\Sigma}^0_1$-\textit{Axiom of Reverse Countable Compact Choice}:\subsection{}$\mathbf{\Sigma}^0_1$-$\overleftarrow{\mathbf{AC}_{\omega,\mathcal{C}}}$:
  $\forall \alpha[  \forall \gamma \in \mathcal{C}
  \exists n [\gamma^{\upharpoonright n} \in \mathcal{G}_{\alpha^{\upharpoonright n}}] \rightarrow \exists n[  \mathcal{C} \subseteq \mathcal{G}_{\alpha^{\upharpoonright n}}] ].$

We also introduce a strong negation:  
  
                                                                                                                                                                                                                                                                                                                                                                                                                                                                                                                                                                                                                                                                                                                                                                                                                                                                                                                                                                                                                                                                                                                                                                                                                                                                                                                                                                                                                                                                                                                                                                                                                                                                               \subsection{}\label{SS:omegacantor}$\neg !\bigl(\mathbf{\Sigma}^0_1$-$\overleftarrow{\mathbf{AC}_{\omega,\mathcal{C}}})$:
  $\exists \alpha[  \forall \gamma \in \mathcal{C}
  \exists n [\gamma^{\upharpoonright n} \in \mathcal{G}_{\alpha^{\upharpoonright n}}]\;\wedge\; \neg \exists n[ \mathcal{C} \subseteq \mathcal{G}_{\alpha^{\upharpoonright n}}] ].$

\smallskip

We  introduce a `real'  version:
\subsection{}$\mathbf{\Sigma}^0_1$-$\overleftarrow{\mathbf{AC}_{\omega,[0,1]}}$:
  $\forall \alpha[  \forall \delta \in [0,1]^\omega
  \exists n [\delta^{\upharpoonright n} \in \mathcal{H}_{\alpha^{\upharpoonright n}}] \rightarrow \exists n [ [0,1] \subseteq \mathcal{H}_{\alpha^{\upharpoonright n}}]].$

 \smallskip and a strong negation:
 \subsection{}\label{SS:omega01}$\neg !(\mathbf{\Sigma}^0_1$-$\overleftarrow{\mathbf{AC}_{\omega,[0,1]}})$:
  $\exists \alpha [  \forall \delta \in [0,1]^\omega
  \exists n [\delta^{\upharpoonright n} \in \mathcal{H}_{\alpha^{\upharpoonright n}}] \;\wedge\; \neg\exists n[  [0,1] \subseteq \mathcal{H}_{\alpha^{\upharpoonright n}}]].$

\medskip
The treatment of real numbers in $\mathsf{BIM}$ is sketched in Subsection \ref{SSS:reals}.
\begin{lemma}\label{L:b45}
 $\mathsf{BIM}$ proves:
\begin{enumerate}[\upshape(i)]
\item $\mathbf{FT}\rightarrow \mathbf{\Sigma}^0_1$-$\overleftarrow{\mathbf{AC}_{\omega,\mathcal{C}}}$ and 
  $  \neg !(\mathbf{\Sigma}^0_1$-$\overleftarrow{\mathbf{AC}_{\omega,\mathcal{C}}})\rightarrow\neg!\mathbf{FT}$.
\item $\mathbf{\Sigma}^0_1$-$\overleftarrow{\mathbf{AC}_{\omega,\mathcal{C}}}\rightarrow \mathbf{\Sigma}^0_1$-$\overleftarrow{\mathbf{AC}_{\omega,[0,1]}}$ and 
$\neg !(\mathbf{\Sigma}^0_1$-$\overleftarrow{\mathbf{AC}_{\omega,[0,1]}})\rightarrow \neg !(\mathbf{\Sigma}^0_1$-$\overleftarrow{\mathbf{AC}_{\omega,\mathcal{C}}})$

\item$\mathbf{\Sigma}^0_1$-$\overleftarrow{\mathbf{AC}_{\omega,[0,1]}}\rightarrow \mathbf{\Sigma}^0_1$-$\overleftarrow{\mathbf{AC}_{\omega,2}}$ and 
$\neg !(\mathbf{\Sigma}^0_1$-$\overleftarrow{\mathbf{AC}_{\omega,2}})\rightarrow \neg !(\mathbf{\Sigma}^0_1$-$\overleftarrow{\mathbf{AC}_{\omega,[0,1]}})$ \end{enumerate}
\end{lemma}
\begin{proof} 
(i)\footnote{The argument may be compared to the arguments for Lemma \ref{L:3resolutions}(i) and for Lemma \ref{L:b3}(ii).} 
 We prove, in $\mathsf{BIM}$: for each $\alpha$, there exists $\beta$ such that 
 $$\forall \gamma \in \mathcal{C} \exists n [\gamma^{\upharpoonright n} \in \mathcal{G}_{\alpha^{\upharpoonright n}}] \rightarrow Bar_\mathcal{C}(D_\beta)\;\mathrm{and}\;\exists m[Bar_\mathcal{C}(D_{\overline\beta m})] \rightarrow \exists n [\mathcal{C} \subseteq \mathcal{G}_{\alpha^{\upharpoonright n}}].$$ 
 The two promised statements then follow easily. 
 
  Let  $\alpha$ be given. Define $\beta$ such that, for every $s$,
  
$\beta(s) \neq 0\leftrightarrow \bigl(s \in \mathit{Bin}\;\wedge\; \exists n < \mathit{length}(s)\exists p \le \mathit{length}(s^{\upharpoonright n})[\alpha^{\upharpoonright n}( \overline {s^{\upharpoonright n}}p) \neq 0]\bigr)$.

Assume:  $\forall \gamma \in \mathcal{C} \exists n \exists p[\alpha^{\upharpoonright n}(\overline{\gamma^{\upharpoonright n}} p)\neq 0]$.\\ Clearly $\forall \gamma \in \mathcal{C}\exists n[\beta(\overline \gamma n)\neq 0]$, i.e. $Bar_\mathcal{C}(D_\beta)$.

Let $m$  be given such that $Bar_\mathcal{C}(D_{\overline\beta m})$.  Suppose there is no $n < m$ such that $Bar_\mathcal{C}(D_{\overline{\alpha^{\upharpoonright n}}m})$.  For each $n<m$, there exists $u$ in $\mathit{Bin}$ such that $\mathit{length}(u)=m$ and $u$ does not meet $D_{\alpha^n}$.  Let $s$ be an element of $\mathit{Bin}$ such that $\mathit{length}(s) = m$ and, for each $n<m$, $s^n$ does not meet $D_{\alpha^{\upharpoonright n}}$.     Note:  $s$ does not meet $D_{\beta}$. Contradiction. Thus we see there must exist   $n<m$ such that
every $s$ in $\mathit{Bin}_m$  meets $D_{\alpha^{\upharpoonright n}}$.  Conclude: $\exists n[\mathcal{C} \subseteq \mathcal{G}_{\alpha^{\upharpoonright n}}]$.

\smallskip
(ii)\footnote{The argument may be compared to the argument for the first half of Theorem \ref{T:compactchoice}.}  We prove, in $\mathsf{BIM}$: for each $\alpha$, there exists $\beta$ such that 
 $$\forall \delta \in [0,1]^\omega \exists n [\delta^{\upharpoonright n} \in \mathcal{H}_{\alpha^{\upharpoonright n}}] \rightarrow \forall \gamma \in \mathcal{C} \exists n [\gamma^{\upharpoonright n} \in \mathcal{G}_{\beta^{\upharpoonright n}}]\;\mathrm{and}$$ $$\exists n [\mathcal{C} \subseteq \mathcal{G}_{\beta^{\upharpoonright n}}] \rightarrow \exists n [[0,1] \subseteq \mathcal{H}_{\alpha^{\upharpoonright n}}].$$ 

Using Lemma \ref{L:Conto[0,1]}, find 
 $\sigma:\mathcal{C}\rightarrow [0,1]$ and $\psi:\omega^\omega\rightarrow\omega^\omega$ such that  \\ $\forall \delta \in [0,1]\exists \gamma \in \mathcal{C}[\sigma|\gamma =_\mathcal{R} \delta]$ and   $\forall \alpha\forall\gamma \in\mathcal{C}[\gamma \in \mathcal{G}_{\psi|\alpha}\leftrightarrow \sigma|\gamma \in \mathcal{H}_{\alpha}]$.

  Let $\alpha$ be given. Define $\beta$ such that, for every $n$, $\beta^{\upharpoonright n} = \psi|(\alpha^{\upharpoonright n})$.

 Assume  $\forall \delta \in [0,1]^\omega\exists n[\delta^{\upharpoonright n} \in \mathcal{H}_{\alpha^{\upharpoonright n}}]$. Then $\forall \gamma \in \mathcal{C}\exists n[\sigma|(\gamma^{\upharpoonright n}) \in \mathcal{H}_{\alpha^{\upharpoonright n}}]$ and\\
 $\forall \gamma \in \mathcal{C} \exists n[\gamma^{\upharpoonright n} \in \mathcal{G}_{\psi|(\alpha^{\upharpoonright n})}]$ and:
 $\forall \gamma \in \mathcal{C}\exists n [\gamma^{\upharpoonright n} \in \mathcal{G}_{\beta^{\upharpoonright n}}]$.

Let $n$ be given such that such that $\mathcal{C} \subseteq \mathcal{G}_{\beta^{\upharpoonright n}}=\mathcal{G}_{\psi|(\alpha^{\upharpoonright n})}$. 

Note: $\forall \gamma \in \mathcal{C}[ \sigma|\gamma \in \mathcal{H}_{\alpha^{\upharpoonright n}}]$
  and, therefore: $[0,1] \subseteq \mathcal{H}_{\alpha^{\upharpoonright n}}$.
  
  \smallskip

 (iii) We prove, in $\mathsf{BIM}$: for each $\alpha$, there exists $\beta$ such that 
 $$\forall \gamma \in \mathcal{C} \exists n [\bigl(n,\gamma(n)\bigr) \in E_{\alpha}] \rightarrow \\\forall \delta \in [0,1]^\omega \exists n [\delta^{\upharpoonright n} \in \mathcal{H}_{\beta^{\upharpoonright n}}]\;\mathrm{and}$$ $$\exists n [[0,1] \subseteq \mathcal{H}_{\beta^{\upharpoonright n}}] \rightarrow \exists n \forall i<2[(n,i)\in E_{\alpha}].$$

 Let $\alpha$ be given. Define $\beta$ such that $\forall n\forall s\in \mathbb{S}[\beta^{\upharpoonright n}(s) \neq 0\leftrightarrow \\\exists i<s[\bigl(\alpha(i) =(n,0)+1\;\wedge\;s''<_\mathbb{Q}1_\mathbb{Q}\bigr)\;\vee\;\bigl(\alpha(i) =(n,1)+1\;\wedge\;0_\mathbb{Q}<_\mathbb{Q}s'\bigr)]]$. 
 
  Note $\forall n[\bigl((n,0) \in E_{\alpha}\leftrightarrow [0,1) \subseteq \mathcal{H}_{\beta^{\upharpoonright n}}]\bigr)\;\wedge\;\bigl((n,1) \in E_{\alpha}\leftrightarrow (0,1] \subseteq \mathcal{H}_{\beta^{\upharpoonright n}}\bigr)] $.

Assume:  $\forall \gamma \in \mathcal{C}\exists n [ \bigl(n,\gamma(n)\bigr) \in E_{\alpha}]$, and $\delta \in [0,1]^\omega$.

 Define $\varepsilon$  such that  $\forall n[\varepsilon(n)=\mu m [0_\mathbb{Q}<_\mathbb{Q}\bigl(\delta^n(m)\bigr)'\;\vee\;\bigl(\delta^n(m)\bigr)''<_\mathbb{Q}1_\mathbb{Q}]$.
 
 Define $\gamma$ in $\mathcal{C}$ such that $\forall n[\gamma(n) = 0\leftrightarrow\bigl(\delta^n(\varepsilon(n))\bigr)''<_\mathbb{Q}1_\mathbb{Q}]$.
 
  Note: $\forall n[\gamma(n) = 1\rightarrow 0_\mathcal{R}<_\mathcal{R} \delta^n ]$. 
  
  Find $n$ such that $\bigl(n,\gamma(n)\bigr) \in E_{\alpha}$ and conclude: 
  
  \textit{either}: $\gamma(n) = 0$ and $\delta^{\upharpoonright n} <_\mathcal{R} 1_\mathcal{R}$ and $[0,1) \subseteq \mathcal{H}_{\beta^{\upharpoonright n}}$, so $\delta^{\upharpoonright n} \in \mathcal{H}_{\beta^{\upharpoonright n}}$, 
  
  \textit{or}: $\gamma(n) =1$ and $0_\mathcal{R} <
  _\mathcal{R}\delta^{\upharpoonright n} $ and $(0,1]\subseteq \mathcal{H}_{\beta^{\upharpoonright n}}$, so, again, $\delta^{\upharpoonright n} \in \mathcal{H}_{\beta^{\upharpoonright n}}$.

  Conclude: $\forall \delta \in [0,1]^\omega \exists n[\delta^{\upharpoonright n} \in \mathcal{H}_{\beta^{\upharpoonright n}}]$.
  
  \smallskip
Let  $n$ be given such that $[0,1] \subseteq \mathcal{H}_{\beta^{\upharpoonright n}}$. Conclude: $\forall i<2[(n,i)\in E_{\alpha}]$. 
\end{proof}\begin{theorem}\label{T:ccc}
\begin{enumerate}[\upshape (i)]
\item $\mathsf{BIM}\vdash \mathbf{FT} \leftrightarrow \mathbf{\Sigma}^0_1$-$\overleftarrow{\mathbf{AC}_{\omega,\mathcal{C}}}\leftrightarrow \mathbf{\Sigma}^0_1$-$\overleftarrow{\mathbf{AC}_{\omega,[0,1]}}$.
\item  $\mathsf{BIM}\vdash \neg !\mathbf{FT} \leftrightarrow \neg !(\mathbf{\Sigma}^0_1$-$\overleftarrow{\mathbf{AC}_{\omega,\mathcal{C}}})\leftrightarrow \neg !(\mathbf{\Sigma}^0_1$-$\overleftarrow{\mathbf{AC}_{\omega,[0,1]}})$.
\end{enumerate}
\end{theorem}

\begin{proof} These statements follow from Lemmas \ref{L:b45} and  \ref{L:3resolutions}.\end{proof}

\section{On the Contraposition of Twofold Compact Choice}\label{S:refinement} 

 We introduce a limited version of  $\mathbf{\Sigma}^0_1$-$\overleftarrow{\mathbf{AC}_{\omega,\mathcal{C}}}$: \subsection{}$\label{SS:sigma01ac2c}\mathbf{\Sigma}^0_1$-$\overleftarrow{\mathbf{AC}_{2,\mathcal{C}}}$:
  $\forall \alpha[  \forall \gamma \in \mathcal{C}
  \exists i<2 [\gamma^{\upharpoonright i} \in \mathcal{G}_{\alpha^i}] \rightarrow \exists i<2 [\mathcal{C} \subseteq \mathcal{G}_{\alpha^{\upharpoonright i}}] ].$
  
This statement should be called the $\mathbf{\Sigma}^0_1$-Axiom of \emph{Reverse Twofold Compact Choice}. It is a contraposition of  a special case of the following scheme:
\[\forall i<2 \exists \gamma \in \mathcal{C}[R(i,\gamma)] \rightarrow 
\exists \gamma \in \mathcal{C}\forall i<2[R(i,\gamma^{\upharpoonright i})].\]
and the latter  scheme is provable in $\mathsf{BIM}$.

 For each $\alpha$, we define the following statement,  
  called $\mathbf{LLPO}^\alpha$: 
\[\forall\varepsilon[\forall p [2p=\mu m[\varepsilon(m)\neq 0]   \rightarrow \mathit{Bar}_\mathcal{C}(D_{\overline \alpha p})]\;\vee\;\]\[ \forall p [2p+1=\mu m[\varepsilon(m)\neq 0]   \rightarrow \mathit{Bar}_\mathcal{C}(D_{\overline \alpha p})]].\]
 
\begin{lemma}\label{L:fteq2} \begin{enumerate}[\upshape (i)] \item $\mathsf{BIM}\vdash \mathbf{LLPO}\leftrightarrow \forall \alpha[\mathbf{LLPO}^\alpha]$.\item $\mathsf{BIM}\;+\mathbf{\Sigma}^0_1$-$\overleftarrow{\mathbf{AC}_{2,\mathcal{C}}}
\vdash \forall \alpha[Bar_\mathcal{C}(D_\alpha)\rightarrow \mathbf{LLPO}^\alpha]$.  \item $\mathsf{BIM}\;+$  $\mathbf{\Pi}^0_1$-$\mathbf{AC}_{\omega,2}\vdash\forall \alpha[Bar_\mathcal{C}(D_\alpha)\rightarrow \mathbf{LLPO}^\alpha]\rightarrow\mathbf{FT}$. \item $\mathsf{BIM}\vdash\mathbf{FT}\rightarrow \mathbf{\Sigma}^0_1$-$\overleftarrow{\mathbf{AC}_{2,\mathcal{C}}}$. \end{enumerate}

 \end{lemma}
 
 \begin{proof} (i) Assume $\mathbf{LLPO}$ and let $\alpha, \varepsilon$ be given.  \\ \textit{Either} $\forall p[2p\neq \mu n[\varepsilon(n)\neq 0]]$ and, therefore,  $\forall\varepsilon[\forall p [2p=\mu m[\varepsilon(m)\neq 0]   \rightarrow \mathit{Bar}_\mathcal{C}(D_{\overline \alpha p})]$, \\\textit{or} $\forall p[2p+1\neq \mu n[\varepsilon(n)\neq 0]]$ and: $\forall\varepsilon[\forall p [2p+1=\mu m[\varepsilon(m)\neq 0]   \rightarrow \mathit{Bar}_\mathcal{C}(D_{\overline \alpha p})]$. \\We thus see: $\mathbf{LLPO}^\alpha$. 
 
 For the converse, note: $\mathbf{LLPO}\leftrightarrow \mathbf{LLPO}^{\underline 0}$. 
 
 \smallskip (ii) Let $\alpha$ be given such that $Bar_\mathcal{C}(D_\alpha)$. \\Using $\mathbf{\Sigma}_0^1$-$\overleftarrow{\mathbf{AC}_{2,\mathcal{C}}}$, we now prove: $\mathbf{LLPO}^\alpha$. 
 
 Let $\varepsilon$ be given.
 Define $\eta$  such that, for each $p$,
\begin{enumerate}
\item  if $ \underline{\overline 0}(2p+2)\sqsubset\varepsilon$, then $\eta^{\upharpoonright 0}(p) = \eta^{\upharpoonright 1}(p) = \alpha(p)$, and,
\item   if $ 2p =\mu m [\varepsilon(m) \neq 0 ]$,  then $\forall m \ge p[\eta^{\upharpoonright 0}(m) =0\;\wedge\;\eta^{\upharpoonright 1}(m) = \alpha(m)]$, and
 \item if $  2p+1=\mu m[\varepsilon(m) \neq 0 ]$,  then $\forall m \ge p[\eta^{\upharpoonright 1}(m) =0\;\wedge\;\eta^{\upharpoonright 0}(m) = \alpha(m)]$.
\end{enumerate}

Note: if $\eta^{\upharpoonright 0} \;\#\; \alpha$, then $\eta^{\upharpoonright 1} = \alpha$.  

 Let $\gamma$ in $\mathcal{C}$ be given. Find $n$ such that $\alpha(\overline{\gamma^{\upharpoonright 0}}n) \neq 0$. \\Either: $\eta^{\upharpoonright 0}(\overline{\gamma^{\upharpoonright 0}}n) = \alpha(\overline{\gamma^{\upharpoonright 0}}n)\neq 0$,\\ or: $\eta^{\upharpoonright 0} \;\#\; \alpha$ and $\eta^{\upharpoonright 1} = \alpha$ and $\exists m[\eta^{\upharpoonright 1}(\overline{\gamma^{\upharpoonright 1}}m) = \alpha(\overline{\gamma^{\upharpoonright 1}}m) \neq 0]$.
 
We thus see: $\forall \gamma \in \mathcal{C}[\gamma^{\upharpoonright 0}  \in \mathcal{G}_{\eta^{\upharpoonright 0}} \;\vee \;\gamma^{\upharpoonright 1} \in \mathcal{G}_{\eta^{\upharpoonright 1}}]$.

Use $\mathbf{\Sigma}_0^1$-$\overleftarrow{\mathbf{AC}_{2,\mathcal{C}}}$ and find $i<2$ such that $\mathcal{C}\subseteq \mathcal{G}_{\eta^{\upharpoonright i}}$.

Assume: $ \mathcal{C}\subseteq  \mathcal{G}_{\eta^{\upharpoonright i}} $, i.e. $Bar_\mathcal{C}(D_{\eta^{\upharpoonright i}})$. 
Assume: $ 2p+i= \mu m [\varepsilon(m) \neq 0] $. \\Note:  $\forall m\ge p[\eta^{\upharpoonright i}(m) = 0]$. Conclude: $\mathit{Bar}_\mathcal{C}(D_{\overline {\eta^{\upharpoonright i}} p})$ and: $\mathit{Bar}_\mathcal{C}(D_{\overline \alpha p})$. 

We thus see:
$\forall p[2p+i= \mu m [\varepsilon(m) \neq 0] \rightarrow\mathit{Bar}_\mathcal{C}(D_{\overline \alpha p})]$.

Conclude: $\exists i<2\forall p[2p+i= \mu m [\varepsilon(m) \neq 0] \rightarrow\mathit{Bar}_\mathcal{C}(D_{\overline \alpha p})]$, i.e. $\mathbf{LLPO}^\alpha$.

\smallskip (iii) Assume $\forall \alpha[Bar_\mathcal{C}(D_\alpha)\rightarrow \mathbf{LLPO}^\alpha]$. \\ Using $\mathbf{\Pi}^0_1$-$\mathbf{AC}_{\omega,2}$, we now prove: $\mathbf{FT}$.

Let $\alpha$ be given such that $Bar_\mathcal{C}(D_\alpha)$ and, therefore, $\mathbf{LLPO}^\alpha$.

Assume: $s\in\mathit{Bin}$. Define $\varepsilon$ such that,  $\forall i<2\forall n[\varepsilon(2n+i) \neq 0 \leftrightarrow Bar_{\mathcal{C}\cap s\ast \langle i \rangle}(D_{\overline\alpha n})] $. Using $\mathbf{LLPO}^\alpha$, find $i<2$ such that 
$\forall p[2p+i= \mu m [\varepsilon(m) \neq 0] \rightarrow\mathit{Bar}_\mathcal{C}(D_{\overline \alpha p})]$.

 Assume we find $n$ such that $Bar_{\mathcal{C}\cap s\ast \langle i\rangle}(D_{\overline \alpha n})$. Then $\varepsilon(2n+i) \neq 0$. \\Find $p:= \mu j[\varepsilon(j) \neq 0]$. Find $q\le n$ such that $p=2q$ or $p=2q+1$. \\\textit{Either} $p=2q+i$ 
and $\mathit{Bar}_\mathcal{C}(D_{\overline \alpha q})$, \textit{or} $p=2q+1-i$ and $Bar_{\mathcal{C}\cap s\ast\langle 1-i\rangle} (D_{\overline \alpha q})$. \\In both cases: $Bar_{\mathcal{C}\cap s\ast\langle 1-i\rangle}(D_{\overline \alpha n})$.   \\We thus see: $\forall n[Bar_{\mathcal{C}\cap s\ast\langle i \rangle} (D_{\overline \alpha n}) \rightarrow Bar_{\mathcal{C}\cap s\ast\langle 1-i \rangle} (D_{\overline \alpha n}) ]$.

Conclude: $\forall s \in \mathit{Bin}\exists i<2\forall n[   Bar_{\mathcal{C}\cap s\ast \langle  i \rangle}( D_{\overline \alpha n})\rightarrow Bar_{\mathcal{C}\cap s\ast\langle 1- i \rangle} (D_{\overline \alpha n})]$.

\smallskip
Now
use  $\mathbf{\Pi}^0_1$-$\mathbf{AC}_{\omega,2}$ and find $\gamma$ in $\mathcal{C}$ such that 

$\forall s \in \mathit{Bin}
 \forall n[ Bar_{\mathcal{C}\cap s\ast \langle \gamma(s) \rangle}( D_{\overline \alpha n})\rightarrow Bar_{\mathcal{C}\cap s\ast\langle 1-\gamma(s) \rangle} (D_{\overline \alpha n})]$.
 
Observe that, for each $s$ in $\mathit{Bin}$, for all $n$, if $Bar_{\mathcal{C}\cap s\ast \langle \gamma(s) \rangle}( D_{\overline \alpha n})$, then also $Bar_{\mathcal{C}\cap s\ast \langle 1- \gamma(s) \rangle}( D_{\overline \alpha n})$,
and, therefore, $Bar_{\mathcal{C}\cap s}( D_{\overline \alpha n})$.

\smallskip

Define $\delta$ in $\mathcal{C}$ such that, for each $n$, $\delta(n) = \gamma(\overline{ \delta} n)$. Find $p$ such that $\alpha(\overline{ \delta} p) \neq 0$ and define $n:=\overline\delta p +1$. Note: $Bar_{\mathcal{C}\cap\overline\delta p}(D_{\overline \alpha n})$.  One may prove, by backwards induction: for each $j \le p$,  $Bar_{\mathcal{C}\cap \overline \delta j}(D_{\overline \alpha n})$. For assume $j+1 \le n$ and $Bar_{\mathcal{C}\cap \overline{\delta} (j+1)}(D_{\overline \alpha n})$.  As $\overline{ \delta}(j+1) = \overline{ \delta} j \ast \langle  \gamma(\overline{\delta} j)\rangle$, one may conclude:   $Bar_{\mathcal{C}\cap \overline{\delta} (j)}(D_{\overline \alpha n})$.
After $n$ steps we obtain the conclusion:  $Bar_\mathcal{C}(D_{\overline\alpha n})$.
 
 Conclude: $\forall \alpha[Bar_\mathcal{C}(D_\alpha)\rightarrow \exists n[Bar_\mathcal{C}(D_{\overline \alpha n})]]$, i.e. $\mathbf{FT}$. 
  
 \smallskip (iv) Assume $\mathbf{FT}$. Use Theorem \ref{T:ccc} and conclude: 
  $ \mathbf{\Sigma}^0_1$-$\overleftarrow{\mathbf{AC}_{\omega,\mathcal{C}}}$ and its corollary: $ \mathbf{\Sigma}^0_1$-$\overleftarrow{\mathbf{AC}_{2,\mathcal{C}}}$. \end{proof}
 
 \begin{theorem}\label{T:fteq2}$\mathsf{BIM}+\mathbf{\Pi}^0_1$-$\mathbf{AC}_{\omega,2}\vdash \mathbf{\Sigma}^0_1$-$\overleftarrow{\mathbf{AC}_{2,\mathcal{C}}}\leftrightarrow \mathbf{FT}$. \end{theorem}
 
 \begin{proof} Use Lemma \ref{L:fteq2}.
 \end{proof}
 
 \subsubsection{}\label{SSS:bickford} Mark Bickford called my attention to the fact that  $\mathbf{\Sigma}^0_1$-$\overleftarrow{\mathbf{AC}_{2, \mathcal{C}}}$ occurs in \cite[\S 2]{moschovakiswkl} and is called there the \textit{separation principle} $\mathrm{SP}$. \\
  After having proven $\mathbf{FT}\rightarrow \mathbf{\Sigma}^0_1$-$\overleftarrow{\mathbf{AC}_{2,\mathcal{C}}}$, see our Lemma \ref{L:fteq2}(iv), the author of \cite{moschovakiswkl} gives  a proof of $\mathbf{FT}\rightarrow \mathbf{WKL!}$.  She does so as follows. \begin{quote} Assume $\mathbf{FT}$. 
  \\Let $\alpha$ be given such that $\forall n[\neg Bar_\mathcal{C}(D_{\overline\alpha n})]$ and \\ $\forall\gamma\in \mathcal{C} [\gamma^{\upharpoonright 0} \perp\gamma^{\upharpoonright 1}\rightarrow \exists i<2\exists n[\alpha(\overline{\gamma^{\upharpoonright i}}n)\neq 0 ]]$. \\We intend to prove: $\exists \gamma \in \mathcal{C}\forall n[\alpha(\overline \gamma n)=0]$. 
 \\Define $\alpha^\ast$ such that $\forall s\in Bin[\alpha^\ast(s)=0\leftrightarrow \forall t\sqsubseteq s[\alpha(t)=0]]$. \\Using induction, we prove: \\$(\ast)$: $\forall n\exists ! s \in Bin_n\forall m>n\exists t \in Bin_m[s\sqsubset t\;\wedge\;\alpha^\ast(t) =0]]$.  \\If $n=0$, then $s=\langle\;\rangle$ satisfies the requirements.
  \\Let $n$ be given and let $s$ be the unique element of $Bin_n$ such that  $\forall m\ge n\exists t \in Bin_m[s\sqsubseteq t\;\wedge\;\alpha^\ast(t)=0]$. \\Observe: $\forall \gamma \in \mathcal{C} \exists i<2\exists n[\alpha(s\ast\langle i \rangle\ast \overline{\gamma^{\upharpoonright i}}n)\neq 0]$.\\Using $\mathbf{\Sigma}^0_1$-$\overleftarrow{\mathbf{AC}_{2,\mathcal{C}}}$, find $j<2$ such that $\forall\gamma \in \mathcal{C}\exists n[\alpha(s\ast\langle j \rangle\ast \overline{\gamma}n)\neq 0]$. \\Using $\mathbf{FT}$, find $m$ such that $\forall\gamma \in \mathcal{C}\exists n\le m[\alpha(s\ast\langle j \rangle\ast \overline{\gamma}n)\neq 0]$. \\Define $p:= n+1+m$ and note: \\ $\forall q> p\forall t \in Bin_q[s\ast\langle j \rangle\sqsubseteq t \rightarrow \alpha^\ast(t)\neq 0]$ and:\\ $s\ast\langle 1-j\rangle$ is the unique element $u$ of $Bin_{n+1}$ such that \\$\forall q\ge n+1\exists t\in Bin_q[u\sqsubseteq t \;\wedge\; \alpha^\ast(t)=0]$.
  \\This completes our proof of $(\ast)$.
  \\ Using $(\ast)$ and $\mathbf{\Pi}^0_1$-$\mathbf{AC}_{\omega,\omega}!$, find $\delta$ such that \\$\forall n[\delta(n)\in Bin_n\;\wedge\;\forall m\ge n\exists t \in Bin_m[\delta(n)\sqsubseteq t\;\wedge \alpha^\ast(t)=0]]$. \\Note: $\forall n[\delta(n)\sqsubset\delta(n+1)\;\wedge\;\alpha^\ast\bigl(\delta(n)\bigr)=0]]$.  \\Define $\gamma$ such that $\forall n[\delta(n)\sqsubset \gamma]$ and note: $\gamma \in \mathcal{C}$ and $\forall n[\alpha(\overline\gamma n)=0]$. \end{quote}

 The author of \cite{moschovakiswkl} quotes the result $\mathbf{WKL}!\rightarrow \mathbf{FT}$, see our Theorem \ref{T:wkl!ft}.
  \\She does not prove $\mathbf{FT}$, or equivalently, $\mathbf{WKL!}$,  from $\mathbf{\Sigma}^0_1$-$\overleftarrow{\mathbf{AC}_{2,\mathcal{C}}}$ and   $\mathbf{\Pi}^0_1$-$\mathbf{AC}_{\omega,\omega}!$.

    Our proof of Theorem \ref{T:ftwkl!} shows that no choice is needed for a proof of \\$\mathbf{FT}\rightarrow\mathbf{WKL!}$. 
\bigskip
 
We introduce a `real' version of  $\mathbf{\Sigma}^0_1$-$\overleftarrow{\mathbf{AC}_{2,\mathcal{C}}}$: \subsection{}$\mathbf{\Sigma}^0_1$-$\overleftarrow{\mathbf{AC}_{2,[0,1]}}$:
  \[\forall \alpha[  \forall \delta \in [0,1]^2
  \exists i<2 [\delta^{\upharpoonright i} \in \mathcal{H}_{\alpha^{\upharpoonright i}}] \rightarrow \exists i<2 [ [0,1] \subseteq \mathcal{H}_{\alpha^{\upharpoonright i}}] ].\]

\begin{theorem}\label{T:c0c1}
 $\mathsf{BIM}\vdash \mathbf{\Sigma}^0_1$-$\overleftarrow{\mathbf{AC}_{2,\mathcal{C}}}\leftrightarrow\mathbf{\Sigma}^0_1$-$\overleftarrow{\mathbf{AC}_{2,[0,1]}}$.

\end{theorem}

\begin{proof}
Using Lemma \ref{L:Conto[0,1]}, find  $\sigma: \mathcal{C} \rightarrow [0,1]$ and $\psi: \omega^\omega\rightarrow\omega^\omega$ such that \\ $\forall \delta \in [0,1]  \exists \gamma[\delta =_\mathcal{R} \sigma|\gamma]$ and  $\forall \alpha\forall \gamma \in \mathcal{C}[\gamma\in \mathcal{G}_{\psi|\alpha}\leftrightarrow \sigma|\gamma \in \mathcal{H}_\alpha]$. 

\smallskip

First assume: $ \mathbf{\Sigma}^0_1$-$\overleftarrow{\mathbf{AC}_{2,\mathcal{C}}}$.

Let $\alpha$ be given such that   $\forall \delta \in [0,1]^2\exists i<2[\delta^{\upharpoonright i} \in \mathcal{H}_{\alpha^{\upharpoonright i}}]$. Define $\beta$ such that, for both $i<2$, $\beta^{\upharpoonright i} = \psi|(\alpha^{\upharpoonright i})$.
 Then $\forall \gamma \in \mathcal{C}\exists i < 2[\sigma|\gamma^{\upharpoonright i} \in \mathcal{H}_{\alpha^{\upharpoonright i}}]$ and, therefore: $\forall \gamma \in \mathcal{C} \exists i <2[\gamma^{\upharpoonright i} \in 
 \mathcal{G}_{\beta^{\upharpoonright i}}]$.
Find $ i <2$ such that  $ \mathcal{C}\subseteq \mathcal{G}_{\beta^{\upharpoonright i}}$. Conclude:
 $ [0,1]\subseteq \mathcal{H}_{\alpha^{\upharpoonright i}}$.
 
 We nay conclude: $\mathbf{\Sigma}^0_1$-$\overleftarrow{\mathbf{AC}_{2,[0,1]}}$.

\smallskip
  Now assume: $\mathbf{\Sigma}^0_1$-$\overleftarrow{\mathbf{AC}_{2,[0,1]}}$.
  
  Let $\alpha$ be given such that 
  
  Using Lemma \ref{L:Cinto[0,1]}, find $\tau:\mathcal{C}\rightarrow [0,1]$ and $\chi:\omega^\omega\rightarrow\omega^\omega$ 
  such that  \begin{enumerate}[\upshape (1)]\item $\forall \gamma\in
   \mathcal{C}\forall \delta\in\mathcal{C}[\gamma\;\#\;\delta\rightarrow \tau|\gamma\;\#_\mathcal{R}\;\tau|\delta]$, and  
  \item $\forall \alpha \forall \gamma \in \mathcal{C}[\gamma \in \mathcal{G}_\alpha\leftrightarrow \tau|\gamma\in\mathcal{H}_{\chi|\alpha}]$ and \item $\forall \delta \in [0,1]\exists\gamma\in\mathcal{C}[\delta\;\#_\mathcal{R}\;\tau|\gamma \rightarrow\forall \alpha[\delta\in\mathcal{H}_{\chi|\alpha}]]$. \end{enumerate}

   Let $\alpha$ be given such that $\forall \gamma \in \mathcal{C}\exists i < 2[\gamma^{\upharpoonright i} \in \mathcal{G}_{\alpha^{\upharpoonright i}} ] $.
  
 Let $\delta$ in $[0,1]^2$ be given.  
 
 Find $\gamma$ in $\mathcal{C}$ such that $\forall i<2[\delta^{\upharpoonright i}\;\#_\mathcal{R}\;\tau|(\gamma^{\upharpoonright i}) \rightarrow\delta^{\upharpoonright i}\in\mathcal{H}_{\chi|(\alpha^{\upharpoonright i})}]$. \\Find $i<2$ such that $\gamma^{\upharpoonright i}\in\mathcal{G}_{\alpha^{\upharpoonright i}}$. \\Conclude: $\tau|(\gamma^{\upharpoonright i})\in\mathcal{H}_{\chi|\alpha^{\upharpoonright i}}$. Find $s,n$ such that $(\chi|\alpha^{\upharpoonright i})(s)\neq 0$ and $\bigl(\tau|(\gamma^{\upharpoonright i})\bigr)(n)\sqsubset_\mathbb{S} s$. \\Using Lemma \ref{L:apart}, find $p$ such that \textit{either}: $\delta^{\upharpoonright i}(p)\sqsubset_\mathbb{S} s$, and, therefore: $\delta^{\upharpoonright i}\in \mathcal{H}_{\chi|(\alpha^i)}$,  \\\textit{or}: $\delta^{\upharpoonright i}(p)\;\#_\mathbb{S}\; \bigl(\tau|(\gamma^{\upharpoonright i})\bigr)(n)$, and, therefore, again: $\delta^{\upharpoonright i}\in \mathcal{H}_{\chi|(\alpha^i)}$.
 
We thus see:  $\forall \delta \in [0,1]^2\exists i<2[\delta^{\upharpoonright i} \in \mathcal{H}_{\chi|(\alpha^{\upharpoonright i})}]$.

 Applying $\mathbf{\Sigma}^0_1$-$\overleftarrow{\mathbf{AC}_{2,[0,1]}}$,  find $i<2$  such that $[0,1]\subseteq\mathcal{H}_{\chi|(\alpha^{\upharpoonright i})}$.
 
  Conclude: $\forall \gamma \in \mathcal{C}[\tau|\gamma\in \mathcal{H}_{\chi|(\alpha^{\upharpoonright i})}]$ and: $\mathcal{C}\subseteq \mathcal{G}_{\alpha^{\upharpoonright i}}$. 
  
  We may conclude: $ \mathbf{\Sigma}^0_1$-$\overleftarrow{\mathbf{AC}_{2,\mathcal{C}}}$.
 \end{proof}

\begin{corollary} $\mathsf{BIM} + \mathbf{\Pi}^0_1$-$\mathbf{AC}_{\omega,2}\vdash\mathbf{FT} \leftrightarrow \mathbf{\Sigma}^0_1$-$\overleftarrow{\mathbf{AC}_{2,[0,1]}}$.\end{corollary} \begin{proof} Use Theorems \ref{T:fteq2} and \ref{T:c0c1}. \end{proof}

\subsection{Some logical consequences}\hfill

In this Subsection, we want to formulate the result of Theorem \ref{T:fteq2} in model-theoretic terms and draw an even sharper conclusion. 

\smallskip
For every $\delta$, we define a proposition $Pr_\delta$, as follows:  $Pr_\delta := \exists n[\delta(n) \neq 0]$.

\smallskip
Tarski's truth definition makes sense intuitionistically as well as classically.

For every structure $\mathfrak{A}=(A, \ldots)$, for every  sentence $\varphi$ in the elementary language of the structure 
$\mathfrak{A}$, we write: \[\mathfrak{A}\models \varphi\] if the sentence $\varphi$ is true in the structure $\mathfrak{A}$.
\\More generally,  for every structure $\mathfrak{A}=(A, \ldots)$, for every  formula\\ 
$\varphi=\varphi(\mathsf{x_0, x_1, \ldots, x_{n-1}})$ in the elementary language of the structure $\mathfrak{A}$, \\for all $a_0, a_1, \ldots, a_{n-1}$ in $A$,
 we write: \[\mathfrak{A}\models \varphi[a_0, a_1, \dots, a_{n-1}]\] if the formula $\varphi$ is true in the structure $\mathfrak{A}$, provided we interpret the individual variables $\mathsf{x_0, x_1, \ldots, x_{n-1}}$ by $a_0, a_1, \ldots, a_{n-1}$, respectively.

\begin{theorem}\label{T:ftlogiceqprep}
The following statements are equivalent in $\mathsf{BIM}$:
\begin{enumerate}[\upshape (i)]
\item  $\forall \alpha[(\mathcal{C}, \mathcal{G}_{\alpha^{\upharpoonright 0}}, Pr_{\alpha^{\upharpoonright 1}})\models \mathsf{\forall x [P(x) \vee A] \rightarrow (\forall x[P(x)] \vee A)}$.
\item  $\forall \alpha[(\mathcal{C}, \mathcal{G}_{\alpha^{\upharpoonright 0}}, \mathcal{G}_{\alpha^{\upharpoonright 1}})\models\mathsf{\forall x[P(x) \vee Q(x)] \rightarrow (\forall x[P(x)] \vee \exists x[Q(x)])}$.
\end{enumerate}
\end{theorem}
\begin{proof} (i) $\Rightarrow$ (ii). Let $\alpha$ be given such that $(\mathcal{C}, \mathcal{G}_{\alpha^{\upharpoonright 0}}, \mathcal{G}_{\alpha^{\upharpoonright 1}}) \models \mathsf{\forall x [P(x) \vee Q(x)]}$.  \\Note: $(\mathcal{C}, \mathcal{G}_{\alpha^{\upharpoonright 0}}, \mathcal{G}_{\alpha^{\upharpoonright 1}}) \models \mathsf{\forall x [P(x) \vee \exists y[Q(y)]]}$. \\Define $\beta$ such that $\beta^{\upharpoonright 0} = \alpha^{\upharpoonright 0}$ and $\forall n[\beta^{\upharpoonright 1}(n)\neq 0\leftrightarrow \bigl(\alpha^{\upharpoonright 1}(n)\neq 0 \;\wedge\;n\in\mathit{Bin}\bigr)]$. Note: $(\mathcal{C}, \mathcal{G}_{\beta^{\upharpoonright 0}}, Pr_{\beta^{\upharpoonright 1}}) \models \mathsf{\forall x [P(x) \vee A]}$, and thus, according to (i): \\ $(\mathcal{C}, \mathcal{G}_{\beta^{\upharpoonright 0}}, Pr_{\beta^{\upharpoonright 1}}) \models \mathsf{\forall x [P(x)] \vee A}$, and,  therefore: $(\mathcal{C}, \mathcal{G}_{\alpha^{\upharpoonright 0}}, \mathcal{G}_{\alpha^{\upharpoonright 1}}) \models \mathsf{\forall x [P(x)] \vee \exists x[Q(x)]}$.
 
\smallskip
(ii) $\Rightarrow$ (i). Let $\alpha$ be  given such that $(\mathcal{C}, \mathcal{G}_{\alpha^{\upharpoonright 0}}, Pr_{\alpha^{\upharpoonright 1}}) \models \mathsf{\forall x[P(x) \vee A]}$. \\Define $\beta$ such that $\beta^{\upharpoonright 0} = \alpha^{\upharpoonright 0}$ and  $\forall n[\beta^{\upharpoonright 1}(n)\neq 0\leftrightarrow \bigl(\exists i<n[\alpha^{\upharpoonright 1}(i) \neq 0] \;\wedge\;n \in \mathit{Bin}\bigr)]$. 
\\Note that, for each $\gamma$ in $\mathcal{C}$, $\gamma \in \mathcal{G}_{\beta^{\upharpoonright 1}}$ if and only if $Pr_{\alpha^{\upharpoonright 1}}$. \\Therefore, $(\mathcal{C}, \mathcal{G}_{\beta^{\upharpoonright 0}}, \mathcal{G}_{\beta^{\upharpoonright 1}}) \models \mathsf{\forall x[P(x) \vee Q(x)]}$ and, according to (ii):\\
$(\mathcal{C}, \mathcal{G}_{\beta^{\upharpoonright 0}}, \mathcal{G}_{\beta^{\upharpoonright 1}})\models \mathsf{\forall x[P(x)] \vee \exists x[Q(x)]}$, and thus:
$(\mathcal{C}, \mathcal{G}_{\alpha^{\upharpoonright 0}}, Pr_{\alpha^{\upharpoonright 1}}) \models \mathsf{\forall x[P(x)] \vee A}$.
\end{proof}
For each $\alpha$, we define  the following statement:
\[\mathbf{LPO}^\alpha:\;\forall \varepsilon[\forall p[\overline{\underline 0}(2p+2)\perp\varepsilon\rightarrow \mathit{Bar}_\mathcal{C}(D_{\overline \alpha p} )]\;\vee \;\exists n[\varepsilon (n) \neq 0]].\]

\begin{lemma}\label{L:lpoalpha}\begin{enumerate}[\upshape (i)]
\item $\mathsf{BIM}\vdash \mathbf{LPO}\rightarrow \forall \alpha[\mathbf{LPO}^\alpha]$.
\item $\mathsf{BIM}\vdash \forall \alpha[\mathbf{LPO}^\alpha\rightarrow \mathbf{LLPO}^\alpha]$.

\end{enumerate} \end{lemma}
\begin{proof} (i) Given any $\varepsilon$, by $\mathbf{LPO}$,\\ \textit{either} $\varepsilon=\underline 0$, and, therefore: $\forall p[\overline{\underline 0}(2p+2)\perp\varepsilon\rightarrow \mathit{Bar}_\mathcal{C}(D_{\overline \alpha p} )]$, \textit{or} $\exists n[\varepsilon(n)\neq 0]$. 

\smallskip (ii) Let $\alpha$ be given such that $\mathbf{LPO}^\alpha$. We have to prove, for all $\varepsilon$, either \\$\forall p[2p=\mu m[\varepsilon(m)\neq 0] \rightarrow \mathit{Bar}_\mathcal{C}(D_{\overline \alpha p})]$ or $ \forall p[2p+1=\mu m[\varepsilon(m)\neq 0 ]\rightarrow\mathit{Bar}_\mathcal{C}(D_{\overline \alpha p})]$.

 Let $\varepsilon$ be given.  Use $\mathbf{LPO}^\alpha$ and distinguish two cases. 
 
 \textit{Case (1)}. $\forall p[ \overline{\underline 0}(2p+2)\perp\varepsilon\rightarrow \mathit{Bar}_\mathcal{C}(D_{\overline \alpha p}) ]$. \\Then $\forall p[2p=\mu m[\varepsilon(m)\neq 0]  \rightarrow\mathit{Bar}_\mathcal{C}(D_{\overline \alpha p})]$, and also\\ $\forall p[2p+1=\mu p[\varepsilon(m)\neq 0]\rightarrow\mathit{Bar}_\mathcal{C}(D_{\overline \alpha p})]$.
 
 \textit{Case (2)}. $\exists n[\varepsilon(n) \neq 0]$. Define $m:=\mu n[\varepsilon(n) \neq 0]$. \\Find $q$ such that $m = 2q$ or $m = 2q+1$. \\If $m=2q+1$, then $\forall p[2p=\mu m[\varepsilon(m)\neq 0]\rightarrow\mathit{Bar}_\mathcal{C}(D_{\overline \alpha p})]$. \\If $m= 2q$, then 
$\forall p[2p+1=\mu m[\varepsilon(m)\neq 0]\rightarrow \mathit{Bar}_\mathcal{C}(D_{\overline \alpha p})]$.

Conclude: $\mathbf{LLPO}^\alpha$.\end{proof} 
\begin{theorem}\label{T:ftlogiceq}
 The following statements are equivalent in $\mathsf{BIM}\;+\;\mathbf{\Pi}^0_1$-$\mathbf{AC}_{\omega,2}$:
 \begin{enumerate}[\upshape(i)]

\item $\mathbf{FT}$.
\item  $\forall \alpha[(\mathcal{C}, \mathcal{G}_{\alpha^{\upharpoonright 0}}, \mathcal{G}_{\alpha^{\upharpoonright 1}})\models \mathsf{\forall x \forall y[P(x) \vee Q(y)] \rightarrow (\forall x[P(x)] \vee \forall y[Q(y)])}] $.
\item  $\forall \alpha[(\mathcal{C}, \mathcal{G}_{\alpha^{\upharpoonright 0}}, Pr_{\alpha^{\upharpoonright 1}})\models\mathsf{\forall x [P(x) \vee A] \rightarrow (\forall x[P(x)] \vee A)}] $.
\item $\forall \alpha[Bar_\mathcal{C}(D_\alpha)\rightarrow \mathbf{LPO}^\alpha]$. 
\end{enumerate}\end{theorem}

\begin{proof}

(i) $\Rightarrow$ (ii). Note that, in $\mathsf{BIM} + \mathbf{\Pi}^0_1$-$\mathbf{AC}_{\omega, 2}$, $\mathbf{FT}$ implies $\mathbf{\Sigma}^0_1$-$\overleftarrow{\mathbf{AC}_{2,\mathcal{C}}}$:\\
  $\forall \alpha[  \forall \gamma \in \mathcal{C}
  \exists i<2 [\gamma^{\upharpoonright i} \in \mathcal{G}_{\alpha^{\upharpoonright i}}] \rightarrow \exists i<2 [\mathcal{C} \subseteq \mathcal{G}_{\alpha^{\upharpoonright i}}] ]$,  see Theorem \ref{T:fteq2}. 

\smallskip
(ii) $\Rightarrow$ (iii). Let $\alpha$ be given such that $(\mathcal{C}, \mathcal{G}_{\alpha^{\upharpoonright 0}}, Pr_{\alpha^{\upharpoonright 1}})\models\mathsf{\forall x [P(x) \vee A]}$, i.e. \\$\forall \gamma \in \mathcal{C}\exists n \exists i<2[\alpha^{\upharpoonright 0}(\overline \gamma n)\neq 0\;\vee\; \alpha^{\upharpoonright 1}(n)\neq 0]$. \\Define $\beta$ such that $\beta^{\upharpoonright 0}=\alpha^{\upharpoonright 0}$ and $\forall n\forall s \in Bin_n[\beta^{\upharpoonright 1}(s)= \alpha^{\upharpoonright 1}(n)]$. 
\\Note $(\mathcal{C}, \mathcal{G}_{\beta^{\upharpoonright 0}}, \mathcal{G}_{\beta^{\upharpoonright 1}})\models \mathsf{\forall x \forall y[P(x) \vee Q(y)]}$ and conclude, using (i), \\$(\mathcal{C}, \mathcal{G}_{\beta^{\upharpoonright 0}}, \mathcal{G}_{\beta^{\upharpoonright 1}})\models \mathsf{\forall x[P(x)] \vee \forall y[Q(y)]}$ and $(\mathcal{C}, \mathcal{G}_{\alpha^{\upharpoonright 0}}, Pr_{\alpha^{\upharpoonright 1}})\models\mathsf{ \forall x[P(x)] \vee A} $.

\smallskip
(iii) $\Rightarrow$ (iv). Assume (iii).

Let $\alpha$ be given such that $\mathit{Bar}_\mathcal{C}(D_\alpha)$. 
We  have to prove: $\mathbf{LPO}^\alpha$.\\
Let $\varepsilon$ be given.\\  
Define $\eta$ in $\mathcal{C}$ such that $\eta^{\upharpoonright 1} = \varepsilon$, and, for each $p$,
if $ \underline{\overline 0}(2p+2)\sqsubset\varepsilon$, then $\eta^{\upharpoonright 0}(p) = \alpha(p)$, and,
 if   $\underline{\overline 0}(2p+2)\perp \varepsilon$,  then $\eta^{\upharpoonright 0}(p) =0$.
Note: if $\eta^{\upharpoonright 0} \;\#\; \alpha$, then $\exists n[\varepsilon(n) =\eta^{\upharpoonright 1} (n) \neq 0]$.

\smallskip
   Let $\gamma$ in $\mathcal{C}$ be given. Find $n$ such that $\alpha(\overline \gamma n) \neq 0$.\\Either: $\eta^0(\overline{\gamma}n) = \alpha(\overline{\gamma}n)\neq 0$ or: $\eta^{\upharpoonright 0} \;\#\; \alpha$ and $\exists m[\eta{\upharpoonright 1}(m) \neq 0]$. \\We thus see: $\forall \gamma \in \mathcal{C}[(\exists n[\eta^{\upharpoonright 0}(\overline{\gamma}n) \neq 0])\;\vee\;(\exists m[\eta^{\upharpoonright 1}(m) \neq 0])]$.\\
 Use (iii) and conclude:   $\forall \gamma \in \mathcal{C} \exists n[\eta^{\upharpoonright 0}(\overline \gamma n) = 1]\;\vee\;\exists m[\eta^{\upharpoonright 1}(m) \neq 0]$.

Assume: $\forall \gamma \in \mathcal{C}\exists n[\eta^{\upharpoonright 0}(\overline \gamma n) \neq 0]$, that is: $\mathit{Bar}_\mathcal{C}(D_{\eta^{\upharpoonright 0}})$. 
Let $p$ be given such that  $ \underline{\overline 0}(2p+2)\perp\varepsilon$. Note $\forall m\ge p[\eta^{\upharpoonright 0}(m) = 0]$. Conclude: $\mathit{Bar}_\mathcal{C}(D_{\overline{\eta^{\upharpoonright 0}} p})$ and: $\mathit{Bar}_\mathcal{C}(D_{\overline \alpha p})$. \\We thus see: $\forall p[ \underline{\overline 0}(2p+2)\perp\varepsilon \rightarrow \mathit{Bar}_\mathcal{C}(D_{\overline \alpha p})]$.

Assume: $\exists n[\eta^{\upharpoonright 1}(n)\neq 0]$, Then, of course: $\exists n [\varepsilon(n) \neq 0]$.

We thus see: $\forall\varepsilon[\overline{\underline 0}(2p+2)\perp \varepsilon \rightarrow Bar_\mathcal{C}(D_{\overline \alpha p})]\;\vee\;\exists n[\varepsilon (n)\neq 0]]$, i.e. $\mathbf{LPO}^\alpha$. 

\smallskip (iv) $\Rightarrow$ (i). Use Lemma \ref{L:lpoalpha}(ii) and Theorem \ref{T:fteq2}(iii).
\end{proof}

The sentences mentioned in Theorems \ref{T:ftlogiceqprep} and \ref{T:ftlogiceq} are true in every structure $(\{0, 1, \ldots, n\}, P, Q, A)$ where $n$ is a natural number, $P, Q$ are arbitrary subsets of $\{0,1, \ldots, n\}$ and $A$ is an arbitrary proposition, that is,  these sentences hold in every \textit{finite} structure.
They sometimes fail to be true in countable structures, as appears from the next two theorems.

\begin{theorem}\label{T:llpological} The following statements are equivalent in $\mathsf{BIM}$. \begin{enumerate}[\upshape(i)]
\item $\mathbf{LLPO}$.
\item $\forall \alpha [(\omega,D_{\alpha^{\upharpoonright 0}}, D_{\alpha^{\upharpoonright 1}})\models \mathsf{\forall x \forall y[P(x) \vee Q(y)]}\rightarrow\mathsf{(\forall x[P(x)] \vee \forall y [Q(y)])}$.

\end{enumerate} \end{theorem}
\begin{proof} (i) $\Rightarrow$ (ii). Let $\alpha $ be given. \\Define $\beta$ such that $\forall p\forall i<2[\beta(2p+i) = 0\leftrightarrow \alpha^{\upharpoonright i}(p)\neq 0]$. \\Assume: $(\omega, D_{\alpha^{\upharpoonright 0}}, D_{\alpha^{\upharpoonright 1}})\models \mathsf{\forall x \forall y[P(x) \;\vee\; Q(y)]}$. Conclude: $\forall p \forall q[\beta(2p) = 0 \;\vee\;\beta(2q+1) = 0]$. Apply $\mathbf{LLPO}$ and distinguish two cases. 

\textit{Case (1)}. $\forall p[2p+1\neq \mu m[\beta(m)\neq 0]]$. Assume we find $p$ such that $\beta(2p+1)\neq 0$. Determine $q$ such that $q\le p$ and $\beta(2q)\neq 0$. Contradiction.  

Conclude: $\forall p[\beta(2p+1)=0]$ and: $D_{\alpha^{\upharpoonright 1}} = \omega$. 

\textit{Case (2)}.  $\forall p[2p\neq\mu m[\beta(m)\neq 0]]$. Then, for similar reasons: $D_{\alpha^{\upharpoonright 0}} = \omega$.

In both cases: $(\omega,D_{\alpha^{\upharpoonright 0}}, D_{\alpha^{\upharpoonright 1}})\models \mathsf{\forall x[P(x)] \vee \forall y [Q(y)]}$.

\smallskip (ii) $\Rightarrow$ (i). Let $\alpha$ be given. Define $\beta$ such that, for each $p$, for both $i<2$, $\beta^{\upharpoonright i}(p) = 0$ if and only if  $\overline \alpha(2p+i) = \underline{\overline 0}(2p+i)$ and $\alpha(2p+i) \neq 0$.  Note:  $(\omega, D_{\beta^{\upharpoonright 0}}, D_{\beta^{\upharpoonright 1}}) \models \mathsf{\forall x \forall y[P(x) \;\vee\; Q(y)]}$.  Conclude:  $(\omega, D_{\beta^{\upharpoonright 0}}, D_{\beta^{\upharpoonright 1}}) \models \mathsf{\forall x [P(x)] \;\vee\; \forall y[Q(y)]}$.

 \textit{Either:} $(\omega, D_{\beta^{\upharpoonright 0}}, D_{\beta^{\upharpoonright 1}}) \models \mathsf{\forall x [P(x)] }$ and: $\forall p[\overline \alpha (2p) = \underline{\overline 0}(2p) \rightarrow \alpha(2p) =0]$,\\ \textit{or:} $(\omega, D_{\beta^{\upharpoonright 0}}, D_{\beta^{\upharpoonright 1}}) \models \mathsf{\forall y [Q(y)] }$ and: $\forall p[\overline \alpha (2p+1) = \underline{\overline 0}(2p+1) \rightarrow \alpha(2p+1) =0]$. \\We thus see: $\mathbf{LLPO}$.
 \end{proof}

\begin{theorem} The following statements are equivalent in $\mathsf{BIM}$. \begin{enumerate}[\upshape(i)]
\item $\mathbf{LPO}$.

\item $\forall \alpha [(\omega,D_{\alpha^{\upharpoonright 0}}, D_{\alpha^{\upharpoonright 1}})\models \mathsf{\forall x [P(x) \vee Q(x)]}\rightarrow\mathsf{(\forall x[P(x)] \vee \exists x [Q(x)])}$.

\item  $\forall \alpha [(\omega,D_{\alpha^{\upharpoonright 0}}, Pr_{\alpha^{\upharpoonright 1}})\models \mathsf{\forall x [P(x) \vee A]}\rightarrow\mathsf{(\forall x[P(x)] \vee A)}$.\end{enumerate} \end{theorem}

\begin{proof} (i) $\Rightarrow$ (ii). Let $\alpha$ be given such that $(\omega,D_{\alpha^{\upharpoonright 0}}, D_{\alpha^{\upharpoonright 1}})\models \mathsf{\forall x [P(x) \vee Q(x)]}$, i.e. $\forall n[\alpha^{\upharpoonright 0}(n)\neq 0\;\vee\;\alpha^{\upharpoonright 1}(n)\neq 0]$.  Using $\mathbf{LPO}$, distinguish two cases. \\\textit{Either} $\forall n[\alpha^{\upharpoonright 0}(n) \neq 0]$ and $(\omega,D_{\alpha^{\upharpoonright 0}}, D_{\alpha^{\upharpoonright 1}})\models \mathsf{\forall x[P(x)]}$, \\\textit{or} $\exists n[\alpha^{\upharpoonright 0}(n) =0]$ and $\exists n[\alpha^{\upharpoonright 1}(n) \neq 0]$ and $(\omega,D_{\alpha^{\upharpoonright 0}}, D_{\alpha^{\upharpoonright 1}})\models \mathsf{\exists x[Q(x)]}$.

\smallskip (ii) $\Rightarrow$ (i). Let $\alpha$  be given. Define $\beta$  such that  $\forall n[\beta^{\upharpoonright 0}(n) = 0\leftrightarrow\alpha(n) \neq 0]$ and $\beta^{\upharpoonright 1}=\alpha$.  Note: $\forall n[\beta^{\upharpoonright 0}(n)\neq 0\;\vee\;\beta^1(n)\neq 0]$, i.e. \\$(\omega, D_{\beta^{\upharpoonright 0}}, D_{\beta^{\upharpoonright 1}})\models \mathsf{\forall x[P(x)\;\vee \;Q(x)]}$. Use (ii) and distinguish two cases.\\ \textit{Either }$(\omega, D_{\beta^{\upharpoonright 0}}, D_{\beta^{\upharpoonright 1}})\models \mathsf{\forall x[P(x)]}$ and $\forall n[\alpha(n) = 0]$, \\\textit{or} $(\omega, D_{\beta^{\upharpoonright 0}}, D_{\beta^{\upharpoonright 1}})\models \mathsf{\exists x[Q(x)]}$ and $\exists n[\alpha(n) \neq 0]$. \\We thus see: $\mathbf{LPO}$. 

\smallskip (i) $\Rightarrow$ (iii). Let $\alpha$ be  given such that $(\omega,D_{\alpha^{\upharpoonright 0}}, Pr_{\alpha^{\upharpoonright 1}})\models \mathsf{\forall x [P(x) \vee A]}$, i.e. \\$\forall n[\alpha^{\upharpoonright 0}(n)\neq 0\;\vee\;\exists m[\alpha^{\upharpoonright 1}(m)\neq 0]]$.  Using $\mathbf{LPO}$, distinguish two cases.  \\\textit{Either} $\forall n[\alpha^{\upharpoonright 0}(n) \neq 0]$ and $(\omega,D_{\alpha^{\upharpoonright 0}}, Pr_{\alpha^{\upharpoonright 1}})\models \mathsf{\forall x[P(x)]}$, \\\textit{or} $\exists n[\alpha^{\upharpoonright 0}(n) =0]$ and $\exists m[\alpha^{\upharpoonright 1}(m) \neq 0]$ and $(\omega,D_{\alpha^{\upharpoonright 0}}, Pr_{\alpha^{\upharpoonright 1}})\models \mathsf{A}$.

\smallskip (iii) $\Rightarrow$ (i).  Let $\alpha$  be given. Define $\beta$ such that $\forall n[\beta(n)=0
\leftrightarrow \alpha(n)\neq 0]$. \\Note: $\forall n[\beta(n)\neq 0\;\vee\; \alpha(n)\neq 0]$, that is, $(\omega, D_\beta, Pr_\alpha)\models \mathsf{\forall x[P(x)\;\vee \;A]}$. \\Use (iii) and distinguish two cases. \\\textit{Either }$(\omega, D_\beta, Pr_\alpha)\models \mathsf{\forall x[P(x)]}$ and $\forall n[\beta(n)\neq 0]$ and $\forall n[\alpha(n) = 0]$, \\\textit{or} $(\omega, D_\beta, Pr_\alpha)\models \mathsf{A}$ and $\exists n[\alpha(n) \neq 0]$. \\We thus see: $\mathbf{LPO}$. \end{proof}
\subsection{A note}

 According to Theorem \ref{T:ftlogiceq}(iii),  $\mathsf{BIM}+\mathbf{\Pi}^0_1$-$\mathbf{AC}_{\omega,2}$ proves that  $\mathbf{FT}$ is equivalent to:\begin{quote}
{\it For all $\alpha$, if $\forall \gamma \in \mathcal{C}\exists n[ \alpha^{\upharpoonright 0}(\overline \gamma n) \neq 0 \vee \alpha^{\upharpoonright 1}(n) \neq 0]$, \\ then  either $\forall \gamma \in \mathcal{C}\exists n[\alpha^{\upharpoonright 0}(\overline \gamma n) \neq 0]$ or $  \exists n[\alpha^{\upharpoonright 1}(n) \neq 0]$}.\end{quote}

One might ask if $\mathsf{BIM}+\mathbf{\Pi}^0_1$-$\mathbf{AC}_{\omega,2}$ proves that  $\neg!\mathbf{FT}$ is equivalent to the following statement: \begin{quote} {\it  There exists $\alpha$ such that $\forall \gamma \in \mathcal{C}\exists n[ \alpha^{\upharpoonright 0}(\overline \gamma n) \neq 0 \vee \alpha^{\upharpoonright 1}(n) \neq 0]$ and  $\neg \forall \gamma \in \mathcal{C}\exists n[\alpha^{\upharpoonright 0}(\overline \gamma n) \neq 0]$ and  $\alpha^{\upharpoonright 1}=\underline 0,$} \end{quote}
i.e.
\begin{quote}{\it 
 $(\ast)$: There exists $\alpha$ such that \\$\forall \gamma \in \mathcal{C}\exists n[ \alpha(\overline \gamma n) \neq 0]$ and $\neg \forall \gamma \in \mathcal{C}\exists n[\alpha(\overline \gamma n) \neq 0]$.} \end{quote}

The  formula $(\ast)$ is an outright contradiction!

\section{The determinacy of finite and infinite games}

 We consider games of perfect information for players $I,II$. First finite games, then games with finitely many moves where the players may choose out of infinitely many alternatives, and then games of infinite length. In the last Subsection we prove: in $\mathsf{BIM}$,  $\mathbf{FT}$ is an equivalent of  the Intuitionistic Determinacy Theorem: \textit{every subset of $(\omega\times 2)^\omega$ is weakly determinate}.

\subsection{Finite games}
\subsubsection{Finite Choice and a contraposition of Finite Choice} \hfill 

\begin{lemma}\label{L:comb0}
$\mathsf{BIM}$ proves the following scheme:
$$\forall m[\forall n<m[A(n)\vee B]\rightarrow (\forall n<m[A(n)]\;\vee\; B)]$$\end{lemma} \begin{proof} The proof is straightforward, by induction. \end{proof}

\begin{lemma}\label{L:comb}
 $\mathsf{BIM}$ proves the following schemes:
\begin{enumerate}[\upshape (i)] \item  $\forall k[\forall n< k\exists m[R(n,m)] \rightarrow \exists s \forall n< k[R\bigl(n,s(n)\bigr)]]$.
\item 

 $\forall k\forall l[\forall s:k \rightarrow l\exists n < k[R\bigl(n,s(n)\bigr)] \rightarrow \exists n < k \forall m< l[R(n,m)]]$.
\end{enumerate}\end{lemma}
\begin{proof} (i) The proof is by induction on $k$ and left to the reader.

\smallskip(ii) The proof is by induction on $k$ and uses Lemma \ref{L:comb0}. Note that there is nothing to prove if $k=0$. Now assume the statement holds for a certain $k$. 

Assume   $\forall s:(k+1)\rightarrow l\exists n < k+1[ R\bigl(n,s(n)\bigr)]$. Note:  $$\forall j<l\forall s:k \rightarrow l [\exists n < k[R\bigl(n, s(n)\bigr)]\;\vee\;R(k,j) ],$$ and, therefore, by Lemma \ref{L:comb0}:$$\forall j<l[R(k,j)\;\vee\;\forall s:k \rightarrow l \exists n < k[R\bigl(n, s(n)\bigr)] ].$$

Using the induction hypothesis, we conclude: $$\forall j <l[R(k,j)]\;\vee\;\exists n<k\forall m<l[ R(n,m)],$$ that is: $\exists n < k +1\forall m<l [R(n,m)]$.
\end{proof}
\subsubsection{Finite games}\label{SS:finitegames} We want to study finite and infinite games for players $I$ and $II$ of perfect information. We first consider finite games: there are finitely many moves, and for each move there are only finitely many alternatives.

\smallskip

 Assume $X\subseteq\omega$. Let $n,l$ be given such that $l>0$. \\Players $I$ and $II$ play the \textit{$I$-game for $X$  in $\mathit{Seq}(n,l)$} and the \textit{$II$-game for $X$ in $\mathit{Seq}(n,l)$} in the same way, as follows. \textit{First,} player $I$ chooses $i_0<l$, \textit{then} player $II$ chooses $i_1<l$, and they continue until a finite sequence $\langle i_0, i_1, \ldots, i_{n-1}\rangle$ of length $n$ has been formed.   \textit{Player $I$ wins the play $\langle i_0, i_1, \ldots, i_{n-1}\rangle$   in the $I$-game for $X$} if and only if $\langle i_0, i_1, \ldots, i_{n-1} \rangle \in X$.    \textit{Player $II$ wins the play $\langle i_0, i_1, \ldots, i_{n-1}\rangle$  
in the $II$-game for $X$} if and only if  $\langle i_0, i_1, \ldots, i_{n-1} \rangle \in X$.

We define  $\boldsymbol{\varphi}$ such that, for all $n$, for all $l>0$,   $$\boldsymbol{\varphi}(n,l):=\mu a\forall i< n\forall c \in Seq(i, l)[c<a].$$

Every number coding a position in $Seq(n,l)$ that is not a final position is smaller than $\boldsymbol{\varphi}(n,l)$.

We define  $\boldsymbol{\psi}$   such that, for all $n$, for all $l>0$,  $$\boldsymbol{\psi}(n,l):=\mu a\forall s\in Seq\bigl(\boldsymbol{\varphi}(n,l),l\bigr)[s<a].$$

When studying games in $Seq(n,l)$ it suffices to consider strategies $s,t$ for player $I,II$, respectively, such that $s,t<\boldsymbol{\psi}(n,l)$. 

The reader should consult Subsubsection \ref{SSS:games} for some of the notations we are going to use, like `$\in_I$' and `$\in_{II}$'.

We define: \textit{$X\subseteq \omega$ is $I$-determinate in $\mathit{Seq}(n,l)$},    $\mathit{Det}^I_{\mathit{Seq}(n,l)}(X)$,
\\if and only if $\forall t < \boldsymbol{\psi}(n,l)\exists c [c \in_{II}t \;\wedge\;  c \in X]\rightarrow \exists s\forall c \in \mathit{Seq}(n,l)[c \in_I s \rightarrow c \in X],$

and: \textit{$X\subseteq \omega$ is $II$-determinate in $\mathit{Seq}(n,l)$},    $\mathit{Det}^{II}_{\mathit{Seq}(n,l)}(X)$,
\\if and only if $\forall s < \boldsymbol{\psi}(n,l)\exists c [c \in_{I}s \;\wedge\;  c \in X]\rightarrow \exists t\forall c \in \mathit{Seq}(n,l)[c \in_{II} t \rightarrow c \in X].$

\begin{theorem}[Determinacy of finite games]\label{T:finitedet}
Let $X\subseteq \omega$ be given.

 For every $l>0$, for every $n$, $\mathit{Det}^I_{\mathit{Seq}(n,l)}(X)$ and $\mathit{Det}^{II}_{\mathit{Seq}(n,l)}(X)$.

\end{theorem}
\begin{proof} Let $X\subseteq \omega$ and $l>0$ be given.

For every $u$, we define: $X_{u}:=\{c \in \omega\mid u\ast c \in X\}$.

We intend to prove a statement seemingly stronger than the statement of the theorem: \begin{quote}$(\ast)$: for every $n$, for every $u$,  $\mathit{Det}^I_{\mathit{Seq}(n,l)}(X_u)$ and $\mathit{Det}^{II}_{\mathit{Seq}(n,l)}(X_u)$.\end{quote}

The proof is by induction on $n$. If $n=0$, then, for every $u$,  the statements:\\ \textit{$X_u$ is $I$-determinate in $\mathit{Seq}(0,l)$} and: \textit{$X_u$ is $II$-determinate in $\mathit{Seq}(0,l)$} both assert:\\ \textit{if $\langle \; \rangle \in X_u$, then $\langle \; \rangle \in X_u$}, and thus are obviously true.

Now assume the statement $(\ast)$ has been established for a certain $n$. We prove that $(\ast)$ is also true for $n+1$.

Let $u$ be given.

\smallskip
First, assume:  $\forall t <\boldsymbol{\psi}(n+1,l)\exists c[c \in_{II} t\;\wedge \;  c \in X_u]$, or, equivalently
\footnote{For the notation $t^{\upharpoonright k}$, see Subsection \ref{SS:subs}.}

 $\forall t <\boldsymbol{\psi}(n+1,l) \exists k <l\exists c[c \in_I t^{\upharpoonright k} \;\wedge\; \langle k \rangle \ast c \in X_u ]$. 
 
 Note: $\forall k<l\forall c[\langle k\rangle\ast c \in X_u\leftrightarrow c \in X_{u\ast\langle k\rangle}]$.

 Conclude:
 $\forall t: l \rightarrow \psi(n,l) \exists k<l\exists c[c \in_I t(k) \;\wedge\;  c \in X_{u\ast\langle k \rangle}]$. 
 
  Use Lemma \ref{L:comb}(ii) and  find $k<l$ such that
  $\forall s <\boldsymbol{\psi}(n,l)\exists c[c \in_I s \; \wedge \;  c \in X_{u\ast\langle k \rangle}]$. 
  
  Use the induction hypothesis and find $t <\boldsymbol{\psi}(n,l)$ such that
  
   $\forall c \in \mathit{Seq}(n,l)[ c \in_{II}t \rightarrow  c \in X_{u\ast\langle k \rangle}]$. 
  
  Define $s<\boldsymbol{\psi}(n+1,l)$  such that $s(\langle \; \rangle )= k$ and $s^{\upharpoonright k} = t$. 
  
  Note: $\forall c\in \mathit{Seq}(n+1,l)[ c \in_I s\rightarrow c \in X_u]$.
 
 We thus see that $X_u$ is $I$-determinate.
 
 \smallskip
Next, assume:  $\forall s <\boldsymbol{\psi}(n+1,l)\exists c[c \in_{I} s\;\wedge \;  c \in X_u]$, or, equivalently:

 $\forall s <\boldsymbol{\psi}(n+1,l)  \exists c[c \in_{II} s^{\upharpoonright s(\langle \; \rangle)} \;\wedge\; \langle s(\langle \; \rangle) \rangle \ast c \in X_u ]$. 
 
 Note: $\forall k<l
  \forall t<\boldsymbol{\psi}(n,l)\exists s<\psi(n+1,l)[s(\langle \; \rangle) = k\;\wedge\;s^{\upharpoonright k} = t]$.  
  
  We thus see:
 $\forall k<l\forall t <\boldsymbol{\psi}(n,l) \exists c[ c \in_{II} t \; \wedge \; \langle k \rangle \ast c \in X_u]$.

 Note: $\forall k<l\forall c[\langle k\rangle\ast c \in X_u\leftrightarrow c \in X_{u\ast\langle k\rangle}]$.
 
 Conclude: $\forall k<l\forall t <\boldsymbol{\psi}(n,l) \exists c[ c \in_{II} t \; \wedge \;  c \in X_{u\ast\langle k \rangle}]$.
 
  Use the induction hypothesis and conclude:
    
   $\forall k<l \exists s <\boldsymbol{\psi}(n,l)\forall c\in \mathit{Seq}(n,l)[ c \in_I s\rightarrow  s \in X_{u\ast\langle k \rangle}]$.

 Use Lemma \ref{L:comb}(i) and find $t<\psi(n+1,l)$ such that  
 
 $\forall k<l \forall c\in \mathit{Seq}(n,l)[ c \in_I t^{\upharpoonright k}\rightarrow  c \in X_{u\ast\langle k\rangle}]$.
 
  Note: $\forall c \in \mathit{Seq}(n+1,l)[ c \in_{II} t\rightarrow c \in X_u]$.
  
  We thus see that $X_u$ is $II$-determinate.
\end{proof}
\subsubsection{Comparison with the classical theorem} Note that, in classical mathematics,  the $I$-determinacy of finite games is stated as follows:

For every $X\subseteq \omega$, for every $l>0$, 
 for every  $n$, \begin{quote}
\textit{either}: $\exists t <\boldsymbol{\psi}(n,l)\forall c \in \mathit{Seq}(n,l)[c \in_{II}t \rightarrow  c \notin X]$,

 \textit{or}: $\exists s <\boldsymbol{\psi}(n,l)\forall c \in \mathit{Seq}(n,l) [c \in_I s \rightarrow c \in X]$,\end{quote} that is: \textit{either} player $II$ has a strategy ensuring that the result of the game will not be in $X$, \textit{or} player $I$ has a  strategy ensuring that it does.

Taken constructively, this statement fails to be true already in the case $n=0$, because it then implies: for every subset $X$ of $\{\langle\;\rangle\}$, either $\langle\;\rangle\notin X$ or $\langle\;\rangle \in X$, and therefore, for any proposition $P$, $\neg P \vee P$, the principle of the excluded third.

\subsection{Infinitely many alternatives}
\subsubsection{Infinitely many alternatives for player $II$}\label{SSS:II} \hfill

 Let  $X\subseteq2 \times \omega$ be given.
  Players $I$ and $II$ play the \textit{$I$-game for $X$ in $2\times\omega$} in the following way.  First, player $I$ chooses $i<2$, then player $II$ chooses $n$ and the play is finished. Player $I$ wins the play $\langle i, n \rangle$ if and only if $\langle i,n\rangle \in X$.
 
We define: \textit{$X$ is $I$-determinate in $2 \times \omega$},   $\mathit{Det}^I_{2 \times \omega}(X)$,
if and only if 
$$\forall t\exists c\in 2 \times \omega [c  \in_{II}t \;\wedge\;  c \in X]\rightarrow \exists s < 2\forall n[\langle s, n \rangle \in X].$$ 

 \begin{theorem}\label{T:detllpo} $\mathsf{BIM}\vdash \forall \alpha[\mathit{Det}^I_{2 \times \omega}(D_\alpha)]\leftrightarrow \mathbf{LLPO}$. \end{theorem}

 \begin{proof} This Theorem is a reformulation of Theorem \ref{T:llpological}.
 
 In order to see this,  make two observations: \smallskip
 
 (i) Note that, for each $\alpha$,  there exists $\beta$ such that
 
 $  Det^I_{2\times\omega}(D_\alpha)\leftrightarrow(\omega,D_{\beta^{\upharpoonright 0}}, D_{\beta^{\upharpoonright 1}})\models \mathsf{\forall x \forall y[P(x) \vee Q(y)]}\rightarrow\mathsf{(\forall x[P(x)] \vee \forall y [Q(y)])}$.
 
 \noindent Given $\alpha$, define $\beta$ such that, for each $n$,  $\beta^{\upharpoonright 0}(n)=\alpha(\langle 0, n\rangle)$ and $\beta^{\upharpoonright 1}(n)=\alpha(\langle 1, n\rangle)$.
 
 \smallskip 
 (ii) Note that, for each $\alpha$, there exists $\beta$ such that 
 
 $ \bigl((\omega,D_{\alpha^{\upharpoonright 0}}, D_{\alpha^{\upharpoonright 1}})\models \mathsf{\forall x \forall y[P(x) \vee Q(y)]}\rightarrow\mathsf{(\forall x[P(x)] \vee \forall y [Q(y)])}\bigr)\leftrightarrow Det^I_{2\times\omega}(D_\beta)$.
 
\noindent Given $\alpha$, define $\beta$ such that, for each $n$, $\beta(\langle 0, n\rangle)=\alpha^0(n)$ and  $\beta(\langle 1, n\rangle)=\alpha^1(n)$.
\end{proof}

\subsubsection{Infinitely many alternatives for player $I$}\hfill

Let $X\subseteq\omega \times 2$ be given. Players $I$ and $II$ play the \textit{$I$-game for $X$ in $\omega\times 2$} in the following way.
\textit{First,} player $I$ chooses a natural number $n$, \textit{then} player $II$ chooses a number $i$ from $\{0,1\}$. Player $I$ wins the \textit{play} $\langle n, i\rangle$
if and only if  $\langle n,i\rangle\in X$.
 
 Note that a strategy for player $I$ in such a two-move-game coincides with his first move and thus is a natural number. A strategy for player $II$, on the other hand, is an infinite sequence $\tau$ in $\mathcal{C}$ that expresses player $II$'s intention to play $\tau(\langle n \rangle)$ once player $I$ has brought them to the position $\langle n \rangle$. 

 We define:   $X\subseteq \omega\times 2$ is \textit{$I$-determinate in $\omega \times 2$},  $\mathit{Det}^I_{\omega \times 2}(X)$, if and only if: $$\forall \tau \in \mathcal{C} \exists c \in \omega \times 2[c \in_{II} \tau \;\wedge\; c \in X] \rightarrow \exists s\forall i<2[ \langle s, i\rangle  \in X].$$

 \begin{theorem}\label{T:decdet} $\mathsf{BIM}\vdash \forall \alpha[Det^I_{\omega \times 2}(D_\alpha)]$.
 \end{theorem}
 
\begin{proof}

 Let $\alpha$  be given. Assume: $\forall \tau \in \mathcal{C} \exists n[\langle n , \tau(\langle n \rangle)\rangle \in D_\alpha]$.    Find $\tau$ in $\mathcal{C}$ such that $\forall n[\tau(\langle n \rangle) = 1\leftrightarrow\langle n,0\rangle \in D_\alpha]$. Find $n$ such that $\langle n, \tau(\langle n\rangle) \rangle \in D_\alpha$. Note: $\tau(\langle n\rangle) = 1$ and   $\forall i<2[\langle n,i \rangle\in D_\alpha]$. \end{proof}
 
 We define:   $X\subseteq \omega\times \omega$ is \textit{$II$-determinate in $\omega \times \omega$},  $\mathit{Det}^{II}_{\omega \times \omega}(X)$, if and only if: $$\forall m \exists n [\langle
 m,n\rangle\in X]\rightarrow \exists \tau\forall m[\langle m, \tau(\langle m\rangle)\rangle\in X].$$
 
 \begin{theorem}\label{T:endetII} $\mathsf{BIM}\vdash \forall \alpha[Det^{II}_{\omega \times \omega}(E_\alpha)]$.
 \end{theorem}
\begin{proof}  Let $\alpha$ be given such that  $\forall m\exists n [\langle m, n\rangle \in E_\alpha]$, that is \\$\forall m\exists n \exists p[\alpha(p)=\langle m, n\rangle +1 ]$. Find   $\gamma$ such that $\forall m[\alpha\bigl(\gamma'(m)\bigr)=\langle m, \gamma''(m)\rangle +1 ]$.  Define $\tau$ such that $\forall m[\tau(\langle m \rangle)=\gamma''(m)]$. \end{proof}

We define:   $X\subseteq 2\times \omega$ is \textit{$II$-determinate in $2 \times \omega$},  $\mathit{Det}^{II}_{2 \times \omega}(X)$, if and only if: $$\forall m<2 \exists n [\langle
 m,n\rangle\in X]\rightarrow \exists t\forall m<2[\langle m, t(\langle m\rangle)\rangle\in X].$$
 
Note:  $\mathsf{BIM}$ proves the scheme ${Det}^{II}_{2 \times \omega}(X)$.

\subsubsection{Longer games}  We also consider games in which  players $I,II$ make more than one move. Which of those games are determinate from the viewpoint of Player $I$? Because of Theorem \ref{T:detllpo}, we restrict ourselves to games in which player $I$ has, for each one of his moves,
countably many alternatives, whereas player $II$ always has to choose one of two possibilities.

For every $n$, for every $X\subseteq(\omega \times 2)^n$,  
we define: 

$X$ is \textit{$I$-determinate in $(\omega \times 2)^n$}, $\mathit{Det}_{(\omega \times 2)^n}(X)$,  
if and only if $$\forall \tau \in \mathcal{C} \exists c[c \in_{II}\tau \;\wedge\; c \in X]\rightarrow \exists s\forall c \in (\omega \times 2)^n[c \in_I s \rightarrow c \in X].$$
 
 This definition  extends in the obvious way to subsets $X$ of $(\omega \times 2)^n \times \omega$.

\subsection{Infinitely many moves}\label{SS:infinitelymanymoves}
We also want to consider  games of infinite length. We imagine players $I, II$ to build together an infinite sequence $\gamma$ in $\omega^\omega$, as follows.
\textit{First,} player $I$ chooses $\gamma(0)$, \textit{then} player $II$ chooses $\gamma(1)$, \textit{then} player $I$ chooses $\gamma(2)$, and so on.

We  define a number of notions of determinacy.

$\mathcal{X}\subseteq(\omega \times 2)^\omega $ is \textit{$I$-determinate in $(\omega \times 2)^\omega$},  $\mathit{Det}^{I}_{(\omega \times 2)^\omega}(\mathcal{X})$,  
if and only if $$\forall \tau \in \mathcal{C} \exists \gamma[\gamma \in_{II}\tau \;\wedge\; \gamma \in \mathcal{X}]\rightarrow \exists \sigma\forall \gamma \in (\omega \times 2)^\omega[\gamma \in_I \sigma \rightarrow \gamma \in \mathcal{X}].$$ 
$\mathcal{X}\subseteq (\omega\times 2)^\omega$ is \textit{finitely $I$-determinate in $(\omega \times 2)^\omega$,}  $\mathit{^\ast Det}^{I}_{(\omega \times 2)^\omega}(\mathcal{X})$
if and only if $$\forall \tau \in \mathcal{C} \exists \gamma[\gamma \in_{II}\tau \;\wedge\; \gamma \in \mathcal{X}]\rightarrow \exists s\forall \gamma \in (\omega \times 2)^\omega[\gamma \in_{I}s \rightarrow \gamma \in \mathcal{X}].$$

 $\mathcal{X}\subseteq\mathcal{C}$ is \textit{$I$-determinate in $\mathcal{C}$,}  $\mathit{Det}^{I}_{\mathcal{C}}(\mathcal{X})$,
if and only if $$\forall \tau \in \mathcal{C} \exists \gamma[\gamma \in_{II}\tau \;\wedge\; \gamma \in \mathcal{X}]\rightarrow \exists \sigma\in \mathcal{C}\forall \gamma \in \mathcal{C}[\gamma \in_{I}\sigma \rightarrow \gamma \in \mathcal{X}].$$ 

\textit{$\mathcal{X} \subseteq\mathcal{C}$ is  finitely  $I$-determinate in $\mathcal{C}$},  $^\ast\mathit{Det}^{I}_{\mathcal{C}}(\mathcal{X})$, if and only if$$\forall \tau \in \mathcal{C} \exists \gamma[\gamma \in_{II}\tau \;\wedge\; \gamma \in \mathcal{X}]\rightarrow \exists s\in Bin\forall \gamma \in \mathcal{C}[\gamma \in_{I}s \rightarrow \gamma \in \mathcal{X}].$$ 

 $\mathcal{X}\subseteq\mathcal{C}$ is \textit{$II$-determinate in $\mathcal{C}$,}  $\mathit{Det}^{II}_{\mathcal{C}}(\mathcal{X})$,
if and only if $$\forall \sigma \in \mathcal{C} \exists \gamma[\gamma \in_{I}\sigma \;\wedge\; \gamma \in \mathcal{X}]\rightarrow \exists \tau\in\mathcal{C}\forall \gamma \in \mathcal{C}[\gamma \in_{II}\tau \rightarrow \gamma \in \mathcal{X}].$$ 

\textit{$\mathcal{X} \subseteq\mathcal{C}$ is  finitely  $II$-determinate in $\mathcal{C}$},  $^\ast\mathit{Det}^{II}_{\mathcal{C}}(\mathcal{X})$, if and only if$$\forall \sigma \in \mathcal{C} \exists \gamma[\gamma \in_{I}\sigma \;\wedge\; \gamma \in \mathcal{X}]\rightarrow \exists t\in Bin\forall \gamma \in \mathcal{C}[\gamma \in_{II}t \rightarrow \gamma \in \mathcal{X}].$$ 
\bigskip
We are going to study the following statements:

\smallskip\noindent
$\mathbf{\Sigma}^0_1$-$\mathit{Det}_{\omega \times 2}^I$:
 $\forall \alpha[Det^I(E_\alpha)]$.

\smallskip\noindent
 $\mathbf{\Delta}^0_1$-$\mathit{Det}^I_{\omega \times 2 \times \omega}$:
 $\forall \alpha[ \mathit{Det}^I_{\omega \times 2\times \omega}(D_\alpha)]. $

\smallskip\noindent
 $\mathbf{\Delta}^0_1$-$\mathit{Det}^I_{(\omega \times 2)^m}$:
 $\forall \alpha[ \mathit{Det}^I_{(\omega \times 2)^m}(D_\alpha)]. $

 \smallskip\noindent
$\mathbf{\Sigma}^0_1$-$\mathit{Det}^I_{(\omega \times 2)^\omega}$:
   $\forall \alpha[ \mathit{Det}^I_{(\omega \times 2)^\omega}(\mathcal{G}_\alpha)]. $

  \smallskip\noindent
$\mathbf{\Sigma}^0_1$-$\mathit{\;^\ast Det}^I_{(\omega \times 2)^\omega}$:
  $\forall \alpha[\mathit{^\ast Det}^I_{(\omega \times 2)^\omega}(\mathcal{G}_\alpha)]. $

 \smallskip\noindent
$\mathbf{\Sigma}^0_1$-$\mathit{Det}^I_{\mathcal{C}}$:
  $\forall \alpha[ \mathit{Det}^I_{\mathcal{C}}(\mathcal{G}_\alpha)]. $

  \smallskip\noindent
$\mathbf{\Sigma}^0_1$-$\mathit{\;^\ast Det}^I_{\mathcal{C}}$:
  $\forall \alpha[ \mathit{^\ast Det}^I_{\mathcal{C}}(\mathcal{G}_\alpha)]. $

  \smallskip\noindent
$\mathbf{\Sigma}^0_1$-$\mathit{Det}^I_{\mathcal{C}}$:
  $\forall \alpha[ \mathit{Det}^{II}_{\mathcal{C}}(\mathcal{G}_\alpha)]. $
  
  \smallskip\noindent
$\mathbf{\Sigma}^0_1$-$\mathit{\;^\ast Det}^{II}_{\mathcal{C}}$:
   $\forall \alpha[\mathit{^\ast Det}^{II}_{\mathcal{C}}(\mathcal{G}_\alpha)]. $

\smallskip
Each of the above formulas $X$ has the form: $\forall \alpha[P(\alpha)\rightarrow Q(\alpha)]$. For each of these nine formulas $X$, we define the statement $\neg! X$, the 
\textit{strong negation} of $X$, as follows:  
\[\neg! X:=\neg!\bigl(\forall \alpha[P(\alpha)\rightarrow Q(\alpha)]\bigr) := \exists \alpha[P(\alpha) \;\wedge\; \neg Q(\alpha)].\]
 
Note that these strong negations contain the negation symbol $\neg$, a possibility we mentioned  in Subsection \ref{SS:strongnegations}.

Note that the symbol $\neg!$ is used as a metamathematical notation, not as part of the language of $\mathsf{BIM}$. One should also not consider $\neg!$ as the name of a syntactical operation on formulas.

\subsection{Simulating a game in $(\omega\times 2)^{\omega}$ by a game in $\mathcal{C}$}
 From the point of view of player $I$, a game in $(\omega \times 2)^\omega$ may be simulated by a game in Cantor space $\mathcal{C}$. Where player $I$ would play $n$ in $(\omega \times 2)^\omega$, he will play $n$ times $0$ and one time $1$ in $\mathcal{C}$.  So he plays the finite sequence $\underline{\overline 0}n\ast\langle 1 \rangle$. Every time he plays $0$,  he makes what we call a \textit{postponing} move. Player $II$ has to react, in $\mathcal{C}$, to these postponing moves of player $I$, but these reactions do not matter. As soon as player $I$ plays $1$ and completes $\overline{\underline 0} n\ast \langle 1 \rangle$, player $II$ gives, in the play in $\mathcal{C}$, the answer he would give to player $I$'s move $n$ in $(\omega \times 2)^\omega$. The reader should keep this in mind when reading the following definitions. 
 
 Define $Halfbin:= (\omega\times 2)^{<\omega}\cup \bigl((\omega\times 2)^{<\omega}\times \omega\bigr)=\bigcup_n\{\overline \gamma n\mid \gamma \in (\omega\times 2)^\omega\}$.

Define $\boldsymbol{\pi_{bin}}$ such that
\begin{enumerate}[\upshape 1.]
  \item $\boldsymbol{\pi_{bin}}(\langle\;\rangle)=\langle\;\rangle$, and,\item for each $c$, if $length(c)$ is even, then, for each $n$, $\boldsymbol{\pi_{bin}}(c\ast\langle n \rangle)=\boldsymbol{\pi_{bin}}(c) \ast\underline{\overline 0}2n\ast\langle 1\rangle$, and, for both $i<2$, $\boldsymbol{\pi_{bin}}(c\ast\langle n, i\rangle)=\boldsymbol{\pi_{bin}}(c\ast\langle n \rangle)\ast\langle i \rangle$. \end{enumerate}
 
 The function $\boldsymbol{\pi_{bin}}$ associates to every position in $Halfbin$ a position in $Bin$.
 
 Note that, for each $c$, $length(\boldsymbol{\pi_{bin}}(c))\ge length(c)$.

 \smallskip
 Define $\boldsymbol{\rho_{bin}}$ in $\omega^\omega$ such that 
  
 \begin{enumerate}[\upshape 1.] 
 \item $\boldsymbol{\rho_{bin}}(\langle\;\rangle)=\langle\;\rangle$, and, \item  for each $d$ in $Bin$, if $length(d)$ is even, then\\  $\boldsymbol{\rho_{bin}}(d\ast\langle 0\rangle)=\boldsymbol{\rho_{bin}}(d\ast\langle 0,0\rangle)=\boldsymbol{\rho_{bin}}(d\ast\langle 0,1\rangle)=\boldsymbol{\rho_{bin}}(d)$, and \item for each $d$ in $Bin$, if $length(d)$ is even, then $\boldsymbol{\rho_{bin}}(d\ast\langle 1\rangle)=\boldsymbol{\rho_{bin}}(d)\ast\langle n \rangle$, and, for both $i<2$, $\boldsymbol{\rho_{bin}}(d\ast\langle 1, i\rangle)=\boldsymbol{\rho_{bin}}(d)\ast\langle n, i\rangle$, where $n$ satisfies:
 
 \noindent \textit{either}: $2n=length(d)$ and $\forall i<n[d(2i)=0]$,
   \textit{or}:
  
   \noindent for some $k>0$, $length(d)=2k + 2n$ and $d(2k-2)=1$ and $\forall i<n[d(2k +2i)=0]$. \end{enumerate}

The function $\boldsymbol{\rho_{bin}}$ associates to every position in $Bin$ a position in $Halfbin$.  

 Note that, for every $c$ in $Halfbin$, $\boldsymbol{\rho_{bin}}\circ \boldsymbol{\pi_{bin}}(c)=c$.

Note that, for each $c$ in $Halfbin$, $length(c)$ is even if and only if $length\bigl(\boldsymbol{\pi_{bin}}(c)\bigr)$ is even.

 \begin{lemma}\label{L:simulate}  The following is provable in $\mathsf{BIM}$. \\For each $\alpha$, there exists $\beta$ such that  $$\forall \tau \in \mathcal{C}\exists \gamma \in (\omega \times 2)^\omega[\gamma \in_{II}\tau\;\wedge\;\gamma\in\mathcal{G}_\alpha] \rightarrow\forall \tau \in \mathcal{C}\exists \delta \in \mathcal{C}[\delta \in_{II}\tau\;\wedge\;\delta\in\mathcal{G}_\beta]\;\mathrm{and}$$   $$\exists \sigma\forall \delta\in\mathcal{C}[\delta\in_I \sigma\rightarrow \delta \in \mathcal{G}_\beta]\rightarrow \exists \sigma \forall \gamma \in (\omega\times 2)^\omega[\gamma\in_I \sigma\rightarrow \gamma \in \mathcal{G}_\alpha]\;\mathrm{and}$$ $$\exists s\forall \delta\in\mathcal{C}[\delta\in_I s\rightarrow \delta \in \mathcal{G}_\beta]\rightarrow \exists s\forall \gamma\in(\omega\times 2)^\omega[\gamma\in_I s\rightarrow \gamma \in\mathcal{G}_\alpha].$$\end{lemma}
 
 \begin{proof} Let $\alpha$ be given. Define $\beta:=\alpha\circ \boldsymbol{\rho_{bin}}$. 
 
 \smallskip Assume $\forall \tau \in \mathcal{C}\exists \gamma \in (\omega \times 2)^\omega[\gamma \in_{II}\tau\;\wedge\;\gamma\in\mathcal{G}_\alpha]$.
 
 Let $\tau$ be given as a strategy for player $II$ in $\mathcal{C}$. 
 
 We want to prove: $\exists \delta \in \mathcal{C}[\delta \in_{II}\tau\;\wedge\; \delta \in \mathcal{G}_\beta]$.
 
 To this end, we define  $\tau^\dag$ as a strategy for player $II$ in $(\omega\times 2)^\omega$. We define $\tau^\dag$ on all positions in $Halfbin$ of odd length, by induction on the length of the position. It suffices to define $\tau^\dag$ on positions $c$ satisfying the condition: $c \in_{II}\tau^\dag$.

 We shall take  care that, for each $c$ in $Halfbin$, if $c\in_{II}\tau^\dag$, then there exists $d$ in $Bin$ such that $\boldsymbol{\rho_{bin}}(d)=c$ and $d\in_{II}\tau$.
 
 \smallskip
 
 We first define $\tau^\dag$  on positions of length $1$.

 Let $n$ be given. We have  to define $\tau^\dag(\langle n \rangle)$.
 
  Find $d$ in $Bin$ such that $length(d)=2n+1$ and $d\in_{II} \tau$ and $d(2n)=1$ and $\forall i<n[d(2i)=0]$.   Define $\tau^\dag(\langle n \rangle):= \tau(d)$. 
  
  Note: $\boldsymbol{\rho_{bin}}(d)=\langle n\rangle$ and  $\boldsymbol{\rho_{bin}}\bigl(d\ast\langle \tau(d)\rangle\bigr)=\langle n, \tau^\dag(\langle n \rangle)\rangle$.
 
 \smallskip
 Now assume $k>0$. Let  $c$ in $(\omega\times 2)^k$ be given such that  $c\in_{II}\tau^\dag$. 
 
 Let $n$ be given. We have to define $\tau^\dag(c\ast\langle n \rangle)$.
  
  First find $d$ in $Bin$ such that $d\in_{II}\tau$ and $\boldsymbol{\rho_{bin}}(d)=c$.   Find $l:=length(d)$. Find $e$ in $Bin$ such that $d\sqsubset e$ and $length(e)=l + 2n+1$ and $e\in_{II} \tau$ and $e(l+2n)=1$ and $\forall i<n[e(l+2i)=0]$. Note: $\boldsymbol{\rho_{bin}}(e)=c\ast\langle n\rangle$. Define $\tau^\dag(c\ast\langle n \rangle):= \tau(e)$.
  
  Note: $\boldsymbol{\rho_{bin}}(e)= c\ast\langle n \rangle$ and $\boldsymbol{\rho_{bin}}\bigl(e\ast\langle\tau(e)\rangle\bigr)=c\ast\langle \tau^\dag(c)\rangle$.

  This completes the definition of $\tau^\dag$.
 
 \smallskip 
  Now
   find $\gamma$ in $(\omega\times 2)^\omega$ such that $\gamma \in_{II}\tau^\dag \;\wedge\; \gamma \in \mathcal{G}_\alpha$. 
   
   Note: $\forall n \exists d \in Bin[d\in_{II}\tau \;\wedge\; \boldsymbol{\rho_{bin}}(d)=\overline \gamma n]$.
   
   Note: $\forall d \in Bin\forall e \in Bin[(d\in_{II}\tau\;\wedge\;e\in_{II}\tau)\rightarrow \bigl(d\sqsubset e\leftrightarrow \boldsymbol{\rho_{bin}}(d)\sqsubset \boldsymbol{\rho_{bin}}(e)\bigr)]$.
   
   Using this fact, 
 construct $\delta$ in $\mathcal{C}$ such that $\delta\in_{II}\tau$ and $\forall n\exists m[\overline \gamma n= \boldsymbol{\rho_{bin}}(\overline\delta m)]$.
 
    Find $n$ such that $\overline \gamma n\in D_\alpha$.  
    
    Find $m$ such that $\overline \gamma n=\boldsymbol{\rho_{bin}}(\overline \delta m)$.
    
     Note: $\overline \gamma n = \boldsymbol{\rho_{bin}}(\overline \delta m)\in D_\alpha$ and: $\overline \delta m \in \mathcal{G}_\beta$ and: $\delta \in \mathcal{G}_\beta$. 
 
 \smallskip We thus see: $\forall \tau \in \mathcal{C}\exists \delta \in \mathcal{C}[\delta\in_{II}\tau\;\wedge\;\delta\in\mathcal{G}_\beta]$.

 \medskip  Now assume $\exists \sigma\forall \delta\in\mathcal{C}[\delta\in_I \sigma\rightarrow \delta \in \mathcal{G}_\beta]$.
 
 Find $\sigma$ such that $\forall \delta\in\mathcal{C}[\delta\in_I \sigma\rightarrow \delta \in \mathcal{G}_\beta]$.
 
 We define $\sigma^\ast$ as a strategy for player $I$ in $(\omega\times 2)^\omega$ such that, for each $c$ in $Halfbin$, if $c\in_I \sigma^\ast$, then  \textit{either}  $ \boldsymbol{\pi_{bin}}(c)\in_I\sigma$ \textit{or}  $\exists e\sqsubset c[e\in D_\alpha]$.

  \smallskip
  We first define $\sigma^\ast(\langle\;\rangle)$.
  
  Define $\delta$ in $\mathcal{C}$ such that $\delta \in_I \sigma$ and $\forall i[\delta(2i+1)=0]$.  Find $m$ such that $\beta(\overline \delta m)\neq 0$ and distinguish two cases.
  
  \textit{Case (a)}. $\exists n[2n<m\;\wedge\;\delta(2n)=1]$. Define:   $ n_0:=\mu n[\delta(2n)=1]$ and  \\$\sigma^\ast(\langle \;\rangle):=n_0$. Note $\langle n_0\rangle\in_I\sigma^\ast$ and $\boldsymbol{\pi_{bin}}(\langle n_0\rangle)=\underline{\overline 0}(2n_0)\ast\langle 1 \rangle =\overline \delta(2n_0+1)\in_I\sigma$. 
  
  \textit{Case (b)}. $\forall n[2n<m \rightarrow\delta(2n)=0]$. Conclude: $\boldsymbol{\rho_{bin}}(\overline \delta m)=\langle\;\rangle$ and $\beta(\langle \;\rangle)\neq 0$ and also $\alpha(\langle \;\rangle)\neq 0$. Define: $\sigma^\ast(\langle\;\rangle):=0$. Note: $\langle 0\rangle \in_I \sigma^\ast$ and $\exists e \sqsubset\langle 0\rangle[e\in D_\alpha]$.

   \smallskip
 Now assume  $k>0$. Let $ c$ in   $ (\omega\times 2)^k$ be given such that  $c\in_I\sigma^\ast$. We  have to define $\sigma^\ast(c)$.
 
 We distinguish two cases.
 
 \textit{Case 1}. $\exists e\sqsubset c[e\in D_\alpha]$. We then define $\sigma^\ast(c):=0$. Note: $\exists e\sqsubset c\ast\langle 0\rangle[e\in D_\alpha]$.
 
 \textit{Case 2}. $\neg \exists e\sqsubset c[e\in D_\alpha]$. Then $\boldsymbol{\pi_{bin}}(c)\in_I \sigma$. 
 
 Note: $length\bigl(\boldsymbol{\pi_{bin}}(c)\bigr)$ is even and find $l$ such that  $2l:=length\bigl(\boldsymbol{\pi_{bin}}(c)\bigr)$. 
  
  Define $\delta$ in $\mathcal{C}$ such that $\delta \in_I \sigma$ and $\boldsymbol{\pi_{bin}}(c)\sqsubset\delta$ and $\forall i[2i+1>2l\rightarrow \delta(2i+1)=0]$.  Find $m$ such that $\beta(\overline \delta m)\neq 0$ and distinguish two cases.
  
  \textit{Case (2a)}. $\exists n[2l\le 2n<m\;\wedge\;\delta(2n)=1]$.
  
   Define:   $ n_0:=\mu n[2l\le 2n<m\;\wedge\;\delta(2n)=1]$ and  $\sigma^\ast(c):=n_0-l$. 
  
  Note $c\ast\langle n_0-l\rangle\in_I\sigma^\ast$ and $\boldsymbol{\pi_{bin}}(c\ast\langle n_0-l\rangle)=\boldsymbol{\pi_{bin}}(c)\ast\underline{\overline 0}(2n_0-2l)\ast\langle 1 \rangle =\overline \delta(2n_0+1)\in_I\sigma$. 
  
  \textit{Case (2b)}. $\forall n[2n<m \rightarrow\delta(2n)=0]$. Conclude: $\boldsymbol{\rho_{bin}}(\overline \delta m)=c$ and $\beta(\overline \delta m)\neq 0$ and  $\alpha(c)=\beta(\overline\delta m)\neq 0$. Define: $\sigma^\ast(\langle\;\rangle):=0$. Note: $c\ast\langle 0\rangle \in_I \sigma^\ast$ and $\exists e \sqsubset c\ast\langle 0\rangle[e\in D_\alpha]$.  
  
   This completes the definition of $\sigma^\ast$. 
   
  \smallskip Now assume $\gamma\in (\omega\times 2)^\omega$ and $\gamma\in_I \sigma^\ast$.  Find $\delta\in\mathcal{C}$ such that $\forall n[\boldsymbol{\pi_{bin}}(\overline \gamma n)\sqsubset \delta]$.  Find $\varepsilon$ in $\mathcal{C}$ such that  such that $\varepsilon \in_I\sigma$ and, for each $n$, if $\overline\delta n \in_I\sigma$, then $\overline\varepsilon n =\overline \delta n$. Find $m$ such that $\overline \varepsilon m\in D_\beta$ and distinguish two cases. 
   
   \textit{Case $(\ast)$}. $\overline\delta m =\overline\varepsilon m$. Conclude: $\overline \delta m \in D_\beta$ and $\boldsymbol{\rho_{bin}}(\overline \delta m)\in D_\alpha$ and $\gamma\in \mathcal{G}_\alpha$.
   
   \textit{Case $(\ast\ast)$}.  $\overline\delta m \neq\overline\varepsilon m$. Then $\overline \delta m \notin_I\sigma$ and $\boldsymbol{\pi_{bin}}(\overline \gamma m)\notin_I\sigma$ and $\exists e\sqsubset\overline \gamma m[e\in D_\alpha]$. Conclude: $\gamma \in \mathcal{G}_\alpha$. 
   
   We thus see: $\forall \gamma \in (\omega\times 2)^\omega[\gamma\in_I\sigma^\ast\rightarrow \gamma \in \mathcal{G}_\alpha]$.

 \smallskip  Assume $\exists s\forall \delta\in\mathcal{C}[\delta\in_I s\rightarrow \delta \in \mathcal{G}_\beta]$.
 
 Find $s$ such that $\forall \delta\in\mathcal{C}[\delta\in_I s\rightarrow \delta \in \mathcal{G}_\beta]$. Consider $p:=length(s)$. Find $q$ such that, for all $c$ in $Halfbin$, if $c\ge q$, then $\boldsymbol{\pi_{bin}}(c)\ge p$.  Define $s^\ast$ such that $length(s^\ast)=q$, inductively.    For each $c<q$ in $\bigcup_k(\omega\times 2)^k$ such that $c\in_Is^\ast$,   $s^\ast(c)$ is defined just as, in the previous paragraph,  where we were given $\sigma \in \mathcal{C}$, $\sigma^\ast(c)$ was defined, for each $c$ in $\bigcup (\omega\times 2)^k$   such that $c\in_I\sigma^\ast$.

 One then may prove: 
 $\forall \gamma \in (\omega\times 2)^\omega[\gamma \in s^\ast\rightarrow\gamma \in \mathcal{G}_\alpha]$.
 \end{proof}

 \begin{lemma}\label{L:prepdet}
One may prove the following statements in $\mathsf{BIM}$.
\begin{enumerate}[\upshape (i)]
\item $\mathbf{\Sigma}^0_1$-$\mathit{Det}_{\omega \times 2}^I \rightarrow \mathbf{\Sigma}^0_1$-$\overleftarrow{\mathbf{AC}_{\omega,2}}$ and
$\neg ! ( \mathbf{\Sigma}^0_1$-$\overleftarrow{\mathbf{AC}_{\omega,2}})\rightarrow \neg !( \mathbf{\Sigma}^0_1$-$\mathit{Det}_{\omega \times 2}^I)$.
\item $\mathbf{\Delta}^0_1$-$\mathit{Det}^I_{\omega \times 2 \times \omega} \rightarrow \mathbf{\Sigma}^0_1$-$\mathit{Det}_{\omega \times 2}^I$ and $\neg!\bigl(\mathbf{\Sigma}^0_1$-$\mathit{Det}_{\omega \times 2}^I\bigr) \rightarrow \neg !\bigl(\mathbf{\Delta}^0_1$-$\mathit{Det}^I_{\omega \times 2 \times \omega}\bigr)$.
\item $\forall m[\mathbf{\Delta}^0_1$-$\mathit{Det}^I_{(\omega \times 2)^m}] \rightarrow \mathbf{\Delta}^0_1$-$\mathit{Det}^I_{\omega \times 2 \times \omega}$ and 

$\neg!(\mathbf{\Delta}^0_1$-$\mathit{Det}_{\omega \times 2\times \omega}^I\bigr)
\rightarrow \exists m[\neg!(\mathbf{\Delta}^0_1$-$\mathit{Det}^I_{(\omega \times 2)^m})]$.

\item  $\forall m[\mathbf{\Sigma}^0_1$-$\mathit{Det}^I_{(\omega \times 2)^\omega} \rightarrow \mathbf{\Delta}^0_1$-$\mathit{Det}^I_{(\omega \times 2)^m} \;\wedge\;$
 
 $\neg !(\mathbf{\Delta}^0_1$-$\mathit{Det}^I_{(\omega \times 2)^m}) \rightarrow \neg !(\mathbf{\Sigma}^0_1$-$\mathit{Det}^I_{(\omega \times 2)^\omega})]$.

\item $\mathbf{\Sigma}^0_1$-$\mathit{Det}^I_\mathcal{C} \rightarrow \mathbf{\Sigma}^0_1$-$\mathit{Det}^I_{(\omega \times 2)^\omega}$ and
 $\neg !(\mathbf{\Sigma}^0_1$-$\mathit{Det}^I_{(\omega \times 2)^\omega}) \rightarrow \neg !(\mathbf{\Sigma}^0_1$-$\mathit{Det}^I_\mathcal{C})$, and \\$\mathbf{\Sigma}^0_1$-$\mathit{\;^\ast Det}^I_\mathcal{C} \rightarrow \mathbf{\Sigma}^0_1$-$\mathit{\;^\ast Det}^I_{(\omega \times 2)^\omega}$ and
 $\neg !(\mathbf{\Sigma}^0_1$-$\mathit\;{^\ast Det}^I_{(\omega \times 2)^\omega}) \rightarrow \neg !(\mathbf{\Sigma}^0_1$-$\mathit{\;^\ast Det}^I_\mathcal{C})$.
 \item $\mathbf{\Sigma}^0_1$-$\mathit{Det}^{II}_\mathcal{C} \rightarrow \mathbf{\Sigma}^0_1$-$\mathit{Det}^I_\mathcal{C}$ and
 $\neg !(\mathbf{\Sigma}^0_1$-$\mathit{Det}^I_\mathcal{C}) \rightarrow \neg !(\mathbf{\Sigma}^0_1$-$\mathit{Det}^{II}_\mathcal{C})$, and 
 
 $\mathbf{\Sigma}^0_1$-$\mathit{^\ast Det}^{II}_\mathcal{C} \rightarrow \mathbf{\Sigma}^0_1$-$\mathit{^\ast Det}^I_\mathcal{C}$ and
 $\neg !(\mathbf{\Sigma}^0_1$-$\mathit{^\ast Det}^I_\mathcal{C}) \rightarrow \neg !(\mathbf{\Sigma}^0_1$-$\mathit{^\ast Det}^{II}_\mathcal{C})$.
 \item $\mathbf{FT} \rightarrow \mathbf{\Sigma}^0_1$-$\mathit{^\ast Det}^{II}_\mathcal{C}$ and $\neg !(\mathbf{\Sigma}^0_1$-$\mathit{^\ast Det}^{II}_\mathcal{C}) \rightarrow \neg!\mathbf{FT}$.
\end{enumerate}
\end{lemma}
\begin{proof} (i) 
 We prove: given any $\alpha$, one may construct $\beta$ such that 
  $$\forall \gamma \in \mathcal{C} \exists n[\bigl(n,\gamma(n)\bigr) \in E_{\alpha}]\rightarrow \forall \tau \in \mathcal{C}\exists n [\langle n, \tau(n)\rangle \in E_\beta]\;\mathrm{and}$$
$$\exists n[\langle n, 0\rangle \in E_\beta \; \wedge \; \langle n, 1\rangle \in E_\beta]\rightarrow \exists n[(n,0) \in E_\alpha\;\wedge\;(n,1)\in E_\alpha].$$
The two promised conclusions then follow easily.

Given $\alpha$, define $\beta$ such that $\forall n\forall i<2[\alpha(p)=(n,i)+1\leftrightarrow \beta(p) = \langle n, i\rangle +1]$  and \\$\forall p[\neg \exists n\exists i<2[\alpha(p)=(n,i)+1]\rightarrow\beta(p) =0]$.

\smallskip Clearly, $\beta$ satisfies the requirements.

\smallskip (ii) We prove: given any $\alpha$, one may construct $\beta$ such that 
$$\forall \tau \in \mathcal{C}\exists n[\langle n, \gamma(n)\rangle\in E_\alpha] \rightarrow \mathit\forall\tau\in\mathcal{C}\exists n\exists p[\langle n, \tau(n),p\rangle\in D_\beta]\;\mathrm{and}$$  $$\exists n\forall i<2\exists p[\langle n,i,p\rangle\in D_\beta] \rightarrow \mathit\exists n\forall i<2[\langle n,i\rangle\in E_\alpha].$$
The two promised conclusions then follow easily.

Given $\alpha$, define $\beta$   such that  $\forall n\forall i<2\forall p[  \beta(\langle n,i,p\rangle)  \neq 0\leftrightarrow\alpha(p)= \langle n,i\rangle +1].$\\ Note: $\forall n\forall i<2[
 \langle n , i\rangle \in E_\alpha \leftrightarrow \exists p[\langle n , i , p \rangle \in D_{\beta}]]$.
 
 \smallskip           Clearly, $\beta$ satisfies the requirements.

\smallskip
(iii) Note: $\mathbf{\Delta}^0_1$-$\mathit{Det}^I_{(\omega \times 2)^2} \rightarrow \mathbf{\Delta}^0_1$-$\mathit{Det}^I_{\omega \times 2 \times \omega}$ and 

$\neg!(\mathbf{\Delta}^0_1$-$\mathit{Det}_{\omega \times 2\times \omega}^I\bigr)
\rightarrow \neg!(\mathbf{\Delta}^0_1$-$\mathit{Det}^I_{(\omega \times 2)^2})$.

\smallskip
(iv) Let $m$ be given. We prove: given any $\alpha$ one may construct $\beta$ such that 
$$\forall \tau \in \mathcal{C} \exists c\in(\omega\times 2)^m[ c \in_{II} \tau \; \wedge\; c \in D_\alpha]\rightarrow \forall \tau \in\mathcal{C}\exists \gamma \in \mathcal (\omega\times 2)^\omega[\gamma \in_{II}\tau\;\wedge\; \gamma\in\mathcal{G}_\beta]\;\mathrm{and}$$  $$\exists\sigma\forall \gamma\in (\omega\times 2)^\omega[\gamma \in_I \sigma \rightarrow \gamma \in \mathcal{G}_\beta]\rightarrow\exists s\forall c\in(\omega\times 2)^m[ c \in_I s  \rightarrow c \in D_\alpha].$$
The two promised conclusions then follow easily.

Given $ \alpha$, define $\beta$  such that $\forall s [\beta(s)\neq 0\leftrightarrow\bigl( s\in (\omega\times 2)^{m}\;\wedge\;\alpha(s)\neq 0\bigr)]$. 

\smallskip Observe that, if  $\forall \tau \in \mathcal{C} \exists c\in(\omega\times 2)^m[ c \in_{II} \tau \; \wedge c \in D_\alpha]$, then \\
 $\forall \tau \in \mathcal{C}\exists \gamma \in \mathcal{C}[\gamma \in_{II}\tau\;\wedge\;\beta\bigl(\overline \gamma(2 m)\bigr) \neq 0 ]$, that is: $\forall \tau \in \mathcal{C}\exists \gamma \in \mathcal{C}[ \gamma \in_{II}\tau\;\wedge\;\gamma \in \mathcal{G}_\beta ]$.

\smallskip Let $\sigma$ in $\mathcal{C}$ be given such that  $ \forall \gamma\in (\omega\times 2)^\omega[\gamma \in_I \sigma \rightarrow \exists n[\beta(\overline \gamma n) \neq 0]]$. \\Conclude: $\forall \gamma\in (\omega\times 2)^\omega[\gamma \in_I \sigma  \rightarrow \alpha\bigl( \overline \delta(2m)\bigr)\neq 0]$. \\Find $N$ such that $\forall c \in \bigcup_{n\le 2m}\omega^n[ c \in_I \gamma\rightarrow c <N]$ and define  $s:= \overline \sigma N$. \\Conclude:
$\forall c\in(\omega\times 2)^m[ c \in_I s  \rightarrow c \in D_\alpha] $.
 
\smallskip We thus see that $\beta$ satisfies the requirements. 
 
 \smallskip
 (v) For each $\alpha$, one may construct $\beta$ such that  $$\forall \tau \in \mathcal{C}\exists \gamma \in (\omega \times 2)^\omega[\gamma \in_{II}\tau\;\wedge\;\gamma\in\mathcal{G}_\alpha] \rightarrow\forall \tau \in \mathcal{C}\exists \delta \in \mathcal{C}[\delta \in_{II}\tau\;\wedge\;\delta\in\mathcal{G}_\beta]\;\mathrm{and}$$   $$\exists \sigma\forall \delta\in\mathcal{C}[\delta\in_I \sigma\rightarrow \delta \in \mathcal{G}_\beta]\rightarrow \exists \sigma \forall \gamma \in (\omega\times 2)^\omega[\gamma\in_I \sigma\rightarrow \gamma \in \mathcal{G}_\alpha]\;\mathrm{and}$$ $$\exists s\forall \delta\in\mathcal{C}[\delta\in_I s\rightarrow \delta \in \mathcal{G}_\beta]\rightarrow \exists s\forall \gamma\in(\omega\times 2)^\omega[\gamma\in_I s\rightarrow \gamma \in\mathcal{G}_\alpha].$$
 
 The proof has been given in Lemma \ref{L:simulate}.

 The promised conclusions  follow easily. 
 
\smallskip
(vi) We prove: given any $\alpha$, one may construct $\beta$ such that $$ \forall \sigma \in \mathcal{C}\exists c \in \mathit{Bin}[c \in_{I} \sigma \;\wedge\; \alpha(c) \neq 0] \rightarrow\forall \tau\in \mathcal{C}\exists d \in \mathit{Bin}[d \in_{II} \tau \;\wedge\; \beta(d) \neq 0]\;\mathrm{and}$$ 
$$ \exists\tau\in\mathcal{C}\forall\delta\in\mathcal{C}[\delta \in_{II} \tau\rightarrow\exists n[\beta(\overline \delta n) \neq 0]]\rightarrow \exists \sigma\in \mathcal{C}\forall \delta \in \mathcal{C}[\delta \in_{I} \sigma\rightarrow \exists n[\alpha(\overline \delta n) \neq 0]]$$  
$$\mathrm{and}\; \exists t\forall\delta\in\mathcal{C}[\delta \in_{II} t\rightarrow\exists n[\beta(\overline \delta n) \neq 0]]\rightarrow \exists t\forall \delta \in \mathcal{C}[\delta \in_{I} t\rightarrow \exists n[\alpha(\overline \delta n) \neq 0]].$$
The promised conclusions then follow easily.

 Given $\alpha$, define $\beta$ such that  $\beta(0) =0$ and  $\forall c \in Bin[\beta(\langle 0 \rangle \ast c) = \beta(\langle 1 \rangle \ast c) = \alpha(c)]$.

\smallskip
Assume: $\forall \sigma \in \mathcal{C}\exists c \in \mathit{Bin}[c \in_{I} \sigma \;\wedge\; \alpha(c) \neq 0]$. 

Let $\tau$ in $\mathcal{C}$ be given.  
\\Define $\sigma$ such that $\forall c \in \mathit{Bin}[\sigma(c) = \tau\bigl(\langle 0\rangle \ast c\bigr)]$.\\ Find $c$ in $\mathit{Bin}$ such that $c \in_{I} \sigma$ and $\alpha(c) \neq 0$.\\ Define $d:=\langle 0\rangle \ast c$ and note $d \in_{II} \tau$ and $\beta(d) \neq 0$.

We thus see: $\forall \tau \in \mathcal{C}\exists d \in \mathit{Bin}[d \in_{II} \tau \;\wedge\; \beta(d) \neq 0]$.

\smallskip
Let $\tau$  in $\mathcal{C}$ be given such that $\forall\delta\in\mathcal{C}[\delta \in_{II} \tau\rightarrow\exists n[\beta(\overline \delta n) \neq 0]]$. 

Define $\sigma$  such that $\forall c\in\mathit{Bin}\forall i<2[\sigma(\langle i\rangle\ast c) = \tau(c)]$.

Note: $\forall\delta\in\mathcal{C}[\delta \in_{I} \sigma\rightarrow  \langle 0 \rangle\ast\delta \in_{II} \tau]$, so $\forall\delta\in\mathcal{C}[\delta \in_{I} \sigma\rightarrow \exists n[\beta( \overline{\langle 0\rangle\ast \delta} n) \neq 0]]$   and: $\forall \delta \in \mathcal{C}[\delta \in_{I} \sigma\rightarrow \exists n[\alpha(\overline \delta n)\neq 0]]$.

\smallskip
Let $t$  be given such that $\forall\delta\in\mathcal{C}[\delta \in_{II} t\rightarrow\exists n[\beta(\overline \delta n) \neq 0]]$. Define $s$  such that, $\forall c\in\mathit{Bin}[\langle 0\rangle\ast c<length(t)\rightarrow s(c) = t(\langle 0 \rangle \ast c)]$.

Conclude, as above: $\forall \delta \in \mathcal{C}[\delta \in_{I} s\rightarrow \exists n[\alpha(\overline \delta n)\neq 0]]$.

\smallskip We thus see that $\beta$ satisfies the requirements.

\smallskip (vii) We prove: for each $\alpha$, there exists $\beta$ such that  $$\forall \gamma \in \mathcal{C}\exists s[s \in_I \gamma \;\wedge\; \alpha(s) \neq 0]\rightarrow \mathit{Bar}_\mathcal{C}(D_\beta)\;\mathrm{and}$$ $$\exists m[Bar_\mathcal{C}(D_{\overline\beta m})]\rightarrow \exists c\forall \delta \in \mathcal{C}[ \delta \in_{II} c \rightarrow \exists n[\alpha(\overline \delta n) \neq 0]].$$
The two promised conclusions then follow easily.

Given $\alpha$, define $\beta$  such that $\forall c \in \mathit{Bin}[\beta(c) \neq 0\leftrightarrow \exists s<c[s \in_{I} c\;\wedge\;\alpha(s) \neq 0]]$.

Assume  $\forall \gamma \in \mathcal{C}\exists s[s \in_I \gamma \;\wedge\; \alpha(s) \neq 0]$.

 Clearly,  $\forall\gamma\in\mathcal{C}\exists n[\beta(\overline \gamma n) \ne 0]$, that is: $\mathit{Bar}_\mathcal{C}(D_\beta)$.

\smallskip
Let  $m$ be given such that $Bar_\mathcal{C}(D_{\overline \beta m})$.  

 Define  $X:=\{s\in\mathit{Bin}_m \mid \exists n \le m[\alpha(\overline s n) \neq 0]\}$. Note: $\forall b\exists s[s \in_I b \;\wedge\; s \in X]$. According to  Theorem \ref{T:finitedet}, $\mathit{Det}^{II}_{\mathit{Bin}_m}(X)$. 
 
 Find $c$ such that $\forall s\in \mathit{Bin}_m[ s \in_{II} c \rightarrow s \in X]$. 
 
 Conclude: $\forall \delta \in \mathcal{C}[\delta \in_{II} c \rightarrow \overline \delta m \in X]$ and: $\forall \delta \in \mathcal{C}[ \delta \in_{II} c \rightarrow \exists n[\alpha(\overline \delta n) \neq 0]]$.  
 
 \smallskip We thus see that $\beta$ satisfies the requirements.
\end{proof}

 \begin{theorem}\label{T:det}\begin{enumerate}[\upshape (i)]
 \item $\mathsf{BIM} \vdash    \mathbf{\Sigma}^0_1$-$\overleftarrow{\mathbf{AC}_{\omega, 2}} \leftrightarrow \mathbf{\Sigma}^0_1$-$\mathit{Det}_{\omega \times 2}^I \leftrightarrow \mathbf{\Delta}^0_1$-$\mathit{Det}^I_{\omega \times 2 \times \omega} \leftrightarrow \\\forall m[\mathbf{\Delta}^0_1$-$\mathit{Det}^{I}_{(\omega \times 2)^m}] \leftrightarrow   \mathbf{\Sigma}^0_1$-$\mathit{Det}^I_{(\omega \times 2)^\omega}\leftrightarrow   \mathbf{\Sigma}^0_1$-$\mathit{\;^\ast Det}^I_{(\omega \times 2)^\omega}\leftrightarrow   \mathbf{\Sigma}^0_1$-$\mathit{Det}^I_{\mathcal{C}}\leftrightarrow  \\ \mathbf{\Sigma}^0_1$-$\mathit{\;^\ast Det}^I_{\mathcal{C}}\leftrightarrow   \mathbf{\Sigma}^0_1$-$\mathit{Det}^{II}_\mathcal{C}\leftrightarrow   \mathbf{\Sigma}^0_1$-$\mathit{\;^\ast Det}^{II}_\mathcal{C}\leftrightarrow \mathbf{FT}$.
 
 \item  $\mathsf{BIM} \vdash   \neg !(\mathbf{\Sigma}^0_1$-$\overleftarrow{\mathbf{AC}_{\omega, 2}}) \leftrightarrow \neg !(\mathbf{\Sigma}^0_1$-$\mathit{Det}_{\omega \times 2}^I) \leftrightarrow \neg !(\mathbf{\Delta}^0_1$-$\mathit{Det}^I_{\omega \times 2 \times \omega}) \leftrightarrow \\\exists m[\neg !(\mathbf{\Delta}^0_1$-$\mathit{Det}^{I}_{(\omega \times 2)^m})] \leftrightarrow   \neg !(\mathbf{\Sigma}^0_1$-$\mathit{Det}^I_{(\omega \times 2)^\omega})\leftrightarrow   \neg !(\mathbf{\Sigma}^0_1$-$\mathit{\;^\ast Det}^I_{(\omega \times 2)^\omega})\leftrightarrow  \\ \neg !(\mathbf{\Sigma}^0_1$-$\mathit{Det}^I_{\mathcal{C}})\leftrightarrow   \neg !(\mathbf{\Sigma}^0_1$-$\mathit{\;^\ast Det}^I_{\mathcal{C}})\leftrightarrow   \neg !(\mathbf{\Sigma}^0_1$-$\mathit{Det}^{II}_\mathcal{C})\leftrightarrow   \neg !(\mathbf{\Sigma}^0_1$-$\mathit{\;^\ast Det}^{II}_\mathcal{C})\leftrightarrow \neg!\mathbf{FT}$.

\end{enumerate}
\end{theorem}
 
\begin{proof}Use Lemmas \ref{L:3resolutions} and \ref{L:prepdet}.  \end{proof}
 
 \subsection{The Intuitionistic Determinacy Theorem} \hfill
Recall that  $\mathcal{X}\subseteq(\omega \times 2)^\omega $ is \textit{$I$-determinate in $(\omega \times 2)^\omega$}
if and only if $$\forall \tau \in \mathcal{C} \exists \gamma[\gamma \in_{II}\tau \;\wedge\; \gamma \in \mathcal{X}]\rightarrow \exists \sigma\forall \gamma \in (\omega \times 2)^\omega[\gamma \in_I \sigma \rightarrow \gamma \in \mathcal{X}].$$ 

 We now define: $\mathcal{X}\subseteq(\omega \times 2)^\omega $ is \textit{weakly $I$-determinate in} $(\omega \times 2)^\omega$   
if and only if $$\forall\varphi:\mathcal{C}\rightarrow (\omega\times 2)^\omega[\forall \tau \in \mathcal{C}[\varphi|\tau \in_{II}\tau \;\wedge\; \varphi|\tau \in \mathcal{X}]\rightarrow \exists \sigma\forall \gamma \in (\omega \times 2)^\omega[\gamma \in_I \sigma \rightarrow \gamma \in \mathcal{X}]].$$ 
  A (continuous) function  $\varphi:\mathcal{C}\rightarrow (\omega \times 2)^\omega$ such that $\forall \tau \in \mathcal{C}[\varphi|\tau \in_{II} \tau]$ will be  called an \textit{anti-strategy for player $I$ in $(\omega \times 2)^\omega$}.
  
Note that  the   Second Axiom of Continuous Choice, $\mathbf{AC}_{\omega^\omega,\omega^\omega}=\mathbf{AC}_{1,1}$, implies,
for every subset $\mathcal{X}\subseteq(\omega \times 2)^\omega$: 
 if $\forall \tau\in\mathcal{C}\exists\gamma \in (\omega\times 2)^\omega[\gamma\in_{II}\tau\;\wedge\;\gamma \in \mathcal{X}]$, then $\exists\varphi: \mathcal{C}\rightarrow(\omega\times 2)^\omega\forall \tau\in\mathcal{C}[\varphi|\tau \in_{II}\tau\;\wedge\;\varphi|\tau \in \mathcal{X}]$.

$\mathbf{AC}_{\omega^\omega,\omega^\omega}$ thus implies: if $\mathcal{X}\subseteq (\omega\times 2)^\omega$ is weakly $I$-determinate in $(\omega\times 2 )^\omega$, then $\mathcal{X}$ is $I$-determinate in $(\omega\times 2 )^\omega$.

\smallskip
 
 Earlier versions of the next Theorem may be found in \cite[Chapter 16]{veldman81} and \cite{veldman2009}.
 
 \begin{theorem}[Intuitionistic Determinacy Theorem]\label{T:det2}
 The following statements are equivalent in $\mathsf{BIM}$:

\begin{enumerate}[\upshape(i)]
 \item $\mathbf{FT}$.
  \item For every anti-strategy $\varphi$ for player $I$ in  $(\omega \times 2)^\omega$ there exists  a strategy $\sigma$ for player I in  $(\omega \times 2)^\omega$ such that  $\forall \gamma\in (\omega\times 2)^\omega[\gamma\in_I \sigma\rightarrow \exists \tau \in \mathcal{C}[\gamma=\varphi|\tau]].$
  
\end{enumerate}
\end{theorem}
  
  \begin{proof}
   (i) $\Rightarrow$ (ii). Let $\varphi:\mathcal{C}\rightarrow (\omega\times 2)^\omega$ satisfy: $\forall \tau\in\mathcal{C}[\varphi|\tau \in_{II} \tau]$. 
   
   We intend to develop a strategy $\sigma$ for player $I$ in $(\omega\times 2)^\omega$ such that Player $I$, whenever he  uses $\sigma$ and develops, together with player $II$, $\gamma$ in $(\omega\times 2)^\omega$, will be able to construct, simultaneously,  a strategy $\tau$ for player $II$ in $(\omega\times 2 )^\omega$ such that $\gamma =\varphi|\tau$.   The infinite sequence $\gamma$ must be the answer given by the anti-strategy $\varphi$ to player $II$'s  strategy $\tau$.  So, while playing $\gamma$, player $I$  \textit{conjectures} a  strategy $\tau$ that player $II$ may be assumed to follow  during this  very play.

  \medskip 
     Let  $c, t$ be given such that $c\in\bigcup_k(\omega\times 2)^k$ and  $t\in Bin$.    We define:  \textit{with respect to the given anti-strategy $\varphi$, $t$ is, at the position $c$, a safe (partial) conjecture by player $I$ about the strategy followed by player $II$},  notation:
   $\mathit{Safe}( c, t)$, if and only if\footnote{For the notation $^cu$, see Subsection \ref{SS:subs}.}
  \[\forall\rho \in \mathcal{C}  \exists u\in \mathit{Bin}[t\sqsubset u\;\wedge\;  ^cu \sqsubset \rho \;\wedge \; c \sqsubseteq \varphi|u].\] 
   
 Note\footnote{For the notation $^c\tau$, see Subsection \ref{SS:subs}.}: $\mathit{Safe}( c, t)\leftrightarrow \forall \rho \in \mathcal{C}  \exists \tau\in\mathcal{C}[t\sqsubset \tau\;\wedge\;  ^c\tau = \rho \;\wedge \; c \sqsubset \varphi|\tau]$.
 
 If $\mathit{Safe}( c, t)$, then    player $I$,  at the position $c$, may extend  every strategy $\rho$ of  player $II$ \textit{`above $c$'}, that is, on  positions $d$ in $(\omega\times 2)^{<\omega}\times\omega$ such  that $c\sqsubseteq d$,  to a strategy $\tau$ of  player $II$ on the  whole of $(\omega\times 2)^{<\omega}\times\omega$ such that  $t\sqsubset \tau$  and $c\sqsubset \varphi|\tau$. 
  
  \smallskip We now prove the following important observation:  $$\forall c\in\bigcup_k(\omega\times 2)^k\forall t\in Bin[\mathit{Safe}( c, t) \vee \neg \mathit{Safe}( c, t)].$$
 
    Let $c$ in $\bigcup_k(\omega\times 2)^k$ and $t$ in $Bin$ be given.  
  
  Define  $\beta$ such that 
   $\forall u [\beta(u)\neq 0\leftrightarrow \bigl(u\in\mathit{Bin}\;\wedge\;( c\sqsubseteq \varphi|u\;\vee\;c\perp\varphi|u)\bigr)]$.  
  
  Note: $Bar_\mathcal{C}(D_\beta)$, and: $\forall u[u\in D_\beta\rightarrow\forall i<2[u\ast\langle i \rangle \in D_\beta]]$.
  
  Using $\mathbf{FT}$, find $n>\mathit{length}(t)$ such that  $Bin_n\subseteq D_\beta$. 
  
   Note: $\forall u \in \mathit{Bin}_n[\mathit{length}(^cu)\le length(u)\le n]$  
  and
  \\$\mathit{Safe}(c,t)\leftrightarrow\forall r \in Bin_n \exists u \in \mathit{Bin}_n[t\sqsubseteq u \;\wedge\;  ^cu \sqsubseteq r\;\wedge \; c\sqsubseteq\varphi|u].$ 
    
   We thus see that one may prove: $\mathit{Safe}(c,t)\;\vee\;\neg \mathit{Safe}(c,t)$.

  \medskip Going one step further, we  want to  define $\sigma$ such that $$\forall c\in \bigcup_k(\omega\times 2)^k\forall t\in Bin[\sigma(c,t)\neq 0\leftrightarrow \mathit{Safe}( c ,t )].$$
  
  To this end, we first define $\psi$ such that, \\for all $c$ in $\bigcup_k(\omega\times 2)^k$, for all $t$ in $\mathit{Bin}$, 
  $$\psi(c,t):=\mu n[n>\mathit{length}(t)\;\wedge\;\forall u \in \mathit{Bin}_n[c\sqsubseteq \varphi| u \;\vee\;c\perp\varphi|u]],$$
  and then define $\sigma$ such that, for all $c$ in $\bigcup_k(\omega\times 2)^k$, for all $t$ in $Bin$,

  $\sigma(c,t)\neq 0\leftrightarrow$ $\forall r \in Bin_{\psi(c,t)} \exists u \in \mathit{Bin}_{\psi(c,t)}[t\sqsubseteq u \;\wedge\;  ^cu \sqsubseteq r\;\wedge \; c\sqsubseteq\varphi|u].$

 \medskip We also define $\chi$ such that, for all $c$ in $\bigcup_k(\omega\times 2)^k$, for all $t$ in $Bin$,  \\ if  $\sigma(c,t)=0$,  then \\$\chi(c,t)=\mu r[r\in Bin_{\psi(c,t)}\;\wedge\;\neg \exists u \in \mathit{Bin}_{\psi(c,t)}[ t \sqsubseteq u\;\wedge\; ^cu \sqsubset r\;\wedge c\sqsubseteq\varphi|u]].$ 
 
 Note that, if   $\sigma( c, t)=0$, then  \\  $\forall\rho\in\mathcal{C}[\chi(c,t)\sqsubset \rho\rightarrow\neg \exists u\in Bin[ t\sqsubseteq u\;\wedge\;^cu\sqsubset \rho \;\wedge\;  c\sqsubseteq \varphi|u]]$.

\medskip

 Let $n,c,t$ be given such that  $c\in(\omega\times 2)^n$ and $t$ in $Bin$  and   $\mathit{Safe}( c, t)$. 

 We  prove:
 \[\forall \rho \in\mathcal{C}\exists d\in \bigcup_{k>0}(\omega\times 2)^k[d\in_{II}\rho \;\wedge\; \exists i<2[\mathit{Safe}( c\ast d, t\ast\langle i \rangle)]].\]

 Let $\rho$ in $\mathcal{C}$ be given. Find $u$ in $\mathit{Bin}$ such that $t\sqsubset u\;\wedge\;^cu\sqsubset \rho \;\wedge\; c\sqsubseteq \varphi|u$. 
  
 Find $i<2$ such that $t\ast\langle i\rangle\sqsubseteq u$.
 
 Find $p$ such that, for all $d$ in $\bigcup_k(\omega\times 2)^k\times \omega$, 
 
 if $d< length(^cu)$, then $length(d)<2p$.

 Define $D:=\{d\in (\omega\times 2)^p\mid d \in_{II}\;\rho\;\wedge\;\exists \tau\in\mathcal{C}[c\ast d\sqsubset \varphi|\tau]\}$. 
 
  It follows from $\mathbf{FT}$ that $E:=\{\overline{\phi|\tau}(2n+2p)\mid\tau\in\mathcal{C}\}$ is a finite\footnote{$D\subseteq \omega$ is {\it finite} if and only if  there exists $a$ such that $D=D_a=\{n<length(a)\mid a(n)\neq 0\}$, see Subsection \ref{SS:decidenum}.} subset of $\omega$. 
  Note that, for all $d$ in $(\omega\times 2)^p$, $d\in D$ if and only if $\exists b \in E[c\ast d = b]$ and conclude:   also $D$ is a finite subset of $\omega$. 
 
 Assume $\forall d\in D[\neg \mathit{Safe}(c\ast d, t\ast\langle i\rangle)]$. 
 
 Define $\rho^\ast$ in $\mathcal{C}$ such that  $^cu\sqsubset \rho ^\ast$ and $\forall d \in D[d\in_{II}\rho^\ast]$ and   \\ $\forall d\in D[\chi(c\ast d, t\ast\langle i\rangle)\sqsubset\; ^{d}(\rho^\ast)]$.

 Find $\tau$ in $\mathcal{C}$ such that $u\sqsubset \tau$ and $^c\tau=\rho^\ast$ and  consider $\varphi|\tau$.  
 
 Note: $c\sqsubseteq\varphi|u\sqsubset \varphi|\tau$.

 Find $d$ in $D$ such that $c\ast d\sqsubset \varphi|\tau$. 
 
 Note: $^{c\ast {d}}\tau =\;^{d}(\rho^\ast)$ and $\chi(c\ast d, t\ast\langle i\rangle)\sqsubset\;^d(\rho^\ast)$.

 Conclude:  $\neg\exists v\in Bin[t\ast\langle i\rangle \sqsubseteq v \;\wedge\; ^{c\ast d}v \sqsubseteq\;^d(\rho^\ast)\;\wedge\; c\ast d
\sqsubseteq \varphi|v]$. 
 
 On the other hand, $t\ast \langle i\rangle\sqsubseteq u\sqsubset \tau$ and $ \exists m[ c\ast d\sqsubseteq \varphi|\overline \tau m]$. 
 
 Note that $^c\tau=\rho^\ast$ and, therefore, $^{c\ast d}\overline \tau m \sqsubset\;^d(\rho^\ast)$.
 
 Conclude: $\exists m[t\ast \langle i\rangle\sqsubset \overline \tau m\;\wedge\;^{c\ast d}\overline \tau m \sqsubset\;^d(\rho^\ast)\;\wedge\;c\ast d\sqsubseteq \varphi|\overline \tau m]$.
 
 Contradiction.
 
 \smallskip We thus see: $\neg\forall d\in D[\neg\mathit{Safe}(c\ast d, t\ast\langle i\rangle)]$. 
 
 As $D$ is finite,  conclude:  $\exists d\in D[\mathit{Safe}( c\ast d, t\ast\langle i \rangle)]$. 
 
 \smallskip
We thus see: $\forall \rho \in\mathcal{C}\exists d\in\bigcup_{k>0}(\omega\times 2)^k[d\in_{II}\rho \;\wedge\; \mathit{Safe}( c\ast d, t\ast\langle i \rangle)]$.
 
\medskip
  
  Using   $\mathbf{\Sigma}^0_1$-$\;\mathit{Det}^I_{(\omega \times 2)^\omega}$, 
   an equivalent of $\mathbf{FT}$, see Theorem   \ref{T:det}(i), find $\sigma$ such that 
   
   $\forall \delta[\delta \in_I \sigma \rightarrow \exists k[ \mathit{Safe}( c\ast\overline\delta (2k+2), t\ast\langle i \rangle)]].$

    Note that  the set $\{\delta\in (\omega\times 2)^\omega\mid \delta \in_I\sigma\}$ is an explicit fan.

    Using $\mathbf{FT}$, find $m$ such that  $\forall \delta[\delta \in_I \sigma \rightarrow \exists k \le m[\mathit{Safe}( c\ast\overline\delta (2k+2), t\ast\langle i \rangle)]]$.
    
    Find $p$ such that $\forall \delta[\delta\in_I\sigma\rightarrow \overline \delta(2m+2)<p]$ and define
    $s:=\overline \sigma p$. 
    
    Note:  
for all $d$ in $(\omega\times 2)^{<\omega}$, if $d>\mathit{length}(s)$ and  $d\in_I s$, then \\$\exists n[0<2n\le\mathit{length}(d) \;\wedge\; \mathit{Safe}(c\ast \overline d(2n), t\ast\langle i \rangle)].$
  
  \medskip
Note that, for all $s$, the set of all $d$ in $(\omega\times 2)^{<\omega}$ such that $d>\mathit{length}(s)$ and $\forall n[2n <length(d) \rightarrow \overline d(2n)<length(s)]$ and  $d\in_Is$ is a finite subset of $\omega$.

Define $\alpha$ such that for all $c$ in $(\omega\times 2)^{<\omega}$,    for all $t$ in $Bin$, for all $s$, \\  $\alpha(c,t,s)\neq 0$ if and only if,
  for all $d$ in $(\omega\times 2)^{<\omega}$, \\if $d>\mathit{length}(s)$ and $\forall n[2n <length(d
  ) \rightarrow \overline d(2n)<length(s)]$ and  $d\in_I s$, then \\$\exists n[0<2n\le\mathit{length}(d) \;\wedge\;\exists i<2[ \mathit{Safe}(c\ast \overline d(2n), t\ast\langle i \rangle)]].$

  We have proven: $\forall c \in (\omega\times 2)^{<\omega}\forall t \in Bin[\mathit{Safe}(c,t)\rightarrow \exists s[\alpha(c,t,s)\neq 0]]$.

   \medskip Define  $\eta$  such that $$\forall  c \in (\omega\times 2)^{<\omega}\forall t\in\mathit{Bin}[\mathit{Safe}( c, t)\rightarrow\eta(c,t)=\mu s[\alpha(c,t,s)\neq 0]].$$

Define $\lambda$ such that
     $\lambda(\langle \; \rangle) = \langle \; \rangle$, and,
      for all  $c$ in $(\omega \times 2)^{<\omega}$, for all $\langle p, i\rangle$ in $\omega\times 2$,  \textit{if}  there exists $j<2$ such that $\mathit{Safe}\bigl( c\ast \langle p, i \rangle, \lambda(c)\ast\langle j \rangle\bigr)$, then $\lambda(c\ast \langle p, i \rangle) = \lambda(c)\ast \langle j_0 \rangle$, where $j_0$ is the least such $j$, and, \textit{if not}, then $\lambda(c\ast\langle p, i\rangle)= \lambda(c)$. 
 
 Note: $\forall c \in (\omega \times 2)^{<\omega}[\mathit{Safe}\bigl( c , \lambda(c)\bigr)]$.
 
 Note: $\forall c\in(\omega\times 2)^{<\omega}\forall d\in(\omega\times 2)^{<\omega}[c\sqsubseteq d\rightarrow \lambda(c)\sqsubseteq \lambda(d)]$.
 
    \smallskip
    Define  $\sigma$ as a strategy for player $I$ in $(\omega\times 2)^\omega$ as follows.  
    
     Let $c$
   in $ (\omega\times 2)^{<\omega}$ be given. 
   
   Find  
       $n_0:=\mu n[\mathit{Safe}( \overline c(2n), \lambda(c))]$.
       
   Find $d$ such that $c = \overline c(2n_0)\ast d$.
   
    Define $\sigma(c) := \bigl(\eta(\overline c(2n_0), \lambda(c))\bigr)(d)$.

     \medskip 
     Let $\gamma$ be an element of $(\omega \times 2)^\omega$ such that $\gamma \in_I  \sigma$. 
      
      \smallskip
      We claim:
     $\forall n\exists m\ge n[\lambda\bigl(\overline \gamma(2m)\bigr) \sqsubset \lambda\bigl(\overline \gamma(2m+2)\bigr)]$. 
     
     We prove this claim as follows. 
     
     Let $n$ be given. Find $n_0:=\mu k[\mathit{Safe}\bigl(\overline \gamma(2k), \lambda(\overline \gamma (2n)\bigr)]$. 
     
     Define $c:=\overline \gamma (2n_0)$ and $t:=\lambda(c)$.
     
     Consider $s:=\eta(c,t)$.

     Find $\delta$ such that $\delta\in_I s$ and $\forall k[\delta(2k+1)= \gamma(2n_0+2k+1)]$. 
     
     Find  $k_0:=\mu k[\exists i<2[\mathit{Safe}\bigl(c\ast \overline \delta(2k+2), t\ast\langle i \rangle\bigr)]$.
     
     Note: $c\ast \overline\delta (2k_0+2) \sqsubset \gamma$.
     
     Note:  $\lambda(c)=\lambda\bigl(\overline \gamma (2n)\bigr)= \lambda\bigl(\overline \gamma (2n_0+2k_0)\bigr)\sqsubset \lambda\bigl(\overline\gamma(2n_0+2k_0+2)\bigr)$. 
     
     Defining $m:=n_0+k_0$, we see: $m\ge n$ and  $\lambda\bigl(\overline \gamma(2m)\bigr) \sqsubset \lambda\bigl(\overline \gamma(2m+2)\bigr)$.
     
     This ends the proof of our claim.

     \smallskip
     Find $\tau$ in $\mathcal{C}$ such that $\forall n[ \lambda\bigl(\overline \gamma(2n)\bigr)\sqsubset \tau]$. 
     
     Note: $\forall n \exists \zeta \in \mathcal{C}[  \overline \tau n\sqsubset\zeta \;\wedge\;  \overline \gamma(2n)\sqsubset\varphi|\zeta]$. Conclude: $\varphi|\tau = \gamma$. 
     
     We thus see: $\forall \gamma\in (\omega\times 2)^\omega[\gamma \in_I\sigma\rightarrow \exists \tau \in \mathcal{C}[\varphi|\tau=\gamma]]$.

\smallskip
     (ii) $\Rightarrow$ (i). Assume (ii). We  prove:  $\mathbf{\Sigma}^0_1$-$\mathit{Det}^I_{\omega \times 2}$.  \\Using Theorem
     \ref{T:det}(i), we then may conclude: $\mathbf{FT}$.

      Let $\alpha$ be given such that $\forall \tau \in \mathcal{C}\exists n [\langle n, \tau(\langle n\rangle)\rangle \in E_\alpha]$, i.e. \\$\forall \tau \in \mathcal{C}\exists n\exists p  [\alpha(p)=\langle n, \tau(\langle n\rangle)\rangle +1]$, i.e.  $\forall \tau \in \mathcal{C}\exists p  [\alpha(p')=\langle p'', \tau(\langle p''\rangle)\rangle +1]$.\\ Define $\psi:\mathcal{C}\rightarrow \omega$ such that $\forall \tau \in \mathcal{C}[\psi(\tau)=\mu p[\alpha(p') = \langle p'',\tau(\langle  p''\rangle)\rangle+1]$. \\Define $\varphi: \mathcal{C}\rightarrow (\omega\times 2)^\omega$ such that $\forall \tau \in \mathcal{C}[\varphi|\tau =\langle \psi''(\tau), \tau(\langle\psi''(\tau)\rangle)\rangle\ast\underline 0]$. 
      \\ Find $\sigma$ such that $\forall \gamma\in(\omega\times 2)^\omega[\gamma\in_I\sigma\rightarrow\exists\tau\in\mathcal{C}[\gamma=\varphi|\tau]]$, and, therefore,\\  $\forall \gamma\in(\omega\times 2)^\omega[\gamma\in_I\sigma\rightarrow\exists q[\alpha(q)=\overline \gamma 2+1]]$ and 
      $\forall \gamma\in(\omega\times 2)^\omega[\gamma\in_I\sigma\rightarrow\overline\gamma 2 \in E_\alpha]$. Define $n:=\sigma(\langle\;\rangle)$ and note: $\forall i<2[\langle n,i\rangle\in_I \sigma]$  and, therefore: $\forall i<2[\langle n, i\rangle \in E_\alpha]$. 
      
      We thus see: $\forall \alpha[\forall \tau \in \mathcal{C}\exists n[\langle n, \tau(\langle n\rangle)\rangle \in E_\alpha]\rightarrow \exists n\forall i<2[\langle n, i\rangle\in E_\alpha]]$, i.e. $\mathbf{\Sigma}^0_1$-$Det^I_{\omega\times 2}$. 
      \end{proof}
 \begin{corollary} $\mathsf{BIM}+\mathbf{FT}$ proves the following scheme:
 
 Every $\mathcal{X}\subseteq (\omega\times 2)^\omega$ is weakly $I$-determinate. \end{corollary}

The theorem and its corollary may be generalized. One may  replace $(\omega \times 2)^\omega$ by any spread $\mathcal{F}$  satisfying the condition: $\exists \zeta\forall\alpha \in\mathcal{F} \forall n[\alpha(2n+1) \le \zeta(n)]$.

  \medskip

Let $\varphi:\mathcal{C}\rightarrow (\omega\times 2)^\omega$ be an anti-strategy for player $I$ in $(\omega \times 2)^\omega$. We define: $\varphi$ \textit{ fails to translate into a strategy for
player $I$} if and only if  \[\neg \exists \sigma \forall \gamma \in (\omega \times 2)^\omega[ \gamma \in_I \sigma \rightarrow \exists \tau \in \mathcal{C}[ \gamma = \varphi|\tau]].\] 
$\neg!\mathbf{FT}$  implies the existence of an anti-strategy for player $I$ in  $(\omega \times 2)^\omega$ that   fails to translate
   into a strategy for player $I$ in  $(\omega \times 2)^\omega$:
   using $\neg!\mathbf{FT}$ and Theorem \ref{T:det}(ii), find $\alpha$ such that  $\forall \tau \in \mathcal{C}\exists n[\langle n , \tau(\langle n\rangle)\rangle \in E_\alpha]$ and  $\neg\exists n\forall i<2[\langle n, i\rangle \in E_\alpha]$.  As  in the proof of Theorem \ref{T:det2}(ii) $\Rightarrow$ (i), find an  anti-strategy $\varphi$  for player $I$ in $(\omega \times 2)^\omega$ such that $\forall \tau \in \mathcal{C}[\overline{(\varphi|\tau)}2 \in E_\alpha]$. 
     Assume $\sigma$ is a
     strategy for player $I$ in $(\omega \times 2)^\omega$ such that $\forall \gamma\in (\omega\times 2)^\omega[\gamma\in_I\sigma\rightarrow\exists\tau\in\mathcal{C}[\gamma=\varphi|\tau]]$.  Consider $n := \sigma(\langle \; \rangle)$\ and conclude:   $\forall i<2[\langle n, i \rangle\in E_\alpha]$. Contradiction.

     We did not find  an argument proving $\neg!\mathbf{FT}$ from the assumption of the existence of 
     an anti-strategy for player $I$ in  $(\omega \times 2)^\omega$ that   fails to translate
   into a strategy for player $I$ in  $(\omega \times 2)^\omega$.
   
  \section{The (uniform) intermediate value theorem}\label{S:ivt}

 \subsection{} The   \textit{Intermediate Value Theorem}, $\mathbf{IVT}$\footnote{For some notation used in this Section, see Subsection \ref{SS:realf}.} 
 
 \textit{For all $\varphi$ in $\mathcal{R}^{[0,1]}$},
 
 \textit{if 
 $\exists \gamma\in [0,1]^2[\ \varphi^{`\mathcal{R}}(\gamma^{\upharpoonright 0} ) \le_\mathcal{R} 0_\mathcal{R} \le_\mathcal{R} \varphi^{`\mathcal{R}}(\gamma^{\upharpoonright 1})]$, then $ \exists \gamma \in 
 [0,1][\varphi^{`\mathcal{R}}(\gamma) =_\mathcal{R} 0_\mathcal{R}].$}
 
 \smallskip
  
  $\mathbf{IVT}$ fails constructively. The next two theorems are similar to \cite[Chapter 3, Theorem 2.4]{bridgesrichman} and  \cite[Theorem 6(iv) and (iii)]{toftdal}.
\begin{theorem}\label{T:ivtllpo} $\mathsf{BIM}\vdash \mathbf{IVT} \rightarrow \mathbf{LLPO}$. \end{theorem}

\begin{proof} Assume $\mathbf{IVT}$. Let $\beta$ be given. 

Define $\delta$ in $\mathcal{R}$ such that, for each $n$, \\if $\underline{\overline 0} n\sqsubset \beta$, then $\delta(n)=(-\frac{1}{2^n}, \frac{1}{2^n})$,   and,  if $\underline{\overline 0}n \perp \beta$ and  $p_0:=\mu p[\beta(p)\neq 0]$, \\then  $\delta(n) =\bigl((-1)^{p_0}\frac{1}{2^{p_0+1}} - \frac{1}{2^{n+3}}, (-1)^{p_0}\frac{1}{2^{p_0+1}} + \frac{1}{2^{n+3}}\bigr)$. 
 Note: \\$\delta>_\mathcal{R} 0_\mathcal{R} \leftrightarrow \exists p[2p=\mu n[\beta(n)\neq 0]]$ and $\delta<_\mathcal{R} 0_\mathcal{R} \leftrightarrow \exists p[2p+1=\mu n[\beta(n)\neq 0]]$. 
\\Find $\varphi$ in $\mathcal{R}^{[0,1]}$  such that $\varphi^{`\mathcal{R}}(0_\mathcal{R}) =_\mathcal{R} (-1)_\mathcal{R}$, and  $\varphi^{`\mathcal{R}}(\frac{1}{3}) =_\mathcal{R} \varphi^{`\mathcal{R}}(\frac{2}{3}) =_\mathcal{R}
 \delta$ and $\varphi^{`\mathcal{R}}(1_\mathcal{R}) = _\mathcal{R}1_\mathcal{R}$ and $\varphi$ is linear on $[0,\frac{1}{3}]$, on
 $[\frac{1}{3}, \frac{2}{3}]$ and on $[\frac{2}{3}, 1]$. 
 \\Note: $\varphi^{`\mathcal{R}}(0_\mathcal{R})\le_\mathcal{R}0_\mathcal{R}\le_\mathcal{R}\varphi^{`\mathcal{R}}(1_\mathcal{R})$. \\Using $\mathbf{IVT}$,  find $\gamma$ in
 $[0,1]$ such that $\varphi^{`\mathcal{R}}(\gamma) =_\mathcal{R} 0_\mathcal{R}$. \\Either $\gamma > _\mathcal{R}\frac{1}{3}$ or $\gamma <_\mathcal{R} \frac{2}{3}$. \\If $\gamma > _\mathcal{R}\frac{1}{3}$, then $\neg(\delta >_\mathcal{R}  0)$ and: $\forall p[2p\neq \mu n[\beta(n)\neq 0]]$, and,  \\if
 $\gamma <_\mathcal{R} \frac{2}{3}$, then $\neg(\delta <_\mathcal{R} 0)$ and: $\forall p[2p+1\neq\mu n[\beta(n)\neq 0]]$.
 
 We thus see: $\forall \beta[\forall p[2p\neq \mu n[\beta(n)\neq 0]]\;\vee\;\forall p[2p+1\neq\mu n[\beta(n)\neq 0]]$, i.e. $\mathbf{LLPO}$.
  \end{proof}

  \begin{theorem}\label{T:llpoivt} $\mathsf{BIM}+ \mathbf{\Pi}^0_1$-$\mathbf{AC}_{\omega,2}\vdash \mathbf{LLPO}\rightarrow \mathbf{IVT}$.\end{theorem}
\begin{proof} Let $\varphi$ in $\mathcal{R}^{[0,1]}$ and $\gamma$ in $[0,1]^2$  be given such that \\$\varphi^{`\mathcal{R}}(\gamma^{\upharpoonright 0} ) \le_\mathcal{R} 0_\mathcal{R}\le_\mathcal{R}\varphi^{`\mathcal{R}}(\gamma^{\upharpoonright 1})$. Define $\beta$ such that, for all $n$, \\$\beta(2n)\neq 0 \leftrightarrow \gamma^{\upharpoonright 0}(n)<_\mathbb{S} \gamma^{\upharpoonright 1}(n)$ and: $\beta(2n+1)\neq 0 \leftrightarrow \gamma^{\upharpoonright 1}(n)<_\mathbb{S} \gamma^{\upharpoonright 0}(n)$.  \\By $\mathbf{LLPO}$,  \\{\it either}: $\forall p[2p+1\ne\mu n[\beta(n)\neq 0]]$ and $\forall n[\gamma^{\upharpoonright 0}(n) \le_\mathbb{S} \gamma^{\upharpoonright 1}(n) ]$ and $\gamma^{\upharpoonright 0}\le_\mathcal{R}\gamma^{\upharpoonright 1}$, \\ {\it or}: $\forall p[2p\ne\mu n[\beta(n)\neq 0]]$ and $\forall n[\gamma^{\upharpoonright 1}(n) \le_\mathbb{S} \gamma^{\upharpoonright 0} (n) ]$ and $\gamma^{\upharpoonright 1}\le_\mathcal{R}\gamma^{\upharpoonright 0}$. 

\smallskip
Assume $\gamma^{\upharpoonright 0} \le_\mathcal{R} \gamma^{\upharpoonright 1}$. Using $\mathbf{LLPO}$ like we used it just now, conclude: \\for all $n$, for all $m\le 2^n$, {\it either}:  $\varphi^{`\mathcal{R}}(\frac{2^n-m}{2^n}\cdot_\mathcal{R}\gamma^{\upharpoonright 0} +_\mathcal{R} \frac{m}{2^n}\cdot_\mathcal{R} \gamma^{\upharpoonright 1}) \le_\mathcal{R} 0_\mathcal{R}$,\\ {\it or}: $  0_\mathcal{R}\le_\mathcal{R} \varphi^{`\mathcal{R}}(\frac{2^n-m}{2^n}\cdot_\mathcal{R}\gamma^{\upharpoonright 0} +_\mathcal{R} \frac{m}{2^n}\cdot_\mathcal{R} \gamma^{\upharpoonright 1}) $. 

Using $\mathbf{\Pi}^0_1$-$\mathbf{AC}_{\omega,2}$, find $\beta$ such that $\beta^0(0)=0$ and $\beta^0(1) \neq 0$ and 

$\forall n\forall m\le 2^n[\beta^n(m)=0\rightarrow\varphi^{`\mathcal{R}}(\frac{2^n-m}{2^n}\cdot_\mathcal{R}\gamma^{\upharpoonright 0} +_\mathcal{R} \frac{m}{2^n}\cdot_\mathcal{R} \gamma^{\upharpoonright 1}) \le_\mathcal{R} 0_\mathcal{R}]$, and 

 $\forall n\forall m\le2^n[\beta^n(m)\neq 0\rightarrow0_\mathcal{R}\le_\mathcal{R}\varphi^{`\mathcal{R}}(\frac{2^n-m}{2^n}\cdot_\mathcal{R}\gamma^{\upharpoonright 0} +_\mathcal{R} \frac{m}{2^n}\cdot_\mathcal{R} \gamma^{\upharpoonright 1})] $. 

Now define $\delta$ such that  $\delta(0) = 0$ and\\ $\forall n[\beta^{n+1}\bigl(2\delta(n) +1\bigr)\neq 0\rightarrow\delta(n+1) =2 \delta(n)]$, and   \\$\forall n[\beta^{n+1}\bigl(2\delta(n) +1\bigr)= 0\rightarrow\delta(n+1) =2 \delta(n)+1]$. 

Note: $\forall n[\delta(n)<2^n\;\wedge\;\beta^n\bigl(\delta(n)\bigr) = 0\;\wedge\;\beta^n\bigl(\delta(n)+1\bigr) \neq 0]$ and  \\ $\forall n[\varphi^{`\mathcal{R}}(\frac{\delta(n)}{2^n})\le_\mathcal{R}0_\mathcal{R}\le_\mathcal{R}\varphi^{`\mathcal{R}}(\frac{\delta(n)+1}{2^n})]$.   

Define  $\varepsilon$ such that, for each $n$, $\varepsilon(n)=(\frac{\delta(n)}{2^n}, \frac{\delta(n)+1}{2^n})$ and define $\rho$ such that, for each $n$, $\rho(n)=\mathit{double}_\mathbb{S}\bigl(\varepsilon(n)\bigr)$.  Note: $\rho\in [0,1]
$. 

Assume $\varphi^{`\mathcal{R}}(\rho)>_\mathcal{R} 0_\mathcal{R}$.

 Find $n,p$ such that  $p\in E_\varphi$ and $\rho(n)\sqsubseteq_\mathbb{S} p'$ and $(p'')'>_\mathbb{Q} 0_\mathbb{Q}$. 
 
 Note $\rho'(n)<_\mathbb{Q}\frac{\delta(n)}{2^n}<_\mathbb{Q}\rho''(n)$ and $\varphi^{`\mathcal{R}}( \frac{\delta(n)}{2^n})\le_\mathbb{R} 0_\mathbb{R}$. Contradiction.

 Conclude: $\varphi^{`\mathcal{R}}(\rho)\le_\mathcal{R} 0_\mathcal{R}$.

For a similar reason,  $0_\mathcal{R}\le_\mathcal{R}\varphi^{`\mathcal{R}}(\rho) $ and, therefore,  $\varphi^{`\mathcal{R}}(\rho) =_\mathcal{R} 0_\mathcal{R}$.

\smallskip The case $\gamma^{\upharpoonright 1} \ge_\mathcal{R}\gamma^{\upharpoonright 0}$ is treated similarly.\end{proof}
 
\begin{corollary}\label{C:ivtwkl} $\mathsf{BIM}+\mathbf{\Pi}^0_1$-$\mathbf{AC}_{\omega,2}\vdash \mathbf{IVT} \leftrightarrow \mathbf{LLPO} \leftrightarrow \mathbf{WKL}$.
\end{corollary}
\begin{proof} Use Theorems  \ref{T:acllpowkl}, \ref{T:ivtllpo} and \ref{T:llpoivt}. \end{proof}
\subsection{} \textit{A contraposition of the Intermediate Value Theorem}, 
 $\overleftarrow{\mathbf{IVT}}$: 
 
\textit{ For each $\varphi$ in $\mathcal{R}^{[0,1]}$},
 \textit{if $ \forall \gamma \in [0,1][\varphi^{`\mathcal{R}}(\gamma)\; \#_\mathcal{R} \; 0_\mathcal{R}]$, then}
 
\textit{either  $ 
 \forall \gamma \in [0,1][\varphi^{`\mathcal{R}}(\gamma) > _\mathcal{R}0_\mathcal{R}]$ or $ \forall\gamma\in [0,1][ \varphi^{`\mathcal{R}}(\gamma) <_\mathcal{R}
 0_\mathcal{R}].$}

 \begin{theorem}\label{T:ivtreverse}$\mathsf{BIM}\vdash\overleftarrow{\mathbf{IVT}}$.
 
\end{theorem}

 \begin{proof}
  Assume $\varphi \in \mathcal{R}^{[0.1]}$ and $\forall \gamma  \in [0,1][\varphi^{`\mathcal{R}}(\gamma)\; \#_\mathcal{R} \;0_\mathcal{R}]$.
  \\Assume:
 $\varphi^{`\mathcal{R}}(0_\mathcal{R}) <_\mathcal{R} 0_\mathcal{R}$. 
 \\Suppose: $\gamma \in [0,1]$ and $\varphi^{`\mathcal{R}}(\gamma) >_\mathcal{R} 0_\mathcal{R}$. \\We will find a contradiction by the method of \textit{successive bisection}. 
 \\Find $q$ in $\mathbb{Q}$ such that $\varphi^{`\mathcal{R}}(q_\mathcal{R}) >_\mathcal{R} 0_\mathbb{Q}$. Define  $\delta$ as follows, by induction. \\We define  $\delta(0) := ( 0, q )$.  Now, let $n, s$ be given such that $\delta(n) = s$. \\Note: $\varphi^{`\mathcal{R}}((\frac{s'+_\mathbb{Q}s''}{2})_\mathcal{R})\;\#\;0_\mathcal{R}$. Find  $(r,t)$ in $E_\varphi$ such that  $r'<_\mathbb{Q} \frac{1}{2}(s'+_\mathbb{Q} s'') <_\mathbb{Q} r''$ and \textit{either}: $0_\mathbb{Q}<_\mathbb{Q} t'$ \textit{or}:  $t''<_\mathbb{Q} 0_\mathbb{Q}$.  If $ 0_\mathbb{Q}<_\mathbb{Q} t'$, define $\delta(n+1) = (s',  \frac{s'+_\mathbb{Q}s''}{2})$, and, if  $ t''<_\mathbb{Q} 0_\mathbb{Q}$, define  $\delta(n+1) = (\frac{s'+ s''}{2}, s'' )$. 
 \\ Note: for each $n$, $\delta(n+1) \sqsubseteq_\mathbb{S} \delta(n)$, and  $\varphi^{`\mathcal{R}}\bigl((\delta'(n))_\mathcal{R}\bigr)<_\mathcal{R} 0_\mathcal{R}<_\mathcal{R}\varphi^{`\mathcal{R}}\bigl((\delta''(n))_\mathcal{R}\bigr)$.
   Note: $\delta\in [0,1]$ and $\varphi^{`\mathcal{R}}(\delta)\;\#_\mathcal{R}\; 0_\mathcal{R}$.  Determine $( r,s)$ in $E_\varphi$ and $n$ in $\omega$ 
 such that $\delta(n) \sqsubseteq_\mathbb{S} r$ and either $s''<_\mathbb{Q} 0$  or  $0<_\mathbb{Q} s'$, that is,   either  $\varphi^{`\mathcal{R}}\bigl((\delta''(n))_\mathcal{R}\bigr) <_\mathcal{R} 0_\mathcal{R}$ or $0_\mathcal{R} <_\mathcal{R}\varphi^{`\mathcal{R}}\bigl((\delta'(n))_\mathcal{R}\bigr)$. 
Contradiction.
 Clearly, then,  $\neg\bigl(\varphi^{`\mathcal{R}}(\gamma) >_\mathcal{R} 0_\mathcal{R}\bigr)$.

From $\neg\bigl(\varphi^{`\mathcal{R}}(\gamma) >_\mathcal{R} 0_\mathcal{R}\bigr)$ and
 $\varphi^{`\mathcal{R}}(\gamma) \;\#_\mathcal{R}\; 0_\mathcal{R}$, we  conclude:  $\varphi^{`\mathcal{R}}(\gamma) <_\mathcal{R} 0_\mathcal{R}$.
 
We thus see: if $\varphi^{`\mathcal{R}}(0_\mathcal{R}) < _\mathcal{R}0_\mathcal{R}$, then $\forall\gamma\in[0,1][\varphi^{`\mathcal{R}}(\gamma) <_\mathcal{R} 0_\mathcal{R}]$.

 In the same way, one proves:  if $\varphi^{`\mathcal{R}}(0_\mathcal{R}) >_\mathcal{R} 0_\mathcal{R}$, then $\forall\gamma\in[0,1][\varphi^{`\mathcal{R}}(\gamma) >_\mathcal{R} 0_\mathcal{R}]$.

 \smallskip We thus see: either  $ 
 \forall \gamma \in [0,1][\varphi^{`\mathcal{R}}(\gamma) > _\mathcal{R}0_\mathcal{R}]$ or $ \forall\gamma\in [0,1][ \varphi^{`\mathcal{R}}(\gamma) <_\mathcal{R}
 0_\mathcal{R}].$
 \end{proof}
 \subsection{$\mathbf{FT}$ is unprovable in $\mathsf{BIM}+ \mathbf{IVT}$ }
  As we observed in Subsection \ref{SSS:countbincunpr}, $\mathsf{BIM}+\mathbf{CT}+X\vee \neg X$  is consistent. According to Theorem \ref{T:ivtreverse}, $\mathsf{BIM}\vdash \overleftarrow{\mathbf{IVT}}$, and, therefore: \\$\mathsf{BIM} + X\vee \neg X\vdash \mathbf{IVT}$. Assume $\mathsf{BIM}\vdash \mathbf{IVT} \rightarrow \mathbf{FT}$. Then  $\mathsf{BIM} + X\vee \neg X\vdash \mathbf{FT}$. \\As we know from Theorem \ref{T:ctka},  $\mathsf{BIM} + \mathbf{CT}\vdash \neg!\mathbf{FT}$, and, therefore, \\$\mathsf{BIM} + \mathbf{CT}\vdash \neg\mathbf{FT}$. Conclude: $\mathsf{BIM}+\mathbf{IVT}\nvdash \mathbf{FT}$, and also, in view of Theorem \ref{T:wklft},  $\mathsf{BIM}+\mathbf{IVT}\nvdash \mathbf{WKL}$. \\Note that, in view of Corollary  \ref{C:ivtwkl}, this gives another proof of $\mathsf{BIM}\nvdash \mathbf{\Pi}_1^0$-$\mathbf{AC}_{\omega, 2}$, a fact  established in Subsection \ref{SSS:countbincunpr}.\\One may even conclude: $\mathsf{BIM}+\mathbf{IVT}\nvdash \mathbf{\Pi}^0_1$-$\mathbf{AC}_{\omega,2}$. 
  \\One may ask if $\mathsf{BIM} +\mathbf{LLPO}\vdash \mathbf{IVT}$, i.e. if the proof of Theorem \ref{T:llpoivt} can be given without recourse to $\mathbf{\Pi}^0_1$-$\mathbf{AC}_{\omega,2}$,  but we do not know the answer to this question.  
 
 \subsection{} 
 
 The   \textit{Uniform Intermediate Value Theorem}, 
 $\mathbf{UIVT}$: 
 
 \textit{For each $\varphi$ such that $\forall n[\varphi^n\in \mathcal{R}^{[0,1]}]$},\\
 \textit{ if $\forall n
 \exists \gamma\in [0,1]^2[\ (\varphi^n)^{`\mathcal{R}}(\gamma^0 ) \le_\mathcal{R} 0_\mathcal{R}\le_\mathcal{R} (\varphi^n)^{`\mathcal{R}} (\gamma^1)]$,}
  \\\textit{ then  $\exists \gamma \in 
 [0,1]^\omega\forall n[(\varphi^n)^{`\mathcal{R}}(\gamma^n) =_\mathcal{R} 0_\mathcal{R}].$}
 
 \medskip
 In \cite[Exercise IV.2.12, page 137]{Simpson},  the reader is asked to prove that, in the classical system $\mathsf{RCA_0}$,  $\mathbf{UIVT}$ is an equivalent of $\mathbf{WKL}$. As $\mathsf{RCA_0}\vdash \mathbf{IVT}$,
   $\mathsf{RCA_0}$ proves the equivalence of  $\mathbf{UIVT}$ and: 
  \subsection{} $\mathbf{Uzero}$:
  
   \textit{For all $\varphi$ such that $\forall n[\varphi^n \in \mathcal{R}^{[0,1]}]$}, 
  
 \textit{if $\forall n\exists \gamma\in [0,1][\ (\varphi^n)^{`\mathcal{R}}(\gamma ) =_\mathcal{R} 0_\mathcal{R}]$, then $\exists \gamma \in 
 [0,1]^\omega\forall n[(\varphi^n)^{`\mathcal{R}}(\gamma^n) =_\mathcal{R} 0_\mathcal{R}].$}
 
 \medskip

  We want to study $\mathbf{Uzero}$ in $\mathsf{BIM}$. We need the following Lemma.
 \begin{lemma}\label{L:uzero}$\mathsf{BIM}$ proves: \begin{enumerate}[\upshape (i)] \item $\exists \psi:\omega^\omega\rightarrow\omega^\omega\forall \varphi \in \mathcal{R}^{[0,1]}[\mathcal{H}_{\psi|\varphi}=\{\gamma\in [0,1]\mid \varphi^{`\mathcal{R}}(\gamma) \;\#_\mathcal{R}\;0_\mathcal{R}\}]$, and\item $\exists \tau:\omega^\omega\rightarrow\omega^\omega\forall \alpha[\tau|\alpha \in \mathcal{R}^{[0,1]}\;\wedge\;\mathcal{H}_\alpha =\{\gamma \in [0,1]\mid (\tau|\alpha)^{`\mathcal{R}}(\gamma)\;\#_\mathcal{R}\;0_\mathcal{R}\}]$. \end{enumerate}\end{lemma}  \begin{proof} Define $\psi:\omega^\omega\rightarrow\omega^\omega$  such that, for each $\varphi$,  for each $s$, $$(\psi|\varphi)(s) \neq 0\leftrightarrow \exists p\in E_{\overline \varphi s}[s\sqsubset_\mathbb{S}p'\;\wedge\;\bigl(0_\mathbb{Q} <_\mathbb{Q} (p'')'\;\vee\; (p'')''<_\mathbb{Q} 0_\mathbb{Q}\bigr)].$$ Note: $\forall \varphi \in \mathcal{R}^{[0,1]}\forall\gamma\in [0,1][\varphi^{`\mathcal{R}}(\gamma) \;\#_\mathcal{R}\; 0_\mathcal{R}\leftrightarrow\gamma \in \mathcal{H}_{\psi|\varphi}]$. 
 
 \smallskip
 
  Now define $\rho$ such that,  for each  $s$ in $\mathbb{S}$, $\rho^s\in\mathcal{R}^{[0,1]}$ and, \\{\it if not $(-1)_\mathbb{Q}\le_\mathbb{Q}s'\le_\mathbb{Q} s''\le_\mathbb{Q}  (2)_\mathbb{Q}$}, then, for all $\gamma$ in $[0,1]$, $(\rho^s)^{`\mathcal{R}}(\gamma) =_\mathcal{R} 0_\mathcal{R}$, and, \\{\it if $(-1)_\mathbb{Q}\le_\mathbb{Q}s'\le_\mathbb{Q} s''\le_\mathbb{Q}  (2)_\mathbb{Q}$}, then,
  for all $\gamma$ in $[0,1]$, \begin{enumerate}[\upshape (1)] \item if $\gamma \le_\mathcal{R} (s')_\mathcal{R}$ or $(s'')_\mathcal{R} \le_\mathcal{R} \gamma$, then $\rho^s(\gamma)  =_\mathcal{R} 0_\mathcal{R}$, and, 
  \item if $(s')_\mathcal{R}\le_\mathcal{R} \gamma \le _\mathcal{R} (s'')_\mathcal{R}$, then $\rho^s(\gamma) =_\mathcal{R} \inf(\gamma -_\mathcal{R} (s')_\mathcal{R},(s'')_\mathcal{R}-_\mathcal{R} \gamma)$. \end{enumerate}
 
 (Note that $\rho^s$ codes the restriction to $[0,1]$ of the `{\it tent}' function from $\mathcal{R}$ tot $\mathcal{R}$ that is zero outside of $[s', s'']$, linear on both $[s', \frac{s'+s''}{2}]$ and $ [\frac{s'+s''}{2}, s'']$ and takes the values $0_\mathcal{R}, \frac{s''-s'}{2}, 0_\mathcal{R}$ at $s', \frac{s'+s''}{2}, s''$, respectively.\\Note that, for all $s$ in $\mathbb{S}$, for all $\gamma$ in $[0,1]$, $(\rho^s)^{`\mathcal{R}}(\gamma)\le_\mathcal{R} (\frac{3}{2})_\mathcal{R}$. ) 

  \smallskip Define $\tau:\omega^\omega\rightarrow\omega^\omega$   such that, for all $\alpha$, $\tau|\alpha\in\mathcal{R}^{[0,1]}$    and, for each $\gamma$ in $[0,1]$, $(\tau|\alpha)^{`\mathcal{R}}(\gamma) =_\mathcal{R} \sum_{s, \alpha(s)\neq 0} (\frac{1}{2^s})_\mathcal{R}\cdot_\mathcal{R}(\rho^s)^{`\mathcal{R}}(\gamma)$.

Note:  $\forall  \gamma\in [0,1][  \gamma \in \mathcal{H}_{\alpha}\leftrightarrow(\tau|\alpha)^{`\mathcal{R}}(\gamma) \;\#_\mathcal{R} \;0_\mathcal{R}]$,  and: \\ $\forall \gamma \in [0,1][\gamma \notin \mathcal{H}_{\alpha} \leftrightarrow (\tau|\alpha)^{`\mathcal{R}}(\gamma) =_\mathcal{R} 0_\mathcal{R}]$.\end{proof} \begin{theorem}\label{T:unifivtcc}
   $\mathsf{BIM}\vdash\mathbf{\Pi}^0_1$-$\mathbf{AC}_{\omega,[0,1]}\leftrightarrow\mathbf{Uzero}$.
  
  \end{theorem}
  \begin{proof} First, assume $\mathbf{\Pi}^0_1$-$\mathbf{AC}_{\omega,[0,1]}$. 
  
  Let $\varphi$ be given such that $\forall n[\varphi^n \in \mathcal{R}^{[0,1]} \;\wedge \; \exists \gamma \in [0,1][(\varphi^n)^{`\mathcal{R}}(\gamma) =_\mathcal{R} 0_\mathcal{R}]]$. Using Lemma \ref{L:uzero}(i), find $\alpha$  such that, $\forall n[\mathcal{H}_{\alpha^n}=\{\gamma \in [0,1]\mid (\varphi^n)^{`\mathcal{R}}(\gamma)\;\#_\mathcal{R}\; 0_\mathcal{R}\}]$. Conclude: $\forall n \exists \gamma \in [0,1] [\gamma \notin \mathcal{H}_{\alpha^n}]$ and, by $\mathbf{\Pi}^0_1$-$\mathbf{AC}_{\omega,[0,1]}$: $\exists \gamma \in [0,1]^\omega \forall n[\gamma^n \notin \mathcal{H}_{\alpha^n}]$ and: $\exists \gamma \in [0,1]^\omega\forall n[(\varphi^n)^{`\mathcal{R}}(\gamma^n) =_\mathcal{R} 0_\mathcal{R}]$. 
  
 Conclude: $\mathbf{Uzero}$.
  
  \smallskip
  Secondly, assume $\mathbf{Uzero}$. 
  
Let $\alpha$ be given such that  $\forall n\exists \gamma \in [0,1][\gamma \notin \mathcal{H}_{\alpha^n}]$. Using Lemma \ref{L:uzero}(ii), find $\varphi$   such that  $\forall n[\varphi^n\in \mathcal{R}^{[0,1]}]$ and  $\forall n[\mathcal{H}_{\alpha^n}=\{\gamma \in [0,1]\mid(\varphi^n)^{`\mathcal{R}}(\gamma) \;\#_\mathcal{R} \;0_\mathcal{R}\}]$.
\\Conclude: $\forall n\exists \gamma \in [0,1][(\varphi^n)^{`\mathcal{R}}(\gamma) =_\mathcal{R} 0_\mathcal{R}]$, and   $\exists \gamma \in [0,1]^\omega\forall n[(\varphi^n)^{`\mathcal{R}}(\gamma^n) =_\mathcal{R} 0_\mathcal{R}]$, and 
 $\exists \gamma \in [0,1]^\omega \forall n [\gamma^n \notin \mathcal{H}_{\alpha^n}]$.
 
 Conclude: $\mathbf{\Pi}^0_1$-$\mathbf{AC}_{\omega,[0,1]}$. 
 \end{proof}
  
 \subsection{}\textit{A uniform contrapositive Intermediate Value Theorem} 
 $\overleftarrow{\mathbf{UIVT}}:$
 
  \textit{For each $\varphi$ such that $\forall n[\varphi^n\in \mathcal{R}^{[0,1]}]$}, 
 $\mathit{if}\;  \forall \gamma \in [0,1]^{\omega} \exists n[(\varphi^n)^{`\mathcal{R}}(\gamma^n)\; \#_\mathcal{R} \; 0_\mathcal{R}],$
 
  $\mathit{then}\; \exists n[
 \forall \gamma \in [0,1][(\varphi^n)^{`\mathcal{R}}(\gamma) > _\mathcal{R}0_\mathcal{R}]\;\vee\; \forall\gamma\in [0,1][(\varphi^n)^{`\mathcal{R}}(\gamma) <_\mathcal{R}
 0_\mathcal{R}]].$

 \smallskip
 As  $\mathsf{BIM}\vdash \overleftarrow{\mathbf{IVT}}$,  $\mathsf{BIM}$ proves the equivalence of $\overleftarrow{\mathbf{UIVT}}$ and:
 
 \subsection{}$\overleftarrow{\mathbf{Uzero}}$:  
 
 \textit{For all $\varphi$ such that $\forall n[\varphi^n\in \mathcal{R}^{[0,1]}]$},
 
 $\mathit{if}\;\forall \gamma \in [0,1]^{\omega} \exists n[(\varphi^n)^{`\mathcal{R}}(\gamma^n)\; \#_\mathcal{R} \; 0_\mathcal{R}], \;\mathit{then}\; \exists n
 \forall \gamma \in [0,1][(\varphi^n)^{`\mathcal{R}}(\gamma) \;\# _\mathcal{R}\;0_\mathcal{R}].$

 \smallskip
 We define a {\it strong negation} of this statement. This strong negation itself contains   a negation sign, a possibility mentioned in Subsection \ref{SS:strongnegations}.  

 \subsection{}$\neg !\overleftarrow{\mathbf{Uzero}}$:  
 
 \textit{There exists $\varphi$ such that $\forall n[\varphi^n\in \mathcal{R}^{[0,1]}]$ and $\forall \gamma \in [0,1]^{\omega} \exists n[(\varphi^n)^{`\mathcal{R}}(\gamma^n)\; \#_\mathcal{R} \; 0_\mathcal{R}]$ and $\neg\exists n
 \forall \gamma \in [0,1][(\varphi^n)^{`\mathcal{R}}(\gamma) \;\# _\mathcal{R}\;0_\mathcal{R}]$.}

 \begin{lemma}\label{L:unifivt} One may prove in $\mathsf{BIM}$:
 \begin{enumerate}[\upshape(i)]
 \item $\mathbf{\Sigma}^0_1$-$\overleftarrow{\mathbf{AC}_{\omega,[0,1]}}\rightarrow \overleftarrow{\mathbf{Uzero}}$ and
  $\neg!\overleftarrow{\mathbf{Uzero}} \rightarrow \neg!\mathbf{\Sigma}^0_1$-$\overleftarrow{\mathbf{AC}_{\omega,[0,1]}}$.
 \item $\overleftarrow{\mathbf{Uzero}} \rightarrow \mathbf{\Sigma}^0_1$-$\overleftarrow{\mathbf{AC}_{\omega,[0,1]}}$ and
 $\neg!\mathbf{\Sigma}^0_1$-$\overleftarrow{\mathbf{AC}_{\omega,[0,1]}}\rightarrow \neg!\overleftarrow{\mathbf{Uzero}}$

 \end{enumerate}
 \end{lemma}
 \begin{proof}
 (i) We prove, in $\mathsf{BIM}$: for all $\varphi$ such that  $\forall n[\varphi^n \in  \mathcal{R}^{[0,1]}]$, there  exists $\beta$ such that 
 $$  \forall \gamma \in [0,1]^\omega\exists n[(\varphi^n)^{`\mathcal{R}}(\gamma^n) \;\#_\mathcal{R}\; 0_\mathcal{R}] \rightarrow  \forall \gamma \in [0,1]^\omega \exists n [\gamma^n \in \mathcal{H}_{\beta^n}]\;\mathrm{and}$$ $$\exists n [ [0,1] \subseteq \mathcal{H}_{\beta^n}] \rightarrow \exists n\forall \gamma \in [0,1] [(\varphi^n)^{`\mathcal{R}}(\gamma) \;\#_\mathcal{R}\; 0_\mathcal{R}].$$
 The two promised conclusions then follow easily.
 
 \smallskip
 Let $\varphi$ be given such that $\forall n[\varphi^n \in  \mathcal{R}^{[0,1]}]$. Using Lemma \ref{L:uzero}(i) find $\beta$ such that $\forall n [\mathcal{H}_{\beta^n}=\{\gamma\in[0,1]\mid (\varphi^n)^{`\mathcal{R}}(\gamma) \;\#_\mathcal{R} \; 0_\mathcal{R}\}]$. 
 
  \smallskip Assume $\forall \gamma \in [0,1]^\omega \exists n[(\varphi^n)^{`\mathcal{R}}(\gamma^n) \;\#_\mathcal{R}\; 0_\mathcal{R}]]$.
  
  Conclude: $\forall \gamma \in [0,1]^\omega \exists n[\gamma^n \in \mathcal{H}_{\beta^n}]$.

  \smallskip
  Now let $n$ be given such that $[0,1]\subseteq \mathcal{H}_{\beta^n}$. 
  Clearly,  $\forall \gamma \in [0,1] [(\varphi^n)^{`\mathcal{R}}(\gamma) \;\#_\mathcal{R}\; 0_\mathcal{R}]$.

  \smallskip
  (ii) We  prove, in $\mathsf{BIM}$: for each $\alpha$, there exists $\varphi$ such that $\forall n[\varphi^n \in  \mathcal{R}^{[0,1]}]$ and $$\forall \gamma \in [0,1]^\omega \exists n [\gamma^n \in \mathcal{H}_{\alpha^n}] \rightarrow \forall \gamma \in [0,1]^\omega\exists n[(\varphi^n)^{`\mathcal{R}}(\gamma^n) \;\#_\mathcal{R}\; 0_\mathcal{R}]\;\mathrm{and}$$  $$\exists n\forall \gamma \in [0,1] [(\varphi^n)^{`\mathcal{R}}(\gamma) \;\#_\mathcal{R}\; 0_\mathcal{R}]) \rightarrow \exists n [ [0,1] \subseteq \mathcal{H}_{\alpha^n}].$$
  The two promised conclusions then follow easily.

  Let $\alpha$ be given. Using Lemma \ref{L:uzero}(ii),  find $\varphi$  such that \\$\forall n[\varphi^n \in \mathcal{R}^{[0,1]}\;\wedge\;  \mathcal{H}_{\alpha^n} =\{\gamma\in [0,1]\mid  (\varphi^n)^{`\mathcal{R}}(\gamma) \;\#_\mathcal{R}\; 0_\mathcal{R}\}]$. 
  
  Assume: $\forall \gamma \in [0,1]^\omega \exists n[\gamma^n \in \mathcal{H}_{\alpha^n}]$.
  Conclude:
   $\forall \gamma \in [0,1]^\omega \exists n[(\varphi^n)^{`\mathcal{R}}(\gamma^n) \;\#_\mathcal{R}\; 0_\mathcal{R}]$.

  Now assume: $\exists n \forall \gamma \in [0,1][(\varphi^n)^{`\mathcal{R}}(\gamma) \;\#_\mathcal{R}\; 0_\mathcal{R}]$. Conclude:
   $\exists n [ [0,1]\subseteq \mathcal{H}_{\alpha^n}]$.
  \end{proof}

 \begin{theorem}\label{T:unifivt}

 $\mathsf{BIM}$ proves $\overleftarrow{\mathbf{Uzero}}\leftrightarrow \overleftarrow{\mathbf{UIVT}}\leftrightarrow \mathbf{FT}$ and \\
  $ \neg !\overleftarrow{\mathbf{Uzero}}\leftrightarrow\neg!\overleftarrow{\mathbf{UIVT}} \leftrightarrow \neg!\mathbf{FT}$.

\end{theorem}
 
\begin{proof} Use Lemma \ref{L:unifivt} and Theorem \ref{T:ccc}.  \end{proof}

 \section{The compactness of classical propositional logic}

In this Section, we prove that $\mathbf{FT}$ is equivalent to a contraposition of a restricted version of the compactness theorem for classical propositional logic.  We prove also the corresponding result for $\neg!\mathbf{FT}$.

We introduce the symbols $\neg$, $\bigwedge$ and $\bigvee$ as natural numbers: $\neg := 1$, $\bigwedge := 2$ and $\bigvee := 3$.  We define (the characteristic function of)  $\mathit{Form}\subseteq\omega$, as follows, by  recursion: for each $n$, $n\in \mathit{Form}$ if and only if 
 $$ n'=0 \;\vee\; (n'=\neg \;\wedge\; n''\in \mathit{Form}) \;\vee\;$$  $$\bigl((n'=\bigwedge \;\vee\;n'=\bigvee) \;\wedge\;\forall i<\mathit{length}(n'')[n''(i)\in \mathit{Form}]\bigr).$$

We define $\top :=(\bigwedge, \langle\;\rangle)$ and $\perp:=(\bigvee, \langle\;\rangle)$. 

Assume $\gamma \in \mathcal{C}$. We define $\tilde \gamma$ in
$\mathcal{C}$ such that,  
 for every $n$ in $\mathit{Form}$, \begin{enumerate}[\upshape (i)]\item if $n\notin \mathit{Form}$, then $\tilde \gamma (n) =0$, and, \item if $n\in \mathit{Form}$ and $n'=0$, then $\tilde \gamma (n) = \gamma(n'')$, and,
\item if $n\in \mathit{Form}$ and $n'=\neg$, then   $\tilde \gamma(n) = 1 -
\tilde \gamma (n'')$, and, \item if $n\in \mathit{Form}$ and $n'=\bigwedge$, then
  $\tilde \gamma (n)=
\min\{\tilde \gamma\bigl(n''(i)\bigr)|i< \mathit{length}(n'')\}$, and  \item if $n\in \mathit{Form}$ and $n'=\bigvee$, then
  $\tilde \gamma (n)=
\max\{\tilde \gamma\bigl(n''(i)\bigr)|i< \mathit{length}(n'')\}$.\end{enumerate}  Note: $0 = \langle \;\rangle$. We define $\min(\emptyset)=1$ and $\max(\emptyset)=0$.\\ Note: $\tilde \gamma(\top) =
\tilde \gamma\bigl((\bigwedge, 0)\bigr) = 1$ and $\tilde \gamma(\perp) =\tilde \gamma\bigl((\bigvee, 0)\bigr) = 0$.  

\smallskip For all $m,n$ in $Form$, we define: $m\equiv n$ if and only if $\forall \gamma \in \mathcal{C}[\tilde \gamma( m) =\tilde \gamma (n)]$.

\smallskip
Assume $c \in  \mathit{Bin}$. We define $\tilde c$ in $\mathit{Bin}$ such that $\mathit{length}(c) = \mathit{length}(\tilde c)$, as follows. 

First, define $\gamma =c\ast\underline 0$. Then define,   for all $m<length(c)$,    $\tilde c (m):=\tilde\gamma (m)$.

\smallskip

 $X\subseteq\omega$ is \textit{realizable}, $\mathit{Real}(X)$, if and only if $\exists \gamma \in \mathcal{C} \forall n \in X[\tilde \gamma(n) =1]$, and  \textit{positively unrealizable}, $\mathit{Unreal}(X)$, if and only if $\forall \gamma \in \mathcal{C}\exists n \in X[\tilde \gamma (n) = 0]$.

\smallskip
 We define a mapping $\mathit{Fm}$ from $\mathit{Bin}$ to
$\mathit{Form}$,  as follows. Assume $a \in \mathit{Bin}$. Find $s$ such that $\mathit{length}(s) = \mathit{length}(a)$, and, for all $i < \mathit{length}(a)$, if $a(i) = 0$, then $s(i) = (\neg, (0,i))$, and, if $a(i) = 1$, then $s(i) = (0,i)$. Define $Fm(a)= (\bigwedge, s)$. 

\begin{lemma}\label{L:helpcompactness}\hfill

\begin{enumerate}[\upshape (i)]\item  $\forall  
a \in \mathit{Bin}\forall\gamma\in\mathcal{C}[\tilde \gamma\bigl(\mathit{Fm}(a)\bigr) = 1\leftrightarrow a \sqsubset \gamma]$.  \item There exists $\beta$ such that \\$\forall m \in Form \forall p[\beta\bigl(m,p)\bigr)>p\;\wedge\; \beta\bigl((m,p)\bigr) \in Form\;\wedge\; \beta\bigl((m,p)\bigr)\equiv m]$. \item\footnote{$[\omega]^\omega=\{\zeta\mid\forall n[\zeta(n)<\zeta(n+1)]\}$, see Subsection \ref{SS:sequences}.} For all $\alpha$ in $\mathcal{C}$,  there exists $\delta$ in $[\omega]^\omega$ such that \\ $\forall m[\delta(m)\in \mathit{Form}\;\wedge\;\forall \gamma \in \mathcal{C}[\tilde\gamma\bigl(\delta(m)\bigr)=1 \leftrightarrow \forall n\le m[\alpha(\overline \gamma n)=0]]]$. \end{enumerate}\end{lemma} \begin{proof}

\smallskip (i) We  prove, by induction: \\for each $n$, $\forall  
a \in \mathit{Bin}_n\forall\gamma\in\mathcal{C}[\tilde \gamma\bigl(\mathit{Fm}(a)\bigr) = 1\leftrightarrow a \sqsubset \gamma]$.

First, note $Fm(\langle\;\rangle)=\top$ and: $\forall \gamma \in \mathcal{C}[\tilde \gamma(\top)=1]$ and   $\forall \gamma \in \mathcal{C}[\langle \;\rangle \sqsubset \gamma]$.

Then, let $a,n$ be given such that $a\in Bin_n$ and $\forall\gamma\in\mathcal{C}[\tilde \gamma\bigl(\mathit{Fm}(a)\bigr) = 1\leftrightarrow a \sqsubset \gamma]$. \\Note that for each $\gamma$ in $\mathcal{C}$, for each $i<2$,  \\
$\tilde \gamma\bigl(Fm(a\ast\langle i\rangle)\bigr)=1\leftrightarrow \bigl(\tilde\gamma\bigl(Fm(a)\bigr) = 1 \;\wedge\; \gamma(n)=i \bigr)\leftrightarrow a\ast\langle i \rangle \sqsubset \gamma$.
 
(ii) The proof is an exercise in calculating codes of formulas. Given $m$ in $Form$ and $p$, one might first find $q:=\max (m,p) $ and then $s$ in $\omega^{q+1}$ such that $s(0)=m$ and $\forall j<q[s(j)=\top]$ and then define $\beta\bigl((m,p)\bigr)
:=(\bigwedge, s)$. 

\smallskip (iii) Let $\alpha$ be given.  We define the promised $\alpha$ by induction. 

If $\alpha(\langle\;\rangle)=0$, define $\delta(0):=\top$, and, if $\alpha(\langle\;\rangle)\neq 0$, define $\delta(0):=\perp$.

Note that $\delta(0)$ satisfies the requirements. 

Let $m$ be given such that $\delta(m)$ has been defined and satisfies the requirements.

Find $t$ such that $\{t(i)\mid i < \mathit{length}(t)\}=\{a\in\mathit{Bin}_{m+1}\mid\forall n\le m+1[\alpha(\overline a n) = 0]\}$. \\Then find $s$ such that $\mathit{length}(s) = \mathit{length}(t)$ and  $\forall i< \mathit{length}(s)[s(i) = \mathit{Fm}\bigl(t(i)\bigr)]$. 
 
 Note that, for each $\gamma$ in $\mathcal{C}$, $\tilde\gamma\bigl((\bigvee, s)\bigr)=1\leftrightarrow$  \\$\exists a \in \mathit{Bin}_{m+1}[\forall n \le m[\alpha(\overline a n) = 0]\;\wedge\;\tilde\gamma\bigl(Fm(a)\bigr)=1]\leftrightarrow$   \\$\exists a \in \mathit{Bin}_{m+1}[\forall n \le m[\alpha(\overline a n) = 0]\;\wedge\;a \sqsubset \gamma]\leftrightarrow\forall n\le m[\alpha(\overline \gamma n) =0]$.
 
 Define $\delta(m+1) =\beta\bigl( (\bigvee, s), \delta(m)\bigr)$, where $\beta$ is the function we found in (ii).  
 
  Note that $\delta(m+1)$ satisfies the requirements. 
 \end{proof}

 \begin{lemma}\label{L:compprop}
 The following statements are provable in  $\mathsf{BIM}$.
 \begin{enumerate}[\upshape (i)]
\item $\mathbf{FT} \rightarrow \forall \alpha[\mathit{Unreal}(E_\alpha) \rightarrow \exists n[\mathit{Unreal}(E_{\overline \alpha n})]]$. \item $\exists  \alpha[\mathit{Unreal}(E_\alpha) \;\wedge\; \forall n[\mathit{Real}(E_{\overline \alpha n})]]\rightarrow \neg!\mathbf{FT}$. 
 \item $\mathbf{WKL}\rightarrow \forall\alpha[\forall n [Real(E_{\overline\alpha n})]\rightarrow Real(E_\alpha)]]$.
  \item $\forall \alpha[\mathit{Unreal}(D_\alpha) \rightarrow \exists n[\mathit{Unreal}(D_{\overline \alpha n})]]\rightarrow \mathbf{FT}$. \item $\neg!\mathbf{FT} \rightarrow \exists \alpha[\mathit{Unreal}(D_\alpha) \;\wedge\; \forall n[\mathit{Real}(D_{\overline \alpha n})]]$.\item $\forall\alpha[\forall n[Real(D_{\overline\alpha n})]\rightarrow Real(D_\alpha)]]\rightarrow \mathbf{WKL}$. 

 \end{enumerate}
 \end{lemma}
 \begin{proof} (i), (ii) and (iii). 
We argue in $\mathsf{BIM}$.

 Let $\alpha$ be given. Define $\beta$  such that, for all $m$, for all $c$ in $Bin_m$, \\ $\beta(c) \neq 0\leftrightarrow \exists n < m[n\in E_{\overline \alpha m}\;\wedge\;\tilde c(n) = 0]]$. We shall prove:
 $$\mathit{Unreal}(E_\alpha) \rightarrow \mathit{Bar}_\mathcal{C}(D_\beta)\;\mathrm{and}$$
 $$\exists m[ Bar_\mathcal{C}(D_{\overline \beta m})]\rightarrow \exists n[\mathit{Unreal}(E_{\overline \alpha n})].$$

 Assume: $\mathit{Unreal}(E_\alpha)$. Let $\gamma$ in $\mathcal{C}$ be given. Find $n,p$ such that  $n \in E_{\overline \alpha p}$ and $\tilde \gamma (n) = 0$. Define $m:=\max\{n, p\}+1 $ and note: $\beta(\overline \gamma m) \neq 0$. 
  Conclude: $\mathit{Bar}_\mathcal{C}(D_\beta)$.

  \smallskip
Let  $m$  be given such that $ Bar_\mathcal{C}(D_{\overline \beta m})$.   For all $c$ in $\mathit{Bin}_m$,  $\exists n\le m[\beta(\overline c n) \neq 0]$, and, therefore:  $\exists n < \mathit{length}(c)[n \in E_{\overline \alpha m}\;\wedge\;\tilde c(n) = 0]$. Conclude: $\mathit{Unreal}(E_{\overline \alpha m})$ and $\exists n[\mathit{Unreal}(E_{\overline \alpha n})]$.
  
  \smallskip Note: if $Unreal(E_\alpha)$, then $Bar_\mathcal{C}(D_\beta)$, and by $\mathbf{FT}$, there exist $m$ such that $Bar_\mathcal{C}(D_{\overline \beta m})$ and $n$ such that $Unreal(E_{\overline \alpha n})$. This establishes (i).
  
  \smallskip Note: if $Unreal(E_\alpha)\;\wedge\;\forall n[Real(E_{\overline\alpha n}]$, then $Bar_\mathcal{C}(D_\beta)$ and $\forall m[\neg Bar(D_{\overline \beta m})]$, i.e. $\neg\mathbf{! FT}$. This establishes (ii).

  \smallskip Note: if $\forall n[Real(E_{\overline\alpha n})]$, then $\forall m[\neg Bar(D_{\overline \beta m})]$, and, by $\mathbf{WKL}$, there exists $\gamma$ such that $\forall n[\beta(\overline   \gamma n)=0 ]$. Conclude: $\forall m \forall n<m[n\in E_{\overline \alpha m}\rightarrow \tilde \gamma(n) =1]$, i.e. $\gamma $ realizes $E_\alpha$ and $Real(E_\alpha)$. This establishes (iii).

  \medskip
 (iv), (v) and (vi).  We argue in $\mathsf{BIM}$. 
 
 Let $\alpha$ be given. Using Lemma \ref{L:helpcompactness}(iii), find $\delta$ in $[\omega]^\omega$ such that\\ $\forall m[\delta(m)\in \mathit{Form}\;\wedge\;\forall \gamma \in \mathcal{C}[\tilde\gamma\bigl(\delta(m)\bigr)=1] \leftrightarrow \forall n\le m[\alpha(\overline \gamma n)=0]]]$.
   
   Note: $\delta$ is strictly increasing and, therefore, \\$\forall m[\exists n[m=\delta(n)]\leftrightarrow \exists n\le m[m=\delta(n)]]$.
   
  Define $\beta$  such that $\forall m[\beta(m) \neq 0\leftrightarrow \exists n[m = \delta(n)]]$.  We shall prove:
$$\mathit{Bar}_\mathcal{C}( D_\alpha)\rightarrow \mathit{Unreal}(D_{ \beta })\;\mathrm{and}$$  $$\exists n[\mathit{Unreal}(D_{\overline \beta n})] \rightarrow \exists m [Bar_\mathcal{C}(D_{\overline \alpha m})].$$

  Assume: $\mathit{Bar}_\mathcal{C}(D_\alpha)$. Given any $\gamma \in \mathcal{C}$,  find $m$ such that $\alpha(\overline \gamma m) \neq 0$ and, therefore, $\tilde\gamma\bigl(\delta(m)\bigr)=0$  and:  $\exists n \in D_\beta[\tilde \gamma(n) \neq 1]$. Conclude: $\mathit{Unreal}(D_\beta)$.

 \smallskip 
  Let $n$ be given such that $\mathit{Unreal}(D_{\overline \beta n})$. Let $m_0$ be the largest $m$ such that $\delta(m)< n$. Note: $\neg \exists \gamma \in \mathcal{C}[\tilde\gamma\bigl(\delta(m_0)\bigr)=1] $,  and, therefore, \\$\forall a \in \mathit{Bin}_{m_0+1}\exists n \le m_0[ \alpha(\overline a n) \neq 0]$. Find $k$ such that $\forall a\in Bin_{m_0+1}[a\le k]$ and conclude:  $Bar_\mathcal{C}(D_{\overline \alpha k})$ and  $\exists m[Bar_\mathcal{C}(D_{\overline \alpha m})]$. 
  
  \smallskip Note: if $Bar_\mathcal{C}(D_\alpha)$ and $Unreal(D_\beta)\rightarrow \exists n[Unreal(D_{\overline\beta n})]$, then $\exists m[Bar_\mathcal{C}(D_{\overline \alpha m})]$. This establishes (iv). 
  
  \smallskip Note: if $Bar_\mathcal{C}(D_\alpha)$ and $\neg \exists n[Bar_\mathcal{C}(D_{\overline \alpha n}]$, then $Unreal(D_\beta)$ and $\forall n[Real(D_{\overline \beta n})]$. This establishes (v).
  
  \smallskip Note: if $\forall n[\neg Bar_\mathcal{C}(D_{\overline \alpha n}]$ and $\forall n[Real(D_{\overline \beta n})]\rightarrow Real(D_\beta)$, then   there exists $\gamma$ in $\mathcal{C}$ realizing $D_\beta$, so $\forall m[\tilde\gamma\bigl(\delta(m)\bigr)=1]$ and $\forall m\forall n\le m[\alpha(\overline \gamma n)=0]$, i.e.  $\forall n[\alpha(\overline \gamma n)=0]$.  This establishes (vi). 
\end{proof}

\begin{theorem}

\begin{enumerate}[\upshape (i)]\label{T:proplog}
\item $\mathsf{BIM}\vdash \mathbf{FT} \leftrightarrow \forall \alpha[\mathit{Unreal}(E_\alpha) \rightarrow \exists n[\mathit{Unreal}(E_{\overline \alpha n})]] \leftrightarrow \\\forall \alpha[\mathit{Unreal}(D_\alpha) \rightarrow \exists n[\mathit{Unreal}(D_{\overline \alpha n})]]$. 
\item $\mathsf{BIM}\vdash \neg!\mathbf{FT} \leftrightarrow \exists \alpha[\mathit{Unreal}(E_\alpha) \;\wedge\; \forall n[\mathit{Real}(E_{\overline \alpha n})]] \leftrightarrow \\\exists \alpha[\mathit{Unreal}(D_\alpha) \;\wedge\; \forall n[\mathit{Real}(D_{\overline \alpha n})]]$. 
\item $\mathsf{BIM}\vdash \mathbf{WKL} \leftrightarrow \forall \alpha[ \forall n[\mathit{Real}(E_{\overline \alpha n})]\rightarrow Real(E_\alpha)] \leftrightarrow \\\forall \alpha[ \forall n[\mathit{Real}(D_{\overline \alpha n})]\rightarrow Real(D_\alpha)]$.
\end{enumerate}
\end{theorem}

\begin{proof} Use Lemma \ref{L:compprop} and: $\forall \alpha\exists \beta[D_\alpha = E_\beta]$. \end{proof}

Theorem \ref{T:proplog}(i) also is a consequence of  \cite[Theorem 6.5]{loeb}. 

 V.N. Krivtsov has shown, among other things,  that $\mathbf{FT}$ is an equivalent of an  intuitionistic (generalized) completeness theorem for intuitionistic first-order predicate logic, see \cite{krivtsov}.

\section{Stronger `Fan Theorems'?}
In this Section, we indicate what, on our opinion, should be  the subject of the next chapter in intuitionistic reverse mathematics. The Fan Theorem may be seen as a  replacement, for tne intuitionistic mathematician, of that enviable tool of the classical mathematician: (Weak) K\"onig's Lemma.  We hope to make clear that the greater subtlety of the language of the intuitionistic mathematician allows for many other possible replacements. 
\subsection{Notions of finiteness}\label{SS:finiteness}

Let $\alpha$ be given.\\ We consider the set $D_\alpha:=\{n\mid\alpha(n)\neq 0\}$, the subset of $\omega$  \textit{decided by $\alpha$}.

\smallskip We define: $D_\alpha$ is \textit{finite} if and only if $\exists n\forall m\ge n[\alpha(m)=0]$.

\smallskip We define: $D_\alpha$ is \textit{bounded-in-number} if and only if \\$\exists n \forall t \in [\omega]^{n+1} \exists i<n+1[\alpha\circ t(i)=0]$.

\smallskip We define: $D_\alpha$ is \textit{almost-finite} if and only if $\forall \zeta \in[\omega]^\omega\exists n[\alpha\circ\zeta(n)=0]$.

\smallskip We define: $D_\alpha$ is \textit{not-not-finite} if and only if $\neg\neg\exists n\forall m\ge n[\alpha(m)=0]$.

\smallskip We define: $D_\alpha$ is \textit{not-infinite}\footnote{A referee suggested to consider this notion too. } if and only if $\neg\forall n \exists m\ge n[\alpha(m)\neq 0]$.\\Note that $D_\alpha$ is infinite if and only if $ \exists \zeta \in [\omega]^\omega \forall n[\alpha\circ\zeta(n)\neq 0]$, and that \\$D_\alpha$ is   not-infinite if and only if $\neg \exists \zeta \in [\omega]^\omega \forall n[\alpha\circ\zeta(n)\neq 0]$. 

\smallskip Decidable subsets of $\omega$ that are bounded-in-number are introduced and discussed in \cite{veldman1995}.

\smallskip
Almost-finite decidable subsets of $\omega$ were introduced in \cite{veldman1995} and \cite{veldman1999}, and  are also studied in \cite{veldman2005}.

\begin{lemma}\label{L:finiteness}\begin{enumerate}[\upshape (i)]\item  $\mathsf{BIM}\vdash\forall \alpha[D_\alpha\; is\; finite\;\rightarrow D_\alpha\; is\;bounded$-$in$-$number]$. \item  $\mathsf{BIM}\vdash\forall \alpha[D_\alpha\; is\; bounded$-$in$-$number\; \rightarrow\; D_\alpha\; is\; finite]\rightarrow\mathbf{LPO}$. \item $\mathsf{BIM}\vdash\forall \alpha[D_\alpha\; is\; bounded$-$in$-$number\rightarrow D_\alpha\; is\; almost$-$finite]$. \item  $\mathsf{BIM}\vdash\forall \alpha[D_\alpha\; is\; almost$-$finite\; \rightarrow \;D_\alpha\; is\; bounded$-$in$-$number]\rightarrow\mathbf{LPO}$.\item $\mathsf{BIM}\vdash \forall \alpha[D_\alpha\; is\; almost$-$finite\;\rightarrow\; D_\alpha\; is\; not$-$infinite]$.    \item $\mathsf{BIM}+\mathbf{BARIND}\footnote{See \ref{SS:barind}. We do not know if the use of this priniple here is unavoidable.}
\vdash\forall \alpha[D_\alpha\; is\; almost$-$finite\;\rightarrow\; D_\alpha\; is\; not$-$not$-$finite]$. 
\end{enumerate}\end{lemma}

\begin{proof}  (i) Let $\alpha, n$ be given such that $\forall m\ge n[\alpha(m)=0]$.

Note: $\forall t\in[\omega]^{n+1}[t(n)\ge n]$ and conclude: $\forall t \in [\omega]^{n+1}[\alpha\circ t(n)=0]$. 

\smallskip (ii) Let $\alpha$ be given. Define $\alpha^\ast$ such that $\forall n[\alpha^\ast(n)\neq 0\leftrightarrow n=\mu m[\alpha(m)\neq 0]]$. 

Note: $\forall t\in [\omega]^2 \exists i<2[\alpha^\ast\circ t(i)=0]$, so $D_{\alpha^\ast}$ is bounded-in-number.

Assuming that $D_{\alpha^\ast}$ is finite, find $n$ such that $\forall m\ge n[\alpha^\ast(m) =0]$. 

\textit{Either} $\exists m<n[\alpha^\ast(m)\neq 0]$ \textit{or} $\forall m[\alpha^\ast(m)=0]$, and, therefore,  \\\indent\textit{either} $\exists m[\alpha(m)\neq 0]$ \textit{or} $\forall m[\alpha(m)=0]$.

We thus see how, starting from the assumption: \\$\forall \alpha[D_\alpha\; is\; bounded$-$in$-$number\; \rightarrow D_\alpha\; is\; finite]$, one obtains the conclusion: $\forall \alpha[\exists m[\alpha(m)\ne 0]\;\vee\;\forall m[\alpha(m)=0]]$, i.e. $\mathbf{LPO}$. 

\smallskip (iii) Let $\alpha, n$ be given such that $\forall t \in [\omega]^{n+1}\exists i<n+1[\alpha\circ t(i)=0]$. 

Conclude: $\forall \zeta \in [\omega]^\omega\exists i < n+1[\alpha\circ\zeta(i)=0]$. 

\smallskip 
(iv) Let $\alpha$ be given. Define $\alpha^\ast$ such that \\$\forall n[\alpha^\ast(n)\neq 0\leftrightarrow \mu m[\alpha(m)\neq 0]\le n<2\cdot\mu m[\alpha(m)\neq 0]]$. 

Observe: for all $k$, if $k=\mu[\alpha(m)\neq 0]$, then $\forall n[k\le n<2\cdot k\leftrightarrow \alpha^\ast(n)\neq 0]$ and $\exists t \in [\omega]^k\forall i<k[\alpha^\ast\circ t(i)\neq 0]$ and $\forall t \in [\omega]^{k+1}\exists i<k+1[\alpha^\ast\circ t(i)= 0]$

Let $\zeta$ in $[\omega]^\omega$ be given. \\We want to prove; $\exists n[\alpha^\ast\circ \zeta(n)=0]$ and  and distinguish two cases. \\\textit{Case (a)}. $\alpha^\ast\circ\zeta(0)= 0 $. Then we are done. \\\textit{Case (b)}. $\alpha^\ast\circ \zeta(0) \neq  0$. Then $\exists m[\alpha(m)\neq 0]$. Define $k:=\mu m[\alpha(m)\neq 0]$ and note: $\forall m\ge 2\cdot k[\alpha^\ast(m)= 0]$, and, in particular, $\alpha^\ast\circ \zeta(2\cdot k) =0$. 

\smallskip We thus see: $\forall \zeta\in [\omega]^\omega\exists n[\alpha^\ast\circ \zeta(n)=0]$, i.e. $D_{\alpha^\ast}$ is almost-finite. 

Assuming that $D_{\alpha^\ast}$ is bounded-in-number, find $n$ such that \\$\forall t \in [\omega]^{n+1}\exists i <n+1[\alpha^\ast\circ t(i) =0]$. \\Conclude: for all $k$, if $k=\mu m[\alpha(m)\neq 0]$, then $k<n+1$.  

\textit{Either} $\exists k<n+1[\alpha\neq 0]$ \textit{or} $\forall k<n+1[\alpha(k)=0]$, and, therefore,  \\\indent\textit{either} $\exists k[\alpha(k)\neq 0]$ \textit{or} $\forall k[\alpha(k)=0]$.

We thus see how, starting from the assumption: \\$\forall \alpha[D_\alpha\; is\;almost$-$finite\; \rightarrow D_\alpha\; is\; bounded$-$in$-$number]$, one obtains the conclusion: $\forall \alpha[\exists k[\alpha(k)\ne 0]\;\vee\;\forall k[\alpha(k)=0]]$, i.e. $\mathbf{LPO}$.

\smallskip (v) Let $\alpha$ be given such that $\forall \zeta\in[\omega]^\omega\exists n[\alpha\circ\zeta(n)=0]$. 
 
 Define $B:=\bigcup_n\{s\in \omega^n\mid s\notin [\omega]^n\;\vee\;\exists i<n[\alpha\circ s(i)=0]\}$. \\We prove: $Bar_{\omega^\omega}(B)$. \\Let $\gamma$ be given. Define $\gamma^\ast$ such that $\gamma^\ast(0)=\gamma(0)$, and, for each $n$, if $\overline \gamma(n+2)\in[\omega]^{<\omega}$, then $\gamma^\ast(n+1)=\gamma(n+1)$, and, if not, then $\gamma^\ast(n+1)=\gamma^\ast(n)+1$. Note: $\gamma^\ast \in [\omega]^\omega$ and find $n$ such that $\overline{\gamma^\ast}n \in B$. \textit{Either} $\overline \gamma n = \overline{\gamma^\ast}n$ \textit{or} $\overline \gamma n\notin [\omega]^n$.\\ In both cases; $\overline \gamma n \in B$.

\noindent Define $C:=\bigcup_n\{s\in \omega^n\mid s\notin[\omega]^n \;\vee\; \exists i<n[\alpha\circ s(i)=0]\;\vee\;D_\alpha\;is\;not$-$not$-$finite]\}$. 
 
 \smallskip
 Note:  $B\subseteq C$. 
 
\smallskip Let $s,n$ be given such that $s\in \omega^n$ and $\forall m[s\ast\langle m\rangle \in C]$. \\We intend to prove: $s\in C$.
 
 We may assume: $s\in [\omega]^n \;\wedge\;\neg \exists i<n[\alpha\circ s(i)=0]$,   
 and distinguish two cases. \\\textit{Case(a)}. $\exists m[s\ast\langle m \rangle \in [\omega]^{n+1} \;\wedge\;\alpha(m)\neq 0]$. Finding such $m$, we consider $s\ast\langle m \rangle$ and conclude: $s\ast\langle m \rangle \in C$ and, therefore: $D_\alpha$ is not-not-finite. \\\textit{Case (b)}. $\neg\exists m[s\ast\langle m \rangle \in [\omega]^{n+1} \;\wedge\;\alpha(m)\neq 0]$, and, therefore, \\$\forall m[s\ast\langle m\rangle \in [\omega]^{n+1}\rightarrow \alpha(m)=0]$. Conclude: $D_\alpha$ is finite.
 
 Defining $P:=\exists m[s\ast\langle m \rangle \in [\omega]^{n+1} \;\wedge\;\alpha(m)\neq 0]$, we thus see: \\$(P\;\vee\;\neg P)\rightarrow D_\alpha$ is not-not-finite. \\By intuitionistic logic\footnote{$\neg\neg(P\vee\neg P)$ is provable, and from $A\rightarrow B$ one may conclude $\neg B \rightarrow \neg A$ and also $\neg\neg A \rightarrow \neg\neg B$. Furthermore, $\neg\neg\neg C$ is equivalent to $\neg C$ and $\neg\neg\neg\neg C$ is equivalent to $\neg\neg C$.}, we conclude: $D_\alpha$ is not-not-finite and: $s\in C$.
 
 We thus see: $\forall s[\forall m[s\ast\langle m \rangle \in C]\rightarrow s\in C]$.
 
 Obviously: $\forall s\forall m[s\in C\rightarrow s\ast\langle m \rangle \in C]$. 
 
 Using $\mathbf{BARIND}$, we conclude: $\langle\;\rangle\in C$, and: $D_\alpha$ is not-not-finite. 
 
 We thus see: \\$\mathsf{BIM}+\mathbf{BARIND}\vdash\forall \alpha[D_\alpha\; is\; almost$-$finite\;\rightarrow\; D_\alpha\; is\; not$-$not$-$finite]$.
 \end{proof}
  
  \begin{lemma}\label{L:not-infinite} $\mathsf{BIM}+\mathbf{MP}$\footnote{For Markov's Principle $\mathbf{MP}$, see Subsubsection \ref{SSS:markov}.} proves \\$ \forall \alpha[D_\alpha\; is\; not$-$infinite\;\leftrightarrow\;D_\alpha\; is\;not$-$not$-$finite\;\leftrightarrow D_\alpha\;is\;almost$-$finite]$.\end{lemma}
  \begin{proof}First, let $\alpha$ be given such that  $D_\alpha$ is not-finite. Then $\neg\exists n\forall m\ge n[\alpha(m) =0]$. \\Conclude $\forall n\neg\neg\exists m\ge n[\alpha(m) \neq 0]$.\\Using $\mathbf{MP}$, conclude: $\forall n\exists m\ge n[\alpha(m)\neq 0]$, i.e. $D_\alpha$ is infinite.

  We thus see, using $\mathbf{MP}$: for all $\alpha$, if $D_\alpha$ is not-finite, then $ D_\alpha$  is infinite,  and may conclude: for all $\alpha$, if $D_\alpha$ is not- infinite, then $D_\alpha$ is not-not-finite.
  
  \smallskip Secondly, let $\alpha$ be given such that $\neg\neg(D_\alpha$ is almost-finite), i.e. \\$\neg\neg\forall \zeta\in [\omega]^\omega\exists n[\alpha\circ \zeta(n)=0]$. Conclude:  $\forall \zeta\in [\omega]^\omega\neg\neg\exists n[\alpha\circ \zeta(n)=0]$, and, by $\mathbf{MP}$, $\forall \zeta\in [\omega]^\omega\exists n[\alpha\circ \zeta(n)=0]$, i.e. $D_\alpha$ is almost-finite.
  
  We thus see: for all $\alpha$, if $\neg\neg(D_\alpha$ is almost-finite), then $D_\alpha$ is almost-finite.
  
  As we also know, from Lemma \ref{L:finiteness}: for all $\alpha$, if $D_\alpha$ is finite, then $D_\alpha$ is almost-finite, we conclude:
  for all $\alpha$, if $D_\alpha$ is not-not-finite, then $D_\alpha$ is almost-finite. 
  
 \smallskip Finally, it is clear that, for all $\alpha$, if $D_\alpha$ is almost-finite, then $D_\alpha$ is not-infinite. \end{proof}
We now extend the notion `almost-finite' from decidable subsets of $\omega$ to enumerable subsets of $\omega$.

 Let $\alpha$ be given. We consider the set \\$E_\alpha:=\{n\mid\exists m[\alpha(m)=n+1]\}$, the subset of $\omega$ \textit{enumerated by} $\alpha$. 
  We define: \\$E_\alpha$ is \textit{almost-finite} if and only if $\forall \zeta\in[\omega]^\omega\exists m\exists n[m<n\;\wedge\;\alpha\circ\zeta(m)=\alpha\circ\zeta(n)]$. 
 
 \smallskip The first  item of the next Lemma show that this definition is consistent with the definition given earlier for decidable subsets of $\omega$.
 
 The second item shows that the definition is a good definition indeed as it 
 does not depend on the enumeration $\alpha$ of $E_\alpha$.

 The fifth item shows that an almost-finite union of almost-finite enumerable subsets of $\omega$ is almost-finite and enumerable. 
 \begin{lemma}\label{L:almost-finite}$\mathsf{BIM}$ proves the following.  \begin{enumerate}[\upshape (i)]\item $\forall\alpha\forall\beta[D_\alpha=E_\beta\rightarrow \\\bigl(\forall \zeta\in[\omega]^\omega\exists n[\alpha\circ\zeta(n)=0]\leftrightarrow\forall \zeta\in[\omega]^\omega\exists m\exists n[m<n\;\wedge\;\beta(m)=\beta(n)]\bigr)$. \item $\forall\alpha\forall\beta[\bigl(E_\alpha=E_\beta\;\wedge\; \forall \zeta\in[\omega]^\omega\exists m\exists n[m<n\;\wedge\;\alpha\circ\zeta(m)=\alpha\circ\zeta(n)]\bigr)\rightarrow\\ \forall \zeta\in[\omega]^\omega\exists m\exists n[m<n\;\wedge\;\beta\circ\zeta(m)=\beta\circ\zeta(n)]]$. \item $\forall \alpha[\forall i<2[E_{\alpha^i}\;is\;almost$-$finite]\rightarrow \bigcup_{i<2}E_{\alpha^i}\; is\; almost$-$finite]$. \item $\forall \alpha\forall n[  \forall i<n[E_{\alpha^i}\;is\;almost$-$finite]\rightarrow\bigcup_{i<n}E_{\alpha^i}\; is\; almost$-$finite]$.  
 \item $\forall \alpha[(\forall n[E_{\alpha^n}\;is\;almost$-$finite]\;\wedge\;\forall \zeta \in [\omega]^\omega\exists n[\alpha^{\zeta(n)}=\underline 0])\rightarrow 
 \\\bigcup_n E_{\alpha^n}\; is\; almost$-$finite]$.\end{enumerate} \end{lemma}
  \begin{proof} (i) The proof is left to the reader. 
  
  \smallskip (ii) The proof is left to the reader.

  \smallskip (iii) Let $\alpha$ be given such that $E_{\alpha^0}, E_{\alpha^1}$ are almost- finite. Define $\alpha^\ast$ such that, for each $n$, $\alpha^\ast(2n)=\alpha^0(n)$ and $\alpha^\ast(2n+1)=\alpha^1(n)$. Note: $E_{\alpha^\ast}=E_{\alpha^0}\cup E_{\alpha^1}$. 
  
\smallskip  Let $\zeta$ in $[\omega]^\omega$ be given. Define $QED:=\exists m\exists n[m<n\;\wedge\;\alpha^\ast\circ\zeta(m)=\alpha^\ast\circ\zeta(n)]$. 
  
  \smallskip We prove: $\forall k\exists l> k[\exists p[\zeta(l)=2p+1]\;\vee\;QED]$. 
  
  Let $k$ be given. Define $\zeta^\ast$ such  that, for each $i$, if $\forall j\le i\exists p[\zeta(k+1+j)=2p]$, then $\zeta^\ast(i)=\zeta(k+1+i)$, and, if not, then $\zeta^\ast(i)=2\cdot \zeta(k+1+i)$. Note: $\forall i\exists p[\zeta^\ast(i)=2p]$. Define $\zeta^{\ast\ast}$ such that $\forall i[\zeta^\ast(i)=2\cdot\zeta^{\ast\ast}(i)]$ and note: $\forall i[\alpha^\ast\circ\zeta^\ast(i)=\alpha^0\circ\zeta^{\ast\ast}(i)]$. Find $m,n$ such that $m<n$ and $\alpha^0\circ\zeta^{\ast\ast}(m)=\alpha^0\circ\zeta^{\ast\ast}(n)$ and note: \\\textit{either} $\zeta^\ast(m)=\zeta(k+1+m)$ and $\zeta^\ast(n)=\zeta(k+1+n)$ and: $QED$ \\\textit{or} $\exists j\le n\exists p[\zeta(k+1+j)=2p+1]$.

  \smallskip Using Theorem \ref{T:acoosigma01},  find $\delta$ such that 
  $\forall k[\delta(k)> k \;\wedge\; (\exists p[\zeta\circ\delta(k)=2p+1]\;\vee\;QED)]$.
  
Define $\zeta^\dag$ such that $\zeta^\dag(0)=\delta(0)$ and, for each $k$, $\zeta^\dag(k+1)=\delta\bigl(\zeta^\dag(k)\bigr)$. \\Note: for each $k$, if $\exists p[\zeta^\dag(k)=2p]$, then $QED$.
  
  Define $\zeta^{\dag\ast}$ such  that, for each $i$, if $\forall j\le i\exists p[\zeta^\dag(j)=2p+1]$, then $\zeta^{\dag\ast}(i)=\zeta^\dag(i)$, and, if not, then $\zeta^{\dag\ast}(i)=2\cdot \zeta^\dag(i)+1$. Note: $\forall i\exists p[\zeta^{\dag^\ast}(i)=2p+1]$. Define $\zeta^{\dag\ast\ast}$ such that $\forall i[\zeta^{\dag\ast}(i)=2\cdot\zeta^{\dag\ast\ast}(i)+1]$ and note: $\forall i[\alpha^\ast\circ\zeta^{\dag\ast}(i)=\alpha^1\circ\zeta^{\dag\ast\ast}(i)]$. Find $m,n$ such that $m<n$ and $\alpha^1\circ\zeta^{\dag\ast\ast}(m)=\alpha^1\circ\zeta^{\dag\ast\ast}(n)$ and note:\\ \textit{either} $\zeta^{\dag\ast}(m)=\zeta^\dag(m)$ and $\zeta^{\dag\ast}(n)=\zeta^\dag(n)$ and: $QED$ \\\textit{or} $\exists j\exists p[\zeta^\dag(j)=2p]$ and $QED$, so in any case $QED$. 
 
 \smallskip
We thus see: $\forall \zeta\in[\omega]^\omega \exists m\exists n[m<n\;\wedge\;\alpha^\ast\circ\zeta(m)=\alpha^\ast\circ\zeta(n)]$, i.e. \\$E_{\alpha^\ast}= E_{\alpha^0}\cup E_{\alpha^1}$ is almost-finite.

  \smallskip (iv) Use (iii) and induction.

  \smallskip (v) Let $\alpha$ be given such that, for all $n$,  $E_{\alpha^n}$ is almost-finite, and \\$\forall \zeta\in[\omega]^\omega\exists n[\alpha^{\zeta(n)}=\underline 0]$. We are going to prove: $\bigcup_n E_{\alpha^n}$ is almost-finite.
  \\Define $\alpha^\ast$ such that, for all $n,m$, $\alpha^\ast(n,m)=\alpha^n(m)$ and note: $E_{\alpha^\ast}= \bigcup_n E_{\alpha^n}$. 
   \\Let $\zeta$ in $[\omega]^\omega$ be given. Define $QED:=\exists m\exists n[m<n\;\wedge\;\alpha^\ast\circ\zeta(m)=\alpha^\ast\circ\zeta(n)]$.
   
   We first prove: $\forall k\forall n\exists l[l>k \;\wedge\;\bigl(QED\;\vee\;\bigl(\zeta(l)\bigr)'>n\bigr)]$. 
   
   Let $k,n$ be given. \\If $\bigl(\zeta(k+1)\bigr)'>n$, there is nothing to prove, so we assume $\bigl(\zeta(k+1)\bigr)'\le n$. \\Define $\zeta^\ast$ such that $\zeta^\ast(0)=\zeta(k+1)$ and, for all $i$, if $\forall j\le i+1[\bigl(\zeta(k+1+j)\bigr)'\le n]$, then $\zeta^\ast(i+1)=\zeta(k+2+i)$, and, if not, then $\zeta^\ast(i+1)=\mu p[p>\zeta^\ast(i)\;\wedge\;p'=n]$.  \\Note: $\forall i[\bigl(\zeta^\ast(i)\bigr)' \le n]$. \\Using (iii), find $p,q$ such that $p<q$ and $\alpha^\ast\circ \zeta^\ast(p)=\alpha^\ast\circ\zeta^\ast(q)$ and note:\\ \textit{either} $\zeta^\ast(p)=\zeta(k+1+p)$ and $\zeta^\ast(q)=\zeta(k+1+q)$ and $QED$, \\\textit{or} $\exists j\le q[\bigl(\zeta(k+1+j)\bigr)'>n]$.
   
   We thus see:  $\forall k\forall n\exists l[l>k \;\wedge\;\bigl(QED\;\vee\;\bigl(\zeta(l)\bigr)'>n\bigr)]$.

   \smallskip Using Theorem \ref{T:acoosigma01}, find $\delta$ such that \\$\forall k\forall n[\delta(k,n)>k\;\wedge\;\big(QED\;\vee\;\bigl(\zeta\circ\delta(k,n)\bigr)'>n\bigr)]$.
   
   Now define $\eta$ such that  $\eta(0)=\delta(0,0)$, and, for each $n$, $\eta(n+1)=\delta\bigl(\eta(n),n\bigr)$. 
   
   Note: $\eta \in [\omega]^\omega$ and $\forall n[\bigl(\zeta\circ\eta(n)\bigr)'>n \;\vee\; QED]$. 
   
   Find $\rho$ in $[\omega]^\omega$ such that $\forall n[((\zeta\circ\eta\circ\rho(n+1)\bigr)'>((\zeta\circ\eta\circ\rho(n)\bigr)' \;\vee\;QED]$.

   Find $p$ such that   $\alpha^{\bigl(\zeta\circ\eta\circ\rho(p)\bigr)'}=\underline 0$. 
   
   Conclude: \textit{either} $QED$, \textit{or} \\$\alpha^\ast\circ\zeta\circ\eta\circ\rho(\langle p,0\rangle) = \alpha^\ast\circ\zeta\circ\eta\circ\rho (\langle p,1\rangle)=0$, and again: $QED$. 
   
   \smallskip
   We thus see: $\forall\zeta\in[\omega]^\omega\exists m\exists n[m<n\;\wedge\;\alpha^\ast\circ\zeta(m)=\alpha^\ast\circ\zeta(n)]$, i.e.\\ $E_{\alpha^\ast}=\bigcup_n E_{\alpha^n}$ is almost-finite.

   \end{proof}
\subsection{Fans, approximate fans and almost-fans}\hfill

\smallskip
Let $\beta$ be given. We define: $\mathcal{F}_\beta:=\{\alpha\mid\forall n[\beta(\overline \alpha n)=0]\}$.

\smallskip
We define: $\beta$ is a \textit{spread-law},  $Spr(\beta)$, if and only if \\$\forall s[\beta(s)=0\leftrightarrow \exists n[\beta(s\ast\langle n \rangle)=0]]$.

If $\beta$ is a spread-law, then $\mathcal{F}_\beta$ is a \textit{spread}. 

\smallskip
We define: $\beta$ is a \textit{fan-law}, $Fan (\beta)$,  if and only if \\$Spr(\beta)$ and $\forall s\exists n\forall m[\beta(s\ast\langle m\rangle)= 0\rightarrow m\le n]$. 

If $\beta$ is a fan-law, then $\mathcal{F}_\beta$ is a \textit{fan}.

\smallskip We define: $\beta$ is an \textit{explicit fan-law} if and only if\\ $Fan(\beta)$ and $\exists \gamma \forall s\forall m[\beta(s\ast\langle m \rangle)=0\rightarrow  m\le\gamma(s)]$. 

If $\beta$ is an explicit fan-law, then $\mathcal{F}_\beta$ is an \textit{explicit fan}. 

\smallskip One may prove in $\mathsf{BIM}$, see Lemma \ref{L:explicitfan}:

$Explfan(\beta)\leftrightarrow \bigl(Spr(\beta)\;\wedge\;\exists \gamma \forall n\forall s \in [\omega]^n[ \beta(s)= 0\rightarrow s \le \gamma(n)]\bigr)$.

\smallskip We define: $\beta$ is an \textit{approximate-fan-law}, $Appfan(\beta)$, if and only if \\$Spr(\beta)$ and $\forall n\exists k\forall t \in [\omega]^{k+1}\exists i<k+1[t(i)\notin \omega^n\;\vee\;\beta\bigl(t(i)\bigr)\neq 0]$. 

If $\beta$ is an approximate-fan-law, then $\mathcal{F}_\beta$ is an \textit{approximate fan}.

\smallskip We define: $\beta$ is an \textit{almost-fan-law}, $Almfan(\beta)$, if and only if \\$Spr(\beta)$ and $\forall s \forall \zeta\in[\omega]^\omega \exists m[\beta\bigl(s\ast\langle \zeta(m)\rangle\bigr)\neq 0]$.  

If $\beta$ is an almost-fan-law, then $\mathcal{F}_\beta$ is an \textit{almost-fan}.

\subsection{}\textit{The Almost-fan Theorem as a Scheme}, $\mathbf{ALMFAN}$:\\ $\forall\beta[\bigl(Almfan(\beta)\;\wedge\;Bar_{\mathcal{F}_\beta}(B)\bigr)\rightarrow \\\exists \alpha[E_\alpha\subseteq B\;\wedge\;E_\alpha\;is\;almost$-$finite\;\wedge\;Bar_{\mathcal{F}_\beta}(E_\alpha)]].$

The following theorem may be compared to Theorem \ref{T:ftscheme}.

\begin{theorem}\label{T:almftscheme} $\mathsf{BIM}+\mathbf{BARIND}+\mathbf{AC}_{\omega,\omega^\omega}\vdash \mathbf{ALMFAN}$. \footnote{We do not know id the use of the Axiom $\mathbf{AC}_{\omega,\omega^\omega}$ is avoidable.}\end{theorem}

\begin{proof} Let $\beta$ be given such that $Almfan(\beta)$ and $\beta(\langle\;\rangle)=0$.\footnote{If $\beta(\langle\;\rangle)\neq 0$, then $\mathcal{F}_\beta=\emptyset$ and there is nothing to prove.} Assume $Bar_{\mathcal{F}_\beta}(B)$.

Define $B':=B\cup\{s\mid \beta(s)\neq 0\}$.

We claim: $Bar_{\omega^\omega}(B')$. In the proof of Theorem \ref{T:ftscheme} we have seen how to prove this claim. 

\smallskip Let $C$ be the set of all $s$ such that either: $\beta(s)\neq 0$ or: \\$\beta(s) = 0$ and $\exists \alpha[ E_\alpha \subseteq B \;\wedge\; E_\alpha\;is\;almost$-$finite\;\wedge\; Bar_{\mathcal{F}_\beta\cap s}(E_\alpha)]$. 

\smallskip
One easily sees: $B\subseteq C$: for every $s$, if $\beta(s) =0$ and $s\in B$, define $\alpha$ such that $\forall n[\alpha(n)=s+1]$ and note: $\{s\}=E_\alpha\subseteq B$ and  $E_\alpha$ is finite and $Bar_{\mathcal{F}_\beta\cap s}(E_\alpha)$. 

\smallskip Now let $s$ be given such that $\forall m[s\ast\langle m \rangle \in C]$. Using $\mathbf{AC}_{\omega, \omega^\omega}$, find $\alpha$ such that, for all $m$, if $\beta(s\ast\langle m \rangle)=0$, then $E_{\alpha^m}\subseteq B$ and $E_{\alpha^m}$ is almost-finite and $Bar_{\mathcal{F}_\beta\cap s\ast\langle m \rangle}(E_{\alpha^m})$, and, if $\beta(s\ast\langle m \rangle)\neq 0$, then $\alpha^m = \underline 0$ and $E_{\alpha^m}=\emptyset$. 
\\Note: $Almfan(\beta)$ and, therefore, $\forall \zeta\in[\omega]^\omega\exists m[\alpha^{\zeta(m)}=\underline 0]$. \\Use Lemma \ref{L:almost-finite}(iv) and conclude: $\bigcup_m E_{\alpha^m}$ is almost-finite. \\Also: $Bar_{\mathcal{F}_\beta\cap s}(\bigcup_n E_{\alpha^n})$. Conclude: $s\in C$.

We thus see: $\forall s[\forall m[s\ast\langle m \rangle\in C]\rightarrow s \in C]$. 

Obviously: $\forall s\forall m[s\in C\rightarrow s\ast\langle m \rangle \in C]$. 

Using $\mathbf{BARIND}$, we conclude: $\langle\;\rangle\in C$, i.e.\\ $\exists \alpha[ E_\alpha \subseteq B \;\wedge\; E_\alpha\;is\;almost$-$finite\;\wedge\; Bar_{\mathcal{F}_\beta}(E_\alpha)]$. 
\end{proof} 
 \subsection{}         \textit{ The Almost-fan Theorem}, $\mathbf{AlmFT}$: \[\forall\beta[Almfan(\beta)\rightarrow\forall\alpha[\bigl( Thinbar_{\mathcal{F}_\beta}(D_\alpha) \;\wedge\; \forall s\in D_\alpha [\beta(s) = 0]\bigr)\rightarrow \]\[ \forall \zeta \in [\omega]^\omega\exists n[\zeta(n) \notin D_\alpha]]],\]
    i.e.: \textit{every decidable subet of $\omega$ that is a thin bar in some almost-fan  is almost-finite}.  
    
  \begin{theorem}\label{T:almft} $\mathsf{BIM}+\mathbf{ALMFAN}\vdash
  \mathbf{AlmFT}$. \end{theorem} 
  
   \begin{proof} Let $\beta, \alpha$ be given such that $Appfan(\beta)$ and $Thinbar_{\mathcal{F}_\beta}(D_\alpha)$ and \\$\forall s \in D_\alpha[\beta(s)=0]$.  Applying Theorem \ref{T:almftscheme}, find $\gamma$ such that $E_\gamma \subseteq D_\alpha$ and $Bar_{\mathcal{F}_\beta}(E_\gamma)$ and $E_{\gamma}$ is almost-finite. As $\forall s\in D_\alpha\forall t\in D_\alpha[s\sqsubseteq t\rightarrow s=t]$, conclude: $E_\gamma = D_\alpha$ and, by Lemma \ref{L:finiteness}(i), $\forall \zeta \in [\omega]^\omega\exists n[\zeta(n)\notin D_\alpha]$. \end{proof}
     The Almost-fan Theorem occurs in \cite{veldman2001a} and \cite{veldman2001b}.  \subsection{}  \textit{The Approximate-fan Theorem}, $\mathbf{AppFT}$:
   \[\forall \beta \forall \alpha[\bigl(Appfan(\beta) \;\wedge\; Thinbar_{\mathcal{F}_\beta}(D_\alpha) \;\wedge\; \forall s \in D_\alpha[ \beta(s) = 0]\bigr)\rightarrow \]\[ \forall \zeta \in [\omega]^\omega\exists n[\zeta(n) \notin D_\alpha]],\]
     i.e.: \textit{every decidable subet of $\omega$ that is a thin bar in some approximate fan  is almost-finite}.
    \begin{theorem}\label{T:almftappft} $\mathsf{BIM}+\mathbf{AlmFT}\vdash \mathbf{AppFT}$.\end{theorem} \begin{proof} Obvious, as every approximate fan is an almost-fan. \end{proof}
    
    I do not know if $\mathsf{BIM}$ proves $\mathbf{AppFT}\rightarrow\mathbf{AlmFT}$.
    \begin{theorem}\label{T:appftft} $\mathsf{BIM}+\mathbf{AppFT}\vdash \mathbf{FT}$. \end{theorem}
    
    \begin{proof} Assume $\mathbf{AppFT}$. We want to prove $\mathbf{FT}$ and use Theorem \ref{T:{thinfan}}.
    
     Let $\alpha$ be given such that $Thinbar_\mathcal{C}(D_\alpha)$.  \\Using $\mathbf{AppFT}$, conclude: $D_{\alpha}$ is almost-finite.\\ Define $\zeta$ in $[\omega]^\omega$ such that, for each $n$, \\ if $\neg \forall\gamma\in\mathcal{C}\exists i<n[\zeta(i)\sqsubset\gamma]$, then $\zeta(n)=\mu p[p\in D_{\alpha}\;\wedge\;\forall i<n[p\neq\zeta(i)]]$.  \\Find $p$ such that $\zeta(p)\notin D_{\alpha}$.  \\Conclude: $\forall n>\zeta(p)[n\notin D_\alpha]$. 
    
    We thus see: $\forall \alpha[Thinbar_\mathcal{C}(D_\alpha)\rightarrow\exists m\forall  n>m[n \notin D_\alpha]]$. 
    
    Using Theorem \ref{T:{thinfan}}, we conclude: $\mathbf{FT}$. \end{proof}

$\mathsf{BIM}$ does not prove $\mathbf{FT}\rightarrow \mathbf{AppFT}$, see \cite[Corollary 10.6]{veldman2011c}.

 In $\mathsf{BIM}$, $\mathbf{AppFT}$ has a number of important equivalents, see \cite{veldman2011c}.  One of them is  the following principle of Open Induction:
 $$\mathbf{OI}([0,1]): \forall \alpha[\forall \delta \in [0,1][[0,\delta)\subseteq \mathcal{H}_\alpha\rightarrow\delta\in\mathcal{H}_\alpha]\rightarrow [0,1]\subseteq\mathcal{H}_\alpha].$$

 The relation between $\mathbf{FT}$ and $\mathbf{AppFT}$ in $\mathsf{BIM}$ may be compared to the relation  between $\mathbf{WKL}$ and $\mathbf{KL}$ in $\mathsf{RCA}_0$. 

 In the classical context of $\mathsf{RCA_0}$, one studies two extensions of $\mathbf{WKL}$,
 
\noindent\textit{Bounded K\"onig's Lemma} $\mathbf{BKL}$, that (for a classical reader) would coincide with (a contraposition of) the second formulation of $\mathbf{FT}$ \ref{SSS:fan} and  \textit{K\"onig's Lemma} $\mathbf{KL}$, that, similarly, would coincide with $\mathbf{FT}^+$ \ref{SSS:fan+}.

$\mathbf{BKL}$ is, in $\mathsf{RCA_0}$,  equivalent to $\mathbf{WKL}$, see   \cite[Lemma IV.1.4]{Simpson}, just as, in $\mathsf{BIM}$, the first and second formulation of $\mathbf{FT}$ are equivalent.

 $\mathbf{KL}$, on the other hand,   is
definitely stronger than $\mathbf{WKL}$, as $\mathsf{RCA_0} +\mathbf{KL}$ is equivalent to $\mathsf{ACA_0}$.

As we observed before, see Theorem \ref{T:fanfan+},
$\mathsf{BIM}+$\textit{Weak-}$\mathbf{\Pi}^0_1$-$\mathbf{AC_{\omega,\omega}}\vdash\mathbf{FT}\leftrightarrow \mathbf{FT}^+.$\footnote{One may learn from \cite[Theorem 9.20]{kohlenbach2}   that, using countable choice, in fact only \textit{Weak}-$\mathbf{\Pi}^0_1$-$\mathbf{AC}_{\omega,\omega}$, one may constructively derive $\mathbf{KL}$ from $\mathbf{WKL}$. On the problem of treating non-constructive assumptions in a constructive context, see the end of Section \ref{S:BIM}.}

 From a constructive point of view,  the axiom  \textit{Weak-}$\mathbf{\Pi}^0_1$-$\mathbf{AC_{\omega,\omega}}$ is weak indeed, as the antecedent is read constructively.
 Also note: $\mathsf{BIM}+ \mathbf{AC}_{\omega,\omega}!\vdash$ \textit{Weak-}$\mathbf{\Pi}^0_1$-$\mathbf{AC}_{\omega,\omega}$.

In a classical context, the r\^ole of a countable axiom of choice is very different from the r\^ole of such an axiom in a constructive context, see  \cite[Appendix 1]{Howard-Kreisel}.

  It seems to us that $\mathbf{FT}^+$ is still too close to $\mathbf{FT}$ for being a good candidate for playing, in the intuitionistic context, a r\^ole like the r\^ole played by $\mathbf{KL}$ in the classical context.

Note that our `axiom' $\mathbf{AppFT}$ is a possibly better candidate, as,  for a classical reader, $\mathbf{AppFT}$ is still  indistinguishable from $\mathbf{KL}$.

\medskip Some authors have called our $\mathbf{FT}$ the \textit{Weak Fan Theorem} $\mathbf{WFT}$, see for instance \cite{loeb}, but we decided not to follow them. 

There is a constructive version of \textit{Weak Weak K\"onig's Lemma} $\mathbf{WWKL}$, see \cite[Definition X.1.7]{Simpson},  that is called $\mathbf{WWFT}$, see \cite{nemoto}.

\section{Notation and conventions}\label{S:notation}

In this Section we explain how, in $\mathsf{BIM}$, some useful notation is introduced  and some elementary results are proven. 
  \subsection{Finite and infinite sequences of natural numbers}\label{SS:sequences}\hfill

 $\mathsf{BIM}$ contains a constant $p$ denoting the function
 enumerating the prime numbers: $p(0)= 2, p(1)= 3, \ldots$
 
  We code finite sequences of natural numbers by
 natural numbers: 
 
 $\langle\;\rangle:=0$ and, for each $k>0$, for all $m_0,m_1, \ldots m_{k-1}$, 
 \[\langle m_0, \ldots, m_{k-1} \rangle = p(k-1)\cdot\prod_{i<k}p(i)^{m_i}-1 \]

 $length(0):=0$ and, 
 
 for each $s>0$,  $\mathit{length}(s) := 1\;+$
\textit{the largest $k$ such that  $p(k)$ divides $s+1$}.

If $i<\mathit{length}(s)-1$, then
 $s(i):=$  \textit{the largest m such that  $p(i)^m$ divides  $s+1$}, and, if $i=\mathit{length}(s)-1$, then $s(i):=$ \textit{the largest m such that  $p(i)^{m+1}$ divides  $s+1$}, and, if $i\ge\mathit{length}(s)$ then $s(i):=0$.
 
 Note:  if $\mathit{length}(s) = k$, then $s =\langle s(0), s(1), \ldots, s(k-1)\rangle$. 
 
 Also note: $\forall s[s\ge\mathit{length}(s)]$.
 
 \smallskip
 $a\ast b$ is the number $s$ satisfying: $\mathit{length}(s) = \mathit{length}(a) + \mathit{length}(b)$ and, 
 \\for each $n$,  if $n < length(a)$, then $s(n) = a(n)$ and, 
 \\if $\mathit{length}(a) \le n <\mathit{length}(s)$, 
 then $s(n) = b \bigl(n - \mathit{length}(a)\bigr)$.

 $a \ast \alpha$ is the element $\beta$ of $\omega^\omega$ satisfying: for each $n$, if $n< \mathit{length}(a)$, \\then $\beta(n) = a(n)$, and, if $ \mathit{length}(a)\le n$, then $\beta(n) = \alpha\bigl(n - \mathit{length}(a)\bigr)$.
 
 For  $n \le length(a)$, $\overline a(n) := \langle
 a(0), \ldots, a(n-1) \rangle$. \\If confusion seems unlikely, we sometimes write:
 ``$\overline a n$'' and not: ``$\overline a (n)$''.

 $a \sqsubseteq b\leftrightarrow \exists n  \le \mathit{length}(b)[a = \overline b n]$, and: 
 $a \sqsubset b\leftrightarrow (a\sqsubseteq b\;\wedge\;a \neq b)$.
  
  $s\le_{lex}t \leftrightarrow \forall m <\mathit{length}(s)[\overline s m =\overline t m \rightarrow s(m)\le t(m)]$.  
 
    $a \perp b\leftrightarrow \neg(a\sqsubseteq b\;\vee\;b \sqsubseteq a)$.

   $\overline{\alpha}(n) :=\overline\alpha n:= \langle
 \alpha(0), \ldots \alpha(n-1) \rangle$. 
 
  $a \sqsubset \alpha\leftrightarrow \exists n[\overline \alpha n = a]$.

\smallskip $a\perp \alpha\leftrightarrow \alpha\perp a\leftrightarrow \neg(a\sqsubset\alpha)$.

\smallskip $\alpha\;\#\;\beta \leftrightarrow \alpha\perp\beta\leftrightarrow\exists n[\alpha(n)\neq \beta(n)]$.

\smallskip For each $p$, $\underline p$ is the element of $\omega^\omega$ satisfying $\forall n[\underline p (n) =p]$.

 \medskip $\forall \alpha\forall n[ \alpha'(n)=\bigl(\alpha(n)\bigr)'\;\wedge\; \alpha''(n)=\bigl(\alpha(n)\bigr)'']$.
 
\medskip
  We extend the language of $\mathsf{BIM}$ by introducing $\in$ and terms denoting subsets of $\omega$. This is not a real extension of the language of $\mathsf{BIM}$. Formulas in which the new symbols occur are abbreviations of  formulas in which they do not occur. 
  
  Given a formula $\varphi=\varphi(n)$ we may introduce a {\it `set term'} $T_\varphi$. The basic formula `$t \in X_\varphi$' is an abbreviation of `$\varphi(t)$'. 
  
   Here is a first example.

$\mathit{Bin}:=\{a\mid  \forall n < \mathit{length}(a)[a(n) = 0 \; \vee\; a(n) = 1]\}$.  

 `$a\in Bin$'  is an abbreviation of `$\forall n < \mathit{length}(a)[a(n) = 0 \; \vee\; a(n) = 1]$'.

\smallskip
$ Bin_m:=\{s\in Bin\mid length(s)=m)\}$. 

\smallskip $\omega^m:=\{s\mid length(s)=m\}$.

\smallskip Given terms $A,B$ denoting subsets of $\omega$,   \\`$A\subseteq B$' is an abbreviation of `$\forall n[n\in A\rightarrow n \in B]$'.

\smallskip
We also introduce terms denoting subsets of $\omega^\omega$, for instance:

   $\mathcal{C}:=\{\gamma\mid\forall n [
 \gamma(n) = 0 \vee \gamma(n) = 1]\}$.

 `$\alpha\in \mathcal{C}$' is ean abbreviation of `$\forall n <length(a)[\alpha(n) = 0 \; \vee\; \alpha(n) = 1]$'.

\smallskip $[\omega]^\omega:=\{\zeta \mid \forall n[\zeta(n)<\zeta(n+1)]\}$.

\smallskip Given terms  $\mathcal{X}, \mathcal{Y}$ denoting subsets of $\omega^\omega$, \\`$\mathcal{X}\subseteq \mathcal{Y}$' abbreviates `$ \forall \alpha[\alpha\in \mathcal{X}\rightarrow\alpha\in\mathcal{Y}]$'.

\smallskip Given a term $\mathcal{X}$ denoting a subset of $\omega^\omega$, we introduce, for all $s$, \\$\mathcal{X}\cap s :=\{\alpha \in \mathcal{X}\mid s\sqsubset \alpha\}$.

\smallskip
Given a term   $\mathcal{X}$ denoting a subset of $\omega^\omega$, and a term $B$ denoting a subset of $\omega$, we let   `$\mathit{Bar}_\mathcal{X}(B)$' be an abbreviation of `$\forall\alpha\in\mathcal{X}\exists n[\overline \alpha n \in B]$'. 

Note that $\mathit{Bar}_\mathcal{X}(B)$ is a  \textit{formula scheme}, that becomes a formula if one substitutes  formulas defining $\mathcal{X}$, $B$, respectively.

\smallskip From now on, we will express ourselves more informally, as in the following example:

For each  $\mathcal{X}\subseteq\omega^\omega$, for each $B\subseteq\omega$,  \\$\mathit{Thinbar}_\mathcal{X}(B)\leftrightarrow\bigl(Bar_\mathcal{X}(B)\;\wedge\; \forall s\in B\forall t\in B[s\sqsubseteq t\rightarrow s=t]\bigr)$.  

Note that we are not extending the language of $\mathsf{BIM}$ by second-order variables.

\subsection{Decidable and enumerable subsets of $\omega$}\label{SS:decidenum}\hfill

  $D_\alpha:=\{i\mid\alpha(i)\neq 0\}$.  $D_\alpha$ is \textit{the subset of $\omega$ decided by $\alpha$}. 

The expression `$i\in D_\alpha$' is an abbreviation of `$\alpha(i)\neq 0$'.

$X\subseteq \omega$ is \textit{decidable} or $\mathbf{\Delta}^0_1$ if and only if $\exists \alpha[X=D_\alpha]$.

  \smallskip $D_a:=\{i\mid i < \mathit{length}(a)\mid a(i) \neq 0\}$.

\smallskip $X\subseteq \omega$ is {\it finite} if and only if $\exists a[X=D_a]$.

\smallskip Note:  for each $\alpha$, $D_\alpha = \bigcup\limits_{n \in \omega} D_{\overline \alpha n}$.

\smallskip
  $E_\alpha := \{n\mid \exists p [\alpha(p) = n+ 1] \}$. $E_\alpha$ is  \textit{the subset of $\omega$ enumerated by $\alpha$}.

$X\subseteq \omega$ is \textit{enumerable} or $\mathbf{\Sigma}^0_1$ if and only if $\exists \alpha[X=E_\alpha]$.

  \smallskip  $E_a := \{n\mid \exists p < \mathit{length}(a)[a(p) = n+ 1] \}$.  
 
 Note:  for each $\alpha$, $E_\alpha = \bigcup\limits_{n \in \omega} E_{\overline \alpha n}$.
 
 $X\subseteq \omega$ is \textit{co-enumerable} or $\mathbf{\Pi}^0_1$ if and only if \\$\exists \alpha[X=\omega\setminus E_\alpha=\{n\mid \forall p[\alpha(p)\neq n+1]\}]$.
 
 Given any $\alpha$, define $\beta$ such that \\$\forall n[\alpha(n)=\beta(n) = 0\;\vee\; \bigl(\alpha(n)\neq 0\;\wedge\;\beta(n)=n+1\bigr)]$, and note: $D_\alpha = E_\beta$. \\We thus see: $\mathsf{BIM}\vdash\forall \alpha \exists \beta[D_\alpha = E_\beta]$.
 
\subsection{Open and closed subsets of $\omega^\omega$,  spreads and fans}\label{SS:fans} \hfill

  $\mathcal{G}_\beta:=\{\gamma\mid\exists n[\beta(\overline \gamma n) \neq 0]\}$.
  
  $\mathcal{G}\subseteq \omega^\omega$ is \textit{  open} or $\mathbf{\Sigma}^0_1$ if and only if $\exists \beta[\mathcal{G}=\mathcal{G}_\beta]$. 

  $\mathcal{F}_\beta := \omega^\omega\setminus \mathcal{G}_\beta = \{\gamma\mid\forall n[\beta(\overline \gamma n) =0]\}$.
  
 $\mathcal{F}\subseteq\omega^\omega$  is  \textit{ closed} or $\mathbf{\Pi}^0_1$ if and only if $\exists\beta[\mathcal{F}= \mathcal{F}_\beta]$.
 
  $Spr(\beta)\leftrightarrow \forall s[\beta(s) =0\leftrightarrow \exists n[\beta(s\ast\langle n \rangle)=0]]$.  \\$\mathcal{X}\subseteq\omega^\omega$ is a \textit{spread} if and only if $\exists \beta[Spr(\beta)\;\wedge\;\mathcal{X}=\mathcal{F}_\beta]$.  
  
  \smallskip In intuitionistic mathematics, not  every  closed subset of $\omega^\omega$ is a spread, see Lemma \ref{L:closedspread}.

 \smallskip
 
 $Fan(\beta)\leftrightarrow \bigl(Spr(\beta)\;\wedge\;\forall s\exists n\forall m[\beta(s\ast \langle m \rangle)=0\rightarrow m\le n]\bigr)$ and  \\$Explfan(\beta)\leftrightarrow \bigl(Spr(\beta)\;\wedge\;\exists \gamma\forall s\forall m[\beta(s\ast \langle m \rangle)=0\rightarrow m\le \gamma(s)]\bigr)$. 
 
  $\mathcal{X}\subseteq\omega^\omega$ is a \textit{fan} if and only if $\exists \beta[Fan(\beta)\;\wedge\;\mathcal{X}=\mathcal{F}_\beta]$.  \\$\mathcal{X}\subseteq\omega^\omega$ is an \textit{explicit fan} if and only if $\exists \beta[Explfan(\beta)\;\wedge\;\mathcal{X}=\mathcal{F}_\beta]$.
  
  \begin{lemma}\label{L:explicitfan} One may prove in $\mathsf{BIM}$:\\$\forall \beta[Explfan(\beta)\leftrightarrow \bigl(Spr(\beta)\;\wedge\;\exists \gamma \forall n\forall s\in\omega^n[\beta(s)=0\rightarrow s\le \gamma(n)]\bigr)]$. \end{lemma} \begin{proof} Let $\beta$ be given such that $Explfan(\beta)$. \\Find $\gamma$ such that $\forall s \forall m[\beta(s\ast\langle m \rangle)=0\rightarrow m\le \gamma(s)]$. \\Define $\delta$ such that $\delta(0)=0=\langle\;\rangle$, and, for each $n$, \\$\delta(n+1) = \max(\{s\ast\langle m\rangle\mid \beta\bigl(\{s\ast\langle m \rangle)=0 \;\wedge\; s\le \delta(n)\;\wedge\;m\le \gamma(s)\}\bigr)$.\\One proves by induction: $\forall n\forall s\in \omega^n[\beta(s)=0\rightarrow s\le \delta(n)]$. 
  
  \smallskip Conversely, let $\beta, \gamma$ be given such that $\forall n\forall s\in \omega^n[\beta(s)=0\rightarrow s\le \gamma(n)]$. \\Define $\delta$ such that, for each $n$, for each $s$ in $\omega^n$ such that $\beta(s)=0$, \\$\delta(s)=\max\bigl(\{m\mid \beta(s\ast\langle m\rangle)=0\;\wedge\; s\ast\langle m \rangle \le \gamma(n+1) \}\bigr)$. \\Note: $\forall s \forall m[\beta(s\ast\langle m\rangle)=0\rightarrow m\le \delta(s)]$ and conclude: $Explfan(\beta)$. \end{proof}
 \subsection{Subsequences}\label{SS:subs} \hfill
 
   $\forall n\forall m[\alpha^{\upharpoonright n}(m) := \alpha(\langle n \rangle \ast m)]$.
 
  $\alpha^{\upharpoonright n}$ is called the \textit{ $n$-th subsequence of the infinite sequence $\alpha$}. 
  
  \smallskip $\mathit{length}(s^{\upharpoonright n}):=\mu p\le s[\langle n \rangle \ast p \ge \mathit{length}(s)]$ and\\ 
$ \forall m<length(s^{\upharpoonright n})[s^{\upharpoonright n}(m) = s(\langle n \rangle \ast m)].$ 

\smallskip Note: $\forall \alpha\forall m\forall n[(\overline\alpha m)^{\upharpoonright n}\sqsubset\alpha^{\upharpoonright n}]$.

\medskip
  $\forall m[^a
  \alpha(m) := \alpha(a \ast m)]$.

Note: $\forall n[^{\langle n\rangle}\alpha = \alpha^{\upharpoonright n}]$.

\smallskip
 
   \smallskip $\mathit{length}(^as):=\mu p\le s[a\ast p \ge \mathit{length}(s)]$ and 
$ \forall m<length(^as)[^as(m) = s(a \ast m)].$

\smallskip Note: $\forall \alpha\forall m\forall a[^a(\overline\alpha m)\sqsubset\;^a\alpha]$.
\subsection{Partial continuous functions from $\omega^\omega$ to $\omega$ and from $\omega^\omega$ to $\omega^\omega$}\label{SS:contfun}\footnote{See \cite[Subsections 7.2 and 7.5]{veldman2011b}}\hfill

$Fun_0(\varphi)\leftrightarrow 
 \forall a \in E_\varphi\forall b \in E_\varphi[a'\sqsubseteq b'\rightarrow a'' = b'']$.
 
 $Dom_0(\varphi):=\{\alpha\mid   \forall \alpha \exists a \in E_\varphi[a'\sqsubset \alpha]\}$.

 \smallskip Assume: $Fun_0(\varphi)$ and   $\alpha\in Dom_0(\varphi)$.
 
 Then $\varphi(\alpha) :=\; the\;number\;  c\; such \;that\; \exists n(\overline \alpha n, c)\in E_\varphi]$. 
 
 \smallskip For every $\mathcal{X}\subseteq \omega^\omega$, for every $\varphi$, $\varphi:\mathcal{X}\rightarrow \omega\leftrightarrow \bigl(Fun_0(\varphi)\;\wedge\; \mathcal{X}\subseteq Dom_0(\varphi)\bigr)$.

 \medskip
 $Fun_1(\varphi)\leftrightarrow\forall a \in E_\varphi\forall b \in E_\varphi[a'\sqsubseteq b'\rightarrow a'' \sqsubseteq b'']$.

 $Dom_1(\varphi):=\{\alpha\mid \forall n \exists a \in E_\varphi[a'\sqsubset \alpha\;\wedge\; length(a'')\ge n]\}$.
 
 $\varphi|a:=\max(\{t\mid\exists b \in E_{\overline\varphi a}[b'\sqsubseteq a \;\wedge b''=t]\})$.
 
 $\varphi:\alpha\mapsto\gamma \leftrightarrow 
 \forall n\exists m[\overline \gamma n \sqsubseteq \varphi|\overline \alpha m]]$.

 \smallskip Assume: $Fun_1(\varphi)$ and   $\alpha\in Dom_1(\varphi)$.
 
 Then $\varphi|\alpha := \; the \; element \;\gamma\;of\; \omega^\omega$ such that $\varphi:\alpha\mapsto \gamma$.

\smallskip For every $\mathcal{X}\subseteq \omega^\omega$, for every $\varphi$, $\varphi:\mathcal{X}\rightarrow \omega^\omega\leftrightarrow \bigl(Fun_1(\varphi)\;\wedge\; \mathcal{X}\subseteq Dom_1(\varphi)\bigr)$.

\subsection{Integers and rationals}\hfill

 $m=_\mathbb{Z}n\leftrightarrow m'+n''=m''+n'$.

 $m<_\mathbb{Z}n\leftrightarrow m'+n''<m''+n'$.

 $0_\mathbb{Z}:=(0,0)$
 
 $m +_\mathbb{Z} n= (m'+n',m''+n'')$.

 $m -_\mathbb{Z} n= (m'+n'',m''+n')$.

 $m\cdot_\mathbb{Z} n:=(m'\cdot n'+m''\cdot n'', m'\cdot n''+m''\cdot n')$. 

\smallskip
 $\mathbb{Q}:=\{n\mid n''>_\mathbb{Z} 0_\mathbb{Z}\}$.

 $m=_\mathbb{Q}n\leftrightarrow m'\cdot_\mathbb{Z}n''=_\mathbb{Z} m''\cdot_\mathbb{Z} n'$.

 $m<_\mathbb{Q}n\leftrightarrow m'\cdot_\mathbb{Z}n''<_\mathbb{Z} m''\cdot_\mathbb{Z} n'$.
 
 $m\le_\mathbb{Q} n\leftrightarrow  m'\cdot_\mathbb{Z}n''\le_\mathbb{Z} m''\cdot_\mathbb{Z} n'$.
 
 $m\le_\mathbb{Q}n \leftrightarrow \max_\mathbb{Q}(m,n) =_\mathbb{Q}n \leftrightarrow  \min_\mathbb{Q}(m,n) =_\mathbb{Q}m$. 

 $m +_\mathbb{Q} n= (m'\cdot_\mathbb{Z}n''+_\mathbb{Z}m''\cdot_\mathbb{Z}n',m''\cdot_\mathbb{Z}n'')$.
 
  $m -_\mathbb{Q} n= (m'\cdot_\mathbb{Z}n''-_\mathbb{Z}m''\cdot_\mathbb{Z} n',m''\cdot_\mathbb{Z}n'')$.
 
  $m\cdot_\mathbb{Q} n=(m'\cdot_\mathbb{Z} n', m''\cdot_\mathbb{Z}n'')$.
 
 \smallskip
 $\mathbb{S}:=\{s\mid s'\in \mathbb{Q} \;\wedge\; s''\in \mathbb{Q}\;\wedge\;s'<_\mathbb{Q}s''\}$. 
 
  $s \sqsubset_\mathbb{S} t\leftrightarrow  (t'<_\mathbb{Q} s' \;\wedge\; s'' <_\mathbb{Q} t'')$. 
 
  $s \sqsubseteq_\mathbb{S} t\leftrightarrow  (t'\le_\mathbb{Q} s' \;\wedge\; s'' \le_\mathbb{Q} t'')$. 
 
  $s<_\mathbb{S} t\leftrightarrow s''<_\mathbb{Q} t'$.
  
   $s\le_\mathbb{S} t\leftrightarrow s'\le_\mathbb{Q} t''$.

 $s\;\#_\mathbb{S} \;t \leftrightarrow (s <_\mathbb{S} t \;\vee\; t<_\mathbb{S} s)$.
 
 \smallskip For each $n$, we define $n_\mathbb{Q}$ in $\mathbb{Q}$ by: $n_\mathbb{Q}=\bigl((n,0),(1,0)\bigr)$.

\smallskip For all $s$ in $\mathbb{S}$, $\mathit{double}_\mathbb{S}(s)$ is the element $u$ of $\mathbb{S}$ satisfying:
 
 $u'+_\mathbb{Q}u''=_\mathbb{Q} s'+_\mathbb{Q} s''$ and $u''-_\mathbb{Q} u' =_\mathbb{Q} 2_\mathbb{Q}\cdot_\mathbb{Q} (s''-_\mathbb{Q} s')$.
 
 Note: $\forall s \in \mathbb{S}\forall t \in \mathbb{S}[s\sqsubseteq_\mathbb{S} t\rightarrow \mathit{double}_\mathbb{S}(s)  \sqsubseteq_\mathbb{S} \mathit{double}_\mathbb{S}(t)]$.

 \smallskip  $s+_\mathbb{S} t:= (s'+_\mathbb{Q} t', s''+_\mathbb{Q} t'')$.
 
  $s\cdot_\mathbb{S} t:= \\\bigl(\min_\mathbb{Q}(s'\cdot_\mathbb{Q}t', s''\cdot_\mathbb{Q}t' , s'\cdot_\mathbb{Q}t'', s''\cdot_\mathbb{Q}t''), \max_\mathbb{Q}(s'\cdot_\mathbb{Q}t', s''\cdot_\mathbb{Q}t' , s'\cdot_\mathbb{Q}t'', s''\cdot_\mathbb{Q}t'')\bigr)$.

 \subsection{Real numbers}\label{SSS:reals}\hfill

 $\mathcal{R}:=\{ \alpha\mid\forall n [\alpha(n)\in \mathbb{S}\;\wedge\;\alpha(n+1) \sqsubseteq_\mathbb{S} \alpha(n)]\;\wedge\;\forall m \exists n[\alpha''(n) -_\mathbb{Q} \alpha'(n) <_\mathbb{Q} \frac{1}{2^m}]\}$.\footnote{This is not the same definition as in \cite[Subsubsection 8.1.4]{veldman2011b}. We replaced `$\sqsubset_\mathbb{S}$' by `$\sqsubseteq_\mathbb{S}$'
 .} 
 
    $\alpha <_\mathcal{R} \beta \leftrightarrow \exists n[\alpha(n) <_\mathbb{S} \beta(n)]$.

    $\alpha \le_\mathcal{R} \beta \leftrightarrow \forall  n[\alpha(n) \le_\mathbb{S} \beta(n)]$.
    
  \smallskip
 $\forall n[\inf(\alpha,\beta)(n):=\bigl(\min_\mathbb{Q}(\alpha'(n),\beta'(n)),\min_\mathbb{Q}(\alpha''(n),\beta''(n))\bigr)]$.
 
 $\forall n[\sup(\alpha,\beta)(n):=\bigl(\max_\mathbb{Q}(\alpha'(n),\beta'(n)),\max_\mathbb{Q}(\alpha''(n),\beta''(n))\bigr)]$.
 
 \smallskip $\alpha \;\#_\mathcal{R}\; \beta \leftrightarrow (\alpha <_\mathcal{R} \beta \;\vee\; \beta <_\mathcal{R} \alpha)$.

   $\alpha =_\mathcal{R} \beta\leftrightarrow (\alpha \le_\mathcal{R} \beta \;\wedge\;\beta \le_\mathcal{R} \alpha)$.
   
   \smallskip $\forall n[(\alpha+_\mathcal{R}\beta)(n):=\alpha(n)+_\mathbb{S}\beta(n)]$. 
   
    $\forall n[(\alpha\cdot_\mathcal{R}\beta)(n):=\alpha(n)\cdot_\mathbb{S}\beta(n)]$.
     
For each $q$ in $\mathbb{Q}$, we define $q_\mathcal{R}$ in $\mathcal{R}$ by: for each $n$, $q_\mathcal{R}(n) = (q -_\mathbb{Q} \frac{1}{2^n}, q+_\mathbb{Q} \frac{1}{2^n})$.

\smallskip $0_\mathcal{R}:=(0_\mathbb{Q})_\mathcal{R}$ and $1_\mathcal{R}:=(1_\mathbb{Q})_\mathcal{R}$.

\begin{lemma}\label{L:apart} One may prove in $\mathsf{BIM}$: \\$\forall s \in \mathbb{S} \forall t \in \mathbb{S}[s\sqsubset_\mathbb{S} t \rightarrow \forall \alpha \in \mathcal{R} \exists n[s\;\#_\mathbb{S} \;\alpha(n)\;\vee\;\alpha(n) \sqsubset_\mathbb{S} t]]$. \end{lemma}
\begin{proof} The proof is left to the reader. \end{proof}
\subsection{$[0,1]$ and $\mathcal{C}$}\label{SSS:01}\hfill

 $[0,1]:=\{\alpha \in \mathcal{R}\mid 0_\mathcal{R} \le_\mathcal{R} \alpha \le_\mathcal{R} 1_\mathcal{R}\}$.

\smallskip $(0,1]:=\{\alpha \in \mathcal{R}\mid 0_\mathcal{R} <_\mathcal{R} \alpha \le_\mathcal{R} 1_\mathcal{R}\}$, and  $[0,1):=\{\alpha \in \mathcal{R}\mid 0_\mathcal{R} \le_\mathcal{R} \alpha <_\mathcal{R} 1_\mathcal{R}\}$.

\smallskip For all $\alpha, \beta $ in $\mathcal{R}$, $[\alpha,\beta):=\{\gamma \in \mathcal{R}\mid\alpha\le_\mathcal{R}\gamma <_\mathcal{R}\beta\}$.

\medskip $[0,1]^2:=\{\gamma\mid \forall i<2[\gamma^{\upharpoonright i}\in [0,1]]\}$ and $[0,1]^\omega:=\{\gamma\mid \forall n[\gamma^{\upharpoonright n}\in [0,1]]\}$. 

 \smallskip
  $\mathcal{H}_\alpha:=\{\gamma \in [0,1]\mid\exists n\exists s \in \mathbb{S}[\alpha(s) \neq 0 \;\wedge\; \gamma(n)\sqsubset_\mathbb{S} s ]\}$.

 \smallskip
 $\mathcal{H}\subseteq \mathcal{R}$ is   \textit{open} if and only if $\exists\alpha[\mathcal{H}=\mathcal{H}_\alpha]$.
  
 \begin{lemma}\label{L:Conto[0,1]} One may prove in $\mathsf{BIM}$:\\There exist $\sigma, \psi$ such that \begin{enumerate}[\upshape(i)]\item $\sigma:\mathcal{C}\rightarrow [0,1]$ and  $\forall \delta\in[0,1]\exists\gamma\in \mathcal{C}[\sigma|\gamma=_\mathcal{R}\delta]$. 
 \item $\psi:\omega^\omega\rightarrow\omega^\omega$ and $\forall\alpha\forall \gamma \in \mathcal{C}[\gamma \in \mathcal{G}_{\psi|\alpha}\leftrightarrow \sigma|\gamma \in \mathcal{H}_\alpha]$. \end{enumerate} \end{lemma}
 
 \begin{proof} (i) Define $\lambda$ and $\rho$ such that, for each $s$ in $\mathbb{S}$, $\lambda(s)=( s', \frac{1}{3}s'+_\mathbb{Q}\frac{2}{3}s'') $ and  $\rho(s)=( \frac{2}{3}s'+_\mathbb{Q}\frac{1}{3}s'', s'')$. 
 
 For each $s$ in $\mathbb{S}$, $\lambda(s)$ is the \textit{left-two-third part of} 
  $s$ and  $\rho(s)$ is the \textit{right-two-third part of} 
  $s$.
  
  Define $\nu$ such that $\nu(\langle\;\rangle)=(0_\mathbb{Q}, 1_\mathbb{Q})$ and, for all $s$ in  $Bin$, 
   $\nu(s\ast\langle 0\rangle)=\lambda\bigl(\nu(s)\bigr)$ and  $\nu(s\ast\langle 1\rangle)=\rho\bigl(\nu(s)\bigr)$.

   Define $\sigma:\mathcal{C}\rightarrow [0,1]$ such that $\forall \gamma \in \mathcal{C}[(\sigma|\gamma)(n)=\nu(\overline \gamma n)]$. \\One may prove:   $\forall \delta\in[0,1]\exists\gamma\in \mathcal{C}[\sigma|\gamma=_\mathcal{R}\delta]$.  
  
   \smallskip (ii)
    Define
$\psi:\omega^\omega\rightarrow\omega^\omega$ such that \\$\forall \alpha\forall s[(\psi|\alpha)(s)\neq 0 \leftrightarrow \bigl(s\in Bin\;\wedge\;\exists t<s[\nu(s)\sqsubset_\mathbb{S}t \;\wedge\;\alpha(t)\neq 0]\bigr)]$.  
\\One may prove:   $\forall\alpha\forall \gamma \in \mathcal{C}[\gamma \in \mathcal{G}_{\psi|\alpha}\leftrightarrow \sigma|\gamma \in \mathcal{H}_\alpha]$.\end{proof}
 \begin{lemma}\label{L:Cinto[0,1]} One may prove in $\mathsf{BIM}$:\\There exist $\tau, \chi$ such that \begin{enumerate}[\upshape(i)]\item $\tau:\mathcal{C}\rightarrow [0,1]$ and  $\forall \gamma \in \mathcal{C}\forall\delta\in\mathcal{C}[\gamma\;\#\;\delta\rightarrow \tau|\gamma\;\#_\mathcal{R}\;\tau|\delta]$. 
 \item $\chi:\omega^\omega\rightarrow\omega^\omega$ and $\forall \alpha \forall \gamma \in \mathcal{C}[\gamma \in \mathcal{G}_\alpha\leftrightarrow \tau|\gamma \in \mathcal{H}_{\chi|\alpha}]$. \item $\forall \delta \in [0,1]\exists \gamma\in \mathcal{C}[\tau|\gamma\;\#_\mathcal{R}\;\delta\rightarrow \forall \alpha[\delta \in \mathcal{H}_{\chi|\alpha}]]$. 
 \item $\forall \delta \in [0,1]^\omega\exists\gamma\in \mathcal{C}\forall n[\tau|\gamma^n\;\#_\mathcal{R}\;\delta^n\rightarrow \forall \alpha[\delta^n \in \mathcal{H}_{\chi|\alpha}]]$. \end{enumerate} \end{lemma}
  \begin{proof} (i)  Define $\pi_0, \pi_1, \pi_2, \pi_3$ and $\pi_4$ such that, for each $s$ in $\mathbb{S}$,  for each $i<5$, $ \pi_i(s):=( \frac{5-i}{5}s' +_\mathbb{Q} \frac{i}{5} s'',\frac{5-i-1}{5}s' +_\mathbb{Q} \frac{i+1}{5} s'') $.

  For each $s$ in $\mathbb{S}$, for each $i<5$, $\pi_i(s)$ is the \textit{$i$-th fifth part of $s$}.
  
  Define $\varepsilon$ such that $\varepsilon(\langle\;\rangle)=(0_\mathbb{Q}, 1_\mathbb{Q})$ and, for all $a$ in  $Bin$,\\ 
   $\varepsilon(a\ast\langle 0\rangle)=\pi_1\bigl(\varepsilon(a)\bigr)$ and  $\varepsilon(s\ast\langle 1\rangle)=\pi_3\bigl(\varepsilon(a)\bigr)$.
   
  Define $\tau:\mathcal{C}\rightarrow [0,1]$ such that $\forall \gamma \in \mathcal{C}[(\tau|\gamma)(n)=\varepsilon(\overline \gamma n)]$. 
   \\One may prove:  $\forall \gamma\in
   \mathcal{C}\forall \delta\in\mathcal{C}[\gamma\;\#\;\delta\rightarrow \tau|\gamma\;\#_\mathcal{R}\;\tau|\delta]$.

\smallskip (ii) Define
$\chi:\omega^\omega\rightarrow\omega^\omega$ such that, for all $\alpha$, for all $s$,  $(\chi|\alpha)(s)\neq 0$ if and only if $\exists t<s[ t\in Bin\;\wedge\; s\sqsubset_\mathbb{S} \varepsilon(t)\;\wedge\;\alpha(t)\neq 0]\;\vee\;\exists n<s\forall t\in Bin_n[s\;\#_\mathbb{S}\;\varepsilon(t)]$. 
\\We now prove:  $\forall\alpha\forall \gamma \in \mathcal{C}[\gamma \in \mathcal{G}_{\alpha}\leftrightarrow \tau|\gamma \in \mathcal{H}_{\chi|\alpha}]$

Let $\alpha$ be given and assume: $\gamma \in \mathcal{C}$. 
 
Assume: $\gamma\in \mathcal{G}_\alpha$. Find $n$ such that  $\alpha(\overline\gamma n)\neq 0$. \\Find $k>n$ such that  $\varepsilon(\overline\gamma k) >  \overline \gamma n$. \\Note  $\varepsilon(\overline\gamma k) \sqsubset_\mathbb{S} \varepsilon(\overline\gamma n)$  and conclude: $(\chi|\alpha)\bigl(\varepsilon(\overline \gamma k)\bigr)\neq 0$.  \\As $(\tau|\gamma)(k+1) \sqsubset_\mathbb{S} (\tau|\gamma)(k)=\varepsilon(\overline \gamma k)$, conclude: $\tau|\gamma \in \mathcal{H}_{\chi|\alpha}$.

Conversely, assume $\tau|\gamma \in \mathcal{H}_{\chi|\alpha}$. Note: $\forall n[\varepsilon(\overline \gamma n) \sqsubset \tau|\gamma]$. 
\\Conclude: $\neg \exists s[s\sqsubset \tau|\gamma \;\wedge\;\exists n\forall t\in Bin_n[s\;\#_\mathbb{S}\;\varepsilon(t)]]$. 
\\Find $n, s$ such that $(\tau|\gamma)(n)\sqsubset_\mathbb{S} s$ and  $(\chi|\alpha)(s)\neq 0$. \\Find $t$ in $Bin$ such that $s\sqsubset \varepsilon(t)$ and $\alpha(t)\neq 0$. \\Note: $t\sqsubset \gamma$ and conclude: $\gamma \in \mathcal{G}_\alpha$.

We thus see: $\forall \gamma \in \mathcal{C}[\gamma \in \mathcal{G}_\alpha\leftrightarrow \tau|\gamma  \in \mathcal{H}_{\chi|\alpha}]$.

\smallskip 
(iii)  Assume $\delta\in[0,1]$.

Note: $\forall a \in \mathit{Bin}[\varepsilon''(a\ast\langle 0\rangle)<_\mathbb{Q}\varepsilon'(a\ast\langle 1\rangle)]$.

Define $\gamma$ in $\mathcal{C}$ such that,   for all $m, p$, \\if $p=\mu q[\varepsilon''(\overline{\gamma}  m\ast\langle 0\rangle) <_\mathbb{Q} \delta'(q)\;\vee\; \delta''(q) <_\mathbb{Q} \varepsilon'(\overline{\gamma} m\ast\langle 1 \rangle)]$,\\ then  $\gamma(m)=0\leftrightarrow\delta''(p) <_\mathbb{Q} \varepsilon'(\overline{\gamma}m\ast\langle 1\rangle)$. 

One may prove, by induction on $n$:\\for each $n$, there exists $p$ such that  $\forall t\in Bin_n[t\perp\overline \gamma n\rightarrow \delta(p)\;\#_\mathbb{S}\;\varepsilon(t)]$.  

Assume: $\delta \;\#_\mathcal{R}\;\tau|\gamma$. Find $n$ such that $\delta(n)\;\#_\mathbb{S} \;(\tau|\gamma)(n)=\varepsilon(\overline \gamma(n))$. 

Find $p>n$ such that $\forall t\in Bin_n[t\perp\overline \gamma n\rightarrow \delta(p) \;\#_\mathbb{S} \;\varepsilon(t)]$.

Conclude: $\forall t \in Bin_n[\delta(p)\;\#_\mathbb{S}\; \varepsilon(t)]$ and, for each $\alpha$, $\delta \in \mathcal{H}_{\chi|\alpha}$. 

We thus see: if $\delta\;\#_\mathcal{R}\;\tau|\gamma$, then $\forall \alpha[\delta \in \mathcal{H}_{\chi|\alpha}]$.

\smallskip (iv)  Assume: $\delta \in [0,1]^\omega$.

Define $\gamma$ in $\mathcal{C}$ such that,  for all $n,m, p$, \\if $p=\mu q[\varepsilon''(\overline{\gamma^n}  m\ast\langle 0\rangle) <_\mathbb{Q} (\delta^n)'(q)\;\vee\; (\delta^n)''(q) <_\mathbb{Q} \varepsilon'(\overline{\gamma^n} m\ast\langle 1 \rangle)]$,\\ then  $\gamma^n(m)=0\leftrightarrow\delta''(p) <_\mathbb{Q} \varepsilon'(\overline{(\gamma^n)}m\ast\langle 1\rangle)$.  

Conclude, following the argument for (iii): 
$\forall n[\delta^n \;\#_\mathcal{R}\;\tau|(\gamma^n)\rightarrow \delta^n\in \mathcal{H}_{\alpha^n}]$.
\end{proof}

\subsection{Real functions from $[0,1]$ to $\mathcal{R}$}\label{SS:realf}\footnote{The definition slightly deviates from the one used in \cite[Subsection 8.4]{veldman2011b}.}\hfill

  $\varphi:[0,1]\rightarrow_\mathcal{R} \mathcal{R}$ if and only if
 \begin{enumerate}
 \item $\forall n\in E_\varphi[n'\in \mathbb{S}\;\wedge\;n'' \in \mathbb{S}]$, and  
 \item  $\forall m\in E_\varphi\forall n\in E_\varphi[m'\sqsubseteq_\mathbb{S} n'\rightarrow m'' \sqsubseteq_\mathbb{S} n'']$, and
 \item $\forall \alpha\in [0,1]\forall n\exists m\exists s[(\alpha(m), s)\in E_\varphi\;\wedge\; s''-_\mathbb{Q}s'<_\mathbb{Q}\frac{1}{2^n}]$. 
 \end{enumerate}
 
 \smallskip $\mathcal{R}^{[0,1]}:=\{\varphi\mid \varphi:[0,1] \rightarrow_\mathcal{R} \mathcal{R}\}$.
 
 \smallskip Assume: $\varphi:[0,1]\rightarrow_\mathcal{R}\mathcal{R}$. 
 
 We define, for each $\alpha$ in $[0,1]$, for each $\beta$ in $\mathcal{R}$,  \\$\varphi:\alpha\mapsto_\mathcal{R}\beta$ if and only if $\forall n\exists m\exists p\in E_\varphi[|\alpha(m)\sqsubseteq_\mathbb{S} p'\;\wedge\;p''\sqsubseteq_\mathbb{S}\beta(n)]$. 
 
   For each $\alpha$ in $[0,1]$,  we  let $\varphi^{`\mathcal{R}}(\alpha)$ be the element $\beta$ of $\mathcal{R}$ such that, for each $n$, $\beta(n)=double_\mathbb{S}(s'')$, where $s$ is the least $t$ such that  $t\in\mathbb{S}$ and $ t''-_\mathbb{Q}t'=_\mathbb{Q} \frac{1}{2^n}$ and  \\$\exists p\le t \exists r\le t\exists m\le t[\varphi(r)= p+1 \;\wedge\; \alpha(m)\sqsubseteq_\mathbb{S} p' \;\wedge\; p''\sqsubseteq_\mathbb{S} t]$. 
   \\Note:  $\varphi:\alpha\mapsto_\mathcal{R} \varphi^{`\mathcal{R}}(\alpha)$.

 \subsection{Game-theoretic terminology}\label{SSS:games}\hfill
 
   $s:n\rightarrow k\leftrightarrow\bigl(length(s)= n\;\wedge\;\forall j<n[s(j) <k]\bigr)$.
 
 \smallskip  $Seq(n,l):=\{s\mid s:n\rightarrow l\}$.  Note: $Bin_n=Seq(n,2)$. 
 
 \smallskip  
 $c \in_{I}  s \leftrightarrow  \forall i[2i < \mathit{length}(c) \rightarrow \bigl(\overline c(2i)<length(s)\; \wedge\;  c(2i) = s\bigl(\overline c (2i)\bigr)\bigr)]$.
 
 \smallskip $c \in_{II}  t \leftrightarrow  \forall i[2i+1 < \mathit{length}(c) \rightarrow \\\bigl(\overline c (2i+1)<length(t) \;\wedge\; c(2i+1) = t\bigl(\overline c (2i+1)\bigr)\bigr)]$.
 
 \smallskip
 (The numbers $s, t$ should be thought of as \textit{strategies} for player $I$, $II$, respectively.)

 \smallskip  
 $c \in_{I}  \sigma \leftrightarrow  \forall i[2i < \mathit{length}(c) \rightarrow  c(2i) = \sigma\bigl(\overline c (2i)\bigr)]$.
 
  \smallskip  $c \in_{II}  \tau \leftrightarrow  \forall i[2i+1 < \mathit{length}(c) \rightarrow c(2i+1) = \tau\bigl(\overline c (2i+1)\bigr)]$.
  
   \smallskip 
 $\gamma \in_{I}  \sigma \leftrightarrow  \forall i[\gamma(2i) = \sigma\bigl(\overline \gamma (2i)\bigr)]$.
 
  \smallskip  $\gamma \in_{II}  \tau \leftrightarrow  \forall i[\gamma(2i+1) = \tau\bigl(\overline \gamma (2i+1)\bigr)]$.
  
  \smallskip  
 $\gamma \in_{I}  s \leftrightarrow  \forall i[\overline \gamma(2i)<length(s)\rightarrow \gamma(2i) = s\bigl(\overline \gamma (2i)\bigr)]$.
 
  \smallskip  $\gamma \in_{II}  t \leftrightarrow  \forall i[\overline\gamma(2i+1)<length(t)\rightarrow \gamma(2i+1) = t\bigl(\overline \gamma (2i+1)\bigr)]$.
  
 \medskip $\omega \times \omega:=\{s\mid \mathit{length}(s) = 2\}$.
  
  \smallskip $2\times  \omega:=\{s\mid \mathit{length}(s) = 2 \;\wedge\; s(0)<2\}$.
  
  \smallskip $\omega\times 2 :=\{s\mid \mathit{length}(s)=2 \;\wedge\; s(1) <2\}$.
  
 \smallskip For each $n>0$,  $(\omega \times 2)^n:=\{s\mid\mathit{length}(s) = 2n\;\wedge\forall i < n[s(2i+1) < 2]\}$ and\\ $(\omega \times 2)^n \times \omega:=\{s\mid\mathit{length}(s) = 2n+1\;\wedge\;\forall i < n[s(2i+1) <2 ]\}$.
 
 \smallskip $(\omega\times 2)^{<\omega}:=\bigcup_n(\omega\times 2)^n$.

 \smallskip
  $(\omega \times 2)^\omega:=\{\delta\mid \forall n[\delta(2n+1) < 2]\}$.
  
 \smallskip  $Halfbin:= (\omega\times 2)^{<\omega}\cup \bigl((\omega\times 2)^{<\omega}\times \omega\bigr) =\bigcup_n\{\overline \gamma n\mid \gamma \in (\omega\times 2)^\omega\}$.

  \smallskip
  For every $\delta$,  $\delta_I$,  $\delta_{II}$ satisfy:  $\forall n[\delta_I(n) = \delta(2n)]$ and $\forall n[ \delta_{II}(n) = \delta(2n+1)]$.
 
 \smallskip For every $s$, $\mathit{length}(s_I)=\mu n[\mathit{length}(s)\le 2n]$ and $\forall n < \mathit{length}(s_I)[s_I(n) = s(2n)]$, and $\mathit{length}(s_{II})=\mu n[\mathit{length}(s)\le 2n+1]$ and $\forall n < \mathit{length}(s_{II})[s_{II}(n) = s(2n+1)]$.

 \end{document}